\documentclass[11pt]{amsart}  

\usepackage{amsmath}
\usepackage{amsthm}
\usepackage{amsfonts}
\usepackage{amssymb}
\usepackage{eucal}
\usepackage{yfonts}
\usepackage[all]{xy}
\usepackage{amsxtra}
\usepackage{appendix} 
\usepackage{tensor} 
\usepackage{url}

\usepackage[pdftex, breaklinks, linktocpage=true, bookmarksopen=true,bookmarksopenlevel=0,bookmarksnumbered=true]{hyperref}

\RequirePackage{color}
\definecolor{bwgreen}{rgb}{0.183,1,0.5}
\definecolor{bwmagenta}{rgb}{0.7,0.0,0.1}
\definecolor{bwblue}{rgb}{0.317,0.161,1}

\hypersetup{
pdfauthor = {\textcopyright\ Bryden Cais},
pdfcreator = {\LaTeX\ with package \flqq hyperref\frqq},
colorlinks = {true},
linkcolor={bwmagenta},	 %color red Color for normal internal links.
anchorcolor={},	 %color	black Color for anchor text.
citecolor={bwblue},	% color	green Color for bibliographical citations in text.
filecolor={},	% color cyan Color for URLs which open local files.
menucolor={},	% color	red	Color for Acrobat menu items.
runcolor={},	% color	filecolor Color for run links (launch annotations).
urlcolor={bwgreen}, %	 color	magenta	Color for linked URLs.
}

\CompileMatrices
\entrymodifiers={+!!<0pt,\fontdimen22\textfont2>}
\DeclareFontFamily{OT1}{rsfs}{}
\DeclareFontShape{OT1}{rsfs}{n}{it}{<-> rsfs10}{}
\DeclareMathAlphabet{\mathscr}{OT1}{rsfs}{n}{it}

\DeclareFontFamily{OT1}{pzc}{}
\DeclareFontShape{OT1}{pzc}{n}{it}{<->s*[2.2]pzc}{}
\DeclareMathAlphabet{\mathpzc}{OT1}{pzc}{b}{sl}

\makeatletter
\newcommand{\rmnum}[1]{\romannumeral #1}
\newcommand{\Rmnum}[1]{\expandafter\@slowromancap\romannumeral #1@}
\makeatother

\DeclareMathOperator{\id}{id}

\DeclareMathOperator{\Frac}{Frac}
\DeclareMathOperator{\ord}{ord} 
\DeclareMathOperator{\nil}{nil}

\newcommand*{\pr}{\rho}
\newcommand*{\ps}{\sigma}

\DeclareMathOperator{\Hom}{Hom}
\DeclareMathOperator{\End}{End}
\DeclareMathOperator{\Ext}{Ext}
\DeclareMathOperator{\Gal}{Gal}
\DeclareMathOperator{\GL}{GL}
\DeclareMathOperator{\SL}{SL}

\DeclareMathOperator{\Aut}{Aut}
\DeclareMathOperator{\Spec}{Spec}

\DeclareMathOperator{\et}{\acute{e}t}

\DeclareMathOperator{\dR}{dR}
\DeclareMathOperator{\Cris}{Cris}
\DeclareMathOperator{\cris}{cris}

\DeclareMathOperator{\tr}{tr}

\DeclareMathOperator{\Pic}{Pic}
\DeclareMathOperator{\Alb}{Alb}
\DeclareMathOperator{\sm}{sm}
\DeclareMathOperator{\Extrig}{Extrig}
\DeclareMathOperator{\Lie}{Lie}
\DeclareMathOperator{\Inf}{Inf}
\DeclareMathOperator{\reg}{reg}

\DeclareMathOperator{\red}{red}

\DeclareMathOperator{\Fil}{Fil}

\DeclareMathOperator{\Rep}{Rep}

\DeclareMathOperator{\sep}{sep}
\DeclareMathOperator{\perf}{rad}

\DeclareMathOperator{\res}{res}

\DeclareMathOperator{\Ig}{Ig}

\DeclareMathOperator{\im}{im}
\DeclareMathOperator{\rad}{rad}

\DeclareMathOperator{\Irr}{Irr}
\DeclareMathOperator{\bal}{bal.}
\DeclareMathOperator{\Ell}{Ell}

\DeclareMathOperator{\Cot}{Cot}
\DeclareMathOperator{\pdiv}{pdiv}

\DeclareMathOperator{\mult}{m}
\DeclareMathOperator{\loc}{ll}

\DeclareMathOperator{\proj}{proj}
\DeclareMathOperator{\incl}{incl}

\DeclareMathOperator{\Null}{null}
\DeclareMathOperator{\sh}{sh}

\DeclareMathOperator{\Sen}{Sen}
\newcommand*{\R}{\ensuremath{\mathbf{R}}}   
\renewcommand*{\c}{\ensuremath{\mathbf{C}}}              
\newcommand*{\Z}{\ensuremath{\mathbf{Z}}}               
\newcommand*{\Q}{\ensuremath{\mathbf{Q}}}                           
             
\newcommand*{\Qbar}{\overline{\Q}}

\newcommand*{\Kbar}{\overline{K}}    
          
\newcommand*{\Gm}{\ensuremath{{\mathbf{G}_m}}}   
\newcommand*{\Ga}{\ensuremath{{\mathbf{G}_a}}}   

\newcommand*{\m}{\mathfrak{M}}

\newcommand*{\s}{\mathfrak{S}}

\newcommand*{\A}{\ensuremath{\mathcal{A}}}
\newcommand*{\scrA}{\mathscr{A}}
\newcommand*{\B}{\mathcal{B}}

\newcommand*{\C}{\mathbf{C}}
\newcommand*{\E}{\mathscr{E}}     
\newcommand*{\F}{\mathbf{F}}
\newcommand*{\scrF}{\mathscr{F}}

\newcommand*{\G}{\mathcal{G}}
\newcommand*{\scrG}{\mathscr{G}}  
\newcommand*{\scrH}{\mathscr{H}}                           
\newcommand*{\h}{\mathscr{H}}                               
\newcommand*{\I}{\mathscr{I}}                               
\newcommand*{\J}{\mathcal{J}}
\renewcommand*{\L}{\mathscr{L}}
\newcommand*{\scrM}{\mathscr{M}}

\renewcommand*{\O}{\mathscr{O}}                    
 
\newcommand*{\X}{\mathcal{X}}     
\newcommand*{\Y}{\mathcal{Y}}

\newcommand*{\scrP}{\mathscr{P}}  
\newcommand*{\scrQ}{\mathscr{Q}}                             

\newcommand*{\scrHom}{\mathscr{H}\mathit{om}}      
\newcommand*{\scrExtrig}{\mathscr{E}\mathit{xtrig}}	
\newcommand*{\scrExt}{\mathscr{E}\mathit{xt}}               
\newcommand*{\scrLie}{\mathscr{L}\mathit{ie}}

\newcommand*{\D}{\ensuremath{\mathbf{D}}}
\newcommand*{\M}{\ensuremath{\mathbf{M}}}
              
\renewcommand*{\H}{\ensuremath{\mathfrak{H}}}

\newcommand*{\Dual}[1]{{{#1}^t}}
\newcommand*{\VDual}[1]{{{#1}^{\vee}}}

\renewcommand*{\int}{\ensuremath{\mathrm{int}}}

\newcommand*{\e}{\ensuremath{\mathbf{E}}}
\renewcommand*{\a}{\ensuremath{\mathbf{A}}}

\renewcommand*{\SS}{\ensuremath{\underline{\mathrm{ss}}}}
\renewcommand*{\u}[1]{\underline{#1}}
\renewcommand*{\o}[1]{\overline{#1}}
\newcommand*{\wh}[1]{\widehat{#1}}
\newcommand*{\wt}[1]{\widetilde{#1}}
\newcommand*{\nor}[1]{{#1}^{\mathrm{n}}}

\newcommand*{\tens}{\mathop{\otimes}\limits}
\newcommand*{\can}{\text{-}\mathrm{can}}
\newcommand*{\fiber}{\mathop{\times}\limits}

\DeclareMathOperator{\BT}{BT}

\theoremstyle{plain}
  \newtheorem{theorem}{Theorem}
  \newtheorem{proposition}[theorem]{Proposition}
  \newtheorem{lemma}[theorem]{Lemma}
  \newtheorem{corollary}[theorem]{Corollary}

\theoremstyle{definition}
  \newtheorem{definition}[theorem]{Definition}

\theoremstyle{remark}
  \newtheorem{example}[theorem]{Example}
  \newtheorem{remark}[theorem]{Remark}

  \newtheorem{warning}[theorem]{Warning}

\include{header}

\numberwithin{theorem}{subsection}  
\numberwithin{equation}{subsection}

\begin{document}
\title{The Geometry of Hida Families and $\Lambda$-adic Hodge Theory}

\author{Bryden Cais}
\address{University of Arizona, Tucson}
\curraddr{Department of Mathematics, 617 N. Santa Rita Ave., Tucson AZ. 85721}
\email{cais@math.arizona.edu}

\thanks{
	During the writing of this paper, the author was partially supported by an NSA Young Investigator grant
	(H98230-12-1-0238) and an NSF RTG (DMS-0838218).
	}

\dedicatory{To Haruzo Hida, on the occasion of his $60^{\text{th}}$ birthday.}

\date{\today}

\begin{abstract}
	We construct $\Lambda$-adic de Rham and crystalline analogues of Hida's
	ordinary $\Lambda$-adic \'etale cohomology, and by exploiting the geometry 
	of integral models of modular curves over the cyclotomic extension
	of $\Q_p$, we prove appropriate finiteness
	and control theorems in each case.  We then employ integral $p$-adic Hodge theory
	to prove $\Lambda$-adic
	comparison isomorphisms between our cohomologies and Hida's \'etale cohomology. 
	As applications of our work, we provide a ``cohomological" construction of the 
	family of $(\varphi,\Gamma)$-modules attached to Hida's ordinary $\Lambda$-adic \'etale
	cohomology by \cite{Dee}, and we give a new and purely geometric 
	proof of Hida's finitenes and control theorems.  We are also able to prove
	refinements of the main theorems in \cite{MW-Hida} and \cite{OhtaEichler};
	in particular, we prove that there is a canonical isomorphism between
	the module of ordinary $\Lambda$-adic cuspforms and the part
	of the crystalline cohomology of the Igusa tower on which Frobenius
	acts invertibly.
\end{abstract}

\maketitle

\section{Introduction}\label{intro}

\subsection{Motivation}

In his landmark papers \cite{HidaGalois} and \cite{HidaIwasawa}, Hida proved that the $p$-adic Galois representations 
attached to ordinary cuspidal Hecke eigenforms by Deligne (\cite{DeligneFormes}, \cite{CarayolReps})
interpolate $p$-adic analytically in the weight variable to a family of $p$-adic representations
whose specializations to integer weights $k\ge 2$ recover
the ``classical" Galois representations attached to weight $k$ cuspidal eigenforms.
Hida's work paved the way for a revolution---
from the pioneering work of Mazur
on Galois deformations to Coleman's construction of $p$-adic families of finite slope overconvergent 
modular forms---and began a trajectory of thought whose fruits include some of the most spectacular
achievements in modern number theory.

Hida's proof is constructive and has at its heart the \'etale cohomology of the tower
of modular curves $\{X_1(Np^r)\}_{r}$ over $\Q$.  More precisely,  
Hida considers the projective limit $H^1_{\et}:=\varprojlim_r H^1_{\et}(X_1(Np^r)_{\Qbar},\Z_p)$
(taken with respect to the trace mappings), which is naturally a module for the 
``big" $p$-adic Hecke algebra $\H^*:=\varprojlim_r \H_r^*$, which is itself an algebra
over the completed group ring $\Lambda:=\Z_p[\![1+p\Z_p]\!]\simeq \Z_p[\![T]\!]$
via the diamond operators.  
Using the idempotent $e^*\in \H^*$ attached to the (adjoint) Atkin operator $U_p^*$ 
to project to the part of $H^1_{\et}$ where $U_p^*$ acts invertibly,
Hida proves in \cite[Theorem 3.1]{HidaGalois} 
(via the comparison isomorphism between \'etale and topological cohomology and explicit 
calculations in group cohomology) that 
$e^* H^1_{\et}$ is finite and free as a module over $\Lambda$, and that the resulting Galois representation
\begin{equation*}
	\xymatrix{
		{\rho: G_{\Q}} \ar[r] & {\Aut_{\Lambda}(e^*H^1_{\et}) \simeq \GL_m(\Z_p[\![T]\!])} 
		}
\end{equation*}
$p$-adically interpolates the representations attached to ordinary cuspidal eigenforms.

By analyzing the geometry of the tower of modular curves, Mazur and Wiles \cite{MW-Hida}
were able to relate the inertial invariants of the local (at $p$) representation $\rho_p$ to the 
\'etale cohomology of the Igusa tower studied in \cite{MW-Analogies}, and in so doing 
proved\footnote{Mazur and Wiles treat only the case of tame level $N=1$.} 
that the ordinary filtration of the Galois representations attached
to ordinary cuspidal eigenforms interpolates:
both the inertial invariants and covariants 
are free of the same finite rank over $\Lambda$ and specialize to the corresponding subquotients
in integral weights $k\ge 2$.  
As an application, they provided examples of cuspforms $f$ and primes $p$
for which the specialization of the associated Hida family of Galois representations 
to weight $k=1$ is not Hodge-Tate,
and so does not arise from a weight one cuspform via the construction of Deligne-Serre
\cite{DeligneSerre}.
Shortly thereafter, Tilouine \cite{Tilouine} clarified the geometric underpinnings
of \cite{HidaGalois} and \cite{MW-Hida}, and removed most of the restrictions on the $p$-component of the
nebentypus of $f$.  Central to both \cite{MW-Hida} and \cite{Tilouine} is a careful study of
the tower of $p$-divisible groups attached to the ``good quotient" modular abelian varieties
introduced in \cite{MW-Iwasawa}.

With the advent of integral $p$-adic Hodge theory, 
and in view of the prominent role
it has played in furthering the trajectory initiated by Hida's work, 
it is natural to ask 
if one can construct Hodge--Tate, de~Rham and crystalline analogues of $e^*H^1_{\et}$,
and if so, to what extent the integral comparison isomorphsms of $p$-adic Hodge theory
can be made to work in $\Lambda$-adic families.
In \cite{OhtaEichler},
Ohta has addressed this question in the case of Hodge cohomology. 
Using the invariant differentials on the tower of $p$-divisible groups studied in \cite{MW-Hida} and \cite{Tilouine},
Ohta constructs a $\Lambda \wh{\otimes}_{\Z_p} \Z_p[\mu_{p^{\infty}}]$-module
from which, via an integral version of the Hodge--Tate comparison isomorphism \cite{Tate}
for ordinary $p$-divisible groups, he is able to recover the semisimplification
of the ``semilinear representation"
$\rho_{p}\wh{\otimes} \O_{\C_p}$, where 
$\C_p$ is, as usual, the $p$-adic completion of an algebraic closure of $\Q_p$.
Using Hida's results, Ohta proves that his Hodge cohomology analogue of
$e^*H^1_{\et}$ is free of finite rank over $\Lambda\wh{\otimes}_{\Z_p} \Z_p[\mu_{p^{\infty}}]$
and specializes to finite level exactly as one expects.  As applications
of his theory, Ohta provides a construction of two-variable $p$-adic $L$-functions
attached to families of ordinary cuspforms differing from that of Kitagawa \cite{Kitagawa},
and, in a subsequent paper \cite{Ohta2}, 
provides a new and streamlined proof of the theorem of Mazur--Wiles \cite{MW-Iwasawa}
(Iwasawa's Main Conjecture for $\Q$; see also \cite{WilesTotallyReal}).
We remark that Ohta's $\Lambda$-adic Hodge-Tate isomorphism is a crucial ingredient
in the forthcoming proof of Sharifi's conjectures \cite{SharifiConj}, \cite{SharifiEisenstein} 
due to Fukaya and Kato \cite{FukayaKato}.

\subsection{Results}\label{resultsintro}

In this paper, we construct the de Rham and crystalline
counterparts to Hida's ordinary $\Lambda$-adic \'etale cohomology and Ohta's 
$\Lambda$-adic Hodge cohomology, and we prove appropriate control 
and finiteness theorems in each case via a careful study of the geometry of
modular curves and abelian varieties.
We then
prove a suitable $\Lambda$-adic 
version of every integral comparison isomorphism one could hope for.
In particular, we are able to recover the entire family
of $p$-adic Galois representations $\rho_{p}$ (and not just its
semisimplification) from our $\Lambda$-adic crystalline cohomology.
As a byproduct of our work, we provide {\em geometric} constructions
of several of the ``cohomologically elusive" semi-linear algebra objects
in $p$-adic Hodge theory, including 
the family of \'etale $(\varphi,\Gamma)$-modules attached to $e^*H^1_{\et}$
by Dee \cite{Dee}.  As an application of our theory, we give a new 
and purely geometric proof of Hida's freeness and control theorems
for $e^*H^1_{\et}$. 

In order to survey our main results more precisely, we introduce
some notation.  Fix an algebraic closure $\Qbar_p$
of $\Q_p$ 
as well as a $p$-power compatible sequence 
$\{\varepsilon^{(r)}\}_{r\ge 0}$ of primitive $p^r$-th roots of unity in $\Qbar_p$.
We set $K_r:=\Q_p(\mu_{p^r})$ and $K_r':=K_r(\mu_N)$, and we
write $R_r$ and $R_r'$ for the rings of integers
in $K_r$ and $K_r'$, respectively.  
Denote by $\scrG_{\Q_p}:=\Gal(\Qbar_p/\Q_p)$ the absolute Galois group
and by $\scrH$
the kernel of the $p$-adic cyclotomic character
$\chi: \scrG_{\Q_p}\rightarrow \Z_p^{\times}$.
We write $\Gamma:=\scrG_{\Q_p}/\scrH \simeq \Gal(K_{\infty}/K_0)$ for the quotient and,
using that $K_0'/\Q_p$ is unramified, we canonically identify $\Gamma$
with $\Gal(K_{\infty}'/K_0')$.
We will denote by $\langle u\rangle$ (respectively $\langle v\rangle_N)$
the diamond operator\footnote{Note that we have $\langle u^{-1}\rangle=\langle u\rangle^*$
and $\langle v^{-1}\rangle_N = \langle v\rangle_N^*$, where $\langle\cdot\rangle^*$
and $\langle \cdot\rangle_N^*$
are the adjoint diamond operators; see \S\ref{tower}.  
}
in $\H^*$ attached to $u^{-1}\in \Z_p^{\times}$
(respectively $v^{-1}\in (\Z/N\Z)^{\times}$) and write 
$\Delta_r$ for the image of the restriction of $\langle\cdot\rangle :\Z_p^{\times}\hookrightarrow \H^*$
to $1+p^r\Z_p\subseteq \Z_p^{\times}$.  For convenience, we put $\Delta:=\Delta_1$,
and for any ring $A$ we write
$\Lambda_{A}:=\varprojlim_r A[\Delta/\Delta_r]$ for the completed group ring
on $\Delta$ over $A$; if $\varphi$ is an endomorphism of $A$, we again write $\varphi$
for the induced endomorphism of $\Lambda_A$ that acts as the identity on $\Delta$.
Finally, we denote by $X_r:=X_1(Np^r)$ the 
usual modular curve over $\Q$ classifying (generalized) elliptic curves
with a $[\mu_{Np^r}]$-structure,
and by $J_r:=J_1(Np^r)$ its Jacobian.

Our first task is to construct a de Rham analogue of Hida's $e^*H^1_{\et}$.
A na\"ive idea would be to mimic Hida's construction, using the
(relative) de Rham cohomology of $\Z_p$-integral models of the modular curves $X_r$
in place of $p$-adic \'etale cohomology.  However, this approach fails due
to the fact that $X_r$ has bad reduction
at $p$, so the relative de Rham
cohomology of integral models does not provide good
$\Z_p$-lattices in the de Rham cohomology of $X_r$ over $\Q_p$.
To address this problem, we use the canoninical integral structures in de Rham cohomology
studied in \cite{CaisDualizing} and the canonical integral model $\X_r$ of $X_r$
over $R_r$ associated to the moduli problem 
$([\bal\ \Gamma_1(p^r)]^{\varepsilon^{(r)}\can};\ [\mu_N])$
\cite{KM} to construct well-behaved integral ``de Rham cohomology" for the tower of modular
curves.  For each $r$, we obtain a short exact sequence of free $R_r$-modules
with semilinear $\Gamma$-action and comuting $\H_r^*$-action
\begin{equation}
	\xymatrix{
		0\ar[r] & {H^0(\X_r, \omega_{\X_r/R_r})} \ar[r] & {H^1(\X_r/R_r)} \ar[r] & {H^1(\X_r,\O_{\X_r})}
		\ar[r] & 0 
	}\label{finiteleveldRseq}
\end{equation}
which is co(ntra)variantly functorial in finite $K_r$-morphisms of the generic fiber $X_r$,
and whose scalar extension to $K_r$ recovers the Hodge filtration of $H^1_{\dR}(X_r/K_r)$.
Extending scalars to $R_{\infty}$ and taking projective limits, 
we obtain a short exact sequence of $\Lambda_{R_{\infty}}$-modules with semilinear $\Gamma$-action
and commuting linear $\H^*$-action 
\begin{equation}
	\xymatrix{
		0\ar[r] & {H^0(\omega)} \ar[r] & {H^1_{\dR}} \ar[r] & {H^1(\O)} 
	}.\label{dRseq}
\end{equation}
Our first main result (see Theorem \ref{main}) is that the ordinary part of (\ref{dRseq}) 
is the correct de Rham analogue of Hida's ordinary $\Lambda$-adic \'etale cohomology:

\begin{theorem}
	There is a canonical short exact sequence of finite free $\Lambda_{R_{\infty}}$-modules
	with semilinear $\Gamma$-action and commuting linear $\H^*$-action
	\begin{equation}
		\xymatrix{
		0\ar[r] & {e^*H^0(\omega)} \ar[r] & {e^*H^1_{\dR}} \ar[r] & {e^*H^1(\O)} \ar[r] & 0
	}.\label{orddRseq}
	\end{equation}
	As a $\Lambda_{R_{\infty}}$-module, $e^*H^1_{\dR}$ is free of rank $2d$, while each of the flanking terms 
	in $(\ref{orddRseq})$ is free of rank $d$, for $d=\sum_{k=3}^{p+1}\dim_{\F_p} S_k(\Gamma_1(N);\F_p)^{\ord}$.
	Applying $\otimes_{\Lambda_{R_{\infty}}} R_{\infty}[\Delta/\Delta_r]$ to $(\ref{orddRseq})$
	recovers the ordinary part of the scalar extension of $(\ref{finiteleveldRseq})$ to $R_{\infty}$.
\end{theorem}

We then show that the $\Lambda_{R_{\infty}}$-adic Hodge filtration (\ref{orddRseq}) is very nearly ``auto dual".
To state our duality result more succintly,
for any ring homomorphism $A\rightarrow B$,
we will write $(\cdot)_B:=(\cdot)\otimes_A B$ and $(\cdot)_B^{\vee}:=\Hom_B((\cdot)\otimes_A B , B)$ for these
functors from $A$-modules to $B$-modules.  If $G$ is any group of automorphisms of $A$ and $M$
is an $A$-module with a semilinear action of $G$, for any ``crossed" 
homomorphism\footnote{That is, $\psi(\sigma\tau) = \psi(\sigma)\cdot\sigma\psi(\tau)$ for all $\sigma,\tau\in\Gamma$,} 
$\psi:G\rightarrow A^{\times}$ we will write $M(\psi)$ for the 
$A$-module $M$ with ``twisted" semilinear $G$-action given by $g\cdot m:=\psi(g)g m$.
Our duality theorem is (see Proposition \ref{dRDuality}):

\begin{theorem}\label{dRdualityThm}
	The natural cup-product auto-duality of 
	$(\ref{finiteleveldRseq})$ over $R_r':=R_r[\mu_N]$ induces a canonical
	$\Lambda_{R_{\infty}'}$-linear and $\H^*$-equivariant isomorphism 
	of exact sequences
	\begin{equation*}
		\xymatrix{
				0\ar[r] & {e^*H^0(\omega)(\langle\chi\rangle\langle a\rangle_N)_{\Lambda_{R_{\infty}'}}}
				\ar[r]\ar[d]^-{\simeq} & 
				{e^*H^1_{\dR}(\langle\chi\rangle\langle a\rangle_N)_{\Lambda_{R_{\infty}'}}} 
				\ar[r]\ar[d]^-{\simeq} & 
				{e^*H^1(\O)(\langle\chi\rangle\langle a\rangle_N)_{\Lambda_{R_{\infty}'}}} 
				\ar[r]\ar[d]^-{\simeq} & 0\\
	0\ar[r] & {(e^*H^1(\O))^{\vee}_{\Lambda_{R_{\infty}'}}} \ar[r] & 
	{({e^*H^1_{\dR}})^{\vee}_{\Lambda_{R_{\infty}'}}} \ar[r] & 
	{(e^*H^0(\omega))^{\vee}_{\Lambda_{R_{\infty}'}}} \ar[r] & 0
		}
	\end{equation*}
	that is compatible with the natural action of $\Gamma \times \Gal(K_0'/K_0)\simeq \Gal(K_{\infty}'/K_0)$ 
	on the bottom row and the twist of the natural action on the top row by the $\H^*$-valued character 
	$\langle \chi\rangle \langle a\rangle_N$,
	where 
	$a(\gamma) \in (\Z/N\Z)^{\times}$ is
	determined for $\gamma\in \Gal(K_0'/K_0)$ by 
	$\zeta^{a(\gamma)}=\gamma\zeta$ for every $N$-th root of unity $\zeta$.
\end{theorem}

We moreover prove that, as one would expect, the $\Lambda_{R_{\infty}}$-module $e^*H^0(\omega)$ 
is canonically isomorphic to the  module $eS(N,\Lambda_{R_{\infty}})$ of ordinary $\Lambda_{R_{\infty}}$-adic
cusp forms of tame level $N$; see Corollary \ref{LambdaFormsRelation}.

To go further, we study the tower of $p$-divisible groups attached to the ``good quotient" modular abelian
varieties introduced by Mazur-Wiles \cite{MW-Iwasawa}. To avoid technical complications with logarithmic
$p$-divisible groups, following \cite{MW-Hida} and \cite{OhtaEichler}, we will henceforth remove the trivial tame character by working with the sub-idempotent ${e^*}'$ of $e^*$ corresponding to projection to the part
where $\mu_{p-1}\subseteq \Z_p^{\times}\simeq \Delta$ acts {\em non}-trivially.
As is well-known (e.g. \cite[\S9]{HidaGalois} and \cite[Chapter 3, \S2]{MW-Iwasawa}), the $p$-divisible group $G_r:={e^*}'J_r[p^{\infty}]$
over $\Q$ extends to a $p$-divisible group $\G_r$ over $R_r$, and we write 
$\o{\G}_r:=\G_r\times_{R_r} \F_p$ for its special fiber.  Denoting by $\D(\cdot)$
the contravariant Dieudonn\'e module functor on $p$-divisible groups over $\F_p$,
we form the projective limits
\begin{equation}
	\D_{\infty}^{\star}:=\varprojlim_r \D(\o{\G}_r^{\star})\quad\text{for}\quad \star\in \{\et,\mult,\Null\},
	\label{DlimitsDef}
\end{equation}
taken along the mappings induced by 
$\o{\G}_r\rightarrow \o{\G}_{r+1}$.  Each of these is naturally a $\Lambda$-module
equipped with linear (!) Frobenius $F$ and Verscheibung $V$ morphisms satisfying $FV=VF=p$,
as well as a linear action of $\H^*$ and a ``geometric inertia" action of $\Gamma$, which reflects
the fact that the generic fiber of $\G_r$ descends to $\Q_p$.
The $\Lambda$-modules (\ref{DlimitsDef}) have the expected structure (see Theorem \ref{MainDieudonne}):

\begin{theorem}\label{DieudonneMainThm}
	There is a canonical split short exact sequence of finite and free $\Lambda$-modules	
	\begin{equation}
		\xymatrix{
			0 \ar[r] & {\D^{\et}_{\infty}} \ar[r] & {\D_{\infty}} \ar[r] & {\D_{\infty}^{\mult}} \ar[r] & 0.
			}.\label{Dieudonneseq}
	\end{equation} 
	with linear $\H^*$ and $\Gamma$-actions.
	As a $\Lambda$-module, $\D_{\infty}$ is free of rank $2d'$, while $\D_{\infty}^{\et}$
	and $\D_{\infty}^{\mult}$ are free of rank $d'$, where $d':=\sum_{k=3}^p \dim_{\F_p} S_k(\Gamma_1(N);\F_p)^{\ord}$.
	For $\star\in \{\mult,\et,\Null\}$, there are canonical isomorphisms
	\begin{equation*}
		\D_{\infty}^{\star}\tens_{\Lambda} \Z_p[\Delta/\Delta_r] \simeq \D(\o{\G}_r^{\star})
	\end{equation*}
	which are compatible with the extra structures.
	Via the canonical splitting of $(\ref{Dieudonneseq})$, $\D_{\infty}^{\star}$ for $\star=\et$
	$($respetively $\star=\mult$$)$ is
	identified with the maximal subpace of $\D_{\infty}$ on which $F$ $($respectively $V$$)$ acts
	invertibly .	
	The Hecke operator $U_p^*\in \H^*$ acts as $F$ on $\D_{\infty}^{\et}$ and as $\langle p\rangle_N V$ on 
	$\D_{\infty}^{\mult}$,
	while $\Gamma$ acts trivially on $\D_{\infty}^{\et}$ and via $\langle \chi(\cdot)\rangle^{-1}$
	on $\D_{\infty}^{\mult}$.
\end{theorem}

We likewise have the appropriate ``Dieudonn\'e" 
analogue of Theorem \ref{dRdualityThm} (see Proposition \ref{DieudonneDuality}):

\begin{theorem}\label{DDuality}
		There is a canonical $\H^*$-equivariant isomorphism of exact sequences of $\Lambda_{R_0'}$-modules
		\begin{equation*}
			\xymatrix{
	0 \ar[r] & {\D_{\infty}^{\et}(\langle \chi \rangle\langle a\rangle_N)_{\Lambda_{R_0'}}} \ar[r]\ar[d]^-{\simeq} & 
	{\D_{\infty}(\langle \chi \rangle\langle a\rangle_N)_{\Lambda_{R_0'}}}\ar[r]\ar[d]^-{\simeq} & 
	{\D_{\infty}^{\mult}(\langle \chi \rangle\langle a\rangle_N)_{\Lambda_{R_0'}}}\ar[r]\ar[d]^-{\simeq} & 0 \\		
		0\ar[r] & {(\D_{\infty}^{\mult})^{\vee}_{\Lambda_{R_0'}}} \ar[r] & 
				{(\D_{\infty})^{\vee}_{\Lambda_{R_0'}}} \ar[r] & 
				{(\D_{\infty}^{\et})^{\vee}_{\Lambda_{R_0'}}}\ar[r] & 0
			}
		\end{equation*}
		that is $\Gamma\times \Gal(K_0'/K_0)$-equivariant, 
		and intertwines $F$
		$($respectively $V$$)$ on the top row with $V^{\vee}$
		$($respectively $F^{\vee}$$)$ on the bottom.\footnote{Here, $F^{\vee}$ (respectively $V^{\vee}$)
		is the map taking a linear functional $f$ to $\varphi^{-1}\circ f\circ F$ 
		(respectively $\varphi\circ f\circ V$), where $\varphi$
		is the Frobenius automorphism of $R_0'=\Z_p[\mu_N]$.}	
\end{theorem}

Just as Mazur-Wiles are able to relate the ordinary-filtration of ${e^*}'H^1_{\et}$
to the \'etale cohomology of the Igusa tower, we can interpret the slope filtraton (\ref{Dieudonneseq})
in terms of the crystalline cohomology of the Igusa tower as follows.
For each $r$, let $I_r^{\infty}$ and $I_r^0$ be the two ``good" irreducible components of
$\X_r\times_{R_r}\F_r$ (see Remark \ref{MWGood}), each of which is isomorphic to the Igusa curve $\Ig(p^r)$
of tame level $N$ and $p$-level $p^r$. For $\star\in \{0,\infty\}$ we form the projective limit
\begin{equation*} 
	H^1_{\cris}(I^{\star}):=\varprojlim_{r} H^1_{\cris}(I_r^{\star}/\Z_p);
\end{equation*}
with respect to the trace mappings on crystalline cohmology induced by the canonical
degeneracy maps on Igusa curves.
Then $H^1_{\cris}(I^{\star})$ is naturally a $\Lambda$-module with linear Frobenius $F$ and Verscheibung $V$ endomorphisms.
Letting $f'$ be the idempotent of $\Lambda$ corresponding to projection to the part
where $\mu_{p-1}\subseteq \Delta\hookrightarrow \Lambda$ acts nontrivially, we prove
(see Theorem \ref{DieudonneCrystalIgusa}):

\begin{theorem}
 There is a canonical isomorphism of $\Lambda$-modules, compatible with $F$ and $V,$
	\begin{equation}
		\D_{\infty} =\D_{\infty}^{\mult}\oplus \D_{\infty}^{\et}\simeq 
		f'H^1_{\cris}(I^{0})^{V_{\ord}} \oplus
						f'H^1_{\cris}(I^{\infty})^{F_{\ord}}.\label{crisIgusa}
	\end{equation}
	which preserves the direct sum decompositions of source and target.
	This isomorphism is Hecke and $\Gamma$-equivariant, with $U_p^*$ and $\Gamma$
	acting as $\langle p\rangle_N V\oplus F$ and 
	$ \langle \chi(\cdot)\rangle^{-1}\oplus \id$, respectively, 
	on each direct sum.	
\end{theorem}

We note that our ``Dieudonn\'e module" analogue (\ref{crisIgusa}) is a significant
sharpening of its \'etale counterpart \cite[\S4]{MW-Hida}, which is formulated
only up to isogeny (i.e. after inverting $p$).  From $\D_{\infty}$, 
we can recover the $\Lambda$-adic Hodge filtration (\ref{orddRseq}), so
the latter is canonically split (see Theorem \ref{dRtoDieudonneInfty}):

\begin{theorem}\label{dRtoDieudonne}
	There is a canonical $\Gamma$ and $\H^*$-equivariant isomorphism of 
	exact sequences 
	\begin{equation}
	\begin{gathered}
		\xymatrix{
		0 \ar[r] & {{e^*}'H^0(\omega)} \ar[r]\ar[d]^-{\simeq} & 
		{{e^*}'H^1_{\dR}} \ar[r]\ar[d]^-{\simeq} & {{e^*}'H^1(\O)} \ar[r]\ar[d]^-{\simeq} & 0 \\
		0 \ar[r] & {\D_{\infty}^{\mult}\tens_{\Lambda} \Lambda_{R_{\infty}}} \ar[r] &
		{\D_{\infty}\tens_{\Lambda} \Lambda_{R_{\infty}}} \ar[r] &
		{\D_{\infty}^{\et}\tens_{\Lambda} \Lambda_{R_{\infty}}} \ar[r] & 0
		}\label{dRcriscomparison}
	\end{gathered}
	\end{equation}
	where the mappings on bottom row are the canonical inclusion and projection morphisms
	corresponding to the direct sum decomposition $\D_{\infty}=\D_{\infty}^{\mult}\oplus \D_{\infty}^{\et}$.
	In particular, the Hodge filtration exact sequence $(\ref{orddRseq})$ is canonically 
	split, and admits a canonical descent to $\Lambda$.
\end{theorem}

	We remark that under the identification (\ref{dRcriscomparison}),  
	the Hodge filtration (\ref{orddRseq}) and slope filtration (\ref{Dieudonneseq})
	correspond, but in the opposite directions.  As a consequence of Theorem 
	\ref{dRtoDieudonne}, we deduce (see Corollary \ref{MFIgusaDieudonne} and
	Remark \ref{MFIgusaCrystal}):

\begin{corollary}
\label{OhtaCor}
	There is a canonical isomorphism of finite free $\Lambda$ $($respectively $\Lambda_{R_0'}$$)$-modules
	\begin{equation*}
		{e}'S(N,\Lambda) \simeq \D_{\infty}^{\mult}
		\qquad\text{respectively}\qquad
		e'\H\tens_{\Lambda} \Lambda_{R_0'} \simeq \D_{\infty}^{\et}(\langle a\rangle_N)\tens_{\Lambda}{\Lambda_{R_0'}}
	\end{equation*}
	that  intertwines $T\in \H:=\varprojlim \H_r$ with $T^*\in \H^*$, where we let
	$U_p^*$ act as $\langle p\rangle_N V$ on $\D_{\infty}^{\mult}$ and as $F$ on $\D_{\infty}^{\et}$. 
	The second of these isomorphisms is in addition $\Gal(K_0'/K_0)$-equivariant.
\end{corollary}	
	
We are also able to recover the semisimplification of ${e^*}'H^1_{\et}$ from $\D_{\infty}$.
Writing $\I\subseteq \scrG_{\Q_p}$ 
for the inertia subgroup at $p$, for any $\Z_p[\scrG_{\Q_p}]$-module $M$, we denote by $M^{\I}$ (respectively $M_{\I}:=M/M^{\I}$)
the sub (respectively quotient) module of invariants (respectively covariants) under $\I$. 

\begin{theorem}\label{FiltrationRecover}
	There are canonical isomorphisms of $\Lambda_{W(\o{\F}_p)}$-modules
	with linear $\H^*$-action and semilinear actions of $F$, $V$, and $\scrG_{\Q_p}$
	\begin{subequations}
		\begin{equation}
			\D_{\infty}^{\et} \tens_{\Lambda} \Lambda_{W(\o{\F}_p)} 
			\simeq ({e^*}'H^1_{\et})^{\I}\tens_{\Lambda} \Lambda_{W(\o{\F}_p)} 
			\label{inertialinvariants}
		\end{equation}
		and
		\begin{equation}
			\D_{\infty}^{\mult}(-1) \tens_{\Lambda} \Lambda_{W(\o{\F}_p)} 
			\simeq ({e^*}'H^1_{\et})_{\I}\tens_{\Lambda} \Lambda_{W(\o{\F}_p)}.
			\label{inertialcovariants}
		\end{equation}
	\end{subequations}
	Writing $\sigma$ for the 
	Frobenius automorphism of $W(\o{\F}_p)$,
	the isomorphism $(\ref{inertialinvariants})$ intertwines $F\otimes \sigma$ with $\id\otimes\sigma$
	and $\id\otimes g$ with $g\otimes g$ for $g\in \scrG_{\Q_p}$, whereas $(\ref{inertialcovariants})$ 
	intertwines $V\otimes \sigma^{-1}$ with $\id\otimes\sigma^{-1}$ and $g\otimes g$ with $g\otimes g$,
	where $g\in\scrG_{\Q_p}$ acts on the Tate twist 
	$\D_{\infty}^{\mult}(-1):=\D_{\infty}\otimes_{\Z_p}\Z_p(-1)$ as
	$\langle \chi(g)^{-1}\rangle \otimes \chi(g)^{-1}$.
\end{theorem}

Theorem \ref{FiltrationRecover} gives the following ``explicit" description
of the semisimplification of ${e^*}'H^1_{\et}$:

\begin{corollary}
	For any $T\in (\H^{*\ord})^{\times}$, let 
	$\lambda(T):\scrG_{\Q_p}\rightarrow \H^{*\ord}$ be the 
	unique continuous $($for the $p$-adic topology on $\H^{*\ord}$$)$
	unramified character whose value on $($any lift of$)$
	$\mathrm{Frob}_p$ is $T$.
	Then $\scrG_{\Q_p}$ acts on $({e^*}'H^1_{\et})^{\I}$ through the
	character $\lambda({U_p^*}^{-1})$ and on $({e^*}'H^1_{\et})_{\I}$ 
	through $\chi^{-1} \cdot \langle \chi^{-1}\rangle \lambda(\langle p\rangle_N^{-1}U_p^*)$.
\end{corollary}

We remark that Corollary \ref{OhtaCor} and Theorem \ref{FiltrationRecover}
combined give a refinement of the main result of \cite{OhtaEichler}.
We are furthermore able to recover the main theorem of \cite{MW-Hida}
(that the ordinary filtration of ${e^*}'H^1_{\et}$ interpolates
$p$-adic analytically):

\begin{corollary}\label{MWmainThmCor}
	Let $d'$ be the integer of Theorem $\ref{DieudonneMainThm}$.  Then each of
	$({e^*}'H^1_{\et})^{\I}$ and $({e^*}'H^1_{\et})_{\I}$ is a free
	$\Lambda$-module of rank $d'$, and for each $r\ge 1$ there are canonical
	$\H^*$ and $\scrG_{\Q_p}$-equivariant isomorphisms of $\Z_p[\Delta/\Delta_r]$-modules
	\begin{subequations}
	\begin{equation}	
		({e^*}'H^1_{\et})^{\I} \tens_{\Lambda} \Z_p[\Delta/\Delta_r] \simeq 
		{e^*}'H^1_{\et}({X_r}_{\Qbar_p},\Z_p)^{\I}\label{HidaResultSub}
	\end{equation}
	\begin{equation}
		({e^*}'H^1_{\et})_{\I} \tens_{\Lambda} \Z_p[\Delta/\Delta_r] \simeq 
		{e^*}'H^1_{\et}({X_r}_{\Qbar_p},\Z_p)_{\I}
		\label{HidaResultQuo}
	\end{equation}
	\end{subequations}
\end{corollary}

To recover the full $\Lambda$-adic local Galois representation ${e^*}'H^1_{\et}$,
rather than just its semisimplification,
it is necessary to work with the full Dieudonn\'e {\em crystal} of $\G_r$ over $R_r$.
Following Faltings \cite{Faltings} and Breuil (e.g. \cite{Breuil}), this is equivalent
to studying the evaluation of the Dieudonn\'e crystal of $\G_r\times_{R_r} R_r/pR_r$
on the ``universal" divided power thickening $S_r\twoheadrightarrow R_r/pR_r$,
where $S_r$ is the $p$-adically completed PD-hull
of the surjection $\Z_p[\![u_r]\!]\twoheadrightarrow R_r$
sending $u_r$ to $\varepsilon^{(r)}-1$.  As the rings $S_r$ are too unwieldly
to directly construct a good crystalline analogue of Hida's
ordinary \'etale cohomology, we must functorially descend  
the ``filtered $S_r$-module" attached to $\G_r$ to the much simpler ring $\s_r:=\Z_p[\![u_r]\!]$.  
While such a descent is provided
(in rather different ways) by the work of Breuil--Kisin and Berger--Wach, neither of these
frameworks is suitable for our application: it is essential for us
that the formation of this descent to $\s_r$ commute with base change as one moves up
the cyclotomic tower, and it is not at all clear that this holds for Breuil--Kisin modules 
or for the Wach modules of Berger.  Instead, we use the theory of \cite{CaisLau}, which works with frames and 
windows \`a la Lau and Zink to provide the desired functorial descent to a ``$(\varphi,\Gamma)$-module" $\m_r(\G_r)$
over $\s_r$.  
We view $\s_r$
as a $\Z_p$-subalgebra of $\s_{r+1}$ via the map sending $u_r$ to $\varphi(u_{r+1}):=(1+u_{r+1})^p -1$,
and we write $\s_{\infty}:=\varinjlim \s_r$ for 
the rising union\footnote{As explained in Remark \ref{Slimits}, 
the $p$-adic completion of $\s_{\infty}$ is actually a very nice ring:
it is canonically and Frobenius equivariantly isomorphic to $W(\F_p[\![u_0]\!]^{\perf})$,
for $\F_p[\![u_0]\!]^{\perf}$ the perfection of the $\F_p$-algebra $\F_p[\![u_0]\!]$.
}
of the $\s_r$, equiped with its Frobenius {\em automorphism} $\varphi$ and commuting action of 
$\Gamma$ determined by $\gamma u_r:=(1+u_r)^{\chi(\gamma)} - 1$.
We then form the projective limits
\begin{equation*}
	\m_{\infty}^{\star}:=\varprojlim (\m_r(\G_r^{\star})\tens_{\s_r} \s_{\infty})\quad\text{for}\quad \star\in\{\et,\mult,\Null\}
\end{equation*}
taken along the mappings induced by $\G_{r}\times_{R_r} R_{r+1}\rightarrow \G_{r+1}$
via the functoriality of $\m_r(\cdot)$
and its compatibility with base change.
These are $\Lambda_{\s_{\infty}}$-modules equipped with a semilinear action of $\Gamma$,
a linear and commuting action of $\H^*$, and a commuting $\varphi$ (respectively $\varphi^{-1}$) semilinear endomorphism $F$ (respectively $V$)
satisfying $FV=VF = \omega$, for $\omega:=\varphi(u_1)/u_1 = u_0/\varphi^{-1}(u_0)\in \s_{\infty}$,
and they provide
our crystalline analogue of Hida's ordinary \'etale cohomology (see Theorem \ref{MainThmCrystal}):

\begin{theorem}
	There is a canonical short exact sequence of finite free $\Lambda_{\s_{\infty}}$-modules
	with linear $\H^*$-action, semilinear $\Gamma$-action, and commuting semilinear
	endomorphisms $F,$ $V$ satisfying $FV=VF=\omega$
	\begin{equation}
		\xymatrix{
			0 \ar[r] & {\m_{\infty}^{\et}} \ar[r] & {\m_{\infty}} \ar[r] & {\m_{\infty}^{\mult}} \ar[r] & 0
		}.\label{CrystallineAnalogue}
	\end{equation}
	Each of $\m_{\infty}^{\star}$ for $\star\in \{\et,\mult\}$ is free of rank $d'$ over 
	$\Lambda_{\s_{\infty}}$, while $\m_{\infty}$ is free of rank $2d'$, where
	$d'$ is as in Theorem $\ref{DieudonneMainThm}$.  Extending scalars on $(\ref{CrystallineAnalogue})$
	along the canonical surjection 
	$\Lambda_{\s_{\infty}}\twoheadrightarrow \s_{\infty}[\Delta/\Delta_r]$ yields the short exact 
	sequence
	\begin{equation*}
		\xymatrix{
			0 \ar[r] & {\m_r(\G_r^{\et})\tens_{\s_r} \s_{\infty}} \ar[r] &
			{\m_r(\G_r)\tens_{\s_r} \s_{\infty}} \ar[r] &
			{\m_r(\G_r^{\mult})\tens_{\s_r} \s_{\infty}} \ar[r] & 
			0
		}
	\end{equation*}
	compatibly with $\H^*$, $\Gamma$, $F$ and $V$.
\end{theorem}

Again, in the spirit of Theorems \ref{dRdualityThm} and \ref{DDuality}, 
there is a corresponding ``autoduality" result for $\m_{\infty}$
(see Theorem \ref{CrystalDuality}).  To state it, we must work over the ring 
$\s_{\infty}':=\varinjlim_r \Z_p[\mu_N][\![u_r]\!]$, with the inductive limit taken along the 
$\Z_p$-algebra maps sending $u_r$ to $\varphi(u_{r+1})$.

\begin{theorem}
		Let $\mu:\Gamma\rightarrow \Lambda_{\s_{\infty}}^{\times}$ be the crossed homomorphism 
		given by $\mu(\gamma):=\frac{u_1}{\gamma u_1}\chi(\gamma) \langle \chi(\gamma)\rangle$.
		There is a canonical $\H^*$ and $\Gal(K_{\infty}'/K_0)$-compatible isomorphism of 
		exact sequences
		\begin{equation*}
		\begin{gathered}
			\xymatrix{
			0\ar[r] & {\m_{\infty}^{\et}(\mu \langle a\rangle_N)_{\Lambda_{\s_{\infty}'}}} \ar[r]\ar[d]_-{\simeq} & 
			{\m_{\infty}(\mu \langle a\rangle_N)_{\Lambda_{\s_{\infty}'}}} \ar[r]\ar[d]_-{\simeq} & 
				{\m_{\infty}^{\mult}(\mu \langle a\rangle_N)_{\Lambda_{\s_{\infty}'}}} \ar[r]\ar[d]_-{\simeq} & 0\\
	0\ar[r] & {(\m_{\infty}^{\mult})_{\Lambda_{\s_{\infty}'}}^{\vee}} \ar[r] & 
				{(\m_{\infty})_{\Lambda_{\s_{\infty}'}}^{\vee}} \ar[r] & 
				{(\m_{\infty}^{\et})_{\Lambda_{\s_{\infty}'}}^{\vee}} \ar[r] & 0 
		}
		\end{gathered}
		\end{equation*}
		that intertwines $F$ 
		$($respectively $V$$)$ on the top row with  
		$V^{\vee}$ $($respectively $F^{\vee}$$)$  on the bottom.
\end{theorem}

The $\Lambda_{\s_{\infty}}$-modules $\m_{\infty}^{\et}$ and $\m_{\infty}^{\mult}$
have a particularly simple structure (see Theorem \ref{etmultdescent}):

\begin{theorem}
	There are canonical $\H^*$, $\Gamma$, $F$ and $V$-equivariant isomorphisms
	of $\Lambda_{\s_{\infty}}$-modules
	\begin{subequations}
	\begin{equation}
		\m_{\infty}^{\et} \simeq \D_{\infty}^{\et}\tens_{\Lambda} \Lambda_{\s_{\infty}},
	\end{equation}	
	intertwining $F$ $($respetcively $V$$)$ with 
	$F\otimes \varphi$ $($respectively $F^{-1}\otimes \omega\cdot \varphi^{-1}$$)$ and $\gamma\in \Gamma$
	with $\gamma\otimes\gamma$, and
	\begin{equation}
		\m_{\infty}^{\mult}\simeq \D_{\infty}^{\mult}\tens_{\Lambda} \Lambda_{\s_{\infty}},	
	\end{equation}
	intertwing $F$ $($respectively $V$$)$ with $V^{-1} \otimes \omega \cdot\varphi$
	$($respectively $V\otimes\varphi^{-1}$$)$
	and $\gamma$ with $\gamma\otimes \chi(\gamma)^{-1} \gamma u_1/u_1)$.
	In particular, $F$ $($respectively $V$) 
	acts invertibly on $\m_{\infty}^{\et}$ $($respectively $\m_{\infty}^{\mult}$$)$.
\end{subequations}
\end{theorem}

From $\m_{\infty}$, we can recover $\D_{\infty}$ and ${e^*}'H^1_{\dR}$,
with their additional structures (see Theorem \ref{SRecovery}):
\begin{theorem}\label{MinftySpecialize}
	Viewing $\Lambda$ as a $\Lambda_{\s_{\infty}}$-algebra via the map induced by $u_r\mapsto 0$,
	there is a canonical isomorphism of short exact sequences of finite free $\Lambda$-modules
	\begin{equation*}
		\xymatrix{
			0 \ar[r] & {\m_{\infty}^{\et}\tens_{\Lambda_{\s_{\infty}}} \Lambda}\ar[d]_-{\simeq} \ar[r] &
			{\m_{\infty}\tens_{\Lambda_{\s_{\infty}}} \Lambda}\ar[r] \ar[d]_-{\simeq}&
			{\m_{\infty}^{\mult}\tens_{\Lambda_{\s_{\infty}}} \Lambda} \ar[r]\ar[d]_-{\simeq} & 0\\
			0 \ar[r] & {\D_{\infty}^{\et}} \ar[r] & {\D_{\infty}} \ar[r] &
			{\D_{\infty}^{\mult}} \ar[r] & 0
		}
	\end{equation*}
	which is $\Gamma$ and $\H^*$-equivariant and carries $F\otimes 1$ to $F$
	and $V\otimes 1$ to $V$.
	Viewing $\Lambda_{R_{\infty}}$ as a $\Lambda_{\s_{\infty}}$-algebra via the map 
	$u_r\mapsto (\varepsilon^{(r)})^p - 1$, there is a canonical
	isomorphism of short exact sequences of 
	$\Lambda_{R_{\infty}}$-modules
	\begin{equation*}
		\xymatrix{
				0 \ar[r] & {\m_{\infty}^{\et}\tens_{\Lambda_{\s_{\infty}}} \Lambda_{R_{\infty}}}
				\ar[d]_-{\simeq} \ar[r] &
			{\m_{\infty}\tens_{\Lambda_{\s_{\infty}}} \Lambda_{R_{\infty}}}\ar[r] \ar[d]_-{\simeq}&
			{\m_{\infty}^{\mult}\tens_{\Lambda_{\s_{\infty}}} \Lambda_{R_{\infty}}} \ar[r]\ar[d]_-{\simeq} & 0\\
		0 \ar[r] & {{e^*}'H^1(\O)} \ar[r]_{i} & 
		{{e^*}'H^1_{\dR}} \ar[r]_-{j} & {{e^*}'H^0(\omega)} \ar[r] & 0 
		}
	\end{equation*}
	with $i$ and $j$ the canonical sections 
	given by the splitting in Theorem $\ref{dRtoDieudonne}$.
\end{theorem}

To recover Hida's ordinary \'etale cohomology from $\m_{\infty}$,
we introduce the ``period" ring of Fontaine\footnote{Though we use the notation introduced by Berger and Colmez.} 
$\wt{\e}^+:=\varprojlim \O_{\C_p}/(p)$, with the projective limit
taken along the $p$-power mapping; this is a perfect valuation ring of characteristic $p$
equipped with a canonical action of $\scrG_{\Q_p}$ via ``coordinates".  We write $\wt{\e}$
for the fraction field of $\wt{\e}^+$ and
$\wt{\a}:=W(\wt{\e})$ for its ring of Witt vectors, equipped 
with its canonical Frobenius automorphism $\varphi$ and $\scrG_{\Q_p}$-action induced by Witt functoriality.
Our fixed choice of $p$-power compatible sequence $\{\varepsilon^{(r)}\}$
determines an element $\u{\varepsilon}:=(\varepsilon^{(r)}\bmod p)_{r\ge 0}$
of $\wt{\e}^+$, and we $\Z_p$-linearly embed $\s_{\infty}$ in $\wt{\a}$ via 
$u_r\mapsto \varphi^{-r}([\u{\varepsilon}]-1)$ where $[\cdot]$ is the Teichm\"uller
section.  This embedding is $\varphi$ and $\scrG_{\Q_p}$-compatible, with $\scrG_{\Q_p}$
acting on $\s_{\infty}$ through the quotient $\scrG_{\Q_p}\twoheadrightarrow \Gamma$.

\begin{theorem}
\label{RecoverEtale}
	Twisting the structure map $\s_{\infty}\rightarrow \wt{\a}$ by the Frobenius automorphism $\varphi$,
	there is a canonical isomorphism of short exact sequences of $\Lambda_{\wt{\a}}$-modules
	with $\H^*$-action
	\begin{equation}
	\begin{gathered}
		\xymatrix{
				0 \ar[r] & {\m_{\infty}^{\et}\tens_{\Lambda_{\s_{\infty}},\varphi} \Lambda_{\wt{\a}}}
				\ar[d]_-{\simeq} \ar[r] &
			{\m_{\infty}\tens_{\Lambda_{\s_{\infty}},\varphi} \Lambda_{\wt{\a}}}\ar[r] \ar[d]_-{\simeq}&
			{\m_{\infty}^{\mult}\tens_{\Lambda_{\s_{\infty}},\varphi} \Lambda_{\wt{\a}}} \ar[r]\ar[d]_-{\simeq} & 0\\
			0 \ar[r] & {({e^*}'H^1_{\et})^{\I}\tens_{\Lambda} \Lambda_{\wt{\a}}} \ar[r] & 
			{{e^*}'H^1_{\et}\tens_{\Lambda} \Lambda_{\wt{\a}}} \ar[r] &
			({e^*}'H^1_{\et})_{\I}\tens_{\Lambda} \Lambda_{\wt{\a}}\ar[r] & 0
		}\label{FinalComparisonIsom}
	\end{gathered}
	\end{equation}
	that is $\scrG_{\Q_p}$-equivariant for the ``diagonal" action of $\scrG_{\Q_p}$
	$($with $\scrG_{\Q_p}$ acting on $\m_{\infty}$ through $\Gamma$$)$
	and intertwines
	$F\otimes \varphi$ with $\id\otimes\varphi$ and $V\otimes\varphi^{-1}$ with $\id\otimes \varphi^{-1}$.
	In particular, there is a canonical isomorphism of $\Lambda$-modules, compatible
	with the actions of $\H^*$ and $\scrG_{\Q_p}$,
	\begin{equation}
		{e^*}'H^1_{\et} \simeq \left( \m_{\infty}\tens_{\Lambda_{\s_{\infty}},\varphi} 
		\Lambda_{\wt{\a}}\right )^{F\otimes\varphi = 1}.\label{RecoverEtaleIsom}
	\end{equation}
\end{theorem}

Theorem \ref{RecoverEtale} allows us to give a new proof of Hida's finiteness and control theorems
for ${e^*}'H^1_{\et}$:

\begin{corollary}[Hida]\label{HidasThm}	
	Let $d'$ be as in Theorem $\ref{DieudonneMainThm}$.  Then
	${e^*}'H^1_{\et}$ is free $\Lambda$-module of rank $2d'$. 
	For each $r\ge 1$ there is a canonical isomorphism of $\Z_p[\Delta/\Delta_r]$-modules
	with linear $\H^*$ and $\scrG_{\Q_p}$-actions
	\begin{equation*}
		{e^*}'H^1_{\et} \tens_{\Lambda} \Z_p[\Delta/\Delta_r] \simeq {e^*}'H^1_{\et}({X_r}_{\Qbar_p},\Z_p)
	\end{equation*}
	which is moreover compatible with the isomorphisms $(\ref{HidaResultSub})$ and $(\ref{HidaResultQuo})$ 
	in the evident manner.
\end{corollary}

We also deduce a new proof of the following duality result
\cite[Theorem 4.3.1]{OhtaEichler} ({\em cf.} \cite[\S6]{MW-Hida}):

\begin{corollary}[Ohta]\label{OhtaDuality}
	Let $\nu:\scrG_{\Q_p}\rightarrow \H^*$ be the 
	character $\nu:=\chi\langle \chi\rangle \lambda(\langle p\rangle_N)$.
	There is a canonical $\H^*$ and $\scrG_{\Q_p}$-equivariant isomorphism
	of short exact sequences of $\Lambda$-modules 
	\begin{equation*}
		\xymatrix{
			0 \ar[r] & {({e^*}'H^1_{\et})^{\I}(\nu)}
			\ar[d]^-{\simeq} \ar[r] & 
			{{e^*}'H^1_{\et}(\nu)}\ar[d]^-{\simeq} \ar[r] &
			{({e^*}'H^1_{\et})_{\I}(\nu)}
			\ar[d]^-{\simeq}\ar[r] & 0 \\
			0 \ar[r] & {\Hom_{\Lambda}(({e^*}'H^1_{\et})_{\I},\Lambda)} \ar[r] & 
			{\Hom_{\Lambda}({e^*}'H^1_{\et},\Lambda)} \ar[r] &
			{\Hom_{\Lambda}(({e^*}'H^1_{\et})^{\I},\Lambda)}\ar[r] & 0
		}
	\end{equation*}
\end{corollary}

The $\Lambda$-adic splitting 
of the ordinary filtration of $e^*H^1_{\et}$ was considered by 
Ghate and Vatsal \cite{GhateVatsal}, who prove (under certain
technical hypotheses of ``deformation-theoretic nature")
that if the $\Lambda$-adic family $\scrF$ associated to 
a cuspidal eigenform $f$
is primitive and $p$-distinguished, then the associated
$\Lambda$-adic local Galois representation $\rho_{\scrF,p}$
is split split if and only if 
some arithmetic specialization of $\scrF$
has CM \cite[Theorem 13]{GhateVatsal}.
We interpret the $\Lambda$-adic splitting of the ordinary filtration
as follows:

\begin{theorem}\label{SplittingCriterion}
	The short exact sequence $(\ref{CrystallineAnalogue})$ admits
	a $\Lambda_{\s_{\infty}}$-linear splitting which is compatible with $F$, $V$,
	and $\Gamma$ if and only if the ordinary filtration of ${e^*}'H^1_{\et}$ 
	admits a $\Lambda$-linear spitting which is compatible with the action of $\scrG_{\Q_p}$.
\end{theorem}

\subsection{Overview of the article}\label{Overview}

Section \ref{Prelim} is preliminary: we review the integral $p$-adic cohomology theories
of \cite{CaisDualizing} and \cite{CaisNeron}, and summarize the relavant facts concerning 
integral models of modular curves
from \cite{KM} that we will need.  Of particular importance is a description of the 
$U_p$-correspondence in characteristic $p$, due to Ulmer \cite{Ulmer}, and recorded in Proposition \ref{UlmerProp}.

In \S\ref{DiffCharp}, we study the de Rham and crystalline cohomolgy of the Igusa tower, and prove
the key ``freeness and control" theorems that form the technical characteristic $p$ backbone of this paper.  
Via an almost combinatorial argument using the description of $U_p$ in characteristic $p$,
we then relate the cohomology of the Igusa tower to the mod $p$ reduction of the ordinary part
of the (integral $p$-adic) cohomology of the modular tower. 

Section \ref{PhiGammaCrystals} is a summary of the theory developed in \cite{CaisLau},
which uses Dieudonn\'e crystals of $p$-divisible groups to
provide a ``cohomological" construction of the $(\varphi,\Gamma)$-modules attached 
to potentially Barsotti--Tate representations.  It is precisely this theory
which allows us to construct our crystalline analogue of Hida's ordinary $\Lambda$-adic
\'etale cohomology.  

Section \ref{results} constitutes the main body of this paper, and the reader who 
is content to refer back to \S\ref{Prelim}--\ref{PhiGammaCrystals} as needed should skip directly there.
In \S\ref{TowerFormalism}, we develop a 
commutative algebra formalism for working with projective limits of ``towers" of cohomology
that we use frequently in the sequel.  Using the canonical lattices in de Rham 
cohomology studied in \cite{CaisDualizing} (and reviewed in \S\ref{GD}), we construct 
our $\Lambda$-adic de Rham analogue of Hida's ordinary $\Lambda$-adic \'etale cohomology in \S\ref{ordfamdR},
and we show that the expected freeness and control results follow
by reduction to characteristic $p$ from the structure
theorems for the de Rham cohomology of the Igusa tower established in \S\ref{DiffCharp}.  
Using work of Ohta \cite{OhtaEichler}, in \S\ref{ordforms} we relate the Hodge filtration of our $\Lambda$-adic de Rham
cohomology to the module of $\Lambda$-adic cuspforms.
In section \ref{BTfamily}, we study the tower of $p$-divisible groups whose cohomology 
allows us to construct our $\Lambda$-adic Dieudonn\'e and crystalline analogues of
Hida's \'etale cohomlogy in \S\ref{OrdDieuSection} and \S\ref{OrdSigmaSection}, respectively. 
We establish $\Lambda$-adic comparison isomorphisms between each of these cohomologies
using the integral comparison isomorphisms of \cite{CaisNeron} and \cite{CaisLau}, recalled in \S\ref{Universal}
and \S\ref{pDivPhiGamma}, respectively.  This enables us to give a new proof of Hida's freeness and 
control theorems and of Ohta's duality theorem in \S\ref{OrdSigmaSection}.

As remarked in \S\ref{resultsintro}, and following \cite{OhtaEichler} and \cite{MW-Hida}, our construction
of the $\Lambda$-adic Dieudonn\'e and crystalline counterparts to Hida's \'etale cohomology
excludes the trivial eigenspace for the action of $\mu_{p-1}\subseteq \Z_p^{\times}$
so as to avoid technical complications with logarithmic $p$-divisible groups.  
In \cite{Ohta2}, Ohta uses the ``fixed part" (in the sense of Grothendieck \cite[2.2.3]{GroModeles})
of N\'eron models with semiabelian reduction to extend his results on 
$\Lambda$-adic Hodge cohomology to allow trivial tame nebentype character.  
We are confident that by using Kato's logarithmic Dieudonn\'e theory \cite{KatoDegen} 
one can appropriately generalize our results in \S\ref{OrdDieuSection} and \S\ref{OrdSigmaSection} to include
the missing eigenspace for the action of $\mu_{p-1}$.

\subsection{Notation}\label{Notation}

If $\varphi:A\rightarrow B$ is any map of rings, we will often write $M_B:=M\otimes_{A} B$ 
for the $B$-module induced from an $A$-module $M$ by extension of scalars.
When we wish to specify $\varphi$, we will write $M\otimes_{A,\varphi} B$.
Likewise, if $\varphi:T'\rightarrow T$ is any morphism of schemes, for any $T$-scheme $X$
we denote by $X_{T'}$ the base change of $X$ along $\varphi$.
If $f:X\rightarrow Y$ is any morphism of $T$-schemes,
we will write $f_{T'}: X_{T'}\rightarrow Y_{T'}$
for the morphism of $T'$-schemes obtained from $f$ by base change along $\varphi$.
When $T=\Spec(R)$ and $T'=\Spec(R')$ are affine, we abuse notation and write
$X_{R'}$ or $X\times_{R} R'$ for $X_{T'}$.

We will frequently work with schemes over a discrete valuation ring $R$.
We will often write $\X,\Y,\ldots$ for schemes over $\Spec(R)$,
and will generally use $X,Y,\ldots$ (respectively $\o{\X},\o{\Y},\ldots$) 
for their generic (respectively special) fibers.

\subsection{Acknowledgements}

It is a pleasure to thank Laurent Berger, Brian Conrad, Adrian Iovita, Joseph Lipman, Tong Liu, and Romyar Sharifi 
for enlightening conversations and correspondence.  I am especially grateful to Haruzo Hida, and Jacques 
Tilouine for their willingness to answer many questions concerning their work.  
This paper owes a great deal to the work of Masami Ohta, and I heartily thank him
for graciously hosting me during a visit to Tokai University in August, 2009.

\tableofcontents

\section{Preliminaries}\label{Prelim}

This somewhat long section is devoted to recalling the geometric 
background we will need in our constructions.  Much (though not all)
of this material is contained in \cite{CaisDualizing}, \cite{CaisNeron}
and \cite{KM}.

\subsection{Dualizing sheaves and de {R}ham cohomology}\label{GD}

We begin by describing a certain modification of the usual de Rham complex for non-smooth curves.  
The hypercohomology of this (two-term) complex is in general much better behaved than algebraic de Rham 
cohomology and will enable us to construct our $\Lambda$-adic de Rham cohomology.
We largely refer to \cite{CaisDualizing}, but remark that our
treatment here is different in some places and better suited to our purposes.

\begin{definition}\label{curvedef}
	A {\em curve} over a scheme $S$ is a morphism $f:X\rightarrow S$ 
	of finite presentation which is a flat local complete 
	intersection\footnote{That is, a {\em syntomic morphism} in the sense of
	Mazur \cite[II, 1.1]{FontaineMessing}.  Here, we use the definition of l.c.i. given 
	in \cite[Exp. \Rmnum{8}, 1.1]{SGA6}.}
	of pure relative dimension 1 with geometrically reduced fibers.
	We will often say that $X$ is a curve over $S$ or that $X$ is 
	a relative $S$-curve when $f$ is clear from context.
\end{definition}

\begin{proposition}\label{curveproperties}
	Let $f:X\rightarrow S$ be a flat morphism of finite type.  The following are equivalent:
	\begin{enumerate}
		\item The morphism $f:X\rightarrow S$ is a curve.\label{fiscrv}
		\item For every $s\in S$, the fiber $f_s:X_s\rightarrow \Spec k(s)$ is a curve.\label{fiberscrv}
		\item For every $x\in X$ with $s=f(x)$, the local ring 
		$\O_{X_s,x}$ is a complete intersection\footnote{That is, the quotient of a regular local ring by 
		a regular sequence.} and $f$ has geometrically reduced fibers of pure dimension 1.\label{localringcrv}
	\end{enumerate}
	Moreover, any base change of a curve is again a curve.
\end{proposition}

\begin{proof}
	Since $f$ is flat and of finite presentation, the definition of local complete 
	intersection that we are using ({\em i.e.} \cite[Exp. \Rmnum{8}, 1.1]{SGA6}) is equivalent
	to the definition given in \cite[$\mathrm{\Rmnum{4}}_4$, 19.3.6]{EGA} by \cite[Exp. \Rmnum{8}, 1.4]{SGA6};
	the equivalence of (\ref{fiscrv})--(\ref{localringcrv}) follows immediately.  
	The final statement of the proposition is an easy consequence of \cite[$\mathrm{\Rmnum{4}}_4$, 19.3.9]{EGA}.    
\end{proof}

\begin{corollary}\label{curvecorollary}
	Let $f:X\rightarrow S$ be a finite type morphism of pure relative dimension $1$.
	\begin{enumerate}
		\item If $f$ is smooth, then it is a curve.\label{smoothcrv}
		\item If $X$ and $S$ are regular and $f$ has geometrically reduced fibers
		then $f$ is a curve.\label{regcrv}
		\item If $f$ is a curve then it is Gorenstein and hence also Cohen Macaulay.\label{crvCM}
	\end{enumerate}
\end{corollary}
			
\begin{proof}
	The assertion (\ref{smoothcrv}) is obvious, and (\ref{regcrv}) follows from the 
	fact that a closed subscheme of a regular scheme
	is regular if and only if it is defined (locally) by a regular sequence; {\em cf.} \cite[6.3.18]{LiuBook}.
	Finally, (\ref{crvCM})
	follows from Proposition \ref{curveproperties} (\ref{localringcrv}) and 
	the fact that every local ring
	that is a complete intersection is Gorenstein and hence Cohen Macaulay  
	(see, e.g., Theorems 18.1 and 21.3 of \cite{matsumura}).  
\end{proof}

Fix a relative curve $f:X\rightarrow S$.    
We wish to apply Grothendieck duality theory to $f$, so we
henceforth assume that $S$ is a noetherian scheme of finite Krull dimension\footnote{Nagata gives
an example \cite[A1, Example 1]{nagata} of an affine and regular noetherian scheme of infinite Krull dimension,
so this hypotheses is not redundant.} that is Gorenstein and excellent, so that
that $\O_S$ is a dualizing complex for $S$ \cite[V,\S10]{RD}.
Since $f$ is CM by Corollary \ref{curvecorollary} (\ref{crvCM}), 
by \cite[Theorem 3.5.1]{GDBC}) the relative dualizing complex $f^!\O_S$ has a 
unique nonzero cohomology sheaf, which is in degree $-1$,
and we define the {\em relative dualizing sheaf} for $X$ over $S$ (or for $f$) to be:

\begin{equation*}
	\omega_f=\omega_{X/S} := H^{-1}(f^!\O_S).
\end{equation*}
Since the fibers of $f$ are Gorenstein, $\omega_{X/S}$ is an invertible $\O_X$-module by \cite[V, Proposition 9.3, Theorem 9.1]{RD}. The formation of $\omega_{X/S}$ is compatible
with arbitrary base change on $S$ and \'etale localization on $X$ \cite[Theorem 3.6.1]{GDBC}.  

\begin{remark}\label{abstractdualizing}
	Since $S$ is Gorenstein and of finite Krull dimension and $f^!$ carries dualizing complexes for $S$ to dualizing complexes for $X$ 
	(see \cite[\Rmnum{5}, \S8]{RD}), 
	the sheaf $\omega_{X/S}$ (thought of as a complex concentrated in some degree) is a dualizing complex for the abstract scheme $X$.   
\end{remark}

\begin{proposition}\label{canmap}
	Let $X\rightarrow S$ be a relative curve. There is a canonical 
	map of $\O_X$-modules
	\begin{equation}
		\xymatrix{
			{c_{X/S}: \Omega^1_{X/S}} \ar[r] & {\omega_{X/S}}
			}\label{cmap}
	\end{equation}
	whose formation commutes with any base change $S'\rightarrow S$, where
	$S'$ is noetherian of finite Krull dimension, Gorenstein, and excellent.
 	 Moreover, the restriction of $c_{X/S}$ to any $S$-smooth subscheme
	of $X$ is an isomorphism. 
\end{proposition}

\begin{proof}
	See \cite{elzeinapp}, especially Th\'eor\`eme \Rmnum{3}.1, and {\em cf.} \cite[6.4.13]{LiuBook}.
\end{proof}

\begin{definition}\label{complexregdiff}
	We define the two-term $\O_S$-linear complex (of $\O_S$-flat coherent $\O_X$-modules) concentrated in degrees 0 and 1
	\begin{equation}
		\xymatrix{
			{\omega_f^{\bullet}=\omega_{X/S}^{\bullet}:=\O_X} \ar[r]^-{d_S} & {\omega_{X/S}}
			}
	\end{equation}
	where $d_S$ is the composite of the map (\ref{cmap}) and the universal 
	$\O_S$-derivation $\O_X\rightarrow \Omega^1_{X/S}$.  We view $\omega_{X/S}^{\bullet}$
	as a filtered complex via ``{\em la filtration b\^ete}" \cite{DeligneHodge2},
	which provides an exact triangle
\begin{equation}
	\xymatrix{
		{\omega_{X/S}[-1]} \ar[r] & {\omega^{\bullet}_{X/S}} \ar[r] & {\O_X}
		}\label{HodgeFilComplex}
\end{equation}
	in the derived category that we call the {\em Hodge Filtration} of $\omega^{\bullet}_{X/S}$.
\end{definition}

Since $c_{X/S}$ is an isomorphism over the $S$-smooth locus $X^{\sm}$ of $f$ in $X$, the complex $\omega^{\bullet}_{X/S}$ coincides with
the usual de Rham complex over $X^{\sm}$.  Moreover, it follows immediately from 
Proposition \ref{canmap} that the formation of $\omega_{X/S}^{\bullet}$ is compatible with 
any base change $S'\rightarrow S$ to a noetherian scheme $S'$ of finite Krull dimension that is
Gorenstein and excellent.

\begin{definition}
	Let $f:X\rightarrow S$ relative curve over $S$.  For each nonnegative integer $i$, we define
	\begin{equation*}
		\mathscr{H}^i(X/S):=\R^i f_*\omega_{X/S}^{\bullet}.
	\end{equation*}
	When $S=\Spec R$ is affine, we will write $H^i(X/R)$ for the global sections of the $\O_S$-module 
	$\mathscr{H}^i(X/S)$.
\end{definition}

The complex $\omega_{X/S}^{\bullet}$ and its filtration (\ref{HodgeFilComplex})
behave extremely well with respect to duality:

\begin{proposition}\label{GDuality}
	Let $f:X\rightarrow S$ be a proper curve over $S$.  There is a canonical
	quasi-isomorphism 
	\begin{equation}
		\omega_{X/S}^{\bullet} \simeq \R\scrHom_X^{\bullet}(\omega_{X/S}^{\bullet},\omega_{X/S}[-1])
		\label{DualityIsom}
	\end{equation}
	which is compatible with the filtrations on both sides induced by $(\ref{HodgeFilComplex})$.
	In particular:
	\begin{enumerate}
		\item There is a natural quasi-isomorphism
		\begin{equation*}
			\R f_*\omega^{\bullet}_{X/S}\simeq \R\scrHom_X^{\bullet}(\R f_*\omega^{\bullet}_{X/S},\O_S)[-2]			
		\end{equation*}
		which is compatible with the filtrations induced by $(\ref{HodgeFilComplex})$.\label{DualityOnS}
		\item If $\rho:Y\rightarrow X$ is any finite morphism of proper curves over $S$, 
		then there is a canonical quasi-isomorphism
	\begin{equation*}
		\R\rho_*\omega_{Y/S}^{\bullet} \simeq \R\scrHom_X^{\bullet}
		(\R\rho_*\omega_{Y/S}^{\bullet},\omega_{X/S}[-1]).
	\end{equation*}
		that is compatible with filtrations.\label{DualityRho}
		
		\end{enumerate}
\end{proposition}

\begin{proof}
	For the first claim, see the proofs of Lemmas 4.3 and 5.4 in \cite{CaisDualizing}, noting that although
	$S$ is assumed to be the spectrum of a discrete valuation ring and the definition of curve in that paper 	
	differs somewhat from the definition here, the arguments themselves apply {\em verbatim} in our 	
	context.  The assertion (\ref{DualityOnS}) (respectvely (\ref{DualityRho})) follows from this by 
	applying $\R f_*$ (respectively $\R\rho_*$) to both sides of (\ref{DualityIsom}) and appealing 
	to Grothendieck duality
	\cite[Theorem 3.4.4]{GDBC} for the proper map $f$ (respectively $\rho$); see the proofs
	of Lemma 5.4 and Proposition 5.8 in \cite{CaisDualizing} for details.
\end{proof}

In our applications, we need to understand the cohomology $H^i(X/S)$ for a proper
curve $X\rightarrow S$ when $S$ is either the spectrum of a discrete valuation
ring $R$ of mixed characteristic $(0,p)$ or the spectrum of a perfect field.
We now examine each of these situations in more detail.

First suppose that $S:=\Spec(R)$ is the spectrum of a discrete valuation ring $R$ having field of fractions $K$ of characteristic zero and perfect residue field $k$ of characteristic $p>0$, and fix a normal curve $f:X\rightarrow S$ that is proper over $S$ with smooth and geometrically connected
generic fiber $X_K$.   This situation is studied extensively
in \cite{CaisDualizing}, and we content ourselves with a summary 
of the results we will need.
To begin, we recall the following ``concrete" description of the relative dualizing sheaf:

\begin{lemma}\label{ConcreteDualizingDescription}
	Let $i:U\hookrightarrow X$ be any Zariski open subscheme of $X$ 
	whose complement consists of finitely many points of codimension $2$
	$($necessarily in the closed fiber of $X$$)$.  Then the canonical map
	\begin{equation*}
		\xymatrix{
			{\omega_{X/S}} \ar[r] & {i_*i^*\omega_{X/S} \simeq i_*\omega_{U/S}}
			}
	\end{equation*}
	is an isomorphism.  In particular, $\omega_{X/S}\simeq i_*\Omega^1_{U/S}$
	for any Zariski open subscheme $i:U\hookrightarrow X^{\sm}$ whose
	complement consists of finitely many points of codimension two.	
\end{lemma}

\begin{proof}
	The first assertion is \cite[Lemma 3.2]{CaisNeron}. The second
	follows from this, since $X^{\sm}$ contains the generic fiber
	and the generic points of the closed fiber by our definition of curve.
\end{proof}

\begin{proposition}\label{ComplexFunctoriality}
	Let $\rho:Y\rightarrow X$ be a finite morphism of normal and proper $S$-curves.
	\begin{enumerate}
		\item Attached to $\rho$ are natural
		pullback and trace morphisms of complexes 
		\begin{equation*}
			\xymatrix{
				{\rho^*: \omega^{\bullet}_{X/S}} \ar[r] & {\rho_*\omega^{\bullet}_{Y/S}}
			}
			\quad\text{and}\quad
						\xymatrix{
				{\rho_*: \rho_*\omega^{\bullet}_{Y/S}} \ar[r] & {\omega^{\bullet}_{X/S}}
			}
		\end{equation*}
		which are of formation compatible with \'etale localization on $X$ and 
		flat base change on $S$ and
		are dual via the duality of Proposition $\ref{GDuality}$ $(\ref{DualityRho})$.
		\label{FunctorialityProps1}
		
			\item For any $S$-smooth point $y\in {Y}^{\sm}$ with image $x:=\rho(y)$
			that lies in $X^{\sm}$, the induced mappings of complexes of $\O_{X,x}$-modules
			$\omega^{\bullet}_{X/S,x}\rightarrow \omega^{\bullet}_{Y/S,y}$ 
			and  
			$\omega^{\bullet}_{Y/S,y}\rightarrow \omega^{\bullet}_{X/S,x}$ 
			coincide with the usual pullback and trace mappings on de Rham complexes
			attached to the finite flat morphism of smooth schemes 
			$\Spec(\O_{Y,y})\rightarrow \Spec(\O_{X,x})$.\label{FunctorialityProps2}
			
	\end{enumerate}
\end{proposition}

\begin{proof}
	The assertions of (\ref{FunctorialityProps1}) follow from the proofs of Propositions 4.5 and 5.5 
	of \cite{CaisDualizing}, while (\ref{FunctorialityProps2})
	is a straightforward consequence of the very construction of $\rho_*$ and $\rho^*$
	as given in \cite[\S4]{CaisDualizing}.
\end{proof}

Since the generic fiber of $X$ is a smooth and proper curve over $K$, 
the Hodge to de Rham spectral sequence degenerates \cite{DeligneIllusie}, and there
is a functorial short exact sequence of $K$-vector spaces
\begin{equation}
	\xymatrix{
		0\ar[r] & {H^0(X_K,\Omega^1_{X_K/K})} \ar[r] & {H^1_{\dR}(X_K/K)} \ar[r] & {H^1(X_K,\O_{X_K})} \ar[r] & 0
	}\label{HodgeFilCrv}
\end{equation}
which we call the {\em Hodge filtration} of $H^1_{\dR}(X_K/K)$.

\begin{proposition}\label{HodgeIntEx}
	Let $f:X\rightarrow S$ be a normal curve that is proper over $S=\Spec(R)$.  
	\begin{enumerate}
		\item There are natural isomorphisms of free $R$-modules of rank $1$
		\begin{equation*}
			H^0(X/R)\simeq H^0(X,\O_X)\quad\text{and}\quad H^2(X/R)\simeq H^1(X,\omega_{X/S}),
		\end{equation*}	
		which are canonically $R$-linearly dual to each other.

	\item There is a canonical short exact sequence of finite free $R$-modules, 
	which we denote $H(X/R)$,
		\begin{equation*}
			\xymatrix{
					0\ar[r] & {H^0(X,\omega_{X/S})} \ar[r] & {H^1(X/R)} \ar[r] & {H^1(X,\O_X)} \ar[r] & 0
			}
		\end{equation*}
	that recovers the Hodge filtration $(\ref{HodgeFilCrv})$ of $H^1_{\dR}(X_K/K)$ after
	extending scalars to $K$.  \label{CohomologyIntegral}
	
	\item Via the canonical cup-product auto-duality of $(\ref{HodgeFilCrv})$, 
	the exact sequence $H(X/R)$ is naturally isomorphic to
	its $R$-linear dual.\label{CohomologyDuality}
	
	\item The exact sequence $H(X/R)$ is contravariantly
	$($respectively covariantly$)$ functorial in finite morphisms $\rho:Y\rightarrow X$
	of normal and proper $S$-curves via pullback $\rho^*$ $($respectively trace $\rho_*$$)$;
	these morphisms recover the usual pullback and trace mappings on Hodge filtrations after extending scalars 
	to $K$ and are adjoint with respect to the canonical cup-product autoduality of $H(X/R)$
	in $(\ref{CohomologyDuality})$. \label{CohomologyFunctoriality} 
	
	\end{enumerate}
\end{proposition}

\begin{proof}
	By Raynaud's ``{\em crit\`ere de platitude cohomologique}" \cite[Th\'eor\`me 7.2.1]{Raynaud}
	(see also \cite[Proposition 2.7]{CaisDualizing}), our requirement that curves have geometrically 
	reduced fibers implies that $f:X\rightarrow S$ is cohomologically flat.\footnote{In other words, the $\O_S$-module $f_*\O_X$ commutes with 
	arbitrary base change.}
	The proposition now follows from Propositions 5.7--5.8 of \cite{CaisDualizing}.
\end{proof}

We now turn to the case that $S=\Spec(k)$ for a perfect field $k$ and
$f:X\rightarrow S$ is a proper and geometrically connected curve over $k$.
Recall that $X$ is required to be geometrically reduced, so that the
$k$-smooth locus $U:=X^{\sm}$ is the complement of finitely many closed
points in $X$.  

\begin{proposition}\label{HodgeFilCrvk} 
	Let $X$ be a proper and geometrically connected curve over $k$.  
	\begin{enumerate}
		\item There are natural isomorphisms of 1-dimensional $k$-vector spaces 
		\begin{equation*}
			H^0(X/k)\simeq H^0(X,\O_X)\quad\text{and}\quad H^2(X/k)\simeq H^1(X,\omega_{X/k}),
		\end{equation*}	
		which are canonically $k$-linearly dual to each other.\label{H0H2overk}
	
		\item There is a natural short exact sequence, which we denote $H(X/k)$
		\begin{equation*}
			\xymatrix{
				0 \ar[r] & {H^0(X,\omega_{X/k})} \ar[r] & {H^1(X/k)} \ar[r] & 
				{H^1(X,\O_{X})} \ar[r] & 0
			}\label{HodgeDegenerationField}
		\end{equation*}
		which is canonically isomorphic to its own $k$-linear dual.\label{HodgeExSeqk}
	\end{enumerate}
\end{proposition}

\begin{proof}
	Consider the long exact cohomology sequence arising from the exact
	triangle (\ref{HodgeFilComplex}).
	Since $X$ is proper over $k$, geometrically connected and reduced,  
	the canonical map $k\rightarrow H^0(X,\O_X)$ is an isomorphism, and it follows that
	the map $d:H^0(X,\O_X)\rightarrow H^0(X,\omega_{X/k})$ is zero,  
	whence the map $H^0(X/k)\rightarrow H^0(X,\O_X)$ is an isomorphism.
	Thanks to Proposition \ref{GDuality} (\ref{DualityOnS}), we have a canonical
	quasi-isomorphism 
	\begin{equation}
		\R\Gamma(X,\omega_{X/k}^{\bullet})\simeq 
		\R\Hom_k^{\bullet}(\R\Gamma(X,\omega_{X/k}^{\bullet}),k)[-2]\label{GDFieldExplicit}
	\end{equation}
	that is compatible with the filtrations induced by (\ref{HodgeFilComplex}).
	Using the spectral sequence
	\begin{equation*}
		E_2^{m,n}:=\Ext_k(\mathbf{H}^{-n}(X,\omega_{X/k}^{\bullet}))
		\implies H^{m+n}(\R\Hom_k^{\bullet}(\R\Gamma(X,\omega_{X/k}^{\bullet}),k))
	\end{equation*}
	and the vanishing of $\Ext_k^m(\cdot,k)$ for $m>0$,
	we deduce that $H^2(X/k)\simeq H^0(X/k)^{\vee}$ is 1-dimensional over $k$.
	Since Grothendieck's trace map $H^1(X,\omega_{X/k})\rightarrow k$ is an isomorphism,
	we conclude that the {\em surjective} map of 1-dimensional $k$-vector spaces 
	$H^1(X,\omega_{X/k})\rightarrow H^2(X/k)$ must be an isomorphism.  It follows that
	the map $d:H^1(X,\O_X)\rightarrow H^1(X,\omega_{X/k})$ is zero as well, as desired.
	The fact that that the resulting short exact sequence in (\ref{HodgeDegenerationField}) is
	canonically isomorphic to its $k$-linear dual, and the fact that 
	the isomorphisms in (\ref{H0H2overk}) are $k$-linearly dual are
	now easy consequences of the isomorphism (\ref{GDFieldExplicit}).	
\end{proof}

We now suppose that $k$ is algebraically closed, and
following \cite[\S5.2]{GDBC}, we recall Rosenlicht's 
explicit description \cite{Rosenlicht} of the relative dualizing sheaf $\omega_{X/k}$
and of Grothendieck duality.

Denote by $k(X)$ the ``function field" of $X$, {\em i.e.} $k(X):=\prod_i k(\xi_i)$
is the product of the residue fields at the finitely many generic points of
$X$, and write $j:\Spec(k(X))\rightarrow X$ for the canonical map. 
By definition, the {\em sheaf of meromorphic differentials on $X$}
is the pushforward $\u{\Omega}^1_{k(X)/k}:=j_*\Omega^1_{k(X)/k}$.
Our hypothesis that $X$ is reduced implies that it is smooth
at its generic points, so $j$ factors through the open immersion 
$i:U:=X^{\sm}\hookrightarrow X$. By \cite[Lemma 5.2.1]{GDBC}, the canonical map
of $\O_X$-modules
\begin{equation}
	\xymatrix{
		{\omega_{X/k}} \ar[r] & {i_*i^*\omega_{X/k}\simeq i_*\Omega^1_{U/k}}
		}\label{OmegaSubSheaf}
\end{equation}  
is injective, and it follows that $\omega_{X/k}$ is a subsheaf of
$\u{\Omega}^1_{k(X)/k}$.  Rosenlicht's theory gives a concrete description of this subsheaf,
as we now explain.

Let $\pi:\nor{X}\rightarrow X$ be the normalization of $X$. 
We have a natural identification
of ``function fields" $k(\nor{X})=k(X)$ and hence a canonical isomorphism 
$\pi_* \u{\Omega}^1_{k(\nor{X})/k}\simeq \u{\Omega}^1_{k(X)/k}$
of sheaves on $X$.

\begin{definition}\label{OmegaReg}
	Let $\omega_{X/k}^{\reg}$ be the sheaf of $\O_{X}$-modules
	whose sections over any open $V\subseteq X$ are those
	meromorphic differentials $\eta$ on $\pi^{-1}(V)\subseteq \nor{X}$
	which satisfy
	\begin{equation}
		\sum_{y\in \pi^{-1}(x)} \res_y(s\eta)=0
	\end{equation}
	for all $x\in V(k)$ and all $s\in \O_{X,x}$, where $\res_{y}$
	is the classical residue map on meromorphic differentials on
	the smooth (possibly disconnected) curve $\nor{X}$ over the algebraically closed field $k$.	
\end{definition}

\begin{remark}\label{OmegaRegMero}
	Let $\Irr(X)$ be the set of irreducible components of $X$.
	Since $\pi$ is an isomorphism over $U$ and $X$ is smooth at its generic
	points, $\nor{X}$ is the disjoint union of the smooth, proper, and irreducible
	$k$-curves $\nor{I}$ for $I\in \Irr(X)$.
	Therefore, a meromorphic differential $\eta$ on $\nor{X}$
	may be viewed as a tuple
	$\eta = \left(\eta_{\nor{I}}\right)_{I\in \Irr(X)}$,	
	with $\eta_{\nor{I}}$ a meromorphic differential on the smooth and irreducible 
	curve $\nor{I}$.
	The condition for a meromorphic differential $\eta$ on $\pi^{-1}(V)$
	to be a section of $\omega_{X/k}^{\reg}$ over $V$ is then
	\begin{equation*}
		\sum_{y\in \pi^{-1}(x)} \res_y(s_y\eta_{\nor{I}_y}) = 0
	\end{equation*}
	for all $x\in V(k)$ and all $s\in \O_{X,x}$, where $\nor{I}_y$ is the 
	unique connected component of $\nor{X}$ on which $y$ lies and $s_y$
	is the image of $s$ under the canonical map $\O_{X,x}\rightarrow \O_{\nor{I}_y,y}$.
\end{remark}

As any holomorphic differential on $\nor{X}$ has zero residue at every closed point,
the pushforward $\pi_*\Omega^1_{\nor{X}/k}$ is naturally a subsheaf of $\omega_{X/k}^{\reg}$,
and this inclusion is an equality at every $x\in U(k)$ since
$\pi$ is an isomorphism over $U$. It likewise follows from the definition
that any section of $\omega^{\reg}_{X/k}$
must be holomorphic at every smooth point of $X$, so there is a natural inclusion
\begin{equation}
	\xymatrix{
		{\omega^{\reg}_{X/k}} \ar@{^{(}->}[r] & {i_*\Omega^1_{U/k}}
	}\label{OmegaRegIncl}
\end{equation}
which is an isomorphism over $U$.  Moreover, by \cite[Lemma 5.2.2]{GDBC},
any section of $\omega^{\reg}_{X/k}$ has poles at the
finitely many non-smooth points of $X$ with order bounded by a constant depending
only on $X$, and it follows that $\omega^{\reg}_{X/k}$ is a coherent sheaf on $X$.
   
Since (\ref{OmegaRegIncl}) is an isomorphism at the generic points of $X$, we have
a quasi-coherent flasque resolution
\begin{equation*}
	\xymatrix{
		0\ar[r] & {\omega_{X/k}^{\reg}} \ar[r] & {\u{\Omega}^1_{k(X)/k}} \ar[r] &
	{\displaystyle\bigoplus_{x\in X^0} {i_x}_*\left(\u{\Omega}^1_{k(X)/k,x}/\omega_{X/k,x}^{\reg}\right)}
		\ar[r] & 0
	},
\end{equation*}
where $X^0$ is the set of closed points of $X$ and $i_x:\Spec(\O_{X,x})\rightarrow X$
is the canonical map.  The associated long exact cohomology 
sequence yields an exact sequence of $k$-vector spaces
\begin{equation}
	\xymatrix{
		{\Omega^1_{k(X)/k}} \ar[r] &  
		{\displaystyle\bigoplus_{x\in X^0} \left(\u{\Omega}^1_{k(X)/k,x}/\omega_{X/k,x}^{\reg}\right)}
		\ar[r] & {H^1(X,\omega^{\reg}_{X/k})} \ar[r] & 0
	}.\label{TrRegCon}
\end{equation}
For $x\in X^0$, the $k$-linear ``residue" map
\begin{equation*}
	\xymatrix{
		{\res_x:\Omega^1_{k(X)/k,x}} \ar[r] & k
		}
	\quad\text{defined by}\quad
		\res_x(\eta):=\sum_{y\in \pi^{-1}(x)} \res_y(\eta)
\end{equation*}
kills $\omega_{X/k,x}^{\reg}$, and the induced composite map
\begin{equation*}
	\xymatrix{
		{\Omega^1_{k(X)/k}} \ar[r] & 
		{\displaystyle\bigoplus_{x\in X^0} 
		\left(\u{\Omega}^1_{k(X)/k,x}/\omega_{X/k,x}^{\reg}\right)} 
		\ar[r]^-{\sum \res_x} & k
	}
\end{equation*}
is zero by the residue theorem on the (smooth) connected components of 
$\nor{X}$.  Thus, from (\ref{TrRegCon}) we obtain a $k$-linear ``trace map"
\begin{equation}
	\xymatrix{
		{\res_X:H^1(X,\omega^{\reg}_{X/k})} \ar[r] & k
		}\label{ResidueMap}
\end{equation}
which coincides with the usual residue map when $X$ is smooth.
Rosenlicht's explicit description of the relative dualizing sheaf and of Grothendieck
duality for $X/k$ is:

\begin{proposition}[Rosenlicht]\label{Rosenlicht}
	Let $X$ be a proper and geometrically connected curve over $k$ with $k$-smooth
	locus $U$. Viewing $\omega_{X/k}$ and $\omega^{\reg}_{X/k}$
	as subsheaves of $i_*\Omega^1_{U/k}$ via $(\ref{OmegaSubSheaf})$ and 
	$(\ref{OmegaRegIncl})$, respectively, we have an equality
	\begin{equation*}
	  \omega_{X/k}=\omega_{X/k}^{\reg}\quad\text{inside}\quad i_*\Omega^1_{U/k}.
	\end{equation*}
	Under this identification, 
	Grothendieck's trace map $H^1(X,\omega_X)\rightarrow k$
	coincides with $-\res_X$.
\end{proposition}

\begin{proof}
	See \cite[Theorem 5.2.3]{GDBC}.
\end{proof}

We now return to the situation that $S=\Spec(R)$ for a discrete valuation ring $R$ with fraction field $K$
of characteristic zero and perfect residue field $k$ of characteristic $p>0$.

\begin{lemma}\label{ReductionCompatibilities}
	Let $X$ be a normal and proper curve over $S=\Spec(R)$ with smooth and geometrically
	connected generic fiber, and denote by $\o{X}:=X_k$
	the special fiber of $X$; it is a proper and geometrically connected curve over $k$ 
	by Proposition $\ref{curveproperties}$ $(\ref{fiberscrv})$.
	\begin{enumerate}
		\item The canonical base change map
		\begin{equation*}
			\xymatrix{
				0 \ar[r] & {H^0(X,\omega_{X/S})\tens_R k} \ar[r]\ar[d]^-{\simeq} & 
				{H^1(X/R) \tens_R k} \ar[r]\ar[d]^-{\simeq} & 
				{H^1(X,\O_X)\tens_R k} \ar[r]\ar[d]^-{\simeq} & 0\\			
				0\ar[r] & {H^0(\o{X},\omega_{\o{X}/k})} \ar[r] & {H^1(\o{X}/k)} \ar[r] & 
				{H^1(\o{X},\O_{\o{X}})} \ar[r] & 0
			}
		\end{equation*}
		is an isomorphism. \label{BaseChngDiagram}  
		
		\item Let $\rho:Y\rightarrow X$ be a finite morphism of
		normal and proper curves over $S$ with smooth and geometrically connected 
		generic fibers.		
		The canonical diagrams $($one for $\rho^*$ and one for $\rho_*$$)$
		\begin{equation*}	
			\xymatrix{
				{H^0(Y,\omega_{Y/S})\tens_R k} \ar@<0.5ex>[r]^-{\rho_*\otimes 1}\ar[d]_-{\simeq} & 
				\ar@<0.5ex>[l]^-{\rho^*\otimes 1} {H^0(X,\omega_{X/S})\tens_R k} \ar[d]^-{\simeq}\\
				{H^0(\o{Y},\omega_{\o{Y}/k})}\ar@{^{(}->}[d]_-{(\ref{OmegaSubSheaf})} & 
				{H^0(\o{Y},\omega_{\o{Y}/k})}\ar@{^{(}->}[d]^-{(\ref{OmegaSubSheaf})}\\
				{H^0(\nor{\o{Y}}, \Omega^1_{k(\nor{\o{Y}})/k})} \ar@<0.5ex>[r]^-{\nor{\o{\rho}}_*} & 
				\ar@<0.5ex>[l]^-{{\nor{\o{\rho}}}^*} {H^0(\nor{\o{X}}, \Omega^1_{k(\nor{\o{X}})/k})} 
			}
		\end{equation*}
		commute, where ${\nor{\o{\rho}}}^*$ and $\nor{\o{\rho}}_*$ are the usual pullback
		and trace morphisms on meromorphic differential forms associated to the finite
		flat map $\nor{\o{\rho}}:\nor{\o{Y}}\rightarrow \nor{\o{X}}$
		of smooth curves over $k$.\label{PTBCCompat}	
	\end{enumerate}	
\end{lemma}

\begin{proof}
	Since $X$ is of relative dimension 1 over $S$, the cohomologies $H^1(X,\O_X)$ and $H^1(X,\omega_{X/S})$
	both commute with base change, and they are both free over $R$ by Proposition \ref{HodgeIntEx}.
	We conclude that $H^i(X,\O_X)$ and $H^i(X,\omega_{X/S})$ commute with base change for all $i$ and
	hence that the left and right vertical maps in the base change diagram (\ref{BaseChngDiagram}) (whose
	rows are exact by Propositions \ref{HodgeIntEx} and \ref{HodgeFilCrvk}) 
	are isomorphisms.  It follows that the middle vertical map in 
	(\ref{BaseChngDiagram}) is an isomorphism as well.	
	The compatibility of pullback and trace under base change to the special fibers, 
	as asserted by the diagram in (\ref{PTBCCompat}),
	is a straightforward consequence of Proposition \ref{ComplexFunctoriality} 
	(\ref{FunctorialityProps2}), using the facts that $X$ and $Y$ are 
	smooth at generic points of closed fibers and that $\o{\rho}:\o{Y}\rightarrow\o{X}$
	takes generic points to generic points as noted in the proof of Lemma \ref{ConcreteDualizingDescription}.
\end{proof}

\subsection{Universal vectorial extensions and Dieudonn\'e crystals}\label{Universal}
 
There is an alternate description of the short exact sequence $H(X/R)$ of Proposition
\ref{HodgeIntEx} (\ref{CohomologyIntegral}) in terms of 
Lie algebras and N\'eron models of Jacobians that will allow us to relate
this cohomology to Dieudonn\'e modules.  To explain this description and its connection
with crystals, we first recall some facts from \cite{MM} and \cite{CaisNeron}.

Fix a base scheme $T$, and let $G$ be an fppf sheaf of abelian groups over $T$.
A {\em vectorial extension} of $G$ is a short exact sequence (of fppf sheaves of abelian
groups)
\begin{equation}
	\xymatrix{
		0 \ar[r] & {V} \ar[r] & {E} \ar[r] & {G} \ar[r] & 0.
		}\label{extension}
\end{equation}
with $V$ a vector group (i.e. an fppf abelian sheaf which is locally represented by a product of $\Ga$'s).  
Assuming that $\Hom(G,V)=0$ for all vector groups $V$, we say that a vectorial extension 
(\ref{extension}) is {\em universal} if, for any vector group $V'$ over $T$, 
the pushout map $\Hom_T(V,V')\rightarrow \Ext^1_T(G,V')$
is an isomorphism.  When a universal vectorial extension of $G$ exists, it is
unique up to canonical isomorphism and covariantly functorial in morphisms $G'\rightarrow G$
with $G'$ admitting a universal extension.  

\begin{theorem}\label{UniExtCompat}
	Let $T$ be an arbitrary base scheme.
	\begin{enumerate}
		\item If $A$ is an abelian scheme over $T$, then a universal vectorial
		extension $\E(A)$ of $A$ exists, with $V=\omega_{\Dual{A}}$,
		and is compatible with arbitrary base change on $T$.
		\label{UniExtCompat1}
		
		\item If $p$ is locally nilpotent on $T$ and $G$ is a $p$-divisible group over
		$T$, then a universal vectorial extension $\E(G)$ of $G$ extsis, with $V=\omega_{\Dual{G}}$,
		and is compatible with arbitrary base change on $T$.	
	\label{UniExtCompat2}
		
		\item If $p$ is locally nilpotent on $T$ and $A$ is an abelian scheme over $T$ with
		associated $p$-divisible group $G:=A[p^{\infty}]$, then the canonical map of fppf sheaves
		$G\rightarrow A$ extends to a natural map
		\begin{equation*}
			\xymatrix{
				0 \ar[r] & {\omega_{\Dual{G}}} \ar[r]\ar[d] & {\E(G)} \ar[r]\ar[d] & {G}\ar[d] \ar[r] & 0\\
				0 \ar[r] & {\omega_{\Dual{A}}} \ar[r] & {\E(A)} \ar[r] & {A} \ar[r] & 0
			}
		\end{equation*}
		which induces an isomorphism of the corresponding short exact sequences of Lie algebras.
		\label{UniExtCompat3}
	\end{enumerate} 
\end{theorem}

\begin{proof}
	For the proofs of (\ref{UniExtCompat1}) and (\ref{UniExtCompat2}), see
	 \cite[\Rmnum{1}, \S1.8 and \S1.9]{MM}.  To prove (\ref{UniExtCompat3}), note that
	 pulling back the universal vectorial extension of $A$ along $G\rightarrow A$
	 gives a vectorial extension $\E'$ of $G$ by $\omega_{\Dual{A}}$.  By universality, there then exists
	 a unique map $\psi:\omega_{\Dual{G}}\rightarrow \omega_{\Dual{A}}$ with the property
	 that the pushout of $\E(G)$ along $\psi$ is $\E'$, and this gives the map on universal extensions.  
	 That the induced map on Lie algebras is an isomorphism follows from \cite[\Rmnum{2}, \S 13]{MM}.
\end{proof}

For our applications, we will need a generalization of the universal extension 
of an abelian scheme to the setting of N\'eron models; in order to describe this
generalization, we first recall the explicit description of the universal
extension of an abelian scheme in terms of rigidified extensions.

For any commutative $T$-group scheme $F$, 
a {\em rigidified extension of $F$ by $\Gm$ over $T$} is a pair $(E,\sigma)$
consisting of an extension (of fppf abelian sheaves)
\begin{equation}
	\xymatrix{
		0 \ar[r] & {\Gm} \ar[r] & {E} \ar[r] & {F} \ar[r] & 0
		}\label{ExtRigDef}
\end{equation}
and a splitting $\sigma: \Inf^1(F)\rightarrow E$ of the pullback of (\ref{ExtRigDef})
along the canonical closed immersion $\Inf^1(F)\rightarrow F$.  Two rigidified 
extensions $(E,\sigma)$ and $(E',\sigma')$ are equivalent if there 
is a group homomorphism $E\rightarrow E'$ carrying $\sigma$ to $\sigma'$
and inducing the identity on $\Gm$ and on $F$.
The set $\Extrig_T(F,\Gm)$ of equivalence classes of rigidified extensions over $T$ is naturally a group
via Baer sum of rigidified extensions\cite[\Rmnum{1}, \S2.1]{MM}, so the functor on $T$-schemes
$T'\rightsquigarrow \Extrig_{T'}(F_{T'},\Gm)$ is naturally a group functor that is
contravariant in $F$ via pullback (fibered product).
We write $\scrExtrig_T(F,\Gm)$ for the fppf sheaf of abelian groups associated to this functor.

\begin{proposition}[Mazur-Messing]\label{MMrep}
	Let $A$ be an abelian scheme over an arbitrary base scheme $T$. 
	The fppf sheaf $\scrExtrig_T(A,\Gm)$ is represented by a smooth and separated $T$-group scheme, 
	and there is a canonical short exact sequence of smooth group schemes over $T$
	\begin{equation}
		\xymatrix{
			0\ar[r] & {\omega_A} \ar[r] & {\scrExtrig_T(A,\Gm)} \ar[r] & {\Dual{A}} \ar[r] & 0
		}.\label{univextabelian}
	\end{equation}
	Furthermore, $(\ref{univextabelian})$ is naturally isomorphic to the universal extension of $\Dual{A}$ by a vector
	group.
\end{proposition}

\begin{proof}
	See \cite{MM}, $\Rmnum{1}, \S2.6$ and Proposition 2.6.7.
\end{proof}

In the case that $T=\Spec R$ for $R$ a discrete valuation ring of mixed
characteristic $(0,p)$ with fraction field $K$,
we have the following genaralization of Proposition \ref{MMrep}:

\begin{proposition}	
	Let $A$ be an abelian variety over $K$, with dual abelian variety $\Dual{A}$, and
	write $\A$ and $\Dual{\A}$ for the N\'eron models of $A$ and $\Dual{A}$ over $T=\Spec(R)$.
	Then the fppf abelian sheaf $\scrExtrig_T(\A,\Gm)$ on the category of smooth $T$-schemes
	is represented by a smooth and separated $T$-group scheme.  Moreover, there
	is a canonical short exact sequence of smooth group schemes over $T$
	\begin{equation}
		\xymatrix{
			0\ar[r] & {\omega_{\A}} \ar[r] & {\scrExtrig_T(\A,\Gm)} \ar[r] & {\Dual{\A}^0} \ar[r] & 0
		}\label{NeronCanExt}
	\end{equation}
	which is contravariantly functorial in $A$ via homomorphisms of abelian varieties over $K$.
	The formation of $(\ref{NeronCanExt})$ is compatible with smooth base change on $T$; in particular,
	the generic fiber of $(\ref{NeronCanExt})$ is the universal extension of $\Dual{A}$ by a vector group.
\end{proposition}

\begin{proof}
	Since $R$ is of mixed characteristic $(0,p)$ with perfect residue field,
    this follows from Proposition 2.6 and the discussion following Remark 2.9 in \cite{CaisNeron}.
\end{proof}

In the particular case that $A$ is the Jacobian of a smooth, proper and geometrically
connected curve $X$ over $K$ which is the generic fiber of a normal proper curve $\X$
over $R$, we can relate the exact sequence of Lie algebras attached to (\ref{NeronCanExt})
to the exact sequence $H(X/R)$ or Proposition \ref{HodgeIntEx} (\ref{CohomologyIntegral}):

\begin{proposition}	\label{intcompare}
	Let $\X$ be a proper relative curve over $T=\Spec(R)$ with smooth generic fiber $X$ over $K$.
	Write $J:=\Pic^0_{X/K}$ for the Jacobian of $X$ and $\Dual{J}$ for its dual, 
	and let $\J$, $\Dual{\J}$ be the corresponding N\'eron models over $R$.
		There is a canonical homomorphism of exact sequences of finite free $R$-modules
			\begin{equation}
			\begin{gathered}
				\xymatrix{
					0 \ar[r] & {\Lie\omega_{\J}} \ar[r]\ar[d] & {\Lie\scrExtrig_T(\J,\Gm)} \ar[r]\ar[d]
					& {\Lie \Dual{\J}^0}  \ar[r]\ar[d] & 0\\
					0 \ar[r] & {H^0(\X,\omega_{\X/T})} \ar[r] & {H^1(\X/R)} \ar[r] & {H^1(\X,\O_{\X})} 
					\ar[r] & 0
			}
			\end{gathered}\label{IntegralComparisonMap}
			\end{equation} 
		that is an isomorphism when $\X$ has rational singularities.\footnote{Recall that $\X$ is 
		said to have {\em rational singularities} if it admits a resolution of singularities 
		$\rho:\X'\rightarrow \X$ with the natural map $R^1\rho_*\O_{{\X'}}=0$.  Trivially, any 
		regular $\X$ has rational singularities.}	
		For any finite morphism $\rho:\Y \rightarrow \X$ of $S$-curves satisfying the above hypotheses,
		the map $(\ref{IntegralComparisonMap})$ intertwines $\rho_*$ 
		$($respectively $\rho^*$$)$ on the bottom row with $\Pic(\rho)^*$ 
		$($respectively $\Alb(\rho)^*$$)$ on the top.
\end{proposition}	
	
\begin{proof}
	 	See Theorem 1.2 and (the proof of) Corollary 5.6 in \cite{CaisNeron}.
\end{proof}

\begin{remark}\label{canonicalproperty}
	Let $X$ be a smooth and geometrically connected curve over $K$
	admitting a normal proper model $\X$ over $R$ that is a curve
	having rational singularities.
	It follows from Proposition \ref{intcompare}
	and the N\'eron mapping property
	that $H(\X/R)$ is a {\em canonical integral structure}
	on the Hodge filtration (\ref{HodgeFilCrv}): it is
	independent of the choice of proper model $\X$ that is normal with rational singularities, 
	and
	is functorial in finite morphisms $\rho:Y\rightarrow X$
	of proper smooth curves over $K$ which admit models over $R$ satisfying these hypotheses.
	These facts can be proved in greater generality by appealing to resolution of singularities
	for excellent surfaces and the flattening techniques of Raynaud--Gruson \cite{RayGrus};
	see \cite[Theorem 5.11]{CaisDualizing} for details.
\end{remark}

We will need to relate universal extensions of $p$-divisible to their Dieudonn\'e crystals.  
In order to explain how this goes, we begin by recalling some basic facts
from crystalline Dieudonn\'e theory, as discussed in \cite{BBM}.

Fix a perfect field $k$ and set $\Sigma:=\Spec(W(k))$, considered as a PD-scheme via the canonical divided powers on the ideal $pW(k)$. Let $T$ be 
a $\Sigma$-scheme on which $p$ is locally nilpotent (so $T$ is naturally a PD-scheme over $\Sigma$), and 
denote by $\Cris(T/\Sigma)$ the big crystalline site of $T$ over $\Sigma$,
endowed with the {\em fppf} topology (see \cite[\S 2.2]{BBM1}). 
If $\scrF$ is a sheaf on $\Cris(T/\Sigma)$ and $T'$ is any PD-thickening of $T$, 
we write $\scrF_{T'}$ for the associated {\em fppf} sheaf on $T'$.
As usual, we denote by $i_{T/\Sigma}:T_{fppf}\rightarrow (T/\Sigma)_{\Cris}$
the canonical morphism of topoi, and 
we abbreviate $\underline{G}:={i_{T/\Sigma}}_{*}G$ for any fppf sheaf $G$ on $T$.

Let $G$ be a $p$-divisible group over $T$, considered as an fppf abelian sheaf on $T$.
As in \cite{BBM}, we define the (contravariant) {\em Dieudonn\'e crystal of $G$ over $T$} to be
\begin{equation}
	\D(G) := \scrExt^1_{T/\Sigma}(\underline{G},\O_{T/\Sigma}).\label{DieudonneDef}
\end{equation}
It is a locally free crystal in $\O_{T/\Sigma}$-modules, which is contravariantly functorial
in $G$ and of formation compatible with base change along PD-morphisms $T'\rightarrow T$ of $\Sigma$-schemes
thanks to 2.3.6.2 and Proposition 2.4.5 $(\rmnum{2})$ of \cite{BBM}.
If $T'=\Spec(A)$ is affine, we will simply write $\D(G)_A$ for the finite locally free $A$-module 
associated to $\D(G)_{T'}$.

The structure sheaf $\O_{T/\Sigma}$ is canonically an extension of $\u{\mathbf{G}}_a$
by the PD-ideal $\J_{T/\Sigma}\subseteq \O_{T/\Sigma}$, and by applying $\scrHom_{T/\Sigma}(\underline{G},\cdot)$
to this extension one obtains (see Propositions 3.3.2 and 3.3.4 as well as
Corollaire 3.3.5 of \cite{BBM})
a short exact sequence (the {\em Hodge filtration})
\begin{equation}
	\xymatrix{
		0\ar[r] & {\scrExt^1_{T/\Sigma}(\underline{G},\J_{T/\Sigma})}\ar[r] &		
		{\D(G)}\ar[r] &		
		{\scrExt^1_{T/\Sigma}(\underline{G},\u{\mathbf{G}}_a)}\ar[r] & 0		
	}\label{HodgeFilCrys}
\end{equation}
that is contravariantly functorial
in $G$ and of formation compatible with base change along PD-morphisms
$T'\rightarrow T$ of $\Sigma$-schemes.
The following ``geometric" description of the value of (\ref{HodgeFilCrys}) on a PD-thickening of the
base will be essential for our purposes:

\begin{proposition}\label{BTgroupUnivExt}
	Let $G$ be a fixed $p$-divisible group over $T$ and let $T'$ be any 
	$\Sigma$-PD thickening of $T$. If $G'$ is any lifting of $G$ to a $p$-divisible
	group on $T'$, then there is a natural isomorphism 
	\begin{equation*}
		\xymatrix{
			0 \ar[r] & {\omega_{G'}} \ar[r]\ar[d]^-{\simeq} & {\scrLie(\E(\Dual{G'}))} \ar[r]\ar[d]^-{\simeq} & 
			{\scrLie (\Dual{G'})}\ar[r]\ar[d]^-{\simeq} & 0\\		
			0\ar[r] & {\scrExt^1_{T/\Sigma}(\underline{G},\J_{T/\Sigma})_{T'}}\ar[r] & {\D(G)_{T'}}\ar[r] &		
			{\scrExt^1_{T/\Sigma}(\underline{G},\underline{\mathbf{G}}_a)_{T'}}\ar[r] & 0		
		}
	\end{equation*} 
	that is moreover compatible with base change in the evident manner.
\end{proposition}

\begin{proof}
	See \cite[Corollaire 3.3.5]{BBM} and \cite[\Rmnum{2}, Corollary 7.13]{MM}.
\end{proof}

\begin{remark}\label{MessingRem}
	In his thesis \cite{Messing}, Messing showed that the Lie algebra of the universal extension
	of $\Dual{G}$ is ``crystalline in nature" and used this as the {\em definition}\footnote{Noting 
	that it suffices to define the crystal $\D(G)$ on $\Sigma$-PD thickenings $T'$ 
	of $T$ to which $G$ admits a lift.} of $\D(G)$.
(See chapter $\Rmnum{4}$, \S2.5 of \cite{Messing} and especially 2.5.2).  Although we
	prefer the more intrinsic description (\ref{DieudonneDef}) of 
	\cite{MM} and \cite{BBM}, it is ultimately Messing's original
	definition that will be important for us.  
\end{remark}

\subsection{Integral models of modular curves}\label{tower}

We record some basic facts about integral models of modular curves that will be needed in what follows.
We assume that the reader is familiar with \cite{KM}, and will freely use the notation and terminology 
therein. Throughout, we fix a prime $p$ and a positive integer $N$ not divisible by $p$.

\begin{definition}
	Let $r$ be a nonnegative integer and $R$ a ring containing a fixed choice $\zeta$ of 
	primitive $p^r$-th root of unity
	in which $N$ is invertible.
	The moduli problem 
	$\scrP_{r}^{\zeta}:=([\bal\ \Gamma_1(p^r)]^{\zeta\can}; [\mu_N]])$ 
	on $(\Ell/R)$
	assigns to $E/S$ the set of quadruples $(\phi:E\rightarrow E',P,Q ; \alpha)	$
	where:
	\begin{enumerate}
		\item $\phi:E\rightarrow E'$ is a $p^r$-isogeny.
		\item $P\in \ker\phi(S)$ and $Q\in \ker\phi^t(S)$ are generators of
		$\ker\phi$ and $\ker\phi^t$, respectively, 
		which pair to $\zeta$ under the canonical pairing
		$\langle\cdot,\cdot\rangle_{\phi}: \ker\phi\times\ker\phi^t\rightarrow \mu_{\deg\phi}$
		\cite[\S2.8]{KM}.
		\item  $\alpha:\mu_N\hookrightarrow E[N]$ is a closed immersion of $S$-group schemes.
	\end{enumerate}	
\end{definition}

\begin{proposition}\label{XrRepresentability}
	If $N \ge 4$, then the moduli problem $\scrP_{r}^{\zeta}$ is represented
	by a regular scheme $\M(\scrP_r^{\varepsilon})$ that is flat of pure relative dimension
	$1$ over $\Spec(R)$.  The moduli scheme $\M(\scrP_r^{\zeta})$ admits a canonical
	compactification $\o{\M}(\scrP_r^{\zeta})$, which is regular and proper flat of pure
	relative dimension $1$ over $\Spec(R)$.
\end{proposition}

\begin{proof}
	Using that $N$ is a unit in $R$, one first shows that for $N\ge 4$, the moduli problem
	$[\mu_N]$ on $(\Ell/R)$ is representable over $\Spec(R)$ and finite \'etale;
	this follows from  2.7.4, 3.6.0, 4.7.1 and 5.1.1 of \cite{KM}, as $[\mu_N]$
	is isomorphic to $[\Gamma_1(N)]$ over any $R$-scheme containing a fixed choice 
	of primitive $N$-th root of unity (see also \cite[8.4.11]{KM}).  By \cite[4.3.4]{KM}, 
	to prove the first assertion it is then enough to show that
	$[\bal \Gamma_1(p^r)]^{\zeta\can}$ on $(\Ell/R)$ is relatively representable and
	regular, which (via \cite[9.1.7]{KM}) is a consequence of \cite[7.6.1 (2)]{KM}.
	For the second assertion, see \cite[\S8]{KM}.
\end{proof}

Recall that we have fixed a compatible sequence $\{\varepsilon^{(r)}\}_{r\ge1}$
of primitive $p^r$-th roots of unity in $\o{\Q}_p$.

\begin{definition}\label{XrDef}
	We set $\X_r:=\o{\M}(\scrP_r^{\varepsilon^{(r)}})$, viewed as a scheme over $T_r:=\Spec(R_r)$.
\end{definition}

There is a canonical action of $\Z_p^{\times}\times (\Z/N\Z)^{\times}$
by $R_r$-automorphisms of $\X_r$,
defined at the level of the underlying moduli problem by
\begin{equation}
		{(u,v)\cdot (\phi:E\rightarrow  E',P,Q; \alpha)} :={(\phi:E\rightarrow E',uP, u^{-1}Q; \alpha\circ v)}
	\label{balcanaction}
\end{equation}
as one checks by means of the computation
	$\langle uP,u^{-1}Q\rangle_{\phi} = \langle P,Q\rangle^{uu^{-1}}_{\phi} = \langle P,Q \rangle_{\phi}$.
Here, we again write $v:\mu_N\rightarrow\mu_N$ for the automorphism of $\mu_N$ 
functorially defined by $\zeta \mapsto \zeta^v$ for any $N$-th root of unity $\zeta$. 
We refer to this action of $\Z_p^{\times}\times (\Z/N\Z)^{\times}$ as the {\em diamond operator}
action, and will denote by $\langle u \rangle$ (respectively $\langle v \rangle_N$) the 
automorphism induced by $u\in \Z_p^{\times}$ (respectively $v\in (\Z/N\Z)^{\times}$).

There is also an $R_r$-semilinear ``geometric inertia" action of $\Gamma:=\Gal(K_{\infty}/K_0)$
on $\X_r$, which allows us to descend the generic fiber of $\X_r$ to $K_0$.
To explain this action,
for $\gamma \in \Gamma$ and any $T_r$-scheme $T'$, let us write 
$T'_{\gamma}$ for the base change of $T'$ along the morphism $T_r\rightarrow T_r$
induced by $\gamma\in \Aut(R_r)$.  
There is a canonical functor $(\Ell/(T_r)_{\gamma})\rightarrow (\Ell/T_r)$ obtained by viewing an elliptic curve
over a $(T_r)_{\gamma}$-scheme $T'$ as 
the same elliptic curve over the same base $T'$,
viewed as a $T_r$-scheme via the projection $(T_r)_{\gamma}\rightarrow T_r$.
For a moduli problem $\scrP$ on $(\Ell/T_r)$, we denote by $\gamma^*\scrP$
the moduli problem on $(\Ell/(T_r)_{\gamma})$ obtained by composing $\scrP$ with this functor; 
see \cite[4.1.3]{KM}.
Each $\gamma\in \Gamma$ gives rise to 
a morphism of moduli problems $\gamma: \scrP_r^{\varepsilon^{(r)}}\rightarrow \gamma^*\scrP_r^{\varepsilon^{(r)}}$
via
\begin{equation}
		{\gamma(\phi:E\rightarrow E',P,Q; \alpha)} := 
		{(\phi_{\gamma}:E_{\gamma}\rightarrow E'_{\gamma},\chi(\gamma)^{-1}P_{\gamma},Q_{\gamma}; \alpha_{\gamma})}
	\label{gammamapsModuli}
\end{equation}
where the subscript of $\gamma$ means ``base change along $\gamma$" (see \S\ref{Notation}).
Since 
\begin{equation*}
	\langle \chi(\gamma)^{-1}P_{\gamma}, Q_{\gamma}\rangle_{\phi_{\gamma}} = 
	\gamma\langle P,Q\rangle_{\phi}^{\chi(\gamma)^{-1}} = \langle P,Q \rangle_{\phi}
\end{equation*}
this really is a morphism of moduli problems on $(\Ell/T_r)$.  We thus obtain
a morphism of $T_r$-schemes
\begin{equation}
	\xymatrix{
		{\gamma:\X_r} \ar[r] & {(\X_r)_{\gamma}}
		}\label{gammamaps}
\end{equation}
for each $\gamma\in \Gamma$, compatibly with change in $\gamma$.  
The induced semilinear action of $\Gamma$ on the generic fiber of ${\X_r}$ 
provides a descent datum with respect to the canonical map $\Spec(K_r)\rightarrow \Spec(K_0)$,
which is necessarily effective as this map is \'etale.  Thus, there is a unique
scheme $X_r$ over $K_0=\Q_p$ with $(X_r)_{K_r}\simeq (\X_r)_{K_r}$;
as the diamond operators visibly commute with the action of $\Gamma$, they 
act on $X_r$ by $\Q_p$-automorphisms in a manner that is compatible with this identification.

\begin{remark}\label{genfiberrem}
We may identify $X_r$ with the base change to $\Q_p$ of the 
modular curve $X_1(Np^r)$ over $\Q$ classifying pairs $(E,\alpha)$
of a generalized elliptic curve $E/S$ together with an embedding of $S$-group schemes
$\alpha:\mu_{Np^r}\hookrightarrow E^{\sm}$ whose image meets each irreducible
component in every geometric fiber.  
If instead we were to use the geometric inertia action on $\X_r$ induced by
	\begin{equation*}
		{\gamma(\phi:E\rightarrow E',P,Q; \alpha)} := 
		{(\phi_{\gamma}:E_{\gamma}\rightarrow E'_{\gamma},P_{\gamma},\chi(\gamma)^{-1}Q_{\gamma}; \alpha_{\gamma})},
	\end{equation*}
	then the resulting descent $X_r'$ of the generic fiber of $\X_r$ to $\Q_p$ would be canonically
	isomorphic to the base change to $\Q_p$ of the modular curve $X_1(Np^r)'$ over $\Q$
	classifying generalized elliptic curves $E/S$ with an embedding of $S$-group schemes
	$\Z/Np^r\Z\hookrightarrow E^{\sm}[Np^r]$ whose image meets each irreducible
	component in every geometric fiber.  Of course, $X_1(Np^r)$ (respectively $X_1(Np^r)'$)
	is the canonical model of the upper half-plane quotient $\Gamma_1(Np^r)\backslash \h^*$ 
	with $\Q$-rational cusp cusp $i\infty$ (respectively $0$).
\end{remark}

Recall (\cite[\S6.7]{KM}) that over any base scheme $S$, a cyclic $p^{r+1}$-isogeny of elliptic curves 
$\phi:E\rightarrow E'$  admits a ``standard factorization" (in the sense of \cite[6.7.7]{KM})
\begin{equation}
	\xymatrix{
		E=:E_0 \ar[r]^-{\phi_{0,1}} & E_1 \ar[r] \cdots & E_{r} \ar[r]^-{\phi_{r,r+1}} & E_{r+1}:=E'
		}.\label{standardFac}
\end{equation}
For each pair of nonnegative integers $a<b\le r+1$ we will write
$\phi_{a,b}$ for the composite $\phi_{a,a+1}\circ\cdots\circ\phi_{b-1,b}$
and $\phi_{b,a}:=\phi_{a,b}^t$ for the dual isogeny.  Using this notion, we define
``degeneracy maps" $\pr,\ps:\X_{r+1} \rightrightarrows \X_r$ (over the map
$T_{r+1}\rightarrow T_r$) at the level of underlying moduli problems
as follows ({\em cf.}: \cite[11.3.3]{KM}):
\begin{equation}
\begin{aligned}
		&{\pr(\phi:E_0\rightarrow E_{r+1},P,Q; \alpha)}:=
		{(\phi_{0,r}:E_0\rightarrow E_{r},pP,\phi_{r+1,r}(Q); \alpha)}\\
		&{\ps(\phi: E_0\rightarrow E_{r+1},P,Q; \alpha)}:=
		{(\phi_{1,r+1}:E_1\rightarrow E_{r+1},\phi_{0,1}(P),pQ; \phi_{0,1}\circ \alpha)}
\end{aligned}\label{XrDegen}
\end{equation} 
By the universal property of fiber products, we obtain morphisms
$T_{r+1}$-schemes
\begin{equation}
	\xymatrix{
		{\X_{r+1}} \ar@<0.5ex>[r]^-{\pr}\ar@<-0.5ex>[r]_-{\ps} &  {\X_r\times_{T_r} T_{r+1}}
	}.\label{rdegen}
\end{equation}
that are compatible with the diamond operators and
the geometric inertia action of $\Gamma$.

\begin{remark}
 On generic fibers, the morphisms (\ref{rdegen}) uniquely descend to degeneracy mappings
 $\pr,\ps:X_{r+1}\rightrightarrows X_r$ of smooth curves over $\Q_p$.  Under the identification
 $X_r\simeq X_1(Np^r)_{\Q_p}$ of Remark \ref{genfiberrem}, the map $\pr$ corresponds
 to the ``standard" projection, induced by ``$\tau\mapsto \tau$" on the complex upper half-plane,
 whereas $\ps$ corresponds to the morphism induced by ``$\tau\mapsto p\tau$."
\end{remark}

Recall that we have fixed a choice of primitive $N$-th root of unity $\zeta_N$ in $\o{\Q}_p$.
The Atkin Lehner ``involution" $w_{\zeta_N}$ on $\X_r\times_{R_r} R_r'$ is defined as
in \cite[\S8]{pAdicShimura}.  Following \cite[11.3.2]{KM}, we define the Atkin Lehner
automorphism $w_{\varepsilon^{(r)}}$ of $\X_r$ over $R_r$ on the
underlying moduli problem $\scrP_r^{\varepsilon^{(r)}}$ as
\begin{equation*}
	{w_{\varepsilon^{(r)}} (\phi:E\rightarrow E',P,Q; \alpha)} :=
		{(\phi^t:E'\rightarrow E,-Q,P;\ \phi\circ\alpha )}
		\label{AtkinLehnerInv}
\end{equation*}
We then define $w_r:=w_{\varepsilon^{(r)}} \circ w_{\zeta_N}=w_{\zeta_N}\circ w_{\varepsilon^{(r)}}$; it is an automorphism of $\X_r \times_{R_r} R_r'$
over $R_r':=R_r[\mu_N]$.

\begin{proposition}\label{ALinv}
	For all $(u,v)\in \Z_p^{\times}\times(\Z/N\Z)^{\times}$ and all 
	$\gamma\in \Gal(K_{\infty}'/K_0)$, the identities
	\begin{align*}
		w_r \langle u\rangle \langle v\rangle_N &= \langle v\rangle_N^{-1}\langle u^{-1}\rangle w_r \\
		(\gamma^*w_r)\gamma & = \gamma w_r\langle \chi(\gamma)\rangle^{-1}\langle a(\gamma)\rangle_N^{-1}\\
		w_r^2 &= \langle -p^r\rangle_N\langle -N\rangle \\
		\pr w_{r+1} &= w_{r}\ps\\
		\ps w_{r+1} &= \langle p\rangle_N w_r\pr
	\end{align*}
	hold, with $a:\Gal(K_{\infty}'/K_0)\rightarrow (\Z/N\Z)^{\times}$ the character determined
	by $\gamma\zeta=\zeta^{a(\gamma)}$ for all $\zeta\in \mu_N(\Qbar_p)$.
\end{proposition}

\begin{proof}
	This is an easy consequence of definitions.
\end{proof}

In order to describe the special fiber of $\X_r$, we must first introduce Igusa curves:

\begin{definition}
	Let $r$ be a nonnegative integer.  The moduli problem $\I_r:=([\Ig(p^r)]; [\mu_N])$ on 
	$(\Ell/\F_p)$ assigns to $(E/S)$ the set of triples $(E,P;\alpha)$ where $E/S$ is an elliptic curve
	and
	\begin{enumerate}
		\item $P\in E^{(p^r)}(S)$ is a point that generates the $r$-fold iterate of 
		Verscheibung $V^{(r)}:E^{(p^r)}\rightarrow E$.   
		\item $\alpha:\mu_N\hookrightarrow E[N]$ is a closed immersion of $S$-group schemes.
	\end{enumerate}
\end{definition}

\begin{proposition}
	If $N\ge 4$, then the moduli problem $\I_r$ on $(\Ell/\F_p)$ is represented by a smooth 
	affine curve $\M(\I_r)$ over $\F_p$ which admits a canonical smooth 
	compactification $\o{\M}(\I_r)$.
\end{proposition}

\begin{proof}
	One argues as in the proof of Proposition \ref{XrRepresentability}, using  \cite[12.6.1]{KM}
	to know that $[\Ig(p^r)]$ is relatively representable on $(\Ell/\F_p)$,
	regular 1-dimensional and finite flat over $(\Ell/\F_p)$.
\end{proof}

\begin{definition}\label{IgusaDef}
	Set $\Ig_r:=\o{\M}(\I_r)$; it is a smooth, proper, and geometrically connected $\F_p$-curve.
\end{definition}

There is a canonical action of the diamond operators $\Z_p^{\times}\times (\Z/N\Z)^{\times}$ on the moduli problem 
$\I_r$ via $(u,v)\cdot (E,P; \alpha):= (E,uP; v\circ\alpha)$; this induces a corresponding action on $\Ig_r$
by $\F_p$-automorphisms.  We again write $\langle u\rangle$ (respectively $\langle v\rangle_N$) for the
action of $u\in \Z_p^{\times}$ (respectively $v\in (\Z/N\Z)^{\times}$).  
Thanks to the ``backing up theorem" \cite[6.7.11]{KM}, one also has natural degeneracy maps
\begin{equation}
	\xymatrix{
		{\pr:\Ig_{r+1}} \ar[r] & {\Ig_r}
	}\qquad\text{induced by}\qquad
	\pr(E,P;\alpha):= (E,VP,\alpha)
	\label{Vmapsch}
\end{equation}
on underlying moduli problems.  This map is visibly
equivariant for the diamond operator action on source and target.
Let $\SS_r$ be the (reduced) closed subscheme of $\Ig_r$
that is the support of the coherent ideal sheaf of relative differentials $\Omega^1_{\Ig_r/\Ig_0}$;
over the unique degree 2 extension of $\F_p$, this scheme breaks up as a disjoint union of rational 
points---the supersingular points. The map (\ref{Vmapsch})
is finite of degree $p$, generically \'etale and totally (wildly) ramified over
each supersingular point. 

We can now describe the special fiber of $\X_r$:

\begin{proposition}	\label{redXr}
	The scheme $\o{\X}_r:=\X_r\times_{T_r} \Spec(\F_p)$ is the disjoint union, with crossings at the supersingular points, of the 		following proper, smooth $\F_p$-curves:
	for each pair $a,b$ of nonnegative integers with $a+b=r$, and for each $u\in (\Z/p^{\min(a,b)}\Z)^{\times}$, 
	one copy of 
	$\Ig_{\max(a,b)}$.
\end{proposition} 

We refer to \cite[13.1.5]{KM} for the definition of ``disjoint union with crossings at the supersingular points".
Note that the special fiber of $\X_r$ is (geometrically) reduced; this will be crucial in our later work.  We often write $I_{(a,b,u)}$ for the irreducible component of $\o{\X}_r$ indexed by the triple $(a,b,u)$ and will refer to it as the {\em $(a,b,u)$-component} (for fixed $(a,b)$ we have $I_{(a,b,u)}=\Ig_{\max(a,b)}$ for all $u$).

For the proof of Proposition \ref{redXr}, we refer the reader to \cite[13.11.2--13.11.4]{KM},
and content ourselves with recalling
the correspondence between (non-cuspidal) points of the $(a,b,u)$-component and 
$[\bal \Gamma_1(p^r)]^{1\can}$-structures on elliptic curves.\footnote{Note that under the canonical ring homomorphism 
$R_r\twoheadrightarrow \F_p$, our fixed choice $\varepsilon^{(r)}$ of primitive $p^r$-th root of unity
maps to $1\in \F_p$, which {\em is} a primitive $p^r$-th root of unity by definition \cite[9.1.1]{KM},
as it is a root of the $p^r$-th cyclotomic polynomial over $\F_p$!}

Let $S$ be any $\F_p$ scheme, fix an ordinary elliptic curve $E_0$ over $S$, and
let $(\phi:E_0\rightarrow E_r,P,Q; \alpha)$ be an element of $\scrP_{r}^1(E_0/S)$.
By \cite[13.11.2]{KM}, there exist unique nonnegative integers $a,b$
with the property that the cyclic $p^r$-isogeny $\phi$ factors as a purely inseparable cyclic $p^a$-isogeny
followed by an \'etale $p^b$-isogeny (this {\em is} the standard factorization of $\phi$).  
Furthermore, there exists a unique elliptic curve $E$ over $S$ and $S$-isomorphisms
$E_0\simeq E^{(p^b)}$ and $E_r\simeq E^{(p^a)}$ such that the cyclic $p^r$ isogeny $\phi$ is: 
\begin{equation*}
	\xymatrix{
		{E_0\simeq E^{(p^b)}} \ar[r]^-{F^a} & {E^{(p^r)}} \ar[r]^-{V^b} & {E^{(p^a)} \simeq E_r}
		}
\end{equation*}
and $P\in E^{(p^b)}(S)$ (respectively $Q\in E^{(p^a)}$) is an Igusa structure of level $p^b$ (respectively $p^a$) on $E$ over $S$.
When $a\ge b$ there is a unique unit $u\in (\Z/p^{b}\Z)^{\times}$ such that $V^{a-b}(Q)=uP$ in $E^{(p^b)}(S)$ and when
$b\ge a$ there is a unique unit $u\in (\Z/p^a\Z)^{\times}$ such that $uV^{b-a}(P)=Q$ in $E^{(p^a)}(S)$.  
Thus, for $a\ge b$ (respectively $b\ge a$) and fixed $u$, the data $(E,Q;p^{-b}V^b\circ\alpha)$ 
(respectively $(E,P;p^{-b}V^b\circ\alpha)$) 
gives an $S$-point of the $(a,b,u)$-component
$\Ig_{\max(a,b)}$.  
Conversely, suppose given $(a,b,u)$ and an $S$-valued point of $\Ig_{\max(a,b)}$ which is neither a cusp nor a supersingular point (in the sense that it corresponds to an ordinary elliptic curve with extra structure).
If $a\ge b$ and $(E,Q;\alpha)$ is the given $S$-point of $\Ig_{a}$ then we set
$P:=u^{-1}V^{a-b}(Q)$, while if $b\ge a$ and $(E,P;\alpha)$ is the given $S$-point of $\Ig_b$ then we set $Q:=uV^{b-a}P$.  Due to \cite[13.11.3]{KM}, the data
\begin{equation*}
	(\xymatrix{
		{E^{(p^b)}} \ar[r]^-{F^a} & {E^{(p^r)}} \ar[r]^-{V^b} & {E^{(p^a)}, P,Q; F^b\circ\alpha} 
	})
\end{equation*}
gives an $S$-point of $\M(\scrP_r^{1})$.  These constructions are visibly inverse to each other.

\begin{remark}\label{rEvenChoice}
	When $r$ is even and $a=b=r/2$, there is a choice to be made as to how one identifies the 
	$(r/2,r/2,u)$-component of $\o{\X}_r$ with $\Ig_{r/2}$: if
	$(\phi:E_0\rightarrow E_r,P,Q;\alpha)$ is an element of  $\scrP_r^1(E_0/S)$
	which corresponds to a point on the $(r/2,r/2,u)$-component, then for $E$ with 
	$E_0\simeq E^{(p^{r/2})}\simeq E_r$,
	{\em both} $(E,P; p^{-r/2}V^{r/2}\circ\alpha)$ and $(E,Q;p^{-r/2}V^{r/2}\circ\alpha)$ are 
	$S$-points of $\Ig_{p^{r/2}}$. 
	Since $uP=Q$,
	the corresponding closed immersions
	$\Ig_{r/2}\hookrightarrow \o{\X}_r$ 	
	are twists of each other by the automorphism $\langle u\rangle$ of the source.
	We will {\em consistently choose} $(E,Q;p^{-r/2}V^{r/2}\circ\alpha)$ to identify the $(r/2,r/2,u)$-component
	of $\o{\X}_r$ with $\Ig_{r/2}$.  
\end{remark} 
  
\begin{remark}\label{MWGood}
	 As in \cite[pg.~236]{MW-Hida}, we will refer to $I_r^{\infty}:=I_{(r,0,1)}$ and $I_r^0:=I_{(0,r,1)}$
  	as the two ``good" components of $\o{\X}_r$.  
	The $\Q_p$-rational cusp $\infty$ of $X_r$ induces a section
	of $\X_r\rightarrow T_r$ which meets $I_r^{\infty}$, while
	the section
	induced by the $K_r'$-rational cusp $0$ meets $I_r^0$.
	It is precisely these irreducible components of 
	$\o{\X}_r$ which contribute to the ``ordinary" part of cohomology. 
	We note that $I_r^{\infty}$ corresponds to the image of $\Ig_r$ under the map $i_1$
	of \cite[pg. 236]{MW-Hida}, and corresponds to the component of $\o{\X}_r$ denoted by $C_{\infty}$
	in \cite[pg. 343]{Tilouine}, by $C_r^{\infty}$ in \cite[pg. 231]{Saby}
	and, for $r=1$, by $I$ in \cite[\S 7]{tameness}.  
\end{remark}

By base change, the degeneracy mappings (\ref{rdegen}) induces morphisms 
$\o{\pr},\o{\ps}:\o{\X}_{r+1}\rightrightarrows \o{\X}_r$
of curves over $\F_p$ which admit the following descriptions on irreducible components:

\begin{proposition}\label{pr1desc}
	Let $a,b$ be nonnegative integers with $a+b=r+1$ and $u\in (\Z/p^{\min(a,b)}\Z)^{\times}$.
	The restriction of the map $\o{\ps}: \o{\X}_{r+1}\rightarrow \o{\X}_r$ to the $(a,b,u)$-component
	of $\o{\X}_{r+1}$ is:
	\begin{equation*}
		\begin{cases}
			\xymatrix{
				{\Ig_{a}=I_{(a,b,u)}} \ar[r]^-{ F\circ \pr} & {I_{(a-1,b,u)}=\Ig_{a-1}}
				}
				&:\quad b < a \le r+1 \\
				\xymatrix{
				{\Ig_{b}=I_{(a,b,u)}} \ar[r]^-{\langle u\rangle^{-1}F} & 
				{I_{(a-1,b,u\bmod p^{a-1})}=\Ig_{b}}
				}
				&:\quad a = b = r/2 \\
				\xymatrix{
				{\Ig_{b}=I_{(a,b,u)}} \ar[r]^-{F} & {I_{(a-1,b,u\bmod p^{a-1})}=\Ig_{b}}
				}
				&:\quad a < b < r+1 \\
				\xymatrix{
				{\Ig_{r+1}=I_{(0,r+1,1)}} \ar[r]^-{\langle p\rangle_N\pr} & {I_{(0,r,1)}=\Ig_{r}}
				}
				&:\quad (a,b,u)=(0,r+1,1)
		\end{cases}
	\end{equation*}
	and the restriction of the map $\overline{\pr}: \o{\X}_{r+1}\rightarrow \o{\X}_r$ to the $(a,b,u)$-component
	of $\o{\X}_{r+1}$ is: 
		\begin{equation*}
		\begin{cases}
			\xymatrix{
				{\Ig_{r+1}=I_{(r+1,0,1)}} \ar[r]^-{\pr} & {I_{(r,0,1)}=\Ig_{r}}
				}
				&:\quad (a,b,u)=(r+1,0,1)\\
				\xymatrix{
				{\Ig_{a}=I_{(a,b,u)}} \ar[r]^-{F} & {I_{(a,b-1,u\bmod p^{b-1})}=\Ig_{a}}
				}
				&:\quad b < a+1 \le r+1 \\
					\xymatrix{
				{\Ig_{b}=I_{(a,b,u)}} \ar[r]^-{\langle u\rangle F\circ \pr} & {I_{(a,b-1,u)}=\Ig_{b-1}}
				}
				&:\quad a+1 = b = r/2+1\\
				\xymatrix{
				{\Ig_{b}=I_{(a,b,u)}} \ar[r]^-{F\circ \pr} & {I_{(a,b-1,u)}=\Ig_{b-1}}
				}
				&:\quad a+1 < b \le r+1  \\
		\end{cases}
	\end{equation*}
	Here, for any $\F_p$-scheme $I$, the map $F:I\rightarrow I$ is the absolute Frobenius morphism.
\end{proposition}

\begin{proof}
	Using the moduli-theoretic definitions (\ref{XrDegen}) of the degeneracy maps,
	the proof is a pleasant exercise in tracing through 
	the functorial correspondence between the points of $\o{\X}_r$
	and points of $\Ig_{(a,b,u)}$.  We leave the details to the reader.
\end{proof}

We likewise have a description of the automorphism 
of $\o{\X}_r$ induced via base change by
the geometric inertia action\footnote{
Since $\Gamma$ acts trivially on $\F_p$,
for each $\gamma\in \Gamma$ the base change of the $R_r$-morphism $\gamma: \X_r\rightarrow (\X_r)_{\gamma}$
along the map induced by the canonical surjection $R_r\twoheadrightarrow \F_p$
is an $\F_p$-morphism $\o{\gamma}:\o{\X}_r\rightarrow (\o{\X}_r)_{\gamma}\simeq \o{\X}_r$.
} 
(\ref{gammamapsModuli}) of $\Gamma$:

\begin{proposition}\label{AtkinInertiaCharp}
	Let $a,b$ be nonnegative integers with $a+b=r$ and $u\in (\Z/p^{\min(a,b)}\Z)^{\times}$.
		For $\gamma\in \Gamma$, the restriction of $\o{\gamma}:\o{\X}_{r}\rightarrow \o{\X}_r$ to 
		the $(a,b,u)$-component of $\o{\X}_{r}$ is:
		\begin{equation*}
			\begin{cases}
				\xymatrix{
				{\Ig_{a}=I_{(a,b,u)}} \ar[r]^-{\id} & 
				{I_{(a,b,\chi(\gamma)u)}=\Ig_{a}}
				}
				&:\quad b\le a \le r\\
				\xymatrix{
				{\Ig_{b}=I_{(a,b,u)}} \ar[r]^-{\langle \chi(\gamma)\rangle^{-1}} & {I_{(a,b,\chi(\gamma)u)}=\Ig_{b}}
				}
				&:\quad a < b \le r \\
		\end{cases}
		\end{equation*}
		\label{InertiaCharp}
\end{proposition}

Following \cite[\S7--8]{Ulmer}, we now define a correspondence $\pi_1,\pi_2:\Y_r\rightrightarrows\X_r$ on $\X_r$ over $R_r$ which naturally 
extends the correspondence on $X_r$ giving the Hecke operator $U_p$ (see below for a brief discussion 
of correspondences).

\begin{definition}
	Let $r$ be a nonnegative integer and $R$ a ring containing a fixed choice $\zeta$ of 
	primitive $p^r$-th root of unity
	in which $N$ is invertible.
	The moduli problem 
	$\scrQ_{r}^{\zeta}:=([\Gamma_0(p^{r+1}); r,r]^{\zeta\can}; [\mu_N])$ on $(\Ell/R)$ 
	assigns to $E/S$ the set of quadruples $(\phi:E\rightarrow E',P,Q;\alpha)$
	where:
	\begin{enumerate}
		\item $\phi$ is a cyclic $p^{r+1}$-isogeny
		with standard factorization 
		\begin{equation*}
		\xymatrix{
			E=:E_0 \ar[r]^-{\phi_{0,1}} & E_1 \ar[r] \cdots & E_{r} \ar[r]^-{\phi_{r,r+1}} & E_{r+1}:=E'
		}
		\end{equation*}
		
		\item $P\in E_1(S)$ and $Q\in E_r(S)$ are generators of 
		$\ker\phi_{1,r+1}$ and $\ker \phi_{r,0}$, respectively, and satisfy
		\begin{equation*}
			\langle P, \phi_{r,r+1}(Q) \rangle_{\phi_{1,r+1}}=\langle \phi_{1,0}(P),Q \rangle_{\phi_{0,r}}=\zeta.
		\end{equation*}
		
		\item  $\alpha:\mu_N\hookrightarrow E[N]$ is a closed immersion of $S$-group schemes.
	\end{enumerate}	
\end{definition}

\begin{proposition}\label{YrRepresentability}
	If $N \ge 4$, then the moduli problem $\scrQ_{r}^{\zeta}$ is represented
	by a regular scheme $\M(\scrQ_r^{\zeta})$ that is flat of pure relative dimension
	$1$ over $\Spec(R)$.  This scheme admits a canonical
	compactification $\o{\M}(\scrP_r^{\zeta})$, which is regular and proper flat of pure
	relative dimension $1$ over $\Spec(R)$.
\end{proposition}

\begin{proof}
	As in the proof of Proposition \ref{XrRepresentability}, it suffices to prove that 
	$[\Gamma_0(p^{r+1}); r,r]^{\zeta\can}$
	is relatively representable and regular, which follows from \cite[7.6.1]{KM};
	see also \S7.9 of {\em op.~cit.}
\end{proof}

\begin{definition}\label{YrDef}
	We set $\Y_r:=\o{\M}(\scrQ_r^{\varepsilon^{(r)}})$, viewed as a scheme over $T_r=\Spec(R_r)$.
\end{definition}

The scheme $\Y_r$ is equipped with an action of the diamond operators $\Z_p^{\times}\times (\Z/N\Z)^{\times}$,
as well as a ``geometric inertia" action of $\Gamma$ given moduli-theoretically
exactly as in (\ref{balcanaction}) and  (\ref{gammamapsModuli}).
The ``semilinear'' action of $\Gamma$ is equivalent to a descent datum---necessarily effective---on the 
generic fiber of $\Y_r$, and we denote by $Y_r$ the resulting unique $\Q_p$-descent
of $(\Y_r)_{K_r}$.  

\begin{remark}\label{Yrgeniden}
 	We may identify $Y_r$ with the base change to $\Q_p$ of the modular curve $X_1(Np^r;Np^{r-1})$
	over $\Q$ classifying triples $(E_1,\alpha,C)$ where $E_1$ is a generalized elliptic curve,
	$\alpha:\mu_{Np^r}\hookrightarrow E^{\sm}_1[Np^r]$ is an embedding of group schemes
	whose image meets each irreducible component in every geometric fiber,
	and $C$ is a locally free subgroup scheme of rank $p$ in $E^{\sm}_1[p]$ with the property that
	$C\cap\im\alpha = 0$.  
	Note that $X_1(Np^r; Np^{r-1})$ is the canonical model
	over $\Q$ with rational cusp $i\infty$ of the modular curve
	$\Gamma_{r+1}^r\backslash \h^*$, for
	$\Gamma_{r+1}^r:=\Gamma_1(p^r)\cap \Gamma_0(p^{r+1})$.	
\end{remark}

%%%%%%%%%%%%%%%%

There is a canonical morphism of curves $\pi:\X_{r+1}\rightarrow \Y_r$ over $T_{r+1}\rightarrow T_r$
induced by the morphism 
\begin{equation}
\begin{gathered}
			\scrP_{r+1}^{\varepsilon^{(r)}}\rightarrow \scrQ_r^{\varepsilon^{(r)}}\quad\text{given by}
		\quad{\pi(\phi:E\rightarrow E',P,Q;\alpha)} := {(\phi:E\rightarrow E',\phi_{0,1}(P),\phi_{r+1,r}(Q); \alpha)}.
\end{gathered}\label{XtoY}	
\end{equation}
One checks that $\pi$ is equivariant with respect to the action of the diamond operators
and of $\Gamma$, and so descends to a map $\pi:Y_r\rightarrow X_r$ of smooth curves over $\Q_p$.
It is likewise straightforward to check that the two projection maps $\ps,\pr: \X_{r+1}\rightrightarrows \X_r$ 
of (\ref{XrDegen}) factor through $\pi$ via unique maps of $T_r$-schemes 
$\pi_1,\pi_2: \Y_{r}\rightrightarrows \X_r$, given as morphisms of underlying
moduli problems on $(\Ell/R_r)$
\begin{equation}
\begin{aligned}
		&{\pi_1(\phi:E_0\rightarrow E_{r+1},P,Q;\alpha)}:=
		{(E_1\xrightarrow{\phi_{1,r+1}} E_{r+1},P,\phi_{r,r+1}(Q);\ \phi_{0,1}\circ\alpha)}\\
		&{\pi_2(\phi:E_0\rightarrow E_{r+1},P,Q;\alpha)}:=
		{(E_0\xrightarrow{\phi_{0,r}} E_{r},\phi_{1,0}(P),Q; \alpha)}
\end{aligned}\label{Upcorr}
\end{equation}
That these morphisms are well defined and
that one has $\pr=\pi\circ \pi_2$ and $\ps=\pi\circ \pi_1$
is easily verified (see \cite[\S7]{Ulmer} and compare to
\cite[\S11.3.3]{KM}).  They are moreover 
finite of generic degree $p$, equivariant for the diamond operators,
and $\Gamma$-compatible; in particular, $\pi_1,\pi_2$ descend to finite 
maps $\pi_1,\pi_2:Y_r\rightrightarrows X_r$
over $\Q_p$.  
Via our identifications in Remarks \ref{genfiberrem} and \ref{Yrgeniden}, the map $\pi_1$
corresponds to the usual ``forget $C$" map, while $\pi_2$ corresponds to ``quotient by $C$".
We stress that the ``standard" degeneracy map $\rho:X_{r+1}\rightarrow X_r$ factors 
through $\pi_2$ (and not $\pi_1$).

\begin{proposition}\label{redYr}
	The scheme $\o{\Y}_r:=\Y_r\times_{T_r} \Spec(\F_p)$ is the disjoint union, with crossings 
	at the supersingular points, of the following proper, smooth $\F_p$-curves:
	for each pair of nonnegative integers $a,b$ with $a+b=r+1$ and for each $u\in (\Z/p^{\min(a,b)\Z})^{\times}$,
	one copy of
	\begin{equation*}
		\begin{cases}
			\Ig_{\max(a,b)} &\text{if}\ ab\neq 0\\
			\Ig_{r} &\text{if}\ (a,b)=(r+1,0)\ \text{or}\ (0,r+1)
		\end{cases}
	\end{equation*}
\end{proposition}

We will write $J_{(a,b,u)}$ for the irreducible component of $\o{\Y}_r$
indexed by $(a,b,u)$, and refer to it as the $(a,b,u)$-component; again, 
$J_{(a,b,u)}$ is independent of $u$.
The proof of Proposition \ref{redYr} is a straightforward adaptation of the arguments of 
\cite[13.11.2--13.11.4]{KM} (see also \cite[Proposition 8.2]{Ulmer}).  
We recall the correspondence between non-cuspidal points of the $(a,b,u)$-component and 
$[\Gamma_0(p^{r+1}); r,r]^{1\can}$-structures on elliptic curves.

Fix an ordinary elliptic curve  $E_0$ over an $\F_p$-scheme $S$, 
and let $(\phi:E_0\rightarrow E_{r+1},P,Q; \alpha)$ be an element of 
$\scrQ_r^{1}(E_0/S)$.  As before, there exist unique 
nonnegative integers $a,b$ with $a+b=r+1$ 
and a unique elliptic curve $E/S$
with the property that the cyclic $p^{r+1}$-isogeny $\phi$ factors as
\begin{equation*}
	\xymatrix{
		{E_0\simeq E^{(p^b)}} \ar[r]^-{F^a} & {E^{(p^{r+1})}} \ar[r]^-{V^b} & {E^{(p^a)} \simeq E_{r+1}}
		}.
\end{equation*}
First suppose that $ab\neq 0$. Then the point $P\in E^{(p^{b+1})}(S)$ (respectively $Q\in E^{(p^{a+1})}(S)$)
is an $[\Ig(p^b)]$ (respectively $[\Ig(p^a)]$) structure on $E^{(p)}$ over $S$.  
If $a\ge b$, there is a unit $u\in (\Z/p^b\Z)^{\times}$ such that $V^{a-b}(Q)=uP$ in $E^{(p^{b+1})}(S)$,
while if $a\le b$ then there is a unique $u\in (\Z/p^a\Z)^{\times}$ with $uV^{b-a}(P)=Q$ in $E^{(p^{a+1})}(S)$.
For $a\ge b$ (respectively $a < b$), and fixed $u$, the data $(E^{(p)}, Q ; p^{1-b}V^{b-1}\circ\alpha)$
(respectively $(E^{(p)}, P; p^{1-b}V^{b-1}\circ \alpha)$) gives an $S$-point of the $(a,b,u)$-component
$\Ig_{\max(a,b)}$.  If $b=0$ (respectively $a=0$), then $Q\in E^{(p^r)}(S)$ (respectively $P\in E^{(p^r)}(S)$)
is an $[\Ig(p^r)]$-structure on $E=E_0$ (respectively $E = E_{r+1}$).
In these extremal cases, the data $(E,Q;\alpha)$
(respectively $(E,P; p^{-r-1}V^{r+1}\circ\alpha)$) 
gives an $S$-point of the $(r+1,0,1)$-component (respectively $(0,r+1,1)$-component)
$Ig_r$.

Conversely, suppose given $(a,b,u)$ and an $S$-point of $\Ig_{\max(a,b)}$ which is neither
cuspidal nor supersingular.  If $r+1 > a\ge b$ and $(E,Q;\alpha)$ is the given point of $\Ig_a$,
then we set $P:=u^{-1}V^{a-b}(Q)\in E^{(p^b)}(S)$, while if $r+1 > b\ge a$ and $(E,P;\alpha)$ is the given point of
$\Ig_b$, we set $Q:=uV^{b-a}P\in E^{(p^{a})}(S)$.  Then 
\begin{equation*}
	(\xymatrix{
		{E^{(p^{b-1})}} \ar[r]^-{F} & {E^{(p^b)}} \ar[r]^-{F^{a-1}} & {E^{(p^r)}} \ar[r]^-{V^{b-1}} & 
		{E^{(p^a)}} \ar[r]^-{V} & {E^{(p^{a-1})}}, P,Q; F^{b-1}\circ\alpha 
	})
\end{equation*}
is an $S$-point of $\M(\scrQ_r^{1})$.
If $b=0$ and $(E,Q,\alpha)$ is an $S$-point of $\Ig_r$, then we let $P\in E^{(p)}(S)$ be the identity 
section and we obtain an $S$-point $(F^{r+1}:E\rightarrow E^{(p^{r+1})},P,Q;\alpha)$
of $\M(\scrQ_r^1)$.
If $a=0$ and $(E,P,\alpha)$ is an $S$-point of $\Ig_r$, then we let $Q\in E^{(p)}(S)$
be the identity section and we obtain an $S$-point $(V^{r+1}:E^{(p^{r+1})}\rightarrow E,P,Q;F^{r+1}\circ\alpha)$
of $\M(\scrQ_r^1)$.

Using the descriptions of $\o{\X}_r$ and $\o{\Y}_r$ furnished by
Propositions \ref{redXr} and \ref{redYr}, we can calculate the restrictions of
the degenercy maps $\o{\pi}_1,\o{\pi}_2:\o{\Y}_r\rightrightarrows \o{\X}_r$ 
to each irreducible component
of the special fiber of $\Y_r$.  The following is due to Ulmer\footnote{We warn the reader, however,
that Ulmer omits the effect of the degeneracy maps on $[\mu_N]$-structures, so his formulae
are slightly different from ours.}
\cite[Proposition 8.3]{Ulmer}: 

\begin{proposition}\label{UlmerProp}
	Let $a,b$ be nonnegative integers with $a+b=r+1$ and $u\in (\Z/p^{\min(a,b)}\Z)^{\times}$.
	The restriction of the map $\o{\pi}_1: \o{\Y}_r\rightarrow \o{\X}_r$ to the $(a,b,u)$-component
	of $\o{\Y}_r$ is:
	\begin{equation*}
		\begin{cases}
			\xymatrix{
				{\Ig_{r}=J_{(r+1,0,1)}} \ar[r]^-{F} & {I_{(r,0,1)}=\Ig_{r}}
				}
				&:\quad (a,b,u)=(r+1,0,1)\\
			\xymatrix{
				{\Ig_{a}=J_{(a,b,u)}} \ar[r]^-{\pr} & {I_{(a-1,b,u)}=\Ig_{a-1}}
				}
				&:\quad b < a < r+1 \\
			\xymatrix{
				{\Ig_{b}=J_{(a,b,u)}} \ar[r]^-{\langle u^{-1}\rangle} & {I_{(a-1,b,u\bmod p^{a-1})}=\Ig_{b}}
				}
				&:\quad a=b=(r+1)/2\\	
			\xymatrix{
				{\Ig_{b}=J_{(a,b,u)}} \ar[r]^-{\id} & {I_{(a-1,b,u\bmod p^{a-1})}=\Ig_{b}}
				}
				&:\quad a < b < r+1 \\
			\xymatrix{
				{\Ig_{r}=J_{(0,r+1,1)}} \ar[r]^-{\langle p\rangle_N} & {I_{(0,r,1)}=\Ig_{r}}
				}
				&:\quad (a,b,u)=(0,r+1,1)
		\end{cases}
	\end{equation*}
	and the restriction of the map $\overline{\pi}_2: \o{\Y}_r\rightarrow \o{\X}_r$ to the $(a,b,u)$-component
	of $\o{\Y}_r$ is: 
		\begin{equation*}
		\begin{cases}
			\xymatrix{
				{\Ig_{r}=J_{(r+1,0,1)}} \ar[r]^-{\id} & {I_{(r,0,1)}=\Ig_{r}}
				}
				&:\quad (a,b,u)=(r+1,0,1)\\
				\xymatrix{
				{\Ig_{a}=J_{(a,b,u)}} \ar[r]^-{\id} & {I_{(a,b-1,u\bmod p^{b-1})}=\Ig_{a}}
				}
				&:\quad b < a+1 \le r+1 \\
				\xymatrix{
				{\Ig_{b}=J_{(a,b,u)}} \ar[r]^-{\langle u \rangle \pr} & {I_{(a,b-1,u)}=\Ig_{b-1}}
				}
				&:\quad a +1 = b =r/2 + 1  \\
				\xymatrix{
				{\Ig_{b}=J_{(a,b,u)}} \ar[r]^-{\pr} & {I_{(a,b-1,u)}=\Ig_{b-1}}
				}
				&:\quad a+1 < b < r+1  \\
				\xymatrix{
				{\Ig_{r}=J_{(0,r+1,1)}} \ar[r]^-{F} & {I_{(0,r,1)}=\Ig_{r}}
				}
				&:\quad (a,b,u)=(0,r+1,1)
		\end{cases}
	\end{equation*}
\end{proposition}

\begin{proof}
	The proof is similar to the proof of Proposition \ref{pr1desc}, using the correspondence between
	irreducible components of $\Y_r$, $\X_r$ and Igusa curves that we have explained, together with the moduli-theoretic 
	definitions (\ref{Upcorr}) of the degeneracy mappings.  We leave the details to the reader.
\end{proof}

We end this section with a brief discussion
of correspondences on curves and their induced action on cohomology and Jacobians,
which we then apply to the specific case of modular curves.
Fix a ring $R$ and a proper normal curve $X$ over $S=\Spec R$.  Throughtout this discussion, we assume either that 
$R$ is a discrete valuation ring of mixed characteristic $(0,p)$ with perfect residue field, 
or that $R$ is a perfect field (and hence the normal $X$ is smooth). 

\begin{definition}
	A {\em correspondence $T:=(\pi_1,\pi_2)$ on $X$} is an ordered pair $\pi_1,\pi_2:Y\rightrightarrows X$
	of finite $S$-morphisms of normal and $S$-proper curves.  The {\em transpose} 
	of a correspondence $T:=(\pi_1,\pi_2)$ on $X$ is the correspondence
	on $X$ given by the ordered pair $T^*:=(\pi_2,\pi_1)$. 
\end{definition}

Thanks to Proposition \ref{HodgeIntEx} (\ref{CohomologyFunctoriality}), any correspondence
$T=(\pi_1,\pi_2)$ on $X$ induces an $R$-linear endomorphism of the short exact sequence
$H(X/R)$ via ${\pi_1}_*\pi_2^*$.  By a slight abuse of notation, we denote this endomorphism
by $T$;  as endomorphisms of $H(X/R)$ we then have
\begin{equation}
	T= {\pi_1}_*\pi_2^*\qquad\text{and}\qquad T^* = {\pi_2}_* \pi_1^*.
	\label{HeckeDef}
\end{equation}
Given a finite map $\pi:X\rightarrow X$, we will consistently view $\pi$ as a correspondence on $X$
via the association $\pi\rightsquigarrow (\id,\pi)$.  
In this way, we may think of correspondences
on $X$ as ``generalized endomorphisms."  This perspective can be made more compelling as follows.

First suppose that $R$ is a field, and fix a correspondence $T$ given by an ordered pair
$\pi_1,\pi_2:Y\rightrightarrows X$ of finite morphisms of smooth and proper curves.
Then $T$ and its transpose $T^*$ induce endomorphisms of the Jacobian $J_X:=\Pic^0_{X/R}$ of $X$, which we again
denote by the same symbols, via
\begin{equation} 
	T:=\Alb(\pi_2)\circ \Pic^0(\pi_1)\qquad\text{and}\qquad T^*:= \Alb(\pi_1) \circ \Pic^0(\pi_2)
	\label{JacHecke}
\end{equation}
Note that when $T=(\id,\pi)$ for a morphism $\pi:X\rightarrow X$, the induced endomorphisms
(\ref{JacHecke}) of $J_X$ are given by $T=\Alb(\pi)$ and $T^*:=\Pic^0(\pi)$.\footnote{Because of this fact, 
for a general correspondence $T=(\pi_1,\pi_2)$
the literature often refers to the induced endomorphism $T$ (respectively $T^*$) of $J_X$
as the {\em Albanese} (respectively {\em Picard}) or 
{\em covariant} (respectively {\em contravariant}) action of the correspondence
$(\pi_1,\pi_2)$.  Since the definitions (\ref{JacHecke}) of $T$ and $T^*$ {\em both}
literally involve Albanese and Picard functoriality, we find this old terminology
confusing, and eschew it in favor of the consistent notation we have introduced. 
}
Abusing notation, we will simply write $\pi$ for the endomorphism $\Alb(\pi)$ of $J_X$
induced by the correspondence $(1,\pi)$, and $\pi^*$ for the endomorphism $\Pic^0(\pi)$
induced by $(\pi,1) = (1,\pi)^*$.  When $\pi:X\rightarrow X$ is an automorphism, an easy argument 
shows that $\pi^* = \pi^{-1}$ as automorphisms of $J_X$.

With these definitions, the canonical filtration compatible isomorphism 
$H^1_{\dR}(X/R) \simeq H^1_{\dR}(J_X/R)$ is $T$ (respectively $T^*$)-equivariant with respect to 
the action (\ref{HeckeDef}) on $H^1_{\dR}(X/R)$ and the action on $H^1_{\dR}(J_X/R)$
induced by pullback along the endomorphisms (\ref{JacHecke}); see \cite[Proposition 5.4]{CaisNeron}.

Now suppose that $R$ is a discrete valuation ring with fraction field $K$
and fix a correspondence $T$ on $X$ given by a pair of finite morphisms
of normal curves $\pi_1,\pi_2:Y\rightrightarrows X$. Let us write $T_K$
for the induced correspondence on the (smooth) generic fiber $X_K$ of $X$.
Via (\ref{JacHecke}) and the N\'eron mapping property, $T_K$ and $T_K^*$
induces endomorphisms of the N\'eron model $J_X$ of the Jacobian of $X_K$,
which we simply denote by $T$ and $T^*$, respectively.
Thanks to Proposition \ref{intcompare}, the filtration compatible morphism
 (\ref{IntegralComparisonMap})  is $T$- and $T^*$-equivariant for the given action (\ref{HeckeDef}) on $H^1(X/R)$
and the action on $\Lie\scrExtrig_R(J_X,\Gm)$ induced by (\ref{JacHecke}) and the
(contravariant) functoriality of $\scrExtrig_R(\cdot,\Gm)$.

\begin{remark} 
	As in Remark \ref{canonicalproperty}, if $X$ is a normal proper curve over $R$
	with rational singularities, then any correspondence on $X_K$
	induces a filtration compatible endomorphism of $H^1(X/R)$
	via its action on $J_{X_K}$, the N\'eron mapping property, and
	the isomorphism (\ref{IntegralComparisonMap}) of Proposition \ref{intcompare}.
\end{remark}

We now specialize this discussion to the case of the modular curve $X_1(Np^r)$ over $\Q$.
For any prime $\ell$, one defines the Hecke correspondences 
$T_{\ell}$ for $\ell\nmid Np$ and $U_{\ell}$ for $\ell | Np$  on $X_1(Np^r)$
as in \cite[\S8]{pAdicShimura} 
({\em cf.} also \cite[\S3]{tameness} and \cite[Chapter 2, \S5.1--5.8]{MW-Iwasawa},
though be aware that the latter works instead with the modular curves $X_1(Np^r)'$
of Remark \ref{genfiberrem}).
If $\ell\neq p$, we have similarly defined correspondences
$T_{\ell}$ and $U_{\ell}$ on $\Ig_r$ over $\F_p$ (see \cite[Chapter 2, \S5.4--5.5]{MW-Iwasawa}).
For $\ell\neq p$, the Hecke correspondences extend to correspondences on $\X_r$ over $R_r$,
essentially by the same definition, while for $\ell=p$ the correspondence $U_p:=(\pi_1,\pi_2)$
on $\X_r$ is defined using the maps (\ref{Upcorr}).  We use the same symbols to 
denote the induced endomorphisms (\ref{JacHecke}) of the Jacobian $J_1(Np^r)$.

\begin{definition}
	We write $\H_r(\Z)$ (respectively $\H_r^*(\Z))$
	for the $\Z$-subalgebra of $\End_{\Q}(J_1(Np^r))$ generated by the
	Hecke operators $T_{\ell}$  (respectively $T_{\ell}^*$)  
	for $\ell\nmid Np$ and $U_{\ell}$ (respectively $U_{\ell}^*$) for $\ell | Np$,
	and the diamond operators $\langle u\rangle$ (respectively $\langle u\rangle^*$)
	for $u\in \Z_p^{\times}$ and 
	$\langle v\rangle_N$ (respectively $\langle v\rangle_N^*$) for $v\in (\Z/NZ)^{\times}$.
	For any commutative ring $A$, we set $\H_r(A):=\H_r(\Z)\otimes_{\Z} A$
	and $\H_r^*(A):=\H_r^*(\Z)\otimes_{\Z} A$, and for ease of notation
	we set $\H_r:=\H_r(\Z_p)$ and $\H_r^*:=\H_r^*(\Z_p)$.   	
\end{definition}

The relation between the Hecke algebras $\H_r$ and $\H_r^*$ is explained by the following:

\begin{proposition}\label{AtkinInterchange}
	Denote by $w_r$ the automorphism of $({J_r})_{K_r'}$ induced by the correspondence $(1,w_r)$
	on $({X_r})_{K_r'}$ over $K_r'$.  Viewing $\H_r$ and $\H_r^*$ as $\Z_p$-subalgebras 
	of $\End_{K_r'}(({J_r})_{K_r'})\otimes_{\Z}\Z_p$, conjugation by $w_r$ carries
	$\H_r$ isomorphically onto $\H_r^*$: that is, $w_rT=T^*w_r$ for all Hecke operators $T$.
\end{proposition}

\begin{proof}
	This is standard; see, e.g., \cite[pg. 336]{Tilouine}, \cite[2.1.8]{OhtaEichler},
	or \cite[Chapter 2, \S5.6 (c)]{MW-Iwasawa}.
\end{proof}

\section{Differentials on modular curves in characteristic \texorpdfstring{$p$}{p}}\label{DiffCharp}

We now analyze the ``modified de Rham cohomology"
(\S\ref{GD}) of the special fibers of the modular
curves $\X_r/R_r$, and we relate its ordinary part to the de Rham cohomology of the ``Igusa Tower."

\subsection{The Cartier operator}

Fix a perfect field $k$ of characteristic $p > 0$ and write $\varphi:k\rightarrow k$ for the $p$-power Frobenius map. 
In this section, we recall the basic theory of the Cartier operator 
for a smooth and proper curve over $k$.  As we will only need the theory
in this limited setting, we will content ourselves with a somewhat {\em ad hoc} formulation of it.
Our exposition follows \cite[\S10]{SerreTopology}, but 
the reader may consult \cite[\S5.5]{Oda} or \cite{CartierNouvelle} for a more general treatment.

Let $X$ be a smooth and proper curve over $k$ and write $F:X\rightarrow X$ 
for the absolute Frobenius map; 
it is finite and flat and is a morphism over the endomorphism of $\Spec(k)$ induced by $\varphi$.
Let $D$ be an effective Cartier (=Weil) divisor on $X$ over $k$, and write $\O_X(-D)$ for 
the coherent (invertible) ideal sheaf determined by $D$. 
The pullback map $F^*:\O_{X}\rightarrow {F}_*\O_{X}$ carries the ideal sheaf
 $\O_X(-nD) \subseteq \O_{X}$ into ${F}_*\O_X(-npD)$, so we obtain a canonical $\varphi$-semilinear pullback map on cohomology
 \begin{equation}
 	\xymatrix{
		{F^*:H^1(X,\O_X(-nD))} \ar[r] & H^1(X,\O_X(-npD)).
		}\label{Fpullback}
 \end{equation}
By Grothendieck--Serre duality, (\ref{Fpullback}) gives a $\varphi^{-1}$-semilinear ``trace" 
map\footnote{This map coincides with Grothendieck's trace morphism on dualizing sheaves
attached to the finite map $F$.} of $k$-vector spaces
\begin{equation}
	\xymatrix{
		{V:={F}_*:H^0(X,\Omega^1_{X/k}(npD))} \ar[r] & {H^0(X,\Omega^1_{X/k}(nD))}
		}\label{cartier}
\end{equation}

\begin{proposition}\label{CartierOp}
	Let $X/k$ be a smooth and proper curve, $D$ an effective Cartier divisor on $X$,
	and $n$ a nonnegative integer.
	\begin{enumerate}
		\item There is a unique $\varphi^{-1}$-linear endomorphism $V:={F}_*$
		of $H^0(X,\Omega^1_{X/k}(nD))$ which is dual, via Grothendieck-Serre
		duality, to pullback by absolute Frobenius on $H^1(X,\O_X(-nD))$.
		\label{CartierExists}	
	
		\item The map $V$ ``improves poles" in the sense that it factors through the canonical inclusion
	  	\begin{equation*}
	  		\xymatrix{
				{H^0(X,\Omega^1_{X/k}(\lceil\frac{n}{p}\rceil D)} \ar@{^{(}->}[r] & 
				{H^0(X,\Omega^1_{X/k}(nD))}
			}.
	  	\end{equation*}
		\label{CartierImproves}
		
		\item If $\rho:Y\rightarrow X$ is any finite morphism of smooth proper curves over $k$, 
		and $\rho^*D$ is the pullback of $D$ to $Y$, then
		the induced pullback and trace maps
		\begin{equation*}
			\xymatrix{
				H^0(Y,\Omega^1_{Y/k}(n\rho^*D)) \ar@<0.5ex>[r]^-{\rho^*} & 
				\ar@<0.5ex>[l]^-{\rho_*} H^0(X,\Omega^1_{X/k}(nD))
			}
		\end{equation*}
		commute with $V$.
		\label{CartierCommutes}
		
		\item Assume that $k$ is algebraically closed.  Then
		for any meromorphic differential $\eta$ on $X$ and any closed point $x$ of $X$, the formula
		\begin{equation*}
			\res_x(V\eta)^p = \res_x(\eta)
		\end{equation*}
		holds, where $\res_x$ is the canonical ``residue at $x$ map" on meromorphic 
		differentials.
		\label{CartierResidue}
	\end{enumerate}
\end{proposition}

\begin{proof}
	Both (\ref{CartierExists}) and (\ref{CartierImproves}) follow from our discussion, while
	(\ref{CartierCommutes}) follows (via duality) from the fact that the $p$-power
	map commutes with any ring homomorphism.  Finally, (\ref{CartierResidue}) follows from the fact
	that the canonical isomorphism $H^1(X,\Omega^1_{X/k})\rightarrow k$ induced by the residue map coincides
	with the negative of Grothendieck's trace isomorphism ({\em cf.}
	Proposition \ref{Rosenlicht}), together with the fact that Grothendieck's
	trace morphism is compatible with compositions; see Appendix B and Corollary 3.6.6 of \cite{GDBC}.
\end{proof}

\begin{remark}\label{poletrace}
	Quite generally, if $\rho:Y\rightarrow X$ is any finite morphism of 
	smooth curves over $k$
	and $y$ is any $k$-point of $Y$ with $x=\rho(y)\in X(k)$, 
	then for any meromorphic differential $\eta$ on $Y$ we have
	\begin{equation}
		\ord_x(\rho_*\eta) \le \left\lceil \frac{\ord_y(\eta)}{e}\right\rceil\label{poleimprovement}
	\end{equation}
	where $e$ is the ramification index of the extension of discrete valuation rings 
	$\O_{X,x}\rightarrow \O_{Y,y}$.
	Indeed, if $\I_x$ and $\I_y$ denote the ideal sheaves of the reduced closed subschemes $x$ and $y$, 
	then the pullback map $\O_X\rightarrow \rho_*\O_Y$ carries $\I_x^n$ into $\rho_*\I_y^{ne}$.  
	Passing to the map on $H^1$'s and using Grothendieck duality,
	we deduce that $\rho_*$ carries $H^0(Y,\Omega^1_{Y/k}\otimes\I_y^{-ne})$ into $
	H^0(X,\Omega^1_X\otimes\I_x^{-n})$, whence the estimate (\ref{poleimprovement}).
	If moreover $k$ is algebraically closed, then we have ({\em cf.} \cite[Theorem 4]{TateResidues})
	\begin{equation}
		\res_x(\rho_*\eta) = \res_y(\eta).
	\end{equation}\label{TateFormula}
\end{remark}

We recall the following (generalization of a) well-known lemma of Fitting:

\begin{lemma}\label{HW}
		Let $A$ be a commutative ring, $\varphi$ an automorphism of $A$,
		and $M$ be an $A$-module equipped with a $\varphi$-semilinear
		endomorphism $F:M\rightarrow M.$
		Assume that one of the following holds:
		\begin{enumerate}
			\item $M$ is a finite length $A$-module.\label{finlen}
			\item $A$ is a complete noetherian adic ring,\label{top}
			with ideal of definition $I\subsetneq A$, and $M$ is a finite $A$-module.
		\end{enumerate}
		Then there is a unique direct sum decomposition 
		\begin{equation}
			M = M^{F_{\ord}} \oplus M^{F_{\nil}},\label{FittingDecomp}
		\end{equation}
		where $M^{\varphi_{\ord}}$ 
		is the maximal $\varphi$-stable
		submodule of $M$ on which $F$ is bijective, and $M^{F_{\nil}}$
		is the maximal $F$-stable submodule of $M$ on which $F$
		is $($topologically$)$ nilpotent.  The assignment $M\rightsquigarrow M^{F_{\star}}$
		for $\star=\ord, \nil$ is an exact functor on the category of $($left$)$
		$A[F]$-modules verifying $(\ref{finlen})$ or $(\ref{top})$.
\end{lemma}

\begin{proof}
	For the proof in case (\ref{finlen}), we refer to \cite[\Rmnum{6}, 5.7]{LazardGroups},
	and just note that one has:
	\begin{equation*}
		M^{F_{\ord}}:=\bigcap_{n\ge 0} \im(F^n)\quad\text{and}\quad
		M^{F_{\nil}}:=\bigcup_{n\ge 0} \ker(F^n),
	\end{equation*}
	where one uses that $\varphi$ is an automorphism to know that the image and kernel 
	of $F^n$ are $A$-submodules of $M$.
	It follows immediately from this that the association $M\rightsquigarrow M^{F_{\star}}$
	is a functor from the category of left $A[F]$-modules of finite $A$-length
	to itself.  It is an exact functor because the canonical inclusion $M^{F_{\star}}\rightarrow M$
	is an $A[F]$-direct summand.  
	In case (\ref{top}), our hypotheses ensure that $M/I^nM$ is a noetherian and Artinian
	$A$-module, and hence of finite length, for all $n$.  Our assertions in this situation
	then follow immediately from (\ref{finlen}), via the uniqueness of (\ref{FittingDecomp}),
	together with fact that $M$ is finite as an $A$-module, and hence $I$-adically complete (as $A$ is).
\end{proof}

We apply \ref{HW} to the $k$-vector space $M:=H^0(X,\Omega^1_{X/k})$
equipped with the $\varphi^{-1}$ semilinear map $V$:

\begin{definition}\label{ordnil}
	The $k[V]$-module $H^0(X,\Omega^1_{X/k})^{V_{\ord}}$ is called the
	{\em $V$-ordinary subspace} of holomorphic differentials on $X$.  It is
	the maximal $k$-subspace of $H^0(X,\Omega^1_{X/k})$ on which $V$ is bijective.
	The nonnegative integer $\gamma_X:=\dim_k H^0(X,\Omega^1_{X/k})^{V_{\ord}}$
	is called the {\em Hasse-Witt invariant} of $X$.
\end{definition}

\begin{remark}\label{DualityOfFVOrd}
	Let $D$ be any effective Cartier divisor.
	Since $V:={F}_*$ and $F:=F^*$ are adjoint under the canonical perfect 
	$k$-pairing between $H^0(X,\Omega^1_{X/k}(D))$ and $H^1(X,\O_X(-D))$, this pairing
	restricts to a perfect duality pairing 
	\begin{equation}
		\xymatrix{
			{H^0(X,\Omega^1_{X/k}(D))^{V_{\ord}} \times H^1(X,\O_X(-D))^{F_{\ord}}} \ar[r] & {k}
		}.\label{DualityOfFVOrdMap}
	\end{equation}
	In particular (taking $D=0$) we also have $\gamma_X=\dim_k H^1(X,\O_X)^{F_{\ord}}$.
\end{remark}

The following ``control lemma" is a manifestation of the fact that the Cartier
operator improves poles (Proposition \ref{CartierOp}, (\ref{CartierImproves})): 

\begin{lemma}\label{sspoles}
	Let $X$ be a smooth and proper curve over $k$ and $D$ an effective Cartier divisor 
	on $X$. Considering $D$ as a closed subscheme of $X$, we write $D_{\red}$ for associated 
	reduced closed subscheme.  
	\begin{enumerate}
		\item For all positive integers $n$, the canonical morphism
			\begin{equation*}
				H^0(X,\Omega^1_{X/k}(D_{\red})) \rightarrow H^0(X,\Omega^1_{X/k}(nD))
			\end{equation*}
			induces a natural isomorphism on $V$-ordinary subspaces.\label{VControl}
			
		\item For each positive integer $n$, the canonical map
			\begin{equation*}
				H^1(X, \O_X(-nD)) \rightarrow H^1(X,\O_X(-D_{\red}))
			\end{equation*}
			induces a natural isomorphism on $F$-ordinary subspaces.\label{FControl} 
		
		\item The identifications in $(\ref{VControl})$ and $(\ref{FControl})$ are canonically 
		$k$-linearly dual, via Remark $\ref{DualityOfFVOrd}$. 	
		
	\end{enumerate}
\end{lemma}

\begin{proof}
	This follows immediately from Proposition \ref{CartierOp}, (\ref{CartierImproves})
	and Remark \ref{DualityOfFVOrd}.
\end{proof}

Now let $\pi:Y\rightarrow X$ be a finite branched covering of smooth, proper and geometrically connected curves over $k$ with group $G$ that is a $p$-group.  
Let $D_X$ be any effective Cartier divisor on $X$ over $k$ with support containing the ramification locus of $\pi$,
and put $D_Y=\pi^*D_X$.   As in Lemma \ref{sspoles},  denote by $D_{X,\red}$ and $D_{Y,\red}$ the underlying reduced closed subschemes;
as $D_{Y,\red}$ is $G$-stable, the $k$-vector spaces 
$H^0(Y,\Omega^1_{Y/k}(nD_{Y,\red}))$ and $H^1(Y,\O_Y(-nD_{Y,\red})$
are canonically $k[G]$-modules for any $n\ge 1$.  The following theorem of Nakajima 
is the key to the proofs of our structure theorems for $\Lambda$-modules:

\begin{proposition}[Nakajima]\label{Nakajima}
	Assume that $\pi$ is ramified, 
	let $\gamma_X$ be the Hasse-Witt invariant of $X$ and set $d:=\gamma_X-1+\deg (D_{X,\red})$.
	Then for each positive integer $n$:
	\begin{enumerate}
		\item The $k[G]$-module 
		 $H^0(Y,\Omega^1_{Y/k}(nD_{Y,\red}))^{V_{\ord}}$ is free of rank $d$ and independent
		 of $n$.\label{NakajimaOne}
		
		\item The $k[G]$-module $H^1(Y,\O_Y(-nD_{Y,\red}))^{F_{\ord}}$ is naturally
		isomorphic to the contragredient of $H^0(Y,\Omega^1_{Y/k}(nD_{Y,\red}))^{V_{\ord}}$;
		as such, it is $k[G]$-free of rank $d$ and independent of $n$.
		\label{NakajimaTwo}
	\end{enumerate}
\end{proposition}

\begin{proof}
	The independence of $n$ is simply Lemma \ref{sspoles};
	using this, the first assertion is then equivalent to Theorem 1 of \cite{Nakajima}.
	The second assertion is immediate from Remark \ref{DualityOfFVOrd},
	once one notes that for $g\in G$ one has the identity $g_*=(g^{-1})^*$
	on cohomology (since $g_*g^*=\deg g = \id$), so $g^*$ and $(g^{-1})^*$
	are adjoint under the duality pairing (\ref{DualityOfFVOrdMap}).
\end{proof}

We end this section with
a brief explanation of the relation between the
de Rham cohomology of $X$ over $k$ and the Dieudonn\'e module
of the $p$-divisible group of the Jacobian of $X$.  This will allow
us to give an alternate description of the $V$-ordinary (respectively $F$-ordinary)
subspace of $H^0(X,\Omega^1_{X/k})$ (respectively $H^1(X,\O_X)$) which
will be instrumental in our applications.

Pullback by the absolute Frobenius gives a semilinear endomorphism
of the Hodge filtration $H(X/k)$ of $H^1_{\dR}(X/k)$
which we again denote by $F=F^*$.
Under the canonical autoduality of $H(X/k)$ provided by Proposition \ref{HodgeFilCrvk} 
(\ref{HodgeDegenerationField}) ,
we obtain $\varphi^{-1}$-semilinear endomorphism
\begin{equation}
	\xymatrix{
		{V:={F}_*: H^1_{\dR}(X/k)} \ar[r] & {H^1_{\dR}(X/k)}
		}\label{CartierOndR}
\end{equation}
whose restriction to $H^0(X,\Omega^1_{X/k})$ coincides with (\ref{cartier}).
Let $A$ be the ``Dieudonn\'e ring", {\em i.e.}~the (noncommutative if $k\neq \F_p$)
ring $A:=W(k)[F,V]$, where  
 $F$, $V$ satisfy $FV=VF=p$, 
$F\alpha=\varphi(\alpha)F$, and $V\alpha=\varphi^{-1}(\alpha)V$ for all $\alpha\in W(k)$.
We view $H^1_{\dR}(X/k)$ as a left $A$-module 
in the obvious way.

\begin{proposition}[Oda]\label{OdaDieudonne}
	Let $J:=\Pic^0_{X/k}$ be the Jacobian of $X$ over $k$ and $G:=J[p^{\infty}]$
	its $p$-divisible group.  Denote by $\D(G)$ the contravariant
	Dieudonn\'e crystal of $G$, so the Dieudonn\'e module $\D(G)_W$ is naturally a left $A$-module, finite and
	free over $W:=W(k)$.  
	\begin{enumerate}
		\item There are canonical isomorphisms of left $A$-modules
		\begin{equation*}
			H^1_{\dR}(X/k)\simeq \D(J)_{k}\simeq \D(G)_k.
		\end{equation*}\label{OdaIsom}
		
		\item For any finite morphism $\rho:Y\rightarrow X$ of smooth and proper curves
		over $k$, the identification of $(\ref{OdaIsom})$ intertwines $\rho_*$
		with $\D(\Pic^0(\rho))$ and $\rho^*$ with $\D(\Alb(\rho))$.\label{OdaIsomFunctoriality}
		
		\item Let $G=G^{\et}\times G^{\mult}\times G^{\loc}$
		be the canonical direct product decomposition of $G$ into its maximal \'etale,
		multiplicative, and local-local subgroups. 
		Via the identification of $(\ref{OdaIsom})$, the canonical mappings
		in the exact sequence $H(X/k)$ induce natural isomorphisms of left $A$-modules
		\begin{equation*}
			H^0(X,\Omega^1_{X/k})^{V_{\ord}} \simeq \D(G^{\mult})_k
			\quad\text{and}\quad
			H^1(X,\O_X)^{F_{\ord}} \simeq \D(G^{\et})_k
		\end{equation*}\label{GetaleGmult}
		
		\item The isomorphisms of $(\ref{GetaleGmult})$ are dual to each other, using the duality pairing 
		of Remark $\ref{DualityOfFVOrd}$ 
		together with the canonical isomorphism $\D(G)_k^t\simeq \D(\Dual{G})_k$
		and the autoduality of $G$ resulting from the autoduality of $J$.\label{BBMDuality}		
	\end{enumerate}
\end{proposition}

\begin{proof}
	Using the characterizing properties of the Cartier operator
	defined by Oda \cite[Definition 5.5]{Oda} and the explicit
	description of the autoduality of $H^1_{\dR}(X/k)$ in terms
	of cup-product and residues, one checks that the endomorphism
	of $H^1_{\dR}(X/k)$ in \cite[Definition 5.6]{Oda} is adjoint
	to $F^*$, and therefore coincides with the endomorphism
	$V:={F}_*$ in (\ref{CartierOndR}); {\em cf.}
	the proof of \cite[Proposition 9]{SerreTopology}.

	We recall that one has a canonical isomorphism 
	\begin{equation}
		H^1_{\dR}(X/k)\simeq H^1_{\dR}(J/k)\label{dRIdenJac}
	\end{equation}
	which is compatible with Hodge filtrations and duality (using the canonical
	principal polarization to identify $J$ with its dual) and which, for	
	any finite morphism of smooth curves $\rho:Y\rightarrow X$ over $k$,
	intertwines $\rho_*$ with $\Pic^0(\rho)^*$ and $\rho^*$ with $\Alb(\rho)^*$; see  
	\cite[Proposition 5.4]{CaisNeron}, noting that the proof given there works over any field $k$,
	and {\em cf.}~Proposition \ref{intcompare}.  It follows from these compatibilities
	and the fact that the Cartier operator as defined in \cite[Definition 5.5]{Oda} is functorial
	that the identification (\ref{dRIdenJac}) is moreover an isomorphism of left $A$-modules,
	with the $A$-structure on $H^1_{\dR}(J/k)$ defined as in \cite[Definition 5.8]{Oda}.
	
	Now by \cite[Corollary 5.11]{Oda} and \cite[Theorem 4.2.14]{BBM}, 
	for any abelian variety $B$ over $k$, 
	there is a canonical isomorphism of left $A$-modules
	\begin{equation}
		H^1_{\dR}(B/k)\simeq \D(B)_k\label{AbVarDieuMod}
	\end{equation} 
	Using the definition of this isomorphism in Proposition 4.2 and Theorem 5.10 of \cite{Oda}, 
	it is straightforward (albeit tedious\footnote{Alternately, 
	one could appeal to 
	\cite{MM}, specifically to Chapter I, 4.1.7, 4.2.1, 3.2.3, 2.6.7
	and to Chapter II, \S13 and \S15 (see especially Chapter II, 13.4 and 1.6).  
	See also \S2.5 and \S4 of \cite{BBM}.
	}) 
	to check that for any homomorphism $h:B'\rightarrow B$
	of abelian varieties over $k$, the identification (\ref{AbVarDieuMod}) intertwines $h^*$ and $\D(h)$.  
	Combining (\ref{dRIdenJac}) and (\ref{AbVarDieuMod})
	yields (\ref{OdaIsom}) and (\ref{OdaIsomFunctoriality}).

	Now since $V={F}_*$ (respectively $F=F^*$) is the zero endomorphism of $H^1(X,\O_X)$ 
	(respectively $H^0(X,\O_X)$), the canonical mapping
	\begin{equation*}
		\xymatrix{
			{H^0(X,\Omega^1_{X/k})} \ar@{^{(}->}[r] & {H^1_{\dR}(X/k)\simeq \D(G)_k}
		}
		\quad\text{respectively}\quad
		\xymatrix{
			{\D(G)_k\simeq H^1_{\dR}(X/k)} \ar@{->>}[r] & {H^1(X,\O_X)}
		}
	\end{equation*}
	induces an isomorphism on $V$-ordinary (respectively $F$-ordinary) subspaces.  
	On the other hand, by Dieudonn\'e theory one knows that
	for {\em any} $p$-divisible group $H$, the semilinear endomorphism $V$ 
	(respectively $F$) of $\D(H)_W$
	is bijective if and only if $H$ is of multiplicative type (respectively \'etale).
	The (functorial) decomposition $G=G^{\et}\times G^{\mult}\times G^{\loc}$
	yields a natural isomorphism of left $A$-modules
	\begin{equation*}
		\D(G)_W\simeq \D(G^{\et})_W\oplus \D(G^{\mult})_W\oplus \D(G^{\loc})_W,
	\end{equation*}
	and it follows that the natural maps $\D(G^{\mult})_W\rightarrow \D(G)_W$,
	$\D(G)_W\rightarrow \D(G^{\et})_W$ induce isomorphisms 
	\begin{equation}
		\D(G^{\mult})_W \simeq \D(G)^{V_{\ord}}_W\quad\text{and}\quad
		\D(G)^{F_{\ord}}_W\simeq \D(G^{\et})_W,\label{VordMultFordEt}
	\end{equation}
	respectively,  which gives (\ref{GetaleGmult}).  Finally, (\ref{BBMDuality})
	follows from Proposition 5.3.13 and the proof of Theorem 5.1.8 in \cite{BBM},
	using Proposition 2.5.8 of {\em op.~cit.}~and the compatibility of the isomorphism
	(\ref{dRIdenJac}) with duality (for which see 
	\cite[Theorem 5.1]{colemanduality} and {\em cf.} \cite[Lemma 5.5]{CaisNeron}). 	
\end{proof}

\subsection{The Igusa tower}\label{IgusaTower}

We apply Proposition \ref{Nakajima} to the Igusa tower
(Definition \ref{IgusaDef}).
The canonical degeneracy map 
$\pr:I_{r}\rightarrow I_1$ defined by (\ref{Vmapsch}) is
finite \'etale outside\footnote{We will frequently 
write simply $\SS$ for the divisor $\SS_r$ on $I_r$
when $r$ is clear from context.
}
$\SS:=\SS_r$ and totally (wildly) ramified
over $\SS_1$, and so makes $I_r$ in to a branched cover
of $I_1$ with group $\Delta/\Delta_r$.
The cohomology groups $H^0(I_r,\Omega^1_{I_r/\F_p}(\SS))$
and $H^1(I_r,\O_{I_r}(-\SS))$ are therefore naturally
$\F_p[\Delta/\Delta_r]$-modules.

\begin{proposition}\label{IgusaStructure}
	Let $r$ be a positive integer, write $\gamma$ for the $p$-rank of
	$J_1(N)_{\F_p}$, and set $\delta:=\deg\SS$.
\begin{enumerate}	
	\item The $\F_p[\Delta/\Delta_r]$-modules $H^0(I_r,\Omega^1_{I_r/\F_p}(\SS))^{V_{\ord}}$
	and $H^1(I_r,\O_{I_r}(-\SS))^{F_{\ord}}$ are both free of rank $d:=\gamma+\delta-1$.
	Each is canonically isomorphic to the contragredient of the other.
	\label{IgusaFreeness}
	 
	\item For any positive integer $s\le r$, the canonical trace mapping
	associated to $\pr:I_r\rightarrow I_s$ induces natural isomorphisms of $\F_p[\Delta/\Delta_s]$-modules 
	\begin{subequations}
		\begin{equation*}
			\xymatrix{
				{\pr_*:H^0(I_r, \Omega^1_{I_r/\F_p}(\SS))^{V_{\ord}}
				\tens_{\F_p[\Delta/\Delta_r]} \F_p[\Delta/\Delta_s]} \ar[r]^-{\simeq} &
				{H^0(I_s, \Omega^1_{I_s/\F_p}(\SS))^{V_{\ord}}}
			}
		\end{equation*}
		\begin{equation*}
			\xymatrix{
				{\pr_*:H^1(I_r, \O_{I_r}(-\SS))^{F_{\ord}}
				\tens_{\F_p[\Delta/\Delta_r]} \F_p[\Delta/\Delta_s]} \ar[r]^-{\simeq} &
				{H^1(I_s, \O_{I_s}(-\SS))^{F_{\ord}}}
			}
		\end{equation*}
	\end{subequations}	
	\label{IgusaControl}
\end{enumerate}
\end{proposition}

\begin{remark}
	Using the moduli interpretation of $I_r$
	and calculations on formal groups of universal elliptic curves,
	one can show \cite[Lemma 12.9.3]{KM} that pullback induces a canonical identification 
	\begin{equation*}
		\pr^*\Omega^1_{I_s/k}=\Omega^1_{I_r/k}(-p^{r-1}(p^r-p^s)\cdot\SS).
	\end{equation*}
	If $n$ is any positive integer, it follows easily from this that $\pr^*$ identifies
	$H^0(I_s,\Omega^1_{I_s/k}(n\cdot\SS))$ with the $\Delta_s/\Delta_r$-invariant subspace of 
	$H^0(I_r, \Omega^1_{I_r/k}(-N_{r,s}(n)\cdot\SS))$, for
	$N_{r,s}(n)=p^{r-1}(p^r-p^s)-p^{r-s}n$.
	In particular, via pullback, $H^0(I_1,\Omega^1_{I_1/k}(p^r-p))$
	is canonically identified with the $\Delta/\Delta_r$-invariant subspace
	of $H^0(I_r,\Omega^1_{I_r/k})$, so the $k$-dimension of this subspace 
	grows {\em exponentially} with $r$.  In this light, it is 
	remarkable that the $V$-ordinary subspace has controlled growth.
	We will not use these facts in what follows, though see Remark \ref{FullIgusaStructure}.  
\end{remark}

In order to prove Proposition \ref{IgusaStructure}, we require the following Lemma
({\em cf.} \cite[p. 511]{MW-Analogies}):

\begin{lemma}\label{MW}
	Let $\pi: Y\rightarrow X$ be a finite flat and generically \'etale morphism of
	smooth and geometrically irreducible curves over a field $k$.
	If there is a geometric point of $X$ over which $\pi$ is totally ramified then the induced map
	of $k$-group schemes $\Pic(\pi):\Pic_{X/k}\rightarrow \Pic_{Y/k}$ has trivial scheme-theoretic kernel.
\end{lemma}

\begin{proof}
	The hypotheses and the conclusion are preserved under extension of $k$, so we may
	assume that $k$ is algebraically closed.  We fix a $k$-point $x\in X(k)$ over which
	$\pi$ is totally ramified, and let $y\in Y(k)$ be the unique $k$-point of $Y$ over $x$.
	To prove that $\Pic_{X/k}\rightarrow \Pic_{Y/k}$
	has trivial kernel, it suffices to prove that the map of groups  
	$\pi_R^*:\Pic(X_R)\rightarrow \Pic(Y_R)$
	is injective for every Artin local $k$-algebra $R$.  
	We fix such a $k$-algebra, and denote by $x_R\in X_R(R)$ and $y_R\in Y_R(R)$
	the points obtained from $x$ and $y$ by base change.  Let $\L$ be a line bundle
	on $X_R$ whose pullback to $Y_R$ is trivial; our claim is that we may choose
	a trivialization $\pi^*\L\xrightarrow{\simeq} \O_{Y_R}$ of $\pi^*\L$
	over $Y_R$ which descends to $X_R$.  In other words, by descent theory,
	we assert that we may choose a trivialization of $\pi^*\L$ with the property that the
	two pullback trivializations under the canonical projection maps
	\begin{equation}
		\xymatrix{
			{Y_R \times_{X_R} Y_R} \ar@<-0.5ex>[r]_-{\pr_2}\ar@<0.5ex>[r]^-{\pr_1} & {Y_R}
			}\label{TwoPullback}
	\end{equation}
	coincide.
	
	We first claim that the $k$-scheme $Z:=Y\times_X Y$ is connected and generically reduced.
	Since $\pi$ is totally ramified over $x$, 
	there is a unique geometric point $(y,y)$ of $Z$ mapping to $x$ under the canonical map 
	$Z\rightarrow X$.  Since this map is moreover finite flat (because $\pi:Y\rightarrow X$
	is finite flat due to smoothness of $X$ and $Y$), every connected component of $Z$ is finite flat  	
	onto $X$ and so passes through $(y,y)$.  Thus, $Z$ is connected.
	On the other hand, $\pi:Y\rightarrow X$ is generically
	\'etale by hypothesis, so there exists a dense open subscheme $U\subseteq X$
	over which $\pi$ is \'etale.  Then $Z\times_X U$ is \'etale---and hence smooth---over $U$
	and the open immersion $Z\times_X U\rightarrow Z$ is schematically dense 
	as $U\rightarrow X$ is schematically dense and $\pi$ is finite
	and flat.  As $Z$ thus contains a $k$-smooth and dense subscheme,
	it is generically reduced.
	
	Fix a choice $e$ of $R$-basis of the fiber $\L(x_R)$ of $\L$ at $x_R$.
	As any two trivializations of $\pi^*\L$ over $Y_R$ differ by an element of 
	$R^{\times}$, we may uniquely choose a trivialization which on $x_R$-fibers	
	\begin{equation}
		\xymatrix{ 
			{\L(x_R)\simeq \pi^*\L(y_R)}\ar[r]^-{\simeq} & {\O_{Y_R}(y_R)\simeq R}
			}\label{TrivOnYunit}
	\end{equation}
	carries $e$ to $1$.  The obstruction to the two pullback trivializations
	under (\ref{TwoPullback}) being equal is a global unit on $Y_R\times_{X_R} Y_R$.
	But since $Y_R\times_{X_R} Y_R = (Y\times_X Y)_R$, we have by flat base change
	\begin{equation*}
		H^0(Y_R\times_{X_R} Y_R, \O_{Y_R\times_{X_R} Y_R}) = H^0(Y\times_X Y,\O_{Y\times_X Y})\otimes_k R=R
	\end{equation*}
	where the last equality rests on the fact that $Y\times_X Y$ is connected, generically reduced, 
	and proper over $k$.
	Thus, the obstruction to the two pullback trivializations being equal is an element of $R^{\times}$,
	whose value may be calculated at any point of $Y_R\times_{X_R} Y_R$.  By our choice (\ref{TrivOnYunit})
	of trivialization of $\pi^*\L$, the value of this obstruction at the point $(y_R,y_R)$ is 1,
	and hence the two pullback trivializations coincide as desired.
\end{proof}

\begin{proof}[Proof of Proposition $\ref{IgusaStructure}$]
	Since $\pr:I_r\rightarrow I_s$ is a finite branched cover with group $\Delta_s/\Delta_r$
	and totally wildly ramified over $\SS_s$,
	we may apply Proposition \ref{Nakajima}, which gives (\ref{IgusaFreeness}).

	To prove (\ref{IgusaControl}), we work over $k:=\o{\F}_p$ and
	argue as follows.  Since $\pr:I_r\rightarrow I_{s}$
	is of degree $p^{r-s}$ and totally ramified over $\SS_{s}$, we have $\pr^*\SS_{s}=p^{r-s}\cdot\SS$;
	it follows that pullback induces a map
	\begin{equation}
		\xymatrix{
				{H^1(I_{s},\O_{I_{s}}(-\SS_{s}))}	\ar[r]^-{\pr^*} & 	{H^1(I_r,\O_{I_r}(-\SS))}
			}\label{pbH1poles}
	\end{equation}
	which we claim is {\em injective}.
	To see this,
	we observe that the long exact cohomology sequence attched to
	the short exact sequence of sheaves on $I_r$ 
	\begin{equation*}
		\xymatrix{
			0\ar[r] & {\O_{I_r}(-\SS)} \ar[r] & {\O_{I_r}} \ar[r] & {\O_{\SS}} \ar[r] & 0
		}
	\end{equation*}
	(with $\O_{\SS}$ a skyscraper sheaf supported on $\SS$) 
	yields a commutative diagram with exact rows
	\begin{equation}
	\begin{gathered}
		\xymatrix@C=15pt{
			0\ar[r] & {H^0(I_{s},\O_{I_{s}})} \ar[r]\ar[d] & 
			{H^0(I_{s},\O_{\SS_{s}})} \ar[r]\ar[d] & 
			{H^1(I_{s},\O_{I_{s}}(-\SS_{s}))} \ar[r]\ar[d] & 
			{H^1(I_{s},\O_{I_{s}})} \ar[r]\ar[d] & 0\\
			0\ar[r] & {H^0(I_r,\O_{I_r})} \ar[r] & 
			{H^0(I_r,\O_{\SS})} \ar[r] & {H^1(I_r,\O_{I_r}(-\SS))} \ar[r] & 
			{H^1(I_r,\O_{I_r})} \ar[r] & 0
		}
	\end{gathered}	
		\label{pbH1inj}
	\end{equation}
	The left-most vertical arrow are is an isomorphism because $I_r$ is geometrically connected for all $r$.
	Since $\SS$ is reduced, we have
		$H^0(I_r,\O_{\SS})=k^{\deg\SS}$
	for all $r$, so since $\pr:I_r\rightarrow I_{s}$ totally ramifies over every 
	supersingular point, the second left-most vertical arrow in 
	(\ref{pbH1inj}) is also an isomorphism.  
	Now the rightmost vertical map in (\ref{pbH1inj}) is	
	identified with the map on Lie algebras
	${\Lie \Pic^0_{I_{s}/k}} \rightarrow {\Lie\Pic^0_{I_r/k}}$
	induced by $\Pic^0(\pr)$, 
	which is injective thanks to Lemma \ref{MW} and the left-exacness of the functor $\Lie$.
	An easy diagram chase 
	using (\ref{pbH1inj}) then shows that (\ref{pbH1poles}) is injective, as claimed.

	Using again the equality $\pr^*(\SS_s)=p^{r-s}\cdot\SS_r$, 
	pullback of meromorphic differentials yields a mapping	
	\begin{equation}
		\xymatrix{
			{H^0(I_s,\Omega^1_{I_s/k}(\SS))} \ar[r] & {H^0(I_r, \Omega^1_{I_r/k}(p^{r-s}\cdot\SS))}
		}\label{diffPBinj}
	\end{equation}
	which is injective since $\pr:I_r\rightarrow I_s$ is separable.
	
	Dualizing the injective maps (\ref{pbH1poles}) and (\ref{diffPBinj}), we see that the canonical
	trace mappings
	\begin{subequations}
		\begin{equation}
			\xymatrix{
				{H^0(I_r,\Omega^1_{I_r/k}(\SS))} \ar[r]^-{\pr_*} & 	
				{H^0(I_{s},\Omega^1_{I_{s}/k}(\SS))}
			}
		\end{equation}
		\begin{equation}
			\xymatrix{
				{H^1(I_r,\O_{I_r}(-p^{r-s}\cdot \SS))} \ar[r]^-{\pr_*} & {H^1(I_{s},\O_{I_{s}}(-\SS))}
			}
		\end{equation}
	\end{subequations}	
	are surjective for all $r\ge s\ge 1$.
	Passing to $V$- (respectively $F$-) ordinary parts and using
	Lemma \ref{sspoles} (\ref{VControl}), we conclude that the canonical trace 
	mappings attached to $I_r\rightarrow I_s$ induce {\em surjective} maps
	as in Proposition \ref{IgusaStructure} (\ref{IgusaControl}).  By (\ref{IgusaFreeness}), 
	these mappings are then surjective
	mappings of free $\F_p[\Delta/\Delta_s]$-modules of the same rank, and are hence
	isomorphisms.  
\end{proof}

\begin{remark}\label{FullIgusaStructure}
	If $G$ is any cyclic group of $p$-power order, then the
	representation theory of $G$ is rather easy, even over a field $k$ of characteristic $p$.  
	Denoting by $\gamma$
	any fixed generator of $G$, for each integer $d$ with $1\le d\le \#G$, there is a unique 
	indecomposable representation of $G$ of dimension $d$, given explicitly by the 
	$k[G]$-module $V_d:=k[G]/(\gamma-1)^d$.
	By using Artin-Schreier theory for a $G$-cover of proper smooth curves $Y\rightarrow X$ over $k$, 
	for any $G$-stable Cartier divisor $D$ on $Y$ it is possible to determine
	the multiplicity of $V_d$ in the $k[G]$-module $H^0(Y,\Omega^1_{Y/k}(D))$ purely in terms
	of the ramification data of $Y\rightarrow X$.  This is carried out for $D=\emptyset$
	in \cite{ValentiniMadan}.  For the $G:=\Delta/\Delta_r$-cover $I_{r}\rightarrow I_1$, one finds
	\begin{equation*}
		H^0(I_r,\Omega^1_{I_r/k})\simeq k[G]^{g(I_1)} \oplus \left(k[G]/(\gamma-1)^{p^{r-1}-1}\right)^{p(\deg\SS_1)-1}
		\oplus\bigoplus_{d=1}^{p^{r-1}-2} \left(k[G]/(\gamma-1)^{d}\right)^{p(\deg\SS_1)}
	\end{equation*} 
	as $k[G]$-modules, where $g(I_1)$ is the genus of $I_1$.
\end{remark}

The space of meromorphic differentials $H^0(I_1,\Omega^1_{I_1/\F_p}(\SS))$
has a natural action of $\F_p^{\times}$ via the diamond operators $\langle \cdot\rangle$,
and the eigenspaces for this action are intimitely connected with mod $p$ cusp forms:

\begin{proposition}\label{MFmodp}
	Let $S_k(N;\F_p)$ be the space of weight $k$ cuspforms for $\Gamma_1(N)$ over $\F_p$,
	and denote by $H^0(I_r,\Omega^1_{I_1/\F_p}(\SS))(k-2)$ the subspace
	of $H^0(I_r,\Omega^1_{I_1/\F_p}(\SS))$ on which $\F_p^{\times}$ acts through the
	character $\langle u \rangle\mapsto u^{k-2}$. 
	For each $k$ with $ 2 < k < p+1$, there are canonical isomorphisms
	of $\F_p$-vector spaces
	\begin{equation}
	\addtocounter{equation}{1}
		S_k(N;\F_p) \simeq H^0(I_1,\Omega^1_{I_1/\F_p})(k-2) \simeq  H^0(I_1,\Omega^1_{I_1/\F_p}(\SS))(k-2) 
		\tag{$\arabic{section}.\arabic{subsection}.\arabic{equation}_k$}
		\label{WtkIsom}
	\end{equation}	
	which are equivariant for the Hecke operators, with $U_p$ acting as usual on modular forms
	and as the Cartier operator $V$ 
	on differential forms.  For $k=2,$ $p+1$, we have the following commutative diagram:
	\begin{equation*}
		\xymatrix{
			{S_2(N;\F_p)} \ar[r]^-{\simeq} \ar@{^{(}->}[d]_-{\cdot A} & 
			{H^0(I_1,\Omega^1_{I_1/\F_p})(0)} \ar@{^{(}->}[d] \\
			{S_{p+1}(N;\F_p)} \ar[r]^-{\simeq} & {H^0(I_1,\Omega^1_{I_1/\F_p}(\SS))(0)}  \\
			}
	\end{equation*}
	where $A$ is the Hasse invariant.
\end{proposition}

\begin{proof}
	This follows from Propositions 5.7--5.10 of \cite{tameness}, using Lemma \ref{CharacterSpaces};
	we note that our forward reference to Lemma \ref{CharacterSpaces} does not result in circular reasoning.
\end{proof}

\begin{remark}\label{dMFmeaning}
	For each $k$ with $2\le k\le p+1$, let us write $d_k:=\dim_{\F_p} S_k(N;\F_p)^{\ord}$
	for the $\F_p$-dimension of the subspace of weight $k$ level $N$ cuspforms over $\F_p$
	on which $U_p$ acts invertibly.
	As in Proposition \ref{IgusaStructure} (\ref{IgusaFreeness}), 
	let $\gamma$ be the $p$-rank of the Jacobian of
	$X_1(N)_{\F_p}$ and $\delta:=\deg\SS$.  It follows immediately from Proposition
	\ref{MFmodp} that we have the equality
	\begin{equation}
		d :=\gamma+\delta-1 = \sum_{k=3}^{p+1} d_k.
	\end{equation}
\end{remark}

\subsection{Structure of the ordinary part of \texorpdfstring{$H^0(\o{\X}_r,\omega_{\o{\X}_r/\F_p})$}{
H0(Xr,omega)}
}\label{OrdStruct}

Keep the notation of \S\ref{IgusaTower} and let $\X_r/R_r$ be as in Definition \ref{XrDef}. 
As before, we denote by $\o{\X}_r:=\X_r\times_{R_r} \F_p$ the special fiber of $\X_r$; it is a 
curve over $\F_p$ in the sense of Definition \ref{curvedef}. 
In this section, using Rosenlicht's theory of the dualizing sheaf as explained in \S\ref{GD}
and the explicit description of $\o{\X}_r$ given by Proposition \ref{redXr},
we will compute the {\em ordinary part} of the cohomology $H(\o{\X}_r/\F_p)$
in terms of the de Rham cohomology of the Igusa tower.

For notational ease, as in Remark \ref{MWGood} we write $I_r^{\infty}:=I_{(r,0,1)}$ and 
$I_r^0:=I_{(0,r,1)}$ for the two ``good" components of $\o{\X}_r$.
Each of these components is abstractly isomorphic to the Igusa curve $\Ig(p^r)$ of level $p^r$ over 
$X_1(N)_{\F_p}$, and we will henceforth make this identification; for $s\le r$, we will 
write simply $\pr:I_r^{\star}\rightarrow I_s^{\star}$ for the 
the canonical degeneracy map induced by (\ref{Vmapsch}).
Using Proposition \ref{UlmerProp}, one checks that
the $\H_r$-correspondences on $\X_r$ restrict to the $\H_r$-correspondences on $I_r^{\infty}$,
(the point is that the degeneracy maps defining $U_p$ on $\X_r$ restrict to a correspondence
on $I_r^{\infty}$), while the $\H_r^*$-correspondences
on $\X_r$ restrict to the $\H_r^*$-correspondences on $I_r^{0}$.
In particular, $U_p=(F,\langle p\rangle_N)$ on $I_r^{\infty}$ and $U_p^*=(F,\id)$
on $I_r^0$.
For $\star=0,\infty$, we denote by $i_r^{\star}: I_r^{\infty}\hookrightarrow \o{\X}_r$ the canonical closed immersion.

\begin{proposition}\label{charpord}
	For each positive integer $r$, pullback of differentials along		
	$i_r^{0}$ $($respectively $i_r^{\infty}$$)$ induces a natural and $\H_r^*$ $($resp.
	$\H_r$$)$-equivariant isomorphism of $\F_p[\Delta/\Delta_r]$-modules
		\begin{equation}
				\xymatrix{
					{e_r^*H^0(\o{\X}_r,\omega_{\o{\X}_r})} \ar[r]^-{\simeq}_-{(i_r^{0})^*} &
					{H^0(I_r^{0},\Omega^1_{I_r^{0}}(\SS))^{V_{\ord}}}
					},
		\ \text{resp.}\ 
				\xymatrix{
					{e_rH^0(\o{\X}_r,\omega_{\o{\X}_r})} \ar[r]^-{\simeq}_-{(i_r^{\infty})^*} &
					{H^0(I_r^{\infty},\Omega^1_{I_r^{\infty}}(\SS))^{V_{\ord}}}
					}.
					\label{I'compIsom}			
		\end{equation}	
	which is $\Gamma$-equivariant for the ``geometric inertia action" $(\ref{gammamaps})$	
	on $\o{\X}_r$ and the action $\gamma\mapsto \langle \chi(\gamma)\rangle^{-1}$ on $I_r^0$
	$($respectively the trivial action on $I_r^{\infty}$$)$.
	The isomorphisms $(\ref{I'compIsom})$ induce identifications that are
	compatible with change in $r$:
	the four diagrams formed
	by taking 
	the interior or the exterior arrows
	\begin{equation}
		\begin{gathered}
			\xymatrix@R=30pt{
				{e_r^*H^0(\o{\X}_r,\omega_{\o{\X}_r})} 
				\ar@<-0.5ex>[r]_-{\langle p\rangle_N^{r} (i_r^0)^*}
				\ar@<0.5ex>[r]^-{F_*^r (i_r^0)^*} \ar@<-0.5ex>[d]_-{\pr_*} &
				{H^0(I_r^{0},\Omega^1_{I_r^{0}}(\SS))^{V_{\ord}}} \ar@<0.5ex>[d]^-{\pr_*}  \\
				{e_{s}^*H^0(\o{\X}_{s},\omega_{\o{\X}_{s}})} 
				\ar@<-0.5ex>[r]_-{F_*^s (i_s^0)^*}
				\ar@<0.5ex>[r]^-{\langle p\rangle_N^{s} (i_s^0)^*}
				\ar@<-0.5ex>[u]_-{\ps^*} &
				{H^0(I_s^{0},\Omega^1_{I_s^{0}}(\SS))^{V_{\ord}}} \ar@<0.5ex>[u]^-{\pr^*} 
			}
		\quad\raisebox{-24pt}{and}\quad
			\xymatrix@R=30pt{
				{e_rH^0(\o{\X}_r,\omega_{\o{\X}_r})} \ar@<-0.5ex>[r]_-{(i_r^{\infty})^*}
				\ar@<0.5ex>[r]^-{F_*^r (i_r^{\infty})^*}
				\ar@<-0.5ex>[d]_-{\ps_*} &
				{H^0(I_r^{\infty},\Omega^1_{I_r^{\infty}}(\SS))^{V_{\ord}}} \ar@<0.5ex>[d]^-{\pr_*}  \\
				{e_{s}H^0(\o{\X}_{s},\omega_{\o{\X}_{s}})} \ar@<-0.5ex>[r]_-{F_*^s (i_s^{\infty})^*}
				\ar@<0.5ex>[r]^-{(i_s^{\infty})^*}
				\ar@<-0.5ex>[u]_-{\pr^*} &
				{H^0(I_s^{\infty},\Omega^1_{I_s^{\infty}}(\SS))^{V_{\ord}}} \ar@<0.5ex>[u]^-{\pr^*} 
			}			
		\end{gathered}\label{rCompatDiagrams}
		\end{equation}
	are all commutative for $s\le r$.
	Via the automorphism $\o{w}_{r}$ of $\o{\X}_r$ and the identification
	$I_{r}^0\simeq \Ig(p^r)\simeq I_{r}^{\infty}$, the first diagram of $(\ref{rCompatDiagrams})$
	is carried isomorphically and compatibly on to the second.
	The same assertions hold true if we replace $\o{\X}_r$ with $\nor{\o{\X}}_r$ and $\Omega^1_{I_r^{\star}}(\SS)$
	with $\Omega^1_{I_r^{\star}}$ throughout.
		\end{proposition}

\begin{proof}
	We may and do work over $k:=\o{\F}_p$, and we abuse notation slightly by 
	writing $\o{\X}_r$ for the {\em geometric} special fiber of $\X_r$.
	If $X$ is an $\F_p$-scheme, we likewiseagain write $X$ it's base change to $k$,
	and we write $F:X\rightarrow X$ for the base change of the absolute Frobenius of
	$X$ over $\F_p$ to $k$.
	Let $\nor{\o{\X}}_r\rightarrow \o{\X}_r$ be the
	normalization map; by Proposition \ref{redXr}, we know that 
	$\nor{\o{\X}}_r$ is the disjoint union of proper smooth and irreducible
	Igusa curves $I_{(a,b,u)}$ indexed by triples $(a,b,u)$ with 
	with $a,b$ nonnegative integers satisfying $a+b=r$ and 
	$u\in(\Z/p^{\min(a,b)}\Z)^{\times}$.  Via Proposition \ref{Rosenlicht}, 
	we identify $\omega_{\o{\X}_r/k}$ with Rosenlicht's sheaf $\omega_{\o{\X}_r/k}^{\reg}$ of
	regular differentials, and we simply write $\omega_{\o{\X}_r}$ for this sheaf.
	By Definition \ref{OmegaReg} and Remark \ref{OmegaRegMero}, we have a functorial
	injection of $k$-vector spaces
	\begin{equation}
		 \xymatrix{
		 	{H^0(\o{\X}_r,\omega_{\o{\X}_r})}\ar@{^{(}->}[r] & 
			{H^0(\nor{\o{\X}}_r,\underline{\Omega}^1_{k(\nor{\o{\X}}_r)})\simeq 
			\prod\limits_{(a,b,u)}\Omega^1_{k(I_{(a,b,u)})}}
			}\label{dualizing2prod}
	\end{equation}
	with image precisely those elements $(\eta_{(a,b,u)})$ of the product that satisfy
	$\sum \res_{x_{(a,b,u)}}(s\eta_{(a,b,u)}) =0$
	for each supersingular point $x\in \o{\X}_r(k)$ and all $s\in \O_{\o{\X}_r,x}$, where
	$x_{(a,b,u)}$ is the unique point of $I_{(a,b,u)}$ lying
	over $x$ and the sum is over all triples $(a,b,u)$ as above.
	We henceforth identify $\eta\in H^0(\o{\X}_r,\omega_{\o{\X}_r})$ with its image 
	under (\ref{dualizing2prod}), and we denote by
	$\eta_{(a,b,u)}$ the $(a,b,u)$-component of $\eta$.

	Recall from (\ref{HeckeDef})
	that the correspondence $U_p:=(\pi_1,\pi_2)$ 
	on $\X_r$ given by the 
	degeneracy maps $\pi_1,\pi_2:\Y_r\rightrightarrows \X_r$ of (\ref{Upcorr}) 
	yields endomorphisms $U_p:=(\pi_1)_*\circ\pi_2^*$ and $U_p^*:=(\pi_2)_*\circ\pi_1^*$
	of $H^0(\X_r,\omega_{\X_r/R_r})$;
	we will again denote by $U_p$ and $U_p^*$ the induced endomorphisms $U_p\otimes 1$ and $U_p^*\otimes 1$ of
	\begin{equation*}
		H^0(\o{\X}_r,\omega_{\o{\X}_r}) \simeq H^0(\X_r,\omega_{\X_r/R_r})\otimes_{R_r} k,
	\end{equation*}
	where the isomorphism is
	the canonical one of Lemma \ref{ReductionCompatibilities} (\ref{BaseChngDiagram}).	
	By the functoriality of normalization,
	we have an induced correspondence $U_p:=(\nor{\o{\pi}}_1,\nor{\o{\pi}}_2)$
	on $\nor{\o{\X}}_r$,
	and we write $U_p$ and $U_p^*$ for the resulting endomorphisms (\ref{HeckeDef})
	of $H^0(\nor{\o{\X}}_r,\underline{\Omega}^1_{k(\nor{\o{\X}}_r)})$.
	By Lemma \ref{ReductionCompatibilities} (\ref{PTBCCompat}), 
	the map (\ref{dualizing2prod}) is then $U_p$ and $U_p^*$-equivariant.  
	The  Hecke correspondences away from $p$
	and the diamond operators act on the source of (\ref{dualizing2prod})
	via ``reduction modulo $p$" and on the target via the induced
	correspondences in the usual way (\ref{HeckeDef}), and 
	the map (\ref{dualizing2prod}) compatible with these
	actions 
	thanks to Lemma 
	\ref{ReductionCompatibilities} (\ref{PTBCCompat}).  Similarly, 
	the semilinear ``geometric inertia" action of $\Gamma:=\Gal(K_{\infty}/K_0)$
	on $\X_r$ induces a linear action on $\nor{\o{\X}}_r$ as in Proposition
	\ref{AtkinInertiaCharp} (\ref{InertiaCharp}), and the map (\ref{dualizing2prod}) is equivariant
	with respect to these actions. 		
	
	We claim that for {\em any} meromorphic differential $\eta = (\eta_{(a,b,u)})$ on $\nor{\o{\X}}_r$, we 
	have
	\begin{subequations}
	\begin{equation}
			\left({U_p}\eta \right)_{(a,b,u)} = \begin{cases}
				 F_*\eta_{(r,0,1)} &:\quad (a,b,u)=(r,0,1)\\
				\pr_*\eta_{(a+1,b-1,u)} &:\quad 0 < b \le a \\
				\sum\limits_{\substack{u'\in (\Z/p^{a+1}\Z)^{\times} \\ u'\equiv u\bmod p^{a}}} 
				\langle u'\rangle \eta_{(a+1,b-1,u)}
				 &:\quad r\ \text{odd},\ a=b-1\\
				 \sum\limits_{\substack{u'\in (\Z/p^{a+1}\Z)^{\times} \\ u'\equiv u\bmod p^{a}}} 
				\rho^*\langle u'\rangle \eta_{(a+1,b-1,u')}
				 &:\quad r\ \text{even},\ a=b-2\\				 
				\sum\limits_{\substack{u'\in (\Z/p^{a+1}\Z)^{\times}\\ u'\equiv u\bmod p^a}} 
				\pr^*\eta_{(a+1,b-1,u')} &:\quad 0 \le a < b-2\\
			\end{cases}\label{Up1}
	\end{equation}
	The proof of this claim is an easy exercise using the definition of $U_p$, the explicit description of 
	the maps $\nor{\o{\pi}}_1$ and $\nor{\o{\pi}}_2$
	given in Proposition \ref{UlmerProp}, and the fact that $F^*$ kills any global 
	meromorphic differential form on a scheme of characteristic $p$.	
	In a similar manner, one derives the explicit description 
	\begin{equation}
			\left(U_p^*\eta \right)_{(a,b,u)} = \begin{cases}
				\langle p\rangle_N^{-1}F_*\eta_{(0,r,1)} &:\quad (a,b,u)=(0,r,1)\\
				\pr_*\eta_{(a-1,b+1,u)} &:\quad 0 < a < b \\
				\langle u\rangle^{-1}\pr_*\eta_{(a-1,b+1,u)} &:\quad r\ \text{even},\ b=a\\				 
				 \sum\limits_{\substack{u'\in (\Z/p^{b+1}\Z)^{\times} \\ u'\equiv u\bmod p^{b}}} 
				\langle u'\rangle^{-1} \eta_{(a-1,b+1,u')}
				 &:\quad r\ \text{odd},\ b=a-1\\
				\sum\limits_{\substack{u'\in (\Z/p^{b+1}\Z)^{\times}\\ u'\equiv u\bmod p^b}} 
				\pr^*\eta_{(a-1,b+1,u')} &:\quad 0 \le b < a-1\\
			\end{cases}\label{Up2}
	\end{equation}
	\end{subequations}
	The crucial observation for our purposes is that for $0 < b \le r$, the $(a,b,u)$-component of 
	${U_p}\eta$ depends only on the $(a+1,b-1,u')$-components of $\eta$ for varying $u'$, and similarly
	for $0<a\le r$ the $(a,b,u)$-component of $U_p^*\eta$ depends only on the $(a-1,b+1,u')$-components 
	of $\eta$.
	By induction, we deduce
	\begin{subequations}
	\begin{equation}
		\left(U_p^n\eta \right)_{(a,b,u)} = \begin{cases}
				\pr_*^b F_*^{n-b}\eta_{(r,0,1)} &:\quad  b \le a \\
				\sum\limits_{\substack{u'\in (\Z/p^{b}\Z)^{\times}\\ u'\equiv u\bmod p^a}} 
				\langle u'\rangle \pr_*^a F_*^{n-b}\eta_{(r,0,1)} &:\quad a < b 			
				\end{cases}\label{Upn1}
	\end{equation}
	and
	\begin{equation}
		\left({U_p^*}^n\eta \right)_{(a,b,u)} = \begin{cases}
				\pr_*^a\langle p\rangle_N^{a-n}F_*^{n-a}\eta_{(0,r,1)} &:\quad  a < b \\
				\sum\limits_{\substack{u'\in (\Z/p^{a}\Z)^{\times}\\ u'\equiv u\bmod p^b}} 
				\langle u'\rangle^{-1} \pr_*^b\langle p\rangle_N^{a-n} F_*^{n-a}\eta_{(0,r,1)} &:\quad b \le a				
				\end{cases}\label{Upn2}
	\end{equation}
	\end{subequations}
	for any $n\ge r\ge 1$.
	
	For any $r>0$ and for $\star=\infty, 0$ we define maps
	\begin{align*}
		\xymatrix{
			{\gamma_r^{\star}: H^0(I_r^{\star},\Omega^1_{I_r^{\star}}(\SS))^{V_{\ord}}} \ar[r] & 
			{H^0(\nor{\o{\X}}_r,\underline{\Omega}^1_{k(\nor{\o{\X}}_r)})}
			}
	\end{align*}
	by
	\begin{subequations}
		\begin{equation}
			(\gamma_{r}^{\infty}(\eta))_{(a,b,u)}: = \begin{cases}
				\pr_*^bF_*^{-b}\eta &:\quad  b\le a\\
		\sum\limits_{\substack{u'\in (\Z/p^{b}\Z)^{\times}\\ u'\equiv u\bmod p^a}} 
				\langle u'\rangle \pr_*^a F_*^{-b}\eta &:\quad a < b \\
				\end{cases}\label{betadef}
		\end{equation}
		and
		\begin{equation}
		(\gamma_{r}^{0}(\eta))_{(a,b,u)}: = \begin{cases}
				\pr_*^a\langle p\rangle_N^{a}F_*^{-a}\eta &:\quad  a < b\\
		\sum\limits_{\substack{u'\in (\Z/p^{a}\Z)^{\times}\\ u'\equiv u\bmod p^b}} 
				\langle u'\rangle^{-1} \pr_*^b \langle p\rangle_N^{a} F_*^{-a}\eta_{(0,r,1)} &:\quad b \le a				

				\end{cases}\label{gammadef}
		\end{equation}
	\end{subequations}
	These maps are well-defined because $F_*=V$ is invertible on the $V$-ordinary subspace,
	and they are immediately seen to be injective by looking at $(r,0,1)$-components.
	Note moreover that the $(a,b,u)$-component of $\gamma_r^{\star}(\eta)$ is independent of $u$.
	
	We claim that 
	the maps $\gamma_r^{\star}$ have image in
	$H^0(\nor{\o{\X}}_r,\omega_{\nor{\o{\X}}_r})$ (i.e. that they factor through (\ref{dualizing2prod})).
	To see this, we proceed as follows. 
	Suppose that $x$ is any supersingular point on $\o{\X}_r$ and $s\in \O_{\o{\X}_r,x}$ is arbitrary. 
	By Proposition \ref{Rosenlicht} and Definition \ref{OmegaReg},
	we must check that the sum of the residues of $s\gamma^{\infty}(\eta)$
	at all $k$-points of $\nor{\o{\X}}_r$ lying over $x$ is zero.  
	Using (\ref{betadef}), we calculate that this sum is equal to
	\begin{align}
		\sum_{b \le a} 
		\sum_{u\in (\Z/p^b\Z)^{\times}}\res_{x_{(a,b,u)}}(s\pr_*^bF_*^{-b}\eta)
		+\sum_{a < b} \sum_{u\in (\Z/p^b\Z)^{\times}}
		\res_{x_{(a,b,u)}}(s\langle u\rangle \pr_*^aF_*^{-b}\eta)
	\label{residuecalc1}
	\end{align}
	where $x_{(a,b,u)}$ denotes the unique point of the $(a,b,u)$-component of $\nor{\o{\X}}_r$ over $x$,
	and the outer sums range over all nonnegative integers $a,b$ with $a+b=r$.
	We claim that for any meromorphic differential $\omega$ on $I_{(a,b,u)}$ and any 
	supersingular point $y$ of $I_{(a,b,u)}$ over $x$, we have
	\begin{subequations}
	\begin{equation}
		 \res_y(\omega) = \res_y(\langle u\rangle\omega)\label{DiamondFix}
	\end{equation}
	for all $u\in \Z_p^{\times}$, and, if in addition $\omega$ is $V$-ordinary,
	\begin{equation}
		\res_y(s\omega) = s(x)\res_y(\omega)\label{FunctionPullOut}
	\end{equation}
	\end{subequations}
	Indeed, (\ref{DiamondFix}) is a consequence of (\ref{TateFormula}),
	using the fact that the automorphism
	$\langle u\rangle$ of $I_{(a,b,u)}$ fixes every supersingular point, 
	while (\ref{FunctionPullOut}) is deduced by thinking about formal expansions of differentials
	at $y$ and using the fact that a $V$-ordinary meromorphic differential has at worst simple poles 
	thanks to Lemma \ref{sspoles}.  Via (\ref{DiamondFix})--(\ref{FunctionPullOut}), we reduce the
	sum (\ref{residuecalc1}) to
	\begin{align}
		\sum_{a + b = r} 
		\sum_{u\in (\Z/p^b\Z)^{\times}} s(x)\res_{x_{(a,b,u)}}(\pr_*^{\min(a,b)}F_*^{-b}\eta)
		&= \sum_{a + b = r} 
		\varphi(p^b)s(x)\res_{x_{(a,b,1)}}(\pr_*^{\min(a,b)}F_*^{-b}\eta)\nonumber\\
		&=s(x)\res_{x_{(r,0,1)}}(\eta) - s(x)\res_{x_{(r,0,1)}}(F_*^{-1}\eta)
		\label{twoterm}
	\end{align}
	where the first equality above follows from the fact
	that for {\em fixed} $a,b$, the points $x_{(a,b,u)}$ for varying $u\in (\Z/p^{\min(a,b)}\Z)^{\times}$
	are all identified with the {\em same} point on $\Ig(p^{\max(a,b)})$, and the second 
	equality is a consequence of (\ref{TateFormula}), since $\rho(x_{(r,0,1)})=x_{(r-1,1,1)}$.
	As $\eta$ is $V$-ordinary, there exists a $V$-ordinary meromorphic differential $\xi$
	on $I_r^0$ with $\eta=F_*\xi$; substituting this expression for $\eta$ in to (\ref{twoterm})
	and applying (\ref{TateFormula}) once more, we conclude that (\ref{twoterm}) is zero,
	as desired.
	That $\gamma_r^{0}$ has image in $H^0(\o{\X}_r,\omega_{\o{\X}_r/k})$ 
	follows from a nearly identical calculation, and we omit the details.

		It follows immediately from our calculations (\ref{Up1})--(\ref{Up2}) and the definitions 
		(\ref{betadef})--(\ref{gammadef}) that
		the relations  $U_p\circ \gamma_r^{\infty}=\gamma_r^{\infty}\circ F_*$
		and ${U_p^*}\circ\gamma^0_r=\gamma_r^0\circ \langle p\rangle_N^{-1}F_*$ hold.  
		Since $F_*$ is invertible on the source of $\gamma_r^{\star}$,
		it follows immediately 
		that $\gamma_r^0$  has image
		contained in $e_r^*H^0(\o{\X}_r,\omega_{\o{\X}_r})$ and that
		$\gamma_r^{\infty}$ has image contained in
		$e_rH^0(\o{\X}_r,\omega_{\o{\X}_r})$.

		To see that these containments an equalities, we proceed as follows.
		Suppose that $\xi\in e_rH^0(\o{\X}_r,\omega_{\o{\X}_r})$ is arbitrary.  We claim that 
		the meromorphic differential $\xi_{(r,0,1)}$ on $I_{r}^{\infty}$ has at worst {\em simple}
		poles along $\SS$ (and is holomorphic outside $\SS$).  Indeed, for each $n>0$ we may find
		$\xi^{(n)}\in {e_r}H^0(\o{\X}_r,\omega_{\o{\X}_r})$ with $\xi=U_p^n\xi^{(n)}$.  
		As discussed in \S\ref{GD}, when viewed as a meromorphic differential on $\nor{\o{\X}}_r$
		any section of $\omega_{\o{\X}_r}$
		has poles of order bounded
		by a constant depending only on $r$ (see \cite[Lemma 5.2.2]{GDBC}).   		
		Since $F:I_r^{\infty}\rightarrow I_r^{\infty}$ is inseparable of degree $p$ (so
		totally ramified over every
		supersingular point), it follows from Remark \ref{poletrace} that there 
		exists $n > r$ such that 
		the meromorphic differential
		$F_*^{n}\xi^{(n)}_{(r,0,1)}$ has at worst simple poles along $\SS$;
		by the formula (\ref{Upn1}) for $U_p^n$, we conclude that the same is true of 
		$$\xi_{(r,0,1)} = (U_p^n\xi^{(n)})_{(r,0,1)}=F_*^n\xi^{(n)}_{(r,0,1)}.$$  
		Applying this with $\xi^{(r)}$ in the role of $\xi$, and using
		(\ref{Upn1}) and (\ref{betadef}) we calculate
		\begin{equation}
			\xi=U_p^{r}\xi^{(r)} = \gamma_r^{\infty}(F_*^r\xi^{(r)}_{(r,0,1)}),
		\end{equation}
		so $\gamma_r^{\infty}$ surjects onto $e_rH^0(\o{\X}_r,\omega_{\o{\X}_r})$ and is
		hence an isomorphism onto this image.  A nearly identical argument shows that 
		$\gamma_r^{0}$ is an isomorphism onto $e_r^*H^0(\o{\X}_r,\omega_{\o{\X}_r})$.
		
		Since pullback of meromorphic differentials along 
		$i_r^{\infty}:I_r^{\infty}\hookrightarrow \nor{\o{\X}}_r$ is given by projection
		\begin{equation}
			\xymatrix@C=35pt{
			{H^0(\nor{\o{\X}}_r,\underline{\Omega}^1_{k(\nor{\o{\X}}_r)})\simeq 
			\prod\limits_{(a,b,u)} H^0(I_{(a,b,u)},\u{\Omega}^1_{k(I_{(a,b,u)})})} \ar[r]^-{\proj_{(r,0,1)}} & 
			{H^0(I_r^{\infty},\u{\Omega}^1_{k(I_r^{\infty})})}
			}\label{PBisProj}
		\end{equation}
		onto the $(r,0,1)$-component,
		the composition of $\gamma_r^{\infty}$ and (the restriction of) 
		$(i_r^{\infty})^*$ in either
		order is the identity map.  Since $i_r^{\infty}$ is compatible with the $\H_r$-correspondences, 
		the resulting isomorphism (\ref{I'compIsom})
		is $\H_r$-equivariant (with $U_p$ acting on the target via $F_*$).  Similarly, 
		since the ``geometric inertia" action (\ref{gammamaps}) of 
		$\Gamma$ on $\X_r$ is compatible via $i_r^{\infty}$ with the trivial action
		on $I_r^{\infty}$ by Proposition \ref{AtkinInertiaCharp},		
		the isomorphism (\ref{I'compIsom}) is equivariant for these actions of $\Gamma$.
		A nearly identical analysis shows that $(i_r^{0})^*$ is $\H_r^*$-compatible
		(with $U_p^*$ acting on the target as $\langle p\rangle_N^{-1} F_*$) and $\Gamma$-equivariant for the
		action of $\Gamma$ on $I_r^{0}$ via $\langle \chi(\cdot)\rangle^{-1}$
		The commutativity of
		the four diagrams in (\ref{rCompatDiagrams}) is an immediate consequence
		of the descriptions of the degeneracy mappings $\o{\pr},\o{\ps}$ on $\nor{\o{\X}}_r$ furnished by Proposition
		\ref{pr1desc} and the explication (\ref{PBisProj}) of pullback by $i_r^{\star}$ in terms
		of projection.  That $\o{w}_r$ interchanges the two diagrams 
		in (\ref{rCompatDiagrams}) is an immediate consequence of Proposition \ref{ALinv}.		
		
		Finally, that the assertions of Proposition \ref{charpord} all hold if $\o{\X}_r$ and
		 $\Omega^1_{I_r^{\star}}(\SS)$
		are replaced by $\nor{\o{\X}}_r$ and $\Omega^1_{I_r^{\star}}$, respectively, follows from a
		a similar---but much simpler---argument.  The point is that the maps $\gamma_r^{\star}$ 
		of (\ref{betadef})--(\ref{gammadef})
		visibly carry $H^0(I_r^{\star},\Omega^1_{I_r^{\star}})^{V_{\ord}}$ into 
		$H^0(\nor{\o{\X}}_r,\Omega_{\nor{\o{\X}}_r}^1)$, from which it follows 
		via our argument that they induce the claimed isomorphisms.
\end{proof}		
		
Since $\o{\X}_r$ is a proper and geometrically connected curve over $\F_p$, Proposition \ref{HodgeFilCrvk} 
(\ref{HodgeDegenerationField}) provides short exact sequences of 
$\F_p[\Delta/\Delta_r]$-modules with linear $\Gamma$ and $\H_r^*$ (respectively $\H_r$)-action
\begin{subequations}
\begin{equation}
	\xymatrix{
		0\ar[r] & {e_r^*H^0(\o{\X}_r,\omega_{\o{\X}_r/\F_p})} \ar[r] & {e_r^*H^1(\o{\X}_r/\F_p)} \ar[r] &
		{e_r^*H^1(\o{\X}_r,\O_{\o{\X}_r})} \ar[r] & 0
	}\label{sesincharp1}
\end{equation}
respectively
\begin{equation}
	\xymatrix{
		0\ar[r] & {e_rH^0(\o{\X}_r,\omega_{\o{\X}_r/\F_p})} \ar[r] & {e_rH^1(\o{\X}_r/\F_p)} \ar[r] &
		{e_rH^1(\o{\X}_r,\O_{\o{\X}_r})} \ar[r] & 0
	}\label{sesincharp2}
\end{equation}
\end{subequations}
which are canonically $\F_p$-linearly dual to each other.  We likewise have such exact sequences
in the case of $\nor{\o{\X}}_r$; note that since $\nor{\o{\X}}_r$ is smooth,
the short exact sequence $H(\nor{\o{\X}}_r/\F_p)$ is simply the Hodge filtration
of $H^1_{\dR}(\nor{\o{\X}}_r/\F_p)$.

\begin{corollary}\label{SplitIgusa}
		The absolute Frobenius morphism of $\o{\X}_r$ over $\F_p$ induces a natural
		$\F_p[\Delta/\Delta_r]$-linear, $\Gamma$-compatible, and $\H_r^*$
		$($respectively $\H_r$$)$ equivariant splitting of $(\ref{sesincharp1})$
		$($respectively $(\ref{sesincharp2})$$)$.
		Furthermore, for each $r$ we have natural isomorphisms of split short exact sequences 
		\begin{subequations}
		\begin{equation}
			\xymatrix@C=15pt{
				0\ar[r] & {e_r^*H^0(\o{\X}_r,\omega_{\o{\X}_r/\F_p})} \ar[r]\ar[d]_-{F_*^r (i_r^0)^*}^-{\simeq} & 
				{e_r^*H^1(\o{\X}_r/\F_p)} \ar[r]\ar[d]^-{\simeq} &
				{e_r^*H^1(\o{\X}_r,\O_{\o{\X}_r})} \ar[r] & 0\\
				0 \ar[r] & {H^0(I_r^{0},\Omega^1(\SS))^{V_{\ord}}} \ar[r] &
				{H^0(I_r^{0},\Omega^1(\SS))^{V_{\ord}}\oplus
				H^1(I_r^{\infty},\O(-\SS))^{F_{\ord}}} \ar[r] & 
				{H^1(I_r^{\infty},\O(-\SS))^{F_{\ord}}} \ar[r]\ar[u]_-{\VDual{(i_r^{\infty})^*}}^-{\simeq} & 0
			}\label{LowerIsom1}
		\end{equation}
		\begin{equation}
			\xymatrix@C=15pt{
				0\ar[r] & {e_rH^0(\o{\X}_r,\omega_{\o{\X}_r/\F_p})} 
				\ar[r]\ar[d]_-{F_*^r (i_r^{\infty})^*}^-{\simeq} & 
				{e_rH^1(\o{\X}_r/\F_p)} \ar[r]\ar[d]^-{\simeq} &
				{e_rH^1(\o{\X}_r,\O_{\o{\X}_r})} \ar[r] & 0\\
				0 \ar[r] & {H^0(I_r^{\infty},\Omega^1(\SS))^{V_{\ord}}} \ar[r] &
				{H^0(I_r^{\infty},\Omega^1(\SS))^{V_{\ord}}\oplus
				H^1(I_r^{0},\O(-\SS))^{F_{\ord}}} \ar[r] & 
				{H^1(I_r^{0},\O(-\SS))^{F_{\ord}}} \ar[r]
				\ar[u]_-{\VDual{(i_r^{0})^*}\langle p\rangle_N^{-r}}^-{\simeq} & 0
			}\label{UpperIsom1}
		\end{equation}
		\end{subequations}
		which are compatible with the extra structures.
		  The identification
		$(\ref{LowerIsom1})$  $($respectively $(\ref{UpperIsom1})$$)$ is moreover
		compatible with change in $r$ using the trace mappings attached
		to $\pr: I_r^{\star}\rightarrow I_{r-1}^{\star}$ and to 
		$\o{\pr}:\o{\X}_r\rightarrow \o{\X}_{r-1}$ $($respectively 
		$\o{\ps}:\o{\X}_r\rightarrow \o{\X}_{r-1}$$)$.
		The same statements hold true if we replace $\o{\X}_r$, $\Omega^1_{I_r^{\star}}(\SS)$,
		and $\O_{I_r^{\star}}(-\SS)$ with $\nor{\o{\X}}_r$, $\Omega^1_{I_r^{\star}}$, 
		and $\O_{I_r^{\star}}$, respectively.
\end{corollary}

\begin{proof}
		Pullback by the absolute Frobenius endomorphism of $\o{\X}_r$
		induces an endomorphism of (\ref{sesincharp1}) which 
		kills $H^0(\o{\X}_r,\omega_{\o{\X}_r/\F_p})$
		and so yields a morphism of $\F_p[\Delta/\Delta_r]$-modules
		\begin{equation}
			\xymatrix{
				{e_r^*H^1(\o{\X}_r,\O_{\o{\X}_r})}\ar[r] & {{e_r^*}H^1(\o{\X}_r/\F_p)} 
			}\label{H1Splitting}
		\end{equation}
		that 
		is $\Gamma$ and $\H_r^*$-compatible 
		and projects to the endomorphism $F^*$ of $e_r^*H^1(\o{\X}_r,\O_{\o{\X}_r})$.
		On the other hand, Proposition \ref{charpord} gives a natural $\Gamma$ and $\H_r^*$-equivariant 
		isomorphism of $\F_p[\Delta/\Delta_r]$-modules
		\begin{equation}
			\xymatrix{
				{H^1(I_r^{\infty},\O_{I_r^{\infty}}(-\SS))^{F_{\ord}}}\ar[r]^-{\VDual{(i_r^{\infty})^*}} &  
				{e_r^*H^1(\o{\X}_r,\O_{\o{\X}_r})}
			}.\label{IsomOnH1}
		\end{equation}
		As this isomorphism intertwines $F^*$ on source and target, 
		we deduce that $F^*$ acts invertibly on 
		${e_r^*H^1(\o{\X}_r,\O_{\o{\X}_r})}$. We may therefore pre-compose (\ref{H1Splitting}) with $(F^*)^{-1}$
		to obtain a canonical splitting of (\ref{sesincharp1}), which by construction is
		$\F_p[\Delta/\Delta_r]$-linear and compatible with $\Gamma$ and $\H_r^*$.  
		The existence of (\ref{LowerIsom1}) as well
		as its compatibility with $\Gamma$, $\H_r^*$ and with change in $r$ now follows
		immediately from Proposition \ref{charpord} and duality (see Remark \ref{DualityOfFVOrd}).
		The corresponding assertions for the exact sequence (\ref{sesincharp2}) and the 
		diagram (\ref{UpperIsom1}) are proved similarly, and we leave the details to the reader.
		A nearly identical argument shows that the same assertions hold true when $\o{\X}_r$, 
		$\Omega^1_{I_r^{\star}}(\SS)$,
		and $\O_{I_r^{\star}}(-\SS)$ are replaced by $\nor{\o{\X}}_r$,
		 $\Omega^1_{I_r^{\star}}$, and $\O_{I_r^{\star}}$, respectively.
\end{proof}

\begin{corollary}\label{FreenessInCharp}
	The exact sequences
	$(\ref{sesincharp1})$ and $(\ref{sesincharp2})$ are split short exact sequences of 
	free $\F_p[\Delta/\Delta_r]$-modules 
	whose terms have $\F_p[\Delta/\Delta_r]$-ranks
	$d$, $2d$, and $d$, respectively, for $d$ as in 
	Remark $\ref{dMFmeaning}$.
	For $s\le r$, the degeneracy maps $\pr,\ps:\X_r\rightrightarrows \X_s$ 
	induce natural isomorphisms of
	exact sequences
	\begin{align*}
		&\xymatrix{
			{\pr_*:e_r^*H(\o{\X}_r/\F_p) \tens_{\F_p[\Delta/\Delta_s]} \F_p[\Delta/\Delta_r]} \ar[r]^-{\simeq} & 
			{e_s^*H(\o{\X}_s/\F_p)}
			}\\
		&\xymatrix{
			{\ps_*:e_rH(\o{\X}_r/\F_p) \tens_{\F_p[\Delta/\Delta_s]} \F_p[\Delta/\Delta_r]} \ar[r]^-{\simeq} & 
			{e_sH(\o{\X}_s/\F_p)}
			}	
	\end{align*}
	that are $\Gamma$ and $\H_r^*$ $($respectively $\H_r$$)$ equivariant.
\end{corollary}

\begin{proof}
	This follows immediately from Proposition \ref{IgusaStructure} and Corollary \ref{SplitIgusa}.
\end{proof}

\begin{remark}
We warn the reader that the na\"ive analogue of Corollary \ref{FreenessInCharp}
in the case of $\nor{\o{\X}}_r$ is false: while
$H^0(I_r,\Omega^1(\SS))^{V_{\ord}}$ is a free $\F_p[\Delta/\Delta_r]$-module, the submodule
of holomorphic differentials need {\em not} be. 
Over $k=\o{\F}_p$, the residue map gives a short exact sequence of $k[\Delta/\Delta_r]$-modules
\begin{equation*}
	\xymatrix{
		0\ar[r] & {H^0(I_r,\Omega^1_{I_r/k})^{V_{\ord}}} \ar[r] & {H^0(I_r,\Omega^1_{I_r/k}(\SS))^{V_{\ord}}}
		\ar[r] & {\ker\left(k^{\delta} \xrightarrow{\sum}  k\right)} \ar[r] & 0
		}
\end{equation*}
with middle term that is free over $k[\Delta/\Delta_r]$;
see Theorem 2 of \cite{Nakajima}.  The splitting of this exact sequence is then equivalent
to the projectivity---hence freeness---of $H^0(I_r,\Omega^1_{I_r/k})^{V_{\ord}}$
over $k[\Delta/\Delta_r]$.  
\end{remark}

In order to formulate the correct analogue of Corollary \ref{FreenessInCharp}
in the case of $\nor{\o{\X}}_r$, we proceed as follows.
Denote by $\tau:\F_p^{\times} \rightarrow \Z_p^{\times}$ the Teichm\"uller character,
and for any $\Z_p$-module $M$ with a linear action of $\F_p^{\times}$ and any 
$j\in \Z/(p-1)\Z$, let
\begin{equation*}
	M(j):=\{m\in M\ :\ d\cdot m = \tau(d)^jm\ \text{for all}\ d\in \F_p^{\times}\}
\end{equation*}
be the subspace of $M$ on which $\F_p^{\times}$ acts via $\tau^j$. 
As $\#\F_p^{\times}=p-1$ is a unit in $\Z_p^{\times}$, the submodule $M(j)$
is a direct summand of $M$.
Explicitly, the idenitity of $\Z_p[\F_p^{\times}]$ admits the  decomposition
\begin{equation}
	1 = \sum_{j\in \Z/(p-1)\Z} f_j\quad\text{with}\quad 
	f_j:=\frac{1}{p-1}\sum_{g\in \F_p^{\times}} \tau^{-j}(g)\cdot g
	\label{GpRngIdem}
\end{equation}
into mutually orthogonal idempotents $f_j$, and we have $M(j)=f_jM$.
In applications, we will consistently need to remove the trivial eigenspace
$M(0)$ from $M$, as this eigenspace in the $p$-adic Galois representations 
we consider is not potentially crystalline at $p$.  We will write
\begin{equation}
	f':=\sum_{\substack{j\in \Z/(p-1)\Z \\ j\neq 0}} f_j\label{TeichmullerIdempotent}
\end{equation}
for the idempotent of $\Z_p[\F_p^{\times}]$ corresponding to projection away
from the 0-eigenspace for $\F_p^{\times}$.  

Applying these considerations
to the identifications of split exact sequences in Corollary \ref{SplitIgusa}, which
are compatible with the canonical diamond operator action of $\Z_p^{\times}\simeq\F_p^{\times}\times \Delta$ 
on both rows, we obtain a corresponding identifiction of split exact sequences
of $\tau^j$-eigenspaces, for each $j\bmod p-1$.
The following is a generalization of \cite[Proposition 8.10 (2)]{tameness}:

\begin{lemma}\label{CharacterSpaces}
	Let $j$ be an integer with $j\not\equiv 0\bmod p-1$.  For each $r$, there are canonical isomorphisms
	\begin{equation}
		\xymatrix{
			{H^0(I_r,\Omega^1_{I_r})(j)}\ar[r]^-{\simeq} & {H^0(I_r,\Omega^1_{I_r}(\SS))(j)}
		}\qquad\text{and}\qquad
		\xymatrix{
			{H^1(I_r,\O(-\SS))(j)}\ar[r]^-{\simeq} & {H^1(I_r,\O)(j)}
		}\label{IgusaEigen}
	\end{equation}
	The normalization map 
	$\nu:\nor{\o{\X}}_r\rightarrow \o{\X}_r$ induces a natural isomorphism of split
	exact sequences
	\begin{equation}
	\begin{gathered}
		\xymatrix{
				0\ar[r] & {{e_r^*}H^0(\o{\X}_r,\Omega^1_{\nor{\o{\X}}_r})(j)} \ar[r]\ar[d]_-{\nu_*}^-{\simeq} & 
				{{e_r^*}H^1_{\dR}(\nor{\o{\X}}_r/\F_p)(j)} \ar[r]\ar[d]^-{\simeq} &
				{{e_r^*}H^1(\nor{\o{\X}}_r,\O_{\nor{\o{\X}}_r})(j)} \ar[r] & 0		\\
				0\ar[r] & {e_r^*H^0(\o{\X}_r,\omega_{\o{\X}_r/\F_p})(j)} \ar[r] & 
				{e_r^*H^1(\o{\X}_r/\F_p)(j)} \ar[r] &
				{e_r^*H^1(\o{\X}_r,\O_{\o{\X}_r})(j)} \ar[r]\ar[u]_-{\nu^*}^-{\simeq} & 0			
		}\label{splitholo}
	\end{gathered}
	\end{equation}
	where the central vertical arrow is deduced from the outer two vertical arrows via the
	splitting of both rows by the Frobenius endomorphism.
	The same assertions hold if we replace $e_r^*$ with $e_r$.	
\end{lemma}

\begin{proof}
	The first map in (\ref{IgusaEigen}) is injective, as it is simply the canonical inclusion.
	To see that it is an isomorphism, 
	we may work over $k:=\o{\F}_p$.  If $\eta$ is {\em any} meromorphic differential on 
	$I_r$ on which $\F_p^{\times}$
	acts via the character $\tau^j$, 
	then since the diamond operators fix every supersingular point on $I_r$ we have
	\begin{equation*}
		\res_x(\eta) = \res_x(\langle u\rangle\eta) = \tau^j(u)\res_x(\eta)
	\end{equation*}
	for any $x\in \SS(k)$ and all $u\in \F_p^{\times}$.  As
	$j\not\equiv 0\bmod p-1$, so $\tau^j$ is nontrivial, we must therefore have $\res_x(\eta)=0$ for all 
	supersingular points $x$.  If in addition $\eta$ is holomorphic
	outside $\SS$ with at worst simple poles along $\SS$, then $\eta$ must be holomorphic everywhere, 
	so the first map in (\ref{IgusaEigen}) is
	surjective, as desired. The second mapping in (\ref{IgusaEigen}) is dual to the first, and 
	hence an isomorphism as well.

	Now for each $j\not\equiv 0\bmod p-1$,
	we have a commutative diagram
	\begin{equation}
	\begin{gathered}
		\xymatrix{
			{e_r^*H^0(\nor{\o{\X}}_r,\Omega^1_{\nor{\o{\X}}_r})(j)} 
			\ar@{^{(}->}[r]^-{\nu_*}\ar[d]_-{(i_r^0)^*}^-{\simeq} &
			{e_r^*H^0(\o{\X}_r,\omega_{\o{\X}_r})(j)} \ar[d]^-{(i_r^0)^*}_-{\simeq} \\
			{H^0(I_r^0,\Omega^1_{I_r^0})(j)^{V_{\ord}}} \ar@{^{(}->}[r]^-{\simeq} & 
			{H^0(I_r^0,\Omega^1_{I_r^0}(\SS))(j)^{V_{\ord}}}
		}
		\label{linkingDiagram}
	\end{gathered}
	\end{equation}
	of $\F_p[\Delta/\Delta_r]$-modules with $\Gamma$ and $\H_r^*$-action in which
	the two vertical arrows are isomorphisms by Proposition \ref{charpord} and the bottom horizontal mapping is 
	an isomorphism as we have just seen.  We conclude that the top horizontal
	arrow of (\ref{linkingDiagram}) is an isomorphism as well.    
	Thus, the left vertical map in (\ref{splitholo}) is an isomorphism, so the same is true
	of the right vertical map by duality.  The diagram (\ref{splitholo}) then follows at once
	from the fact the both rows are canonically split by the Frobenius endomorphism, thanks to 
	Corollary \ref{SplitIgusa}. 	
	A nearly identical argument shows that the same assertions hold if we replace $e_r^*$
	with $e_r$ throughout.
	\end{proof}

If $A$ is any $\Z_p[\F_p^{\times}]$-algebra and $a\in A$, we will 
write $a':=f'a$ for the product of $a$ with the idempotent $f'$
of (\ref{TeichmullerIdempotent}), or equivalently the 
projection of $a$ to the complement 
of the trivial eigenspace for $\F_p^{\times}$.  We will apply this
to $A=\H_r,\,\H_r^*$, viewed as $\Z_p[\F_p^{\times}]$-algebras
in the usual manner, via the diamond operators and the Teichm\"uller
section $\tau:\F_p^{\times}\hookrightarrow \Z_p^{\times}$.

\begin{proposition}\label{NormalizationCoh}
	For each $r$ there are natural isomorphisms of split short exact sequences
	\begin{subequations}
		\begin{equation}
		\begin{gathered}
			\xymatrix@C=15pt{
				0\ar[r] & {{e_r^*}'H^0(\o{\X}_r,\Omega^1_{\nor{\o{\X}}_r})} 
				\ar[r]\ar[d]_-{F_*^r (i_r^0)^*}^-{\simeq} & 
				{{e_r^*}'H^1_{\dR}(\nor{\o{\X}}_r/\F_p)} \ar[r]\ar[d]^-{\simeq} &
				{{e_r^*}'H^1(\nor{\o{\X}}_r,\O_{\nor{\o{\X}}_r})} \ar[r] & 0\\
				0 \ar[r] & {f'H^0(I_r^{0},\Omega^1)^{V_{\ord}}} \ar[r] &
				{f'H^0(I_r^{0},\Omega^1)^{V_{\ord}}\oplus
				f'H^1(I_r^{\infty},\O)^{F_{\ord}}} \ar[r] & 
				{f'H^1(I_r^{\infty},\O)^{F_{\ord}}} \ar[r]\ar[u]_-{\VDual{(i_r^{\infty})^*}}^-{\simeq} & 0
			}\label{LowerIsom2}
		\end{gathered}
		\end{equation}
		\begin{equation}
		\begin{gathered}
			\xymatrix@C=15pt{
				0\ar[r] & {e_r'H^0(\nor{\o{\X}}_r,\Omega^1_{\nor{\o{\X}}_r})} 
				\ar[r]\ar[d]_-{F_*^r (i_r^{\infty})^*}^-{\simeq} & 
				{e_r'H^1_{\dR}(\nor{\o{\X}}_r/\F_p)} \ar[r]\ar[d]^-{\simeq} &
				{e_r'H^1(\nor{\o{\X}}_r,\O_{\nor{\o{\X}}_r})} \ar[r] & 0\\
				0 \ar[r] & {f'H^0(I_r^{\infty},\Omega^1)^{V_{\ord}}} \ar[r] &
				{f'H^0(I_r^{\infty},\Omega^1)^{V_{\ord}}\oplus
				f'H^1(I_r^{0},\O)^{F_{\ord}}} \ar[r] & 
				{f'H^1(I_r^{0},\O)^{F_{\ord}}} \ar[r]
				\ar[u]_-{\VDual{(i_r^{0})^*}\langle p\rangle_N^{-r}}^-{\simeq} & 0
			}\label{UpperIsom2}
		\end{gathered}
		\end{equation}
		\end{subequations}
		Setting $d':=\sum_{k=3}^p d_k$ where $d_k:=\dim_{\F_p} S_k(N;\F_p)^{\ord}$
		as in Remark $\ref{dMFmeaning}$,
		the terms in the top rows of $(\ref{LowerIsom2})$ and $(\ref{UpperIsom2})$
		are free $\F_p[\Delta/\Delta_r]$-modules of ranks $d'$, $2d'$, and $d'$.		
		The identification
		$(\ref{LowerIsom2})$  $($respectively $(\ref{UpperIsom2})$$)$ is
		$\Gamma$ and $\H_r^*$ $($respectively $\H_r$$)$-equivariant, and
		compatible with change in $r$ using the trace mappings attached
		to $\pr: I_r^{\star}\rightarrow I_s^{\star}$ and to 
		$\o{\pr}:\o{\X}_r\rightarrow \o{\X}_{s}$ $($respectively 
		$\o{\ps}:\o{\X}_r\rightarrow \o{\X}_{s}$$)$.
\end{proposition}

\begin{proof}
	This follows immediately from Corollaries \ref{SplitIgusa}--\ref{FreenessInCharp}
	and Lemma \ref{CharacterSpaces}, using the fact that the group ring $\F_p[\Delta/\Delta_r]$
	is local, so any projective $\F_p[\Delta/\Delta_r]$-module 
	is free.
\end{proof}

As usual, we write $\Pic^0_{\nor{\o{\X}}_r/\F_p}[p^{\infty}]$ for the $p$-divisible group
of the Jacobian of 
$\nor{\o{\X}}_r$ over $\F_p$; it is equipped
with canonical actions of 
$\H_r$ and $\H_r^*$, as well as a ``geometric inertia"
action of $\Gamma$ over $\F_p$.
\begin{definition}\label{pDivGpSpecial}
	We define $\Sigma_r:={e_r^*}'\Pic^0_{\nor{\o{\X}}_r/\F_p}[p^{\infty}]$,
	equipped with the induced actions of $\H_r^*$ and $\Gamma$. 
\end{definition}	
We will employ Proposition \ref{NormalizationCoh} and Oda's description (Proposition \ref{OdaDieudonne}) of 
Dieudonn\'e modules in terms of de Rham cohomology to analyze the structure of 
$\Sigma_r$.

\begin{proposition}\label{GisOrdinary}
	For each $r$, 
	there is a natural isomorphism of $A:=\Z_p[F,V]$-modules
	\begin{equation}
		\D(\Sigma_r)_{\F_p} \simeq {e_r^*}'H^1_{\dR}(\nor{\o{\X}}_r/\F_p)\simeq 
		f'H^0(I_r^{\infty},\Omega^1)^{V_{\ord}}\oplus f'H^1(I_r^0,\O)^{F_{\ord}}.\label{DieudonneDesc}
	\end{equation}
	which is compatible with $\H_r^*$, $\Gamma$, and change in $r$ and which
	carries $\D(\Sigma_r^{\mult})_{\F_p}$ $($respectively $\D(\Sigma_r^{\et})_{\F_p}$$)$ isomorphically 
	onto $f'H^0(I_r^{0},\Omega^1)^{V_{\ord}}$ $($respectively
	$f'H^1(I_r^{\infty},\O)^{F_{\ord}}$$)$.  In particular, 
	$\Sigma_r$ is ordinary.
\end{proposition}

\begin{proof}
	First note that since the identifications (\ref{LowerIsom2}) and (\ref{UpperIsom2})
	are induced by the canonical closed immersions $i_r^{\star}:I_r^{\star}\hookrightarrow \nor{\o{\X}}_r$,
	they are compatible with the natural actions of Frobenius and the Cartier operator.
	The isomorphism (\ref{DieudonneDesc}) is therefore an immediate consequence of
	Propositions \ref{OdaDieudonne} and \ref{NormalizationCoh}.  Since this isomorphism is
	compaible with $F$ and $V$, we have 
	\begin{subequations}
	\begin{equation}  
		\D(\Sigma_r^{\mult})_{\F_p}   
		\simeq \D(\Sigma_r)_{\F_p}^{V_{\ord}} 
		\simeq f'H^0(I_r^{0},\Omega^1)^{V_{\ord}}
	\end{equation}
	and
	\begin{equation}
		\D(\Sigma_r^{\et})\otimes_{\Z_p}\F_p 
		\simeq \D(\Sigma_r)_{\F_p}^{F_{\ord}} 
		\simeq f'H^1(I_r^{\infty},\O)^{F_{\ord}}
	\end{equation}
	\end{subequations}
	and we conclude that the canonical inclusion 
	$\D(\Sigma_r^{\mult})_{\Z_p}\oplus\D(\Sigma_r^{\et})_{\Z_p}\hookrightarrow \D(\Sigma_r)_{\Z_p}$
	is surjective, whence $\Sigma_r$ is ordinary by Dieudonn\'e theory.
\end{proof}

We now analyze the ordinary $p$-divisible group $\Sigma_r$ in more detail.
Since $\nor{\o{\X}}_r$
is the disjoint union of proper smooth and irreducible Igusa curves $I_{(a,b,u)}$
(see Proposition \ref{redXr}) with $I_r^0:=I_{(0,r,1)}$ and $I_r^{\infty}=I_{(r,0,1)}$, 
we have
a canonical identification
\begin{equation}
	\Pic^0_{\nor{\o{\X}}_r/\F_p} = \prod_{(a,b,u)} \Pic^0_{I_{(a,b,u)}/\F_p}.
	\label{Pic0Iden}
\end{equation}
For $\star=0,\infty$ let us write $j_r^{\star}:=\Pic^0_{I_r^{\star}/\F_p}$
for the Jacobian of $I_r^{\star}$ over $\F_p$.
The canonical closed immersions 
$i_r^{\star}:I_r^{\star}\hookrightarrow \nor{\o{\X}}_r$ yield (by Picard and Albanese functoriality) 
homomorphisms of abelian varieties over $\F_p$
\begin{equation}
	\xymatrix{
		{\Alb(i_r^{\star}):j_r^{\star}} \ar[r] & {\Pic^0_{\nor{\o{\X}}_r/\F_p}}
		}
		\quad\text{and}\quad
		\xymatrix{
		{\Pic^0(i_r^{\star}):\Pic^0_{\nor{\o{\X}}_r/\F_p}} \ar[r] & {j_r^{\star}}
		}.\label{AlbPicIncl}
\end{equation}
Via the identification (\ref{Pic0Iden}), we know that $j_r^{\star}$ is a direct factor
of $\Pic^0_{\nor{\o{\X}}_r/\F_p}$; in these terms $\Alb(i_r^{\star})$ is the unique mapping
which projects to the identity on $j_r^{\star}$ and to the zero map on all other factors,
while $\Pic^0(i_r^{\star})$ is simply projection onto the factor $j_r^{\star}$.
As $\Sigma_r$ is a direct factor of ${f' \Pic^0_{\nor{\o{\X}}_r/\F_p}[p^{\infty}]}$,
these mappings induce homomorphisms of $p$-divisible groups over $\F_p$
\begin{subequations}
\begin{equation}
	\xymatrix@C=35pt{
		{f'j_r^{0}[p^{\infty}]^{\mult}} \ar[r]^-{\Alb(i_r^{0})} & 
		{f' \Pic^0_{\nor{\o{\X}}_r/\F_p}[p^{\infty}]^{\mult}} \ar[r]^-{\proj} & 
		{\Sigma_r^{\mult}}
		}\label{Alb0}
\end{equation}		
\begin{equation}
	\xymatrix@C=35pt{
		{\Sigma_r^{\et}} \ar[r]^-{\incl} &
		{f' \Pic^0_{\nor{\o{\X}}_r/\F_p}[p^{\infty}]^{\et}} \ar[r]^-{\Pic^0(i_r^{\infty})} & 
		{f'j_r^{\infty}[p^{\infty}]^{\et}}
		}\label{Picinfty}	
\end{equation}
\end{subequations}
which we (somewhat abusively) again denote by $\Alb(i_r^{0})$ and $\Pic^0(i_r^{\infty})$, respectively.
The following is a sharpening of \cite[Chapter 3, \S3, Proposition 3]{MW-Iwasawa} 
(see also \cite[Proposition 3.2]{Tilouine}):

\begin{proposition}\label{MWSharpening}
	The mappings $(\ref{Alb0})$ and $(\ref{Picinfty})$ are isomorphisms.  They induce 
	a canonical split short exact sequences of $p$-divisible groups over $\F_p$
	\begin{equation}
		\xymatrix@C=45pt{
			0 \ar[r] & {f'j^0_r[p^{\infty}]^{\mult}} \ar[r]^-{\Alb(i_r^0)\circ V^r} &
			{\Sigma_r} \ar[r]^-{\Pic^0(i_r^{\infty})} & {f'j^{\infty}_r[p^{\infty}]^{\et}} \ar[r] & 0
		}\label{pDivUpic}
	\end{equation}
	which is:
	\begin{enumerate}
		\item $\Gamma$-equivariant for the geometric inertia action on $\Sigma_r$, the trivial
		action on $f'j_r^{\infty}[p^{\infty}]^{\et}$, and the action via $\langle \chi(\cdot) \rangle^{-1}$
		on $f'j_r^{0}[p^{\infty}]^{\mult}$.  \label{GammaCompatProp}
		
		\item $\H_r^*$-equivariant with $U_p^*$ acting on $f'j_r^{\infty}[p^{\infty}]^{\et}$
		as $F$ and on $f'j_r^{0}[p^{\infty}]^{\mult}$ as $\langle p\rangle_N V$.
		
		\item	Compatible with change in $r$ via the mappings $\Pic^0(\pr)$ on $j_r^{\star}$ and $\Sigma_r$.
		\label{ChangerProp}
\end{enumerate}	
\end{proposition}

\begin{proof}
	It is clearly enough to prove that the sequence (\ref{pDivUpic}) induced by $(\ref{Alb0})$ and $(\ref{Picinfty})$
	is exact.
	Since the contravariant Dieudonn\'e module functor from the category of $p$-divisible groups
	over $\F_p$ to the category of $A$-modules which are $\Z_p$ finite and free is an exact 
	anti-equivalence, it suffices to prove such exactness
	after applying $\D(\cdot)_{\Z_p}$.  
	As the resulting sequence consist of finite free $\Z_p$-modules, 
	exactness may be checked modulo $p$ where it follows immediately from 
	Propositions \ref{NormalizationCoh} and \ref{GisOrdinary}.
	The claimed compatibility with $\Gamma$, $\H_r^*$, and change in $r$ is deduced from 
	Propositions \ref{AtkinInertiaCharp},
	\ref{UlmerProp}, and \ref{pr1desc}, respectively.
\end{proof}

\begin{remark}
	It is possible to give a short proof of Proposition \ref{MWSharpening} 
	along the lines of \cite{MW-Iwasawa} or \cite{Tilouine} by using Proposition \ref{UlmerProp} directly.
	We stress, however, that our approach via Dieudonn\'e modules gives more refined information,
	most notably that the Dieudonn\'e module of $\Sigma_r[p]$ is free as an $\F_p[\Delta/\Delta_r]$-module.
	This fact will be crucial in our later arguments. 
\end{remark}

\section{Dieudonn\'e crystals and \texorpdfstring{$(\varphi,\Gamma)$}{(Phi,Gamma)}-modules}\label{PhiGammaCrystals}

In this section, we summarize the main results of \cite{CaisLau},
which provides a classification of $p$-divisible groups 
over $R_r$ by certain semi-linear algebra structures.
These structures---which arise naturally via the Dieudonn\'e crystal functor---
are cyclotomic analogues of Breuil and Kisin modules, and are closely
related to Wach modules.\footnote{See \cite{CaisLau} for the precise relationship.}

\subsection{\texorpdfstring{$(\varphi,\Gamma)$}{(Phi,Gamma)}-modules attached to \texorpdfstring{$p$}{p}-divisible groups}\label{pDivPhiGamma}

Fix a perfect field $k$ of characteristic $p$.
Write $W:=W(k)$ for the Witt vectors of $k$ and $K$ for its fraction field,
and denote by $\varphi$ the unique automorphism of $W(k)$ lifting the $p$-power map
on $k$.  Fix an algebraic closure $\overline{K}$ of $K$, as well
as a compatible sequence $\{\varepsilon^{(r)}\}_{r\ge 1}$ of primitive $p$-power roots of unity in $\o{K}$,
and set $\scrG_K:=\Gal(\o{K}/K)$.
For $r\ge 0$, we put $K_r:=K(\mu_{p^r})$ and $R_r:=W[\mu_{p^r}]$,
and we set $\Gamma_r:=\Gal(K_{\infty}/K_r)$, and $\Gamma:=\Gamma_0$.

Let $\s_r:=W[\![u_r]\!]$ be the power series ring in one variable $u_r$ over $W$,
viewed as a topological ring via the $(p,u_r)$-adic topology. 
We equip $\s_r$ with the unique continuous action of $\Gamma$ and extension of $\varphi$ 
determined by
\begin{align}
	&\gamma u_r := (1+u_r)^{\chi(\gamma)} -1\quad \text{for $\gamma\in \Gamma$} && \text{and} &&
	\varphi(u_r) := (1+u_r)^p -1.\label{gamphiact}
\end{align}
We denote by $\O_{\E_r}:=\widehat{\s_r[\frac{1}{u_r}]}$ the $p$-adic completion of the localization
${\s_r}_{(p)}$, which is a complete discrete valuation ring with uniformizer $p$ and 
residue field $k(\!(u_r)\!)$.  One checks that the actions of $\varphi$ and $\Gamma$
on $\s_r$ uniquely extend to $\O_{\E_r}$.

For $r>0$, we write $\theta: \s_r\twoheadrightarrow R_r$ for the continuous and $\Gamma$-equivariant $W$-algebra 
surjection
sending $u_r$ to $\varepsilon^{(r)}-1$, whose kernel is the principal ideal generated by
the Eisenstein polynomial $E_r:=\varphi^r(u_r)/\varphi^{r-1}(u_r)$,
and we denote by $\tau:\s_r\twoheadrightarrow W$ the continuous and $\varphi$-equivariant surjection of $W$-algebras 
determined by $\tau(u_r)=0$.
We lift the canonical inclusion $R_r\hookrightarrow R_{r+1}$
to a $\Gamma$- and $\varphi$-equivariant $W$-algebra injection
${\s_r} \hookrightarrow {\s_{r+1}}$
determined by $u_r\mapsto \varphi(u_{r+1})$;
this map uniquely extends to a continuous injection
$\O_{\E_r}\hookrightarrow \O_{\E_{r+1}}$, compatibly with $\varphi$ and $\Gamma$.
We will frequently identify $\s_r$ (respectively $\O_{\E_r}$) with its image in $\s_{r+1}$ 
(respectively $\O_{\E_{r+1}}$),
which coincides with the image of $\varphi$ on $\s_{r+1}$ (respectively $\O_{\E_{r+1}})$.
Under this convention, we have 
$E_{r}(u_r) = E_1(u_1) = u_0/u_1$ for all $r>0$, so we will simply write $\omega:=E_r(u_r)$
for this common element of $\s_r$ for $r>0$.

\begin{definition}	
	We write $\BT_{\s_r}^{\varphi}$ for the category of {\em Barsotti-Tate modules over $\s_r$},
	{\em i.e.} the category whose objects are pairs $(\m,\varphi_{\m})$ where
	\begin{itemize}
		\item $\m$ is a free $\s_r$-module of finite rank.
		\item $\varphi_{\m}:\m\rightarrow \m$ is a $\varphi$-semilinear
		map whose linearization has cokernel killed by $\omega$,
	\end{itemize}
	and whose morphisms are $\varphi$-equivariant $\s_r$-module homomorphisms. 
	We write $\BT_{\s_r}^{\varphi,\Gamma}$ for the subcategory of $\BT_{\s_r}^{\varphi}$
	consisting of objects $(\m,\varphi_{\m})$ which admit a semilinear $\Gamma$-action 
	(in the category $\BT_{\s_r}^{\varphi}$) with the property that $\Gamma_r$ acts trivially on
	$\m/u_r\m$.  Morphisms in $\BT_{\s_r}^{\varphi,\Gamma}$ are $\varphi$ and $\Gamma$-equivariant
	morphisms of $\s_r$-modules.
	We often abuse notation by writing $\m$ for the pair 
	$(\m,\varphi_{\m})$ and $\varphi$ for $\varphi_{\m}$. 
\end{definition}

If $(\m,\varphi_{\m})$ is any object of $\BT_{\s_r}^{\varphi,\Gamma}$, then 
$1\otimes\varphi_{\m}:\varphi^*\m\rightarrow \m$
is injective with cokernel killed by $\omega$, so there is a unique
$\s_r$-linear homomorphism $\psi_{\m}:\m\rightarrow \varphi^*\m$
with the property that the composition of $1\otimes\varphi_{\m}$ and $\psi_{\m}$ (in either order)
is multiplication by $\omega$.  Clearly, $\varphi_{\m}$ and $\psi_{\m}$ determine eachother.

\begin{definition}\label{DualBTDef}
	Let $\m$ be an object of $\BT_{\s_r}^{\varphi,\Gamma}$.  The {\em dual
	of $\m$} is the object $(\m^{t},\varphi_{\m^{t}})$ of $\BT_{\s_r}^{\varphi,\Gamma}$
	whose underlying $\s_r$-module is $\m^{t}:=\Hom_{\s_r}(\m,\s_r)$, equipped with
	the $\varphi$-semilinear endomorphism
	\begin{equation*}
		\xymatrix@C=32pt{
			{\varphi_{\m^{t}}: \m^{t}} \ar[r]^-{1\otimes \id_{\m^{t}}} & {\varphi^*\m^{t} \simeq (\varphi^*\m)^{t}}
			\ar[r]^-{\psi_{\m}^{t}} & {\m^{t}}
		}
	\end{equation*}
	and the commuting action of $\Gamma$ given for $\gamma\in \Gamma$ by
	\begin{equation*}
		(\gamma f)(m) := \chi(\gamma)^{-1}\varphi^{r-1}(\gamma u_r/u_r)\cdot\gamma (f(\gamma^{-1} m )).
	\end{equation*}
\end{definition}

There is a natural notion of base change for Barsotti--Tate modules.
Let $k'/k$ be an algebraic extension (so $k'$ is automatically perfect), and write $W':=W(k')$,
$R_r':=W'[\mu_{p^r}]$, $\s_r':=W'[\![u_r]\!]$, and so on.  
The canonical
inclusion $W\hookrightarrow W'$ extends to a $\varphi$ and $\Gamma$-compatible
$W$-algebra injection $\iota_r:\s_r\hookrightarrow \s_{r+1}'$, and extension
of scalars along $\iota_r$ yields a canonical 
canonical base change functor ${\iota_r}_*: \BT_{\s_r}^{\varphi,\Gamma}\rightarrow \BT_{\s_{r+1}}^{\varphi,\Gamma}$
which one checks is compatible with duality.

Let us write $\pdiv_{R_r}^{\Gamma}$ for the subcategory of $p$-divisible groups over $R_r$
consisting of those objects and morphisms which descend (necessarily uniquely) to $K=K_0$ on generic fibers.
By Tate's Theorem, this is of course equivalent to the full subcategory of $p$-divisible
groups over $K_0$ which have good reduction over $K_r$.   Note that for any algebraic extension $k'/k$, 
base change along the inclusion $\iota_r:R_r\hookrightarrow R_{r+1}'$ gives a covariant functor
${\iota_r}_*:\pdiv_{R_r}^{\Gamma}\rightarrow \pdiv_{R_{r+1}'}^{\Gamma}$.

The main result of \cite{CaisLau} is the following:

\begin{theorem}\label{CaisLauMain}
	For each $r>0$, there is a contravariant functor 
	$\m_r:\pdiv_{R_r}^{\Gamma}\rightarrow \BT_{\s_r}^{\varphi,\Gamma}$ such that:
	\begin{enumerate}
		\item The functor $\m_r$ is an exact equivalence of categories, compatible with duality.
		\label{exequiv}
		\item The functor $\m_r$ is of formation compatible with base change: 
		for any algebraic extension $k'/k$, there is a natural isomorphism
		of composite functors ${\iota_r}_*\circ \m_r \simeq \m_{r+1}\circ {\iota_{r}}_*$ on $\pdiv_{R_r}^{\Gamma}$.
		\label{BaseChangeIsom}
		\item For $G\in \pdiv_{R_r}^{\Gamma}$, put $\o{G}:=G\times_{R_r} k$ and $G_0:=G\times_{R_r} R_r/pR_r$.
		\begin{enumerate}		
			\item There is a functorial and $\Gamma$-equivariant isomorphism of $W$-modules
			\begin{equation*}
				\m_r(G)\tens_{\s_r,\varphi\circ \tau} W \simeq \D(\o{G})_W,
			\end{equation*}		 
			carrying 
			$\varphi_{\m}\otimes \varphi$ to $F:\D(\o{G})_W\rightarrow \D(\o{G})_W$ 
			and $\psi_{\m}\otimes 1$ to $V\otimes 1: \D(\o{G})_W \rightarrow  \varphi^*\D(\o{G})_W$.
			\label{EvaluationONW}
		\item There is a functorial and $\Gamma$-equivariant isomorphism of $R_r$-modules
			\begin{equation*}
				\m_r(G)\tens_{\s_r,\theta\circ\varphi} R_{r} \simeq \D(G_0)_{R_r}.
			\end{equation*}\label{EvaluationONR}
		\end{enumerate}
	\end{enumerate}
\end{theorem}

We wish to explain how to functorially recover the $\scrG_K$-representation afforded 
by the $p$-adic Tate module $T_pG_K$ from $\m_r(G)$.  In order to do so, we must first recall 
the necessary period rings; for a more detailed synopsis of these rings and their properties,
we refer the reader to \cite[\S6--\S8]{Colmez}.

As usual, we put\footnote{Here we use the notation introduced by Berger and Colmez; in Fontaine's
original notation, this ring is denoted $\R$.} 
$$\wt{\e}^+:=\varprojlim_{x\mapsto x^p} \O_{\c_K}/(p),$$
equipped with its canonical $\scrG_K$-action via ``coordinates"
and $p$-power Frobenius map $\varphi$.
This is a perfect ({\em i.e.} $\varphi$ is an automorphism) valuation ring  
of charteristic $p$ with residue field $\overline{k}$
and fraction field $\wt{\e}:=\Frac(\wt{\e}^+)$ that is algebraically closed.
We view $\wt{\e}$ as a topological field via its valuation topology, with respect to which 
it is complete.
Our fixed choice of $p$-power compatible
sequence $\{\varepsilon^{(r)}\}_{r\ge 0}$ 
induces an element $\u{\varepsilon}:=(\varepsilon^{(r)}\bmod p)_{r\ge 0}$ of $\wt{\e}^+$
and we set $\e_{K}:=k(\!(\u{\varepsilon} - 1)\!)$, viewed as a 
topological\footnote{The valuation $v_{\e}$ on $\wt{\e}$ induces the usual discrete valuation 
on $\e_{K,r}$, with the unusual 
normalization $1/p^{r-1}(p-1)$.} subring of $\wt{\e}$; note that this is 
a $\varphi$- and $\scrG_K$-stable subfield of $\wt{\e}$ that is independent
of our choice of $\u{\varepsilon}$.  We write $\e:=\e_K^{\sep}$
for the separable closure of $\e_K$ in the algebraically closed field $\wt{\e}$.
The natural $\scrG_K$-action on $\wt{\e}$ induces a canonical identification
$\Gal(\e/\e_{K}) = \scrH:=\ker(\chi)\subseteq \scrG_K$, so
$\e^{\scrH}=\e_{K}$.  
If $E$ is any subring of $\wt{\e}$, we write $E^+:=E\cap \wt{\e}^+$
for the intersection (taken inside $\wt{\e}$).

We now construct Cohen rings for each of the above subrings of $\wt{\e}$.
To begin with, we put
\begin{equation*}
 \wt{\a}^+:=W(\wt{\e}^+),\qquad\text{and}\qquad \wt{\a}:=W(\wt{\e});
\end{equation*} 
each of these rings is equipped with a canonical Frobenius automorphism $\varphi$
and action of $\scrG_K$ via Witt functoriality.  
Set-theoretically identifying $W(\wt{\e})$ with $\prod_{m=0}^{\infty} \wt{\e}$ in the usual way, we 
endow each factor with its valuation topology and give $\wt{\a}$ the product topology.\footnote{This 
is what is called the {\em weak topology} on $\wt{\a}$.
If each factor of $\wt{\e}$ is instead given the discrete topology, then the
product topology on $\wt{\a}=W(\wt{\e})$ is the familiar $p$-adic 
topology, called the {\em strong} topology.}
The $\scrG_K$ action on $\wt{\a}$ is then continuous
and the canonical $\scrG_K$-equivariant $W$-algebra surjection 
$\theta:\wt{\a}^+\rightarrow \O_{\c_K}$ is continuous when
$\O_{\c_K}$ is given its usual $p$-adic topology.
For each $r\ge 0$, there is a unique continuous $W$-algebra map $j_r:\O_{\E_r}\hookrightarrow \wt{\a}$ 
determined by $j_r(u_r):=\varphi^{-r}([\u{\varepsilon}] - 1)$.  These maps are 
moreover $\varphi$ and $\scrG_K$-equivariant, with $\scrG_K$ acting on $\O_{\E_r}$ through the quotient
$\scrG_K\twoheadrightarrow \Gamma$, and compatible with change in $r$.
We define $\a_{K,r}:=\im(j_r:\O_{\E_r}\rightarrow \wt{\a}),$
which is
naturally a $\varphi$ and $\scrG_K$-stable subring of $\wt{\a}$ that is independent of our choice
of $\u{\varepsilon}$.  
We again omit the subscript when $r=0$. 
Note that $\a_{K,r}=\varphi^{-r}(\a_K)$ inside $\wt{\a}$, and that $\a_{K,r}$
is a discrete valuation ring with uniformizer $p$ and residue field $\varphi^{-r}(\e_K)$
that is purely inseparable over $\e_K$.
We define $\a_{K,\infty}:=\bigcup_{r\ge 0} \a_{K,r}$
and write $\wt{\a}_K$ (respectively $\wh{\a}_K$) 
for the closure of $\a_{K,\infty}$ in $\wt{\a}$ with respect to the weak
(respectively strong) topology.  

Let $\a_{K,r}^{\sh}$ be the strict Henselization of $\a_{K,r}$
with respect to the separable closure of its residue field inside $\wt{\e}$.
Since $\wt{\a}$ is strictly Henselian, there is a unique local morphism 
$\a_{K,r}^{\sh}\rightarrow \wt{\a}$ recovering the given inclusion on residue fields,
and we henceforth view $\a_{K,r}^{\sh}$ as a subring of $\wt{\a}$.
We denote by $\a_r$ the topological closure
of $\a_{K,r}^{\sh}$ inside $\wt{\a}$ with respect to the strong topology,
which
is a $\varphi$ and $\scrG_K$-stable subring of $\wt{\a}$,
and we note that $\a_r = \varphi^{-r}(\a)$ and $\a_r^{\scrH}= \a_{K,r}$
inside $\wt{\a}$. 
We note also that the canonical map $\Z_p\hookrightarrow \wt{\a}^{\varphi=1}$
is an isomorphism, from which it immediately follows that the same is true
if we replace $\wt{\a}$ by any of its subrings constructed above.
If $A$ is any subring of $\wt{\a}$, we define $A^+:=A\cap \wt{\a}^+$,
with the intersection taken inside $\wt{\a}$.

\begin{remark}\label{Slimits}
	We will identify $\s_r$ and $\O_{\E_r}$ with their respective images
	$\a_{K,r}^+$ and $\a_{K,r}$ in $\wt{\a}$ under $j_r$.
	Writing $\s_{\infty}:=\varinjlim \s_r$
	and $\O_{\E_{\infty}}:=\varinjlim \s_r$, we likewise 
	identify $\s_{\infty}$ with $\a_{K,\infty}^+$ and $\O_{\E_{\infty}}$
	with $\a_{K,\infty}$.  
	Denoting by $\wh{\s}_{\infty}$ (respectively $\wt{\s}_{\infty}$) the $p$-adic (respectively 
	$(p,u_0)$-adic) completions, one has
	\begin{equation*}
		\wh{\s}_{\infty} = \wh{\a}_K^+ = W(\e_K^{\rad,+})\quad\text{and}\quad
		\wt{\s}_{\infty} = \wt{\a}_K^+ = W(\wt{\e}_K^{+}),
	\end{equation*}
	for $\e_K^{\rad}:=\cup_{r\ge 0} \varphi^{-r}(\e_K)$ the radiciel ($=$perfect) closure of 
	$\e_K$ in $\wt{\e}$ and $\wt{\e}_K$ its topological completion.
	Via these identifications, $\omega :=u_0/u_1\in \a_{K,1}^+$ is
	a principal generator
	of $\ker(\theta:\wt{\a}^+\twoheadrightarrow \O_{\C_K})$.
\end{remark}

We can now explain the functorial relation between $\m_r(G)$ and $T_pG_K$:

\begin{theorem}\label{comparison}
	Let $G\in \pdiv_{R_r}^{\Gamma}$, and write $H^1_{\et}(G_K):=(T_pG_K)^{\vee}$
	for the $\Z_p$-linear dual of $T_pG_K$.
	There is a canonical mapping
	of finite free $\a_r^+$-modules with semilinear Frobenius and $\scrG_K$-actions
	\begin{equation}
		\xymatrix{
			{\m_r(G)\tens_{\s_r,\varphi} \a_r^+} \ar[r] & {H^1_{\et}(G_K)\otimes_{\Z_p} \a_r^+}
		}
	\end{equation}
	that is injective with cokernel killed by $u_1$.  
	Here, $\varphi$ acts as $\varphi_{\m_r(G)}\otimes \varphi$ on source
	and as $1\otimes\varphi$ on target, while $\scrG_K$ acts diagonally 
	on source and target through the quotient $\scrG_K\twoheadrightarrow \Gamma$
	on $\m_r(G)$.
	  In particular, there is a natural $\varphi$ and $\scrG_K$-equivariant
	isomorphism
	\begin{equation}
				{\m_r(G)\tens_{\s_r,\varphi} \a_r} \simeq  {H^1_{\et}(G_K)\otimes_{\Z_p} \a_r}.
				\label{comparisonb}
	\end{equation}
	These mappings are compatible with duality and with change in $r$ in the obvious manner.
\end{theorem}

\begin{corollary}\label{GaloisComparison}
	For $G\in \pdiv_{R_r}^{\Gamma}$, there are functorial isomorphisms of $\Z_p[\scrG_K]$-modules
	\begin{subequations}
		\begin{align}
			T_pG_K &\simeq \Hom_{\s_r,\varphi}(\m_r(G),\a_r^+)\\
			H^1_{\et}(G_K) &\simeq (\m_r(G) \tens_{\s_r,\varphi} \a_r)^{\varphi_{\m_r(G)}\otimes \varphi=1}.
			\label{FontaineModule}
		\end{align}
	\end{subequations}
	which are compatible with duality and change in $r$.  In the first isomorphism,
	we view $\a_r^+$ as a $\s_r$-algebra via the composite of the usual structure map with 
	$\varphi$. 
\end{corollary}

\begin{remark}
	By definition, the map $\varphi^r$ on $\O_{\E_r}$ is injective with image $\O_{\E}:=\O_{\E_0}$,
	and so induces a $\varphi$-semilinear isomorphism of $W$-algebras 
	$\xymatrix@C=15pt{{\varphi^{r}:\O_{\E_r}} \ar[r]^-{\simeq}&{\O_{\E}} }$.
	It follows from (\ref{FontaineModule}) of Corollary \ref{GaloisComparison} and Fontaine's theory of 
	$(\varphi,\Gamma)$-modules
	over $\O_{\E}$ that $\m_r(G)\otimes_{\s_r,\varphi^r} \O_{\E}$ {\em is} the $(\varphi,\Gamma)$-module
	functorially associated to the $\Z_p[\scrG_K]$-module $H^1_{\et}(G_K)$.
\end{remark}

For the remainder of this section, we recall the construction of the functor $\m_r$, both because we shall need
to reference it in what follows, and because we feel it is enlightening.  
For details, including the proofs of Theorems \ref{CaisLauMain}--\ref{comparison}
and Corollary \ref{GaloisComparison}, we refer the reader to \cite{CaisLau}.

Fix $G\in \pdiv_{R_r}^{\Gamma}$ and set $G_0:=G\times_{R_r}{R_r/pR_r}.$  
The $\s_r$-module $\m_r(G)$ is a functorial descent of the evaluation of
the Dieudonn\'e crystal $\D(G_0)$ on a certain ``universal" PD-thickening of $R_r/pR_r$, 
which we now describe.
Let $S_r$ be the $p$-adic completion of the PD-envelope of $\s_r$ with respect to the 
ideal $\ker\theta$, viewed as a (separated and complete) topological ring via the $p$-adic topology.
We give $S_r$ its PD-filtration: for $q\in \Z$ the ideal $\Fil^q S_r$ is the 
topological closure of the ideal generated by $\{\alpha^{[n]}\,:\,\alpha\in \ker\theta,\,n\ge q\}$.
By construction, the map $\theta:\s_r\twoheadrightarrow R_r$
uniquely extends to a continuous surjection of $\s_r$-algebras $S_r\twoheadrightarrow R_r$ 
(which we again denote by $\theta$) whose kernel $\Fil^1 S_r$ is equipped with topologically 
PD-nilpotent\footnote{Here we use our assumption that $p>2$.} divided powers.
One shows that there is a unique continuous endomorphism $\varphi$ of $S_r$ extending $\varphi$ on $\s_r$,
and that $\varphi(\Fil^1 S_r)\subseteq pS_r$; in particular, we may define
$\varphi_1: \Fil^1 S_r\rightarrow S_r$ by $\varphi_1:=\varphi/p$,
which is a $\varphi$-semilinear homomorphism of $S_r$-modules. Note that
$\varphi_1(E_r)$ is a unit of $S_r$, so the image of $\varphi_1$ generates
$S_r$ as an $S_r$-module.

Since the action of $\Gamma$ on $\s_r$ preserves $\ker\theta$, it follows from the universal mapping property of
divided power envelopes and $p$-adic continuity considerations that this action uniquely extends to
a continuous and $\varphi$-equivariant action of $\Gamma$ on $S_r$ which
is compatible with the PD-structure and the filtration.
Similarly, the transition map $\s_r\hookrightarrow \s_{r+1}$ uniquely extends to a 
continuous $\s_r$-algebra homomorphism $:S_r\rightarrow S_{r+1}$ which is moreover compatible with filtrations
(because $E_r(u_r)=E_{r+1}(u_{r+1})$ under our identifications),
and for nonnegative integers $s < r$ we view $S_r$ as an $S_s$-algebra  
via these maps. 

\begin{definition}
	Let $\BT_{S_r}^{\varphi}$ be the category of triples $(\scrM,\Fil^1\scrM, \varphi_{\scrM,1})$ where
	\begin{itemize}
		\item $\scrM$ is a finite free $S_r$-module and $\Fil^1\scrM\subseteq \scrM$ is an $S_r$-submodule.
		\item $\Fil^1\scrM$ contains $(\Fil^1 S_r)\scrM$ and the quotient $\scrM/\Fil^1\scrM$ is a free 
		$S_r/\Fil^1S_r=R_r$-module.
		\item $\varphi_{\scrM,1}:\Fil^1\scrM_r\rightarrow \scrM$ is a $\varphi$-semilinear map whose image
		generates $\scrM$ as an $S_r$-module.
	\end{itemize}
	Morphisms in $\BT_{S_r}^{\varphi}$ are $S_r$-module homomorphisms which are compatible with the
	extra structures.  As per our convention, we will often write $\scrM$ for a triple 
	$(\scrM,\Fil^1\scrM,\varphi_{\scrM,1})$, and $\varphi_1$ for $\varphi_{\scrM,1}$ when it can cause
	no confusion.  We denote by $\BT_{S_r}^{\varphi,\Gamma}$ the subcategory of $\BT_{S_r}^{\varphi}$
	consisting of objects $\scrM$ that are equipped
	with a semilinear action of $\Gamma$ which preserves $\Fil^1\scrM$, commutes with $\varphi_{\scrM,1}$,
	and whose restriction to $\Gamma_r$ is trivial on $\scrM/u_r\scrM$; morphisms in $\BT_{S_r}^{\varphi,\Gamma}$
	are $\Gamma$-equivariant morphisms in $\BT_{S_r}^{\varphi}$.
\end{definition}

The kernel of the surjection $S_r/p^nS_r\twoheadrightarrow R_r/pR_r$ is the image of the ideal 
$\Fil^1 S_r + pS_r$, which by construction is equipped topologically PD-nilpotent divided powers.
We may therefore define
\begin{equation}
	\scrM_r(G)=\D(G_0)_S:=\varprojlim_n \D(G_0)_{S/p^nS},
\end{equation}
which is a finite free $S_r$-module that depends contravariantly functorially on $G_0$.
We promote $\scrM_r(G)$ to an object of $\BT_{S_r}^{\varphi,\Gamma}$ as follows.
As the quotient map $S_r\twoheadrightarrow R_r$ induces a PD-morphism of PD-theckenings
of $R_r/pR_r$, there is a natural isomorphism of free $R_r$-modules
\begin{equation}
	\scrM_r(G)\otimes_{S_r} R_r \simeq \D(G_0)_{R_r}.\label{surjR}
\end{equation}
By Proposition \ref{BTgroupUnivExt}, there is a canonical ``Hodge" filtration $\omega_G \subseteq \D(G_0)_{R_r}$,
which reflects the fact that $G$ is a $p$-divisible group over $R_r$ lifting $G_0$,
and we define $\Fil^1\scrM_r(G)$ to be the preimage of $\omega_G$ under the 
composite of the isomorphism (\ref{surjR}) with the natural surjection 
$\scrM_r(G)\twoheadrightarrow \scrM_r(G)\otimes_{S_r} R_r$; note that this depends on $G$ and not just on 
$G_0$.  The Dieudonn\'e crystal is compatible with arbitrary base change, so the 
relative Frobenius $F_{G_0}:G_0\rightarrow G_0^{(p)}$ induces an canonical morphism of $S_r$-modules
\begin{equation*}
	\xymatrix{
		{\varphi^*\D(G_0)_{S_r} \simeq \D(G_0^{(p)})_{S_r}} \ar[r]^-{\D(F_{G_0})} & {\D(G_0)_{S_r}}
	},
\end{equation*}
which we may view as a $\varphi$-semilinear map $\varphi_{\scrM_r(G)}:\scrM_r(G)\rightarrow \scrM_r(G)$.
As the relative Frobenius map $\omega_{G_0^{(p)}}\rightarrow \omega_{G_0}$ is zero, 
it follows that the restriction of $\varphi_{\scrM_r(G)}$ to $\Fil^1 \scrM_r(G)$ has image contained in
$p\scrM_r(G)$, so we may define $\varphi_{\scrM_r(G),1}:=\varphi_{\scrM_r(G)}/p$, and one proves as in 
\cite[Lemma A.2]{KisinFCrystal}
that the image of $\varphi_{\scrM_r(G),1}$ generates $\scrM_r(G)$ as an $S_r$-module.

It remains to equip $\scrM_r(G)$ with a canonical semilinear action of $\Gamma$.
Let us write $G_{K_r}$ for the generic fiber of $G$ and $G_{K}$ for its unique
descent to $K=K_0$.  The existence of this descent is reflected by the 
existence of a commutative diagram with cartesian square
\begin{equation}
\begin{gathered}
	\xymatrix{
{G_{K}\fiber_K K_r} \ar@/^/[rrd]^-{1\times \gamma} \ar@/_/[ddr] \ar@{.>}[dr]|-{\gamma} 	&         &					\\
	&{\big(G_{K}\fiber_K K_r\big)_{\gamma}} \ar[r]_-{\pr_1} \ar[d]^-{\pr_2}\ar@{} [dr] |{\square} & 
	{G_{K}\fiber_K K_r} \ar[d]\\
	&{\Spec(K_r)} \ar[r]_-{\gamma} &{\Spec(K_r)}
	}
\end{gathered}
\label{GammaAction}
\end{equation}
for each $\gamma\in \Gamma$, compatibly with change in $\gamma$; here, the subscript of $\gamma$ denotes base change
along the map of schemes induced by $\gamma$.
Since $G$ has generic fiber $G_{K_r}=G_K\times_K K_r$, Tate's Theorem ensures that the
dotted arrow above uniquely extends to an isomorphism
of $p$-divisible groups over $R_r$
\begin{equation}
	\xymatrix{
		{G}\ar[r]^-{\gamma} & {G_{\gamma}} 
	},\label{TateExt}
\end{equation}
compatibly with change in $\gamma$.
By assumption, the action of $\Gamma$ on $S_r$ commutes with the divided powers
on $\Fil^1 S_r$ and induces the given action on the quotient $S_r\twoheadrightarrow R_r$;
in other words, $\Gamma$ acts by automorphisms on the object
$(\Spec(R_r/pR_r)\hookrightarrow \Spec(S_r/p^nS_r))$ of $\Cris((R_r/pR_r)/W)$.
Since $\D(G_0)$ is a crystal, each $\gamma\in \Gamma$ therefore gives an $S_r$-linear map
\begin{equation*}
	\xymatrix{
		{\gamma^*\D(G_0)_{S_r} \simeq \D((G_0)_{\gamma})_{S_r}} \ar[r] & {D(G_0)_{S_r}}
	}
\end{equation*}
and hence an $S_r$-semilinear (over $\gamma$) endomorphism $\gamma$ of $\scrM_r(G)$.
One easily checks that the resulting action of $\Gamma$ on $\scrM_r(G)$
commutes with $\varphi_{\scrM,1}$ and preserves $\Fil^1\scrM_r(G)$.
By the compatibility of $\D(G_0)$ with base change and the obvious fact that
the $W$-algebra surjection $S_r\twoheadrightarrow W$ sending $u_r$ to $0$
is a PD-morphism over the canonical surjection $R_r/pR_r\twoheadrightarrow k$,
there is a natural isomorphism
\begin{equation}
	\scrM_r(G)\otimes_{S_r} W \simeq \D(\o{G})_W. 
\end{equation}
It follows easily from this and the diagram (\ref{GammaAction})
that the action of $\Gamma_r$ on $\scrM_r(G)/u_r\scrM_r(G)$
is trivial.

To define $\m_r(G)$, we functorially descend the $S_r$-module $\scrM_r(G)$
along the structure morphism $\alpha_r:\s_r\rightarrow S_r$.  More precisely, 
for $\m\in \BT_{\s_r}^{\varphi,\Gamma}$, we define 
${\alpha_r}_*(\m):=(M,\Fil^1M,\Phi_1)\in \BT_{S_r}^{\varphi,\Gamma}$ via:

\begin{equation}
	\begin{gathered}
		M:=\m\tens_{\s_r,\alpha_r\circ\varphi} S_r\qquad\text{with diagonal $\Gamma$-action}\\
		\Fil^1 M :=\left\{ m\in M\ :\ \varphi_{\m}\otimes\id (m) \in  \m\otimes_{\s_r} \Fil^1 S_r
		\subseteq \m\otimes_{\s_r} S_r \right\}  \\
		\xymatrix{
			{\Phi_1: \Fil^1 M} \ar[r]^-{\varphi_{\m}\otimes\id} & { \m\tens_{\s_r} \Fil^1 S_r}
			\ar[r]^-{\id\otimes\varphi_1} & {\m\tens_{\s_r,\varphi} S_r = M} 
		}.
\end{gathered}
\label{BreuilSrDef}	
\end{equation}
The following is the key technical point of \cite{CaisLau}, and is proved using 
the theory of windows:
\begin{theorem}\label{Lau}
	For each $r$, the functor ${\alpha_r}_*:\BT_{\s_r}^{\varphi,\Gamma}\rightarrow \BT_{S_r}^{\varphi,\Gamma}$
	is an equivalence of categories, compatible with change in $r$.
\end{theorem}

\begin{definition}
	For $G\in \pdiv_{R_r}^{\Gamma}$, we write $\m_r(G)$
	for the functorial descent of $\scrM_r(G)$ to an object of $\BT_{\s_r}^{\varphi,\Gamma}$
	as guaranteed by Theorem \ref{Lau}.
	By construction, we have a natural isomorphism
of functors ${\alpha_r}_*\circ \m_r\simeq \scrM_r$ on $\pdiv_{R_r}^{\Gamma}$.
\end{definition}

\begin{example}\label{GmQpZpExamples}
Using Messing's description of the Dieudonn\'e crystal of a $p$-divisible group
in terms of the Lie algebra of its universal extension (cf. remark \ref{MessingRem}),
one calculates that for $r\ge 1$
	\begin{subequations}
	\begin{equation}
		\m_r(\Q_p/\Z_p)  = \s_r,\qquad \varphi_{\m_r(\Q_p/\Z_p)}:= \varphi,\qquad \gamma:=\gamma
		\label{MrQpZp}
	\end{equation}
	\begin{equation}
		\m_r(\mu_{p^{\infty}})  = \s_r,\qquad \varphi_{\m_r(\mu_{p^{\infty}})}:= \omega\cdot\varphi,
		\qquad \gamma:=\chi(\gamma)^{-1}\varphi^{r-1}(\gamma u_r/u_r)\cdot \gamma 
	\label{MrMu}
	\end{equation}
\end{subequations}
with $\gamma\in \Gamma$ acting as indicated.  
Note that both $\m_r(\Q_p/\Z_p)$ and $\m_r(\Gm[p^{\infty}])$
arise by base change from their incarnations when $r=1$,
as follows from the fact that $\omega = \varphi(u_1)/u_1$ and
$\varphi^{r-1}(\gamma u_r/u_r)=\gamma u_1/u_1$ via our identifications.
\end{example}

\subsection{The case of ordinary \texorpdfstring{$p$}{p}-divisible groups}

When $G\in \pdiv_{R_r}^{\Gamma}$ is ordinary,
one can say significantly more about the structure of the $\s_r$-module $\m_r(G)$.
To begin with, we observe that for arbitrary $G\in \pdiv_{R_r}^{\Gamma}$,
the formation of the maximal \'etale quotient of $G$ and of the maximal 
connected and multiplicative-type sub $p$-divisible groups of $G$ are functorial in $G$, 
so each of $G^{\et}$, $G^0$, and $G^{\mult}$ is naturally an object of $\pdiv_{R_r}^{\Gamma}$
as well.   We thus (functorially) obtain objects 
$\m_r(G^{\star})$ of $\BT_{\s_r}^{\varphi, \Gamma}$ which admit particularly simple 
descriptions when $\star=\et$ or $\mult$, as we now explain.

As usual, we write $\o{G}^{\star}$ for the special fiber of $G^{\star}$ and $\D(\o{G}^{\star})_W$
for its Dieudonn\'e module.
Twisting the $W$-algebra structure on $\s_r$ by the automorphism $\varphi^{r-1}$
of $W$, 
we define objects of $\BT_{\s_r}^{\varphi,\Gamma}$
\begin{subequations}
	\begin{equation}
		\m_r^{\et}(G) : = \D(\o{G}^{\et})_W\tens_{W,\varphi^{r-1}} \s_r,
		\qquad \varphi_{\m_r^{\et}}:= F\otimes \varphi,
		\qquad \gamma:=\gamma \otimes \gamma 
		\label{MrEtDef}
	\end{equation}
	\begin{equation}
		\m_r^{\mult}(G) : = \D(\o{G}^{\mult})_W\tens_{W,\varphi^{r-1}} \s_r,
		\qquad \varphi_{\m_r^{\mult}}:= V^{-1}\otimes E_r\cdot\varphi,
		\qquad \gamma:=\gamma \otimes \chi(\gamma)^{-1}\varphi^{r-1}(\gamma u_r/u_r)\cdot \gamma 
	\label{MrMultDef}
	\end{equation}
\end{subequations}
with $\gamma\in \Gamma$ acting as indicated.
Note that these formulae make sense and do indeed give objects of $\BT_{\s_r}^{\varphi,\Gamma}$
as $V$ is 
invertible\footnote{A $\varphi^{-1}$-semilinear map of $W$-modules $V:D\rightarrow D$ 
is {\em invertible} if there exists a $\varphi$-semilinear endomorphism $V^{-1}$ whose composition
with $V$ in either order is the identity.  This is easily seen to be equivalent to 
the invertibility of the linear map $V\otimes 1: D\rightarrow \varphi^* D$, with $V^{-1}$
the composite of  $(V\otimes 1)^{-1}$ and the $\varphi$-semilinear map $\id\otimes 1:D\rightarrow \varphi^*D$.
}
 on $\D(\o{G}^{\mult})_W$ and $\gamma u_r/u_r \in \s_r^{\times}$.
It follows easily from these definitions that $\varphi_{\m_r^{\star}}$
linearizes to an isomorphism when $\star=\et$ and has image
contained in $\omega\cdot \m_r^{\mult}(G)$ when $\star=\mult$ 
Of course, $\m_r^{\star}(G)$ 
is contravariantly functorial in---and depends only on---the closed fiber $\o{G}^{\star}$ of $G^{\star}$.

\begin{proposition}\label{EtaleMultDescription}
	Let $G$ be an object of $\pdiv_{R_r}^{\Gamma}$ and let $\m_r^{\et}(G)$ and $\m_r^{\mult}(G)$
	be as above.  The map $F^r:G_0 \rightarrow G_0^{(p^r)}$ 
	$($respectively $V^r:G_0^{(p^r)}\rightarrow G_0$$)$ induces a natural isomorphism
	in $\BT_{\s_r}^{\Gamma}$
	\begin{equation}
			\m_r(G^{\et}) \simeq \m_r^{\et}(G)\qquad\text{respectively}\qquad
			\m_r(G^{\mult}) \simeq \m_r^{\mult}(G).\label{EtMultSpecialIsoms}
	\end{equation}
	These identifications are compatible with change in $r$
	in the sense that for $\star=\et$ $($respectively $\star=\mult$$)$ there is a canonical
	commutative diagram in $\BT_{\s_{r+1}}^{\Gamma}$
	\begin{equation}
	\begin{gathered}
		\xymatrix{
			{\m_{r+1}(G^{\star}\times_{R_r} R_{r+1})} 
			\ar[r]_-{\simeq}^-{(\ref{EtMultSpecialIsoms})}\ar[d]_-{\simeq} & 
			{\m_{r+1}^{\star}(G\times_{R_r} R_{r+1})} \ar@{=}[r] &
			 {\D(\o{G}^{\star})_W\tens_{W,\varphi^r} \s_{r+1}}  
			 \ar[d]^-{F\otimes\id\ (\text{respectively}\ V^{-1}\otimes\id)}_-{\simeq} \\
			{\m_r(G^{\star})\tens_{\s_r} \s_{r+1}} \ar[r]^-{\simeq}_-{(\ref{EtMultSpecialIsoms})} & 
			{\m_r^{\star}(G)\tens_{\s_r} \s_{r+1}} \ar@{=}[r] & 
			 {\D(\o{G}^{\star})_W\tens_{W,\varphi^{r-1}} \s_{r+1}}
		}
	\end{gathered}
	\label{EtMultSpecialIsomsBC}
	\end{equation}
	where the left vertical isomorphism is deduced from Theorem $\ref{CaisLauMain}$ $(\ref{BaseChangeIsom}).$
\end{proposition}

\begin{proof}
	For ease of notation, we will write $\m_r^{\star}$ and
	and $\D^{\star}$ for  $\m_r^{\star}(G)$ and $\D(\o{G}^{\star})_W$, respectively. 
	Using (\ref{BreuilSrDef}), one finds that $\scrM_r^{\et}:={\alpha_r}_*(\m_r^{\et})\in \BT_{S_r}^{\varphi,\Gamma}$
	is given by the triple
	\begin{subequations}
	\begin{equation}
		\scrM_r^{\et}:=(\D^{\et}\otimes_{W,\varphi^r} S_r,\ \D^{\et}\otimes_{W,\varphi^r} \Fil^1 S_r,\  
		F\otimes \varphi_1)
	\end{equation}	
	with $\Gamma$ acting diagonally on the tensor product.  Similarly,
	${\alpha_r}_*(\m_r^{\mult})$ is given by the triple
	\begin{equation}
		(\D^{\mult}\otimes_{W,\varphi^r} S_r,\ \D^{\mult}\otimes_{W,\varphi^r} S_r,\ 
		V^{-1} \otimes v_r\cdot\varphi)\label{WindowMultCase}
	\end{equation}
	\end{subequations}
	where $v_r=\varphi(E_r)/p$ and $\gamma\in \Gamma$ acts on $\D^{\mult}\otimes_{W,\varphi^r} S_r$
	as $\gamma \otimes \chi(\gamma)^{-1} \varphi^r(\gamma u_r/u_r)\cdot \gamma$.	
	Put $\lambda := \log(1+u_0)/{u_0},$
	where $\log(1+X):\Fil^1 S_r\rightarrow S_r$ is the usual (convergent for the $p$-adic topology) power series
	and $u_0$ is viewed as an element of $S_r$ via the structure map $S_0\rightarrow S_r$
	(concretely, $u_0= \varphi^r(u_r)$).
	For each $r\ge 0$, one checks that $\lambda$ admits the convergent product expansion 
	$\lambda=\prod_{i\ge 0} \varphi^i(v_r)$, so $\lambda\in S_r^{\times}$ and
	\begin{equation}
		\frac{\lambda}{\varphi(\lambda)}  = \varphi(E_r)/p= v_r\qquad\text{and}\qquad
		\frac{\lambda}{\gamma\lambda} = \chi(\gamma)^{-1}\varphi^r(\gamma u_r/u_r) \quad\text{for}\ 
		\gamma\in 		\Gamma.\label{lambdaTransformation}
	\end{equation}
	It follows from (\ref{lambdaTransformation})
	that the $S_r$-module automorphism of $\D^{\mult}\otimes_{W,\varphi^r} S_r$
	given by multiplication by $\lambda$
	carries (\ref{WindowMultCase}) isomorphically onto the object of $\BT_{S_r}^{\varphi,\Gamma}$
	given by the triple
	\begin{equation}
		\scrM_r^{\mult}:=(\D^{\mult}\otimes_{W,\varphi^r} S_r,\ \D^{\mult}\otimes_{W,\varphi^r} S_r,\  
		V^{-1}\otimes\varphi)
	\end{equation}
	with $\Gamma$ acting {\em diagonally} on the tensor product.

	On the other hand, since $G_0^{\et}$ (respectively $G_0^{\mult}$) is \'etale 
	(respectively of multiplicative type) over $R_r/pR_r$, the relative Frobenius 
	(respectively Verscheibung) morphism of $G_0$ induces isomorphisms
	\begin{subequations}
	\begin{equation}
		\xymatrix{
			{G_0^{\et}} \ar[r]_-{\simeq}^-{F^r} &  {(G_0^{\et})^{(p^r)} \simeq 
			{\varphi^r}^*\o{G}^{\et} \times_k R_r/pR_r} 
		}\label{Ftrick}
	\end{equation}
	respectively
	\begin{equation}
		\xymatrix{
			{G_0^{\mult}} & \ar[l]^-{\simeq}_-{V^r}   {(G_0^{\mult})^{(p^r)} 
			\simeq {\varphi^r}^*\o{G}^{\mult} \times_k R_r/pR_r} 
		}\label{Vtrick}
	\end{equation} 
	\end{subequations}
	of $p$-divisible groups over $R_r/pR_r$, where we have used the fact that the map $x\mapsto x^{p^r}$
	of $R_r/pR_r$ factors as $R_r/pR_r \twoheadrightarrow k \hookrightarrow R_r/pR_r$
	in the final isomorphisms of both lines above.  Since the Dieudonn\'e crystal is compatible
	with base change and the canonical map $W\rightarrow S_r$ extends to a PD-morphism 
	$(W,p)\rightarrow (S_r, pS_r+\Fil^1 S_r)$ over $k\rightarrow R_r/pR_r$, 
	applying $\D(\cdot)_{S_r}$ to (\ref{Ftrick})--(\ref{Vtrick}) yields natural isomorphisms
	$\D(G_0^{\star})_{S_r} \simeq \D^{\star}\otimes_{W,\varphi^r} S_r$ for $\star=\et,\mult$
	which carry $F$ to $F\otimes \varphi$.  It is a straightforward exercise using the construction
	of $\scrM_r(G^{\star})$ given in \S\ref{pDivPhiGamma} to check
	that these isomorphisms extend to give isomorphisms $\scrM_r(G^{\et}) \simeq \scrM_r^{\et}$
		and $\scrM_r(G^{\mult}) \simeq \scrM_r^{\mult}$ in $\BT_{S_r}^{\varphi,\Gamma}$.
	By Theorem \ref{Lau},
	we conclude that we have natural isomorphisms in $\BT_{\s_r}^{\varphi,\Gamma}$
	as in (\ref{EtMultSpecialIsoms}).  The commutativity of (\ref{EtMultSpecialIsomsBC}) 
	is straightforward, using the definitions of the base change isomorphisms.
\end{proof}

Now suppose that $G$ is ordinary.
As $\m_r$ is exact by Theorem \ref{CaisLauMain} (\ref{exequiv}),
applying $\m_r$ to the connected-\'etale sequence of $G$ gives 
a short exact sequence in $\BT_{\s_r}^{\varphi,\Gamma}$
\begin{equation}
	\xymatrix{
		0\ar[r] & {\m_r(G^{\et})} \ar[r] & {\m_r(G)} \ar[r] & {\m_r(G^{\mult})} \ar[r] & 0
	}\label{ConEtOrdinary}
\end{equation}
which is contravariantly functorial and exact in $G$.
Since $\varphi_{\m_r}$ linearizes to an isomorphism on $\m_r(G^{\et})$
and is topologically nilpotent on $\m_r(G^{\mult})$, we think of
(\ref{ConEtOrdinary}) as the ``slope flitration" for Frobenius acting on $\m_r(G)$.
On the other hand, Proposition \ref{BTgroupUnivExt} and Theorem \ref{CaisLauMain} (\ref{EvaluationONR})
provide 
a canonical ``Hodge filtration" of $\m_r(G)\tens_{\s_r,\varphi} R_r\simeq \D(G_0)_{R_r}$:
\begin{equation}
	\xymatrix{
		0\ar[r] & {\omega_{G}} \ar[r] & {\D(G_0)_{R_r}} \ar[r] & {\Lie(G^t)} \ar[r] & 0
	}\label{HodgeFilOrd}
\end{equation}
which is contravariant and exact in $G$.  
Our assumption that $G$ is ordinary yields 
({\em cf.} \cite{KatzSerreTate}):

\begin{lemma}\label{HodgeFilOrdProps}
	With notation as above, there are natural and $\Gamma$-equivariant 
	isomorphisms 
\begin{equation}
	 \Lie(G^t)\simeq \D(G_0^{\et})_{R_r}  \qquad\text{and} \qquad \D(G_0^{\mult})_{R_r}\simeq \omega_G.
	 \label{FlankingIdens}
\end{equation}
	Composing these isomorphisms with the canonical maps obtained by applying $\D(\cdot)_{R_r}$ 
	to the connected-\'etale sequence of $G_0$
	yield functorial $R_r$-linear splittings of the Hodge filtration $(\ref{HodgeFilOrd})$.
	Furthermore, there is a canonical and $\Gamma$-equivariant isomorphism of split exact
	sequences of $R_r$-modules	
	\begin{equation}
	\begin{gathered}
			\xymatrix{
			0\ar[r] & {\omega_{G}} \ar[r]\ar[d]^-{\simeq} & {\D(G_0)_{R_r}} \ar[r]\ar[d]^-{\simeq} & 
			{\Lie(G^t)} \ar[r]\ar[d]^-{\simeq} & 0\\
				0 \ar[r] & {\D(\o{G}^{\mult})_W\tens_{W,\varphi^r} R_r} \ar[r]_-{i} & 
			{\D(\o{G})_W\tens_{W,\varphi^r} R_r} \ar[r]_-{j} & 
			{\D(\o{G}^{\et})_W\tens_{W,\varphi^r} R_r}\ar[r] & 0
				}
	\end{gathered}
	\label{DescentToWIsom}
	\end{equation}
	with $i,j$ the inclusion and projection mappings
	obtained from the canonical direct sum decomposition 
	$\D(\o{G})_W\simeq \D(\o{G}^{\mult})_W\oplus \D(\o{G}^{\et})_W$.
\end{lemma}

\begin{proof}
Applying $\D(\cdot)_{R_r}$ to the connected-\'etale sequence
of $G_0$ and using Proposition \ref{BTgroupUnivExt} yields a commutative diagram
with exact columns and rows
\begin{equation}
\begin{gathered}
	\xymatrix{
						&     & 0\ar[d] & 0 \ar[d] & \\
		 & 0 \ar[r]\ar[d] & {\omega_{G}}\ar[r]\ar[d] & 
		{\omega_{G^{\mult}}} \ar[r]\ar[d] & 0\\
		0\ar[r] & {\D(G_0^{\et})_{R_r}} \ar[r]\ar[d] & {\D(G_0)_{R_r}}\ar[r]\ar[d] & 
		{\D(G_0^{\mult})_{R_r}} \ar[r]\ar[d] & 0\\
		0 \ar[r] & {\Lie({G^{\et}}^t)} \ar[r]\ar[d] & {\Lie(G^t)}\ar[r]\ar[d] & 
		0 & \\
		& 0 & 0 &  &
		}
\end{gathered}
\label{OrdinaryDiagram}
\end{equation}
where we have used the fact that that the invariant differentials
and Lie algebra of an \'etale $p$-divisible group
(such as $G^{\et}$ and ${G^{\mult}}^t\simeq {G^t}^{\et}$)
are both zero.  The isomorphisms (\ref{FlankingIdens})
follow at once.  We likewise immediately see that the short exact sequence 
in the center column of (\ref{OrdinaryDiagram}) is functorially and $R_r$-linearly
split.  Thus, to prove the claimed identification in (\ref{DescentToWIsom}),
it suffices to exhibit natural isomorphisms of free $R_r$-modules with $\Gamma$-action
\begin{equation}
	\D(G_0^{\et})_{R_r} \simeq \D(\o{G}^{\et})_W\tens_{W,\varphi^r} R_r
	\qquad\text{and}\qquad
	\D(G_0^{\mult})_{R_r} \simeq \D(\o{G}^{\mult})_W\tens_{W,\varphi^r} R_r,
	\label{TwistyDieuIsoms}
\end{equation} 
both of which follow easily by applying $\D(\cdot)_{R_r}$ to
(\ref{Ftrick}) and (\ref{Vtrick}) and using the compatibility 
of the Dieudonn\'e crystal with base change as in the proof of Proposition (\ref{EtaleMultDescription}).
\end{proof}

From the slope filtration (\ref{ConEtOrdinary}) of $\m_r(G)$
we can recover both the (split) slope filtration of $\D(\o{G})_W$
and the (split) Hodge filtration (\ref{HodgeFilOrd}) of $\D(G_0)_{R_r}$:

\begin{proposition}\label{MrToHodge}
	There are canonical and $\Gamma$-equivariant isomorphisms of short exact sequences
	\begin{subequations}
	\begin{equation}
	\begin{gathered}
		\xymatrix{
			0\ar[r] & {\m_r(G^{\et})\tens_{\s_r,\varphi\circ\tau} W} \ar[r]\ar[d]^-{\simeq} & 
			{\m_r(G)\tens_{\s_r,\varphi\circ\tau} W} \ar[r]\ar[d]^-{\simeq} & 
			{\m_r(G^{\mult})\tens_{\s_r,\varphi\circ\tau} W} \ar[r]\ar[d]^-{\simeq} & 0 \\
			0 \ar[r] & {\D(\o{G}^{\et})_W} \ar[r] & {\D(\o{G})_W} \ar[r] & {\D(\o{G}^{\mult})_W}
			\ar[r] & 0
		}\label{MrToDieudonneMap}
	\end{gathered}
	\end{equation}
	\begin{equation}
	\begin{gathered}
		\xymatrix{
			0\ar[r] & {\m_r(G^{\et})\tens_{\s_r,\theta\circ\varphi} R_r} \ar[r]\ar[d]^-{\simeq} & 
			{\m_r(G)\tens_{\s_r,\theta\circ\varphi} R_r} \ar[r]\ar[d]^-{\simeq} & 
			{\m_r(G^{\mult})\tens_{\s_r,\theta\circ\varphi} R_r} \ar[r]\ar[d]^-{\simeq} & 0 \\
			0 \ar[r] &  {\Lie(G^t)} \ar[r]_-{i} & {\D(G_0)_{R_r}} \ar[r]_-{j} & {\omega_{G}} \ar[r] & 0\\
		}
	\end{gathered}\label{MrToHodgeMap}
	\end{equation}
	\end{subequations}
	Here, $i:\Lie(G^t)\hookrightarrow \D(G_0)_{R_r}$
	and $j:\D(G_0)_{R_r}\twoheadrightarrow \omega_{G}$
	are the canonical splittings of Lemma $\ref{HodgeFilOrdProps}$,
	the top row of $(\ref{MrToHodgeMap})$ is obtained from $(\ref{ConEtOrdinary})$ by extension of scalars,
	and the isomorphism $(\ref{MrToDieudonneMap})$ intertwines $\varphi_{\m_r(\cdot)}\otimes \varphi$ with $F\otimes \varphi$
	and $\psi\otimes 1$ with $V\otimes 1$.
\end{proposition}

\begin{proof}
	This follows immediately from Theorem \ref{CaisLauMain} (\ref{EvaluationONW}) and Lemma \ref{HodgeFilOrdProps}.
\end{proof}

\section{Results and Main Theorems}\label{results}

In this section, we will state and prove our main results as described in \S\ref{resultsintro}.
Throughout, we will keep the notation of \S\ref{resultsintro} and of 
\S\ref{pDivPhiGamma} with $k:=\F_p$.

\subsection{The formalism of towers}\label{TowerFormalism}

In this preliminary section, we set up a general commutative algebra framework 
for dealing with the various projective limits of cohomology modules
that we will encounter.

\begin{definition}
	A {\em tower of rings} is an inductive system $\scrA:=\{A_r\}_{r\ge 1}$ of local rings with local
	transition maps.  A {\em morphism of towers} $\scrA\rightarrow \scrA'$
	is a collection of local ring homomorphisms $A_r\rightarrow A_r'$ which are compatible
	with change in $r$.
	A {\em tower of $\scrA$-modules} $\scrM$
	consists of the following data:
	\begin{enumerate}
		\item For each integer $r\ge 1$, an $A_r$-module $M_r$.
	  	\item A collection of $A_r$-module homomorphisms
			$\varphi_{r,s}:M_r\rightarrow M_{s}\otimes_{A_{s}} A_r $
		for each pair of integers $r\ge s\ge 1$, which are compatible 
		in the obvious way under composition.		
	\end{enumerate}
	A {\em morphism of towers of $\scrA$-modules} $\scrM\rightarrow \scrM'$
	is a collection of $A_r$-module homomorphisms $M_r\rightarrow M_r'$ which
	are compatible with change in $r$ in the evident manner.  For a tower of 
	rings $\scrA=\{A_r\}$, we will write $A_{\infty}$ for the inductive limit,
	and for
	a tower of $\scrA$-modules $\scrM=\{M_r\}$, we set
	\begin{equation*}
		M_B := \varprojlim_r \left( M_r\otimes_{A_r} B\right)\quad\text{and write simply}\quad
		M_{\infty}:=M_{A_{\infty}},
	\end{equation*}
	for any $A_{\infty}$-algebra $B$, with the projective limit taken with respect to the induced transition maps.
\end{definition}

\begin{lemma}\label{Technical}
	Let $\scrA=\{A_r\}_{r\ge 0}$ be a tower of rings
	and suppose that $I_r\subseteq A_r$ is a sequence of proper principal ideals
	such that $A_r$ is $I_r$-separated and
	the image of $I_{r}$ in $A_{r+1}$ is contained in $I_{r+1}$ for all $r$.	
	Write $I_{\infty}:=\varinjlim I_r$
	for the inductive limit, and set $\o{A}_r:=A_r/I_r$ for all $r$.
	Let $\scrM=\{M_r,\pr_{r,s}\}$ be a tower of $\scrA$-modules
	equipped with an action\footnote{That is, a homomorphism of 
	groups $\Delta\rightarrow \Aut_{\scrA}(\scrM)$,
	or equivalently, an $A_r$-linear action of $\Delta$ on $M_r$
	for each $r$ that is compatible with change in $r$.
	}
	of $\Delta$ by $\scrA$-automorphisms.  Suppose that 
	$M_r$ is free of finite rank over $A_r$ for all $r$, and that 
	$\Delta_r$ acts trivially on $M_r$.
	Let $B$ be an $A_{\infty}$-algebra, and observe that $M_B$
	is canonically a module over the completed group ring $\Lambda_B$.
	Assume that $B$ is either flat over $A_{\infty}$ or 
	that $B$ is a flat $\o{A}_{\infty}$-algebra,
	and that the following two conditions hold for all $r>0$
	\begin{enumerate}
	\setcounter{equation}{1}
		\renewcommand{\theenumi}{\theequation{\rm\alph{enumi}}}
		{\setlength\itemindent{10pt} 
			\item $\o{M}_r:=M_r/I_rM_r$ is a free $\o{A}_r[\Delta/\Delta_r]$-module of rank 
			$d$ that is independent of $r$.\label{freehyp}}
		{\setlength\itemindent{10pt} 
		\item For all $s\le r$ the induced maps 
		$\xymatrix@1{
				{\overline{\pr}_{r,s}: \o{M}_r}\ar[r] & 
				{\o{M}_{s}\otimes_{\o{A}_{s}} \o{A}_{r}}
				}$
		are surjective.\label{surjhyp}} 
	\end{enumerate}
	Then:
	\begin{enumerate}
	 	\item $M_r$ is a free $A_r[\Delta/\Delta_r]$-module of rank $d$ for all $r$.\label{red2pfree}
	 
		\item The induced maps of $A_r[\Delta/\Delta_{s}]$-modules
		\begin{equation*}
			\xymatrix{
				{M_r \otimes_{A_r[\Delta/\Delta_r]} A_r[\Delta/\Delta_{s}]} \ar[r] & {M_{s}\otimes_{A_{s}} A_r}
			}
		\end{equation*}
		are isomorphisms for all $r\ge s$.\label{red2psurj}
		
		\item $M_B$ is a finite free $\Lambda_{B}$-module of rank $d$.	\label{MBfree}
		
		\item For each $r$, the canonical map 
		\begin{equation*}
			\xymatrix{
				{M_B\otimes_{\Lambda_B} B[\Delta/\Delta_r]} \ar[r] & {M_r\otimes_{A_r} B}
				}
		\end{equation*}
		is an isomorphism of $B[\Delta/\Delta_r]$-modules.
		\label{ControlLimit}

		\item If $B'$ is any 
		$B$-algebra which is flat over 
		$A_{\infty}$ or $\o{A}_{\infty}$, then the canonical map
		\begin{equation*}
			\xymatrix{
				{M_B\otimes_{\Lambda_B} \Lambda_{B'}} \ar[r] & {M_{B'}}
				}
		\end{equation*}
		is an isomorphism of finite free $\Lambda_{B'}$-modules.\label{CompletedBaseChange}
		
	\end{enumerate}
\end{lemma}

\begin{proof}
	For notational ease, let us put $\Lambda_{A_r,s}:=A_r[\Delta/\Delta_{s}]$ for all pairs of nonnegative integers
	$r,s$.  Note that $\Lambda_{A_r,s}$ is a local $A_r$-algebra, so the principal
	ideal $\widetilde{I}_r:=I_r\Lambda_{A_r,s}$ is 
	contained in the radical of $\Lambda_{A_r,s}$. 
	
	Let us fix $r$ and choose a principal generator $f_r\in A_r$ of
	$I_r$ (hence also of $\widetilde{I}_r$).  The module $M_r$ is obviously finite
	over $\Lambda_{A_r,r}$ (as it is even finite over $A_r$), so by hypothesis (\ref{freehyp})
	we may choose $m_1,\ldots,m_{d}\in M_r$
	with the property that the images of the $m_i$ in $\o{M}_r=M_r/\widetilde{I}_rM_r$ 
	freely generate $\o{M}_r$ as an 
	$\o{A}_r[\Delta/\Delta_r]=\Lambda_{A_r,r}/\widetilde{I}_r$-module.   
	By Nakayama's Lemma \cite[Corollary to Theorem 2.2]{matsumura}, we conclude
	that $m_1,\ldots,m_{d}$ generate $M_r$ as a $\Lambda_{A_r,r}$-module.  
	If 
	\begin{equation}
		\sum_{i=1}^{d} x_i m_i =0\label{genreln}
	\end{equation}
	is any relation on the $m_i$ with $x_i\in \Lambda_{A_r,r}$, then necessarily 
	$x_i\in \widetilde{I}_r\Lambda_{A_r,r}$,	
	and we claim that $x_i\in \widetilde{I}_r^j$ for all $j\ge 0$.  To see this, we proceed by induction and suppose
	that our claim holds for $j\le N$.  Since $\widetilde{I}_r$ is principal, 
	for each $i$ there exists $x_i'\in \Lambda_{A_r,r}$
	with $x_i = f_r^N x_i'$, and the relation (\ref{genreln}) reads $f_r^Nm=0$ with $m\in M_r$ given by
	$m:=\sum_{i=1}^{d} x_i'm_i.$
	Since $M_r$ is free as an $A_r$-module, it is in particular torsion free, so we conclude that
	$m=0$.  Since the images of the $m_i$ {\em freely} generate $M_r/\wt{I}_rM_r$,
	it follows that $x_i'\in \widetilde{I}_r$ and hence that $x_i\in \widetilde{I}_r^{N+1}$,
	which completes the induction.
	By our assumption that $A_r$ is $I_r$-adically separated, we must have $x_i=0$ for all $i$
	and the relation (\ref{genreln})
	is trivial.  We conclude that $m_1,\ldots,m_{d}$ freely generate $M_r$ over $\Lambda_{A_r,r}$, giving 
	(\ref{red2pfree}).
	
	To prove (\ref{red2psurj}), note that our assumption (\ref{surjhyp}) that the maps $\overline{\pr}_{r,s}$ are 
	surjective for all $r\ge s$ implies that the same is true of the maps $\pr_{r,s}$ 
	(again by Nakayama's Lemma) and hence that the induced map of 
	$\Lambda_{A_r,s}$-modules in (\ref{red2psurj}) is surjective.
	As this map is then a surjective map of free $\Lambda_{A_r,s}$-modules of the same
	rank $d$, it must be an isomorphism.  
	
	Since the kernel of the canonical surjection $\Lambda_{A_r,r}\twoheadrightarrow \Lambda_{A_r,s}$
	lies in the radical of $\Lambda_{A_r,r}$, we deduce by Nakayama's Lemma that 
	any lift to $M_r$ of a $\Lambda_{A_r,s}$-basis of $M_{s}\otimes_{A_{s}} A_r$
	is a $\Lambda_{A_r,r}$-basis of $M_r$.
	It follows easily from this that the projective limit $M_B$ is a free $\Lambda_{B}$-module of rank 
	$d$ for any flat $A_{\infty}$-algebra $B$.  The corresponding assertions for any flat 
	$\o{A}_{\infty}$-algebra $B$ follow similarly, using the hypotheses (\ref{freehyp})
	and (\ref{surjhyp}) directly, and this gives (\ref{MBfree}).
	
	Observe that the mapping of (\ref{ControlLimit}) is obtained from the canonical
	surjection $M_B\twoheadrightarrow M_r\otimes_{A_r} B$ by extension of scalars,
	keeping in mind the natural identification 
	$M_r\otimes_{A_r} B \otimes_{\Lambda_B} B[\Delta/\Delta_r] \simeq M_r \otimes_{A_r} B.$
	It follows at once that this mapping is surjective.  By (\ref{red2pfree}) and (\ref{MBfree}), we conclude that
	the mapping in (\ref{ControlLimit}) is a surjection of free $B[\Delta/\Delta_r]$-modules
	of the same rank and is hence an isomorphism as claimed.

	It remains to
	prove (\ref{CompletedBaseChange}). 
	Extending scalars,
	the canonical maps
	$M_B\twoheadrightarrow M_r\otimes_{A_r} B$ induce surjections
	\begin{equation*}
		\xymatrix{
			{M_{B}\otimes_{\Lambda_B} \Lambda_{B'}} \ar@{->>}[r] & 
			{(M_r\otimes_{A_r} B)\otimes_{\Lambda_B} \Lambda_{B'} \simeq M_r\otimes_{A_r} B'}
			}
	\end{equation*}
	that are compatible in the evident manner with change in $r$.  Passing to inverse
	limits gives the mapping $M_{B}\otimes_{\Lambda_{B}}\Lambda_{B'}\rightarrow M_{B'}$ 
	of (\ref{CompletedBaseChange}).  Due to (\ref{MBfree}),
	this is then a map of finite free $\Lambda_{B'}$-modules of the same rank, 
	so to check that it is an isomorphism it suffices by Nakayama's Lemma
	to do so after applying $\otimes_{\Lambda_{B'}} B'[\Delta/\Delta_r]$, which is
	an immediate consequence of (\ref{ControlLimit}). 
\end{proof}

We record the following elementary commutative algebra fact,
which will be extremely useful to us:

\begin{lemma}\label{fflatfreedescent}
	Let $A\rightarrow B$ be a local homomorphism of local rings which makes $B$ into
	a flat $A$-algebra, and let $M$ be an arbitrary $A$-module.
	Then $M$ is a free $A$-module of finite rank if and only if $M\otimes_A B$
	is a free $B$-module of finite rank.
\end{lemma}

\begin{proof}
	First observe that since $A\rightarrow B$ is local and flat, it is faithfully flat.
	We write $M=\varinjlim M_{\alpha}$ as the direct limit of its finite $A$-submodules,
	whence $M\otimes_A B = \varinjlim (M_{\alpha}\otimes_A B)$ with each of $M_{\alpha}\otimes_A B$
	naturally a finitely generated $B$-submodule of $M\otimes_A B$.  Assume that $M\otimes_A B$
	is finitely generated as a $B$-module. Then there exists $\alpha$ with 
	$M_{\alpha}\otimes_A B\rightarrow M\otimes_A B$ surjective, and as $B$ is faithfully flat over $A$,
	this implies that $M_{\alpha}\rightarrow M$ is surjective, whence $M$ is finitely generated over $A$.
	Suppose in addition that $M\otimes_A B$ is free as a $B$-module.  In particular, $M\otimes_A B$
	is $B$-flat, which implies by faithful flatness of $B$ over $A$ that $M$ is $A$-flat
	(see, e.g. \cite[Exercise 7.1]{matsumura}).  Then $M$ is a finite flat module over the local ring $A$,
	whence it is free as an $A$-module by \cite[Theorem 7.10]{matsumura}.
\end{proof}

Finally, we analyze duality for towers with $\Delta$-action.  

\begin{lemma}\label{LambdaDuality}
	With the notation of Lemma $\ref{Technical}$, 
	let
	$\scrM:=\{M_r,\pr_{r,s}\}$ and $\scrM':=\{M_r',\pr_{r,s}'\}$ be two 
	towers of $\scrA$-modules with $\Delta$-action satisfying $(\ref{freehyp})$ and $(\ref{surjhyp})$.
	Suppose that for each $r$ there exist $A_r$-linear perfect duality pairings
	\begin{equation}
		\xymatrix{
			{\langle \cdot,\cdot \rangle_{r}:M_r\times M_r'} \ar[r] & A_r
		}\label{pairinghyp}
	\end{equation}
	with respect to which $\delta$ is self-adjoint for all $\delta\in \Delta$,
	and which satisfy the compatibility condition\footnote{By abuse of notation,
	for any map of rings $A\rightarrow B$ and any $A$-bilinear pairing of $A$-modules 
	$\langle\cdot,\cdot\rangle:M\times M'\rightarrow A$, we again write 
	$\langle\cdot,\cdot\rangle: M_B\times M_B'\rightarrow B$ for the $B$-bilinear pairing induced
	by extension of scalars.}
	\begin{equation}
		\langle\pr_{r,s}m, \pr_{r,s}'m'\rangle_{s} = 
		\sum_{\delta\in \Delta_{s}/\Delta_{r}} \langle m,\delta^{-1} m'\rangle_{r}
		\label{pairingchangeinr}
	\end{equation}
	for all $r\ge s$.  Then
 	for each $r$, the pairings
		$
			\xymatrix@1{
				{(\cdot,\cdot)_{r}: M_r \times M_r'} \ar[r] &  \Lambda_{A_r,r}				}
		$
		defined by
		\begin{equation*}
			(m,m')_{r} := \sum_{\delta\in \Delta/\Delta_r} \langle m , \delta^{-1} m'\rangle_r \cdot \delta
		\end{equation*}
		are $\Lambda_{A_r,r}$-bilinear and perfect, and  compile to give a $\Lambda_B$-linear
		perfect pairing
		\begin{equation*}
			\xymatrix{
				{(\cdot,\cdot)_{\Lambda_B}: M_B \times M_B'} \ar[r] & {\Lambda_B}
				}.
		\end{equation*}
		In particular, $M_B'$ and $M_B$ are canonically $\Lambda_B$-linearly dual to eachother. 
\end{lemma}

\begin{proof}
	An easy reindexing argument shows that $(\cdot,\cdot)_{r}$ is $\Lambda_{A_r,r}$-linear
	in the right factor, from which it follows that it is also $\Lambda_{A_r,r}$-linear
	in the left due to our assumption that $\delta\in \Delta$ is self-adjoint with respect to
	$\langle\cdot, \cdot\rangle_{r}$.  To prove that 
	$(\cdot,\cdot)_{r}$ is a perfect duality pairing, we analyze the $\Lambda_{A_r,r}$-linear map
	\begin{equation}
		\xymatrix@C=45pt{
			{M_r} \ar[r]^-{m\mapsto (m,\cdot)_r} & {\Hom_{\Lambda_{A_r,r}}(M_r',\Lambda_{A_r,r})}
		}.\label{GroupRingDuality}
	\end{equation}
	Due to Lemma \ref{Technical}, both $M_r$ and $M_r'$ are free $\Lambda_{A_r,r}$-modules,
	necessarily of the same rank by the existence of the perfect $A_r$-duality pairing (\ref{pairinghyp}).
	It follows that (\ref{GroupRingDuality}) is a homomorphism of free $\Lambda_{A_r,r}$-modules
	of the same rank. To show that it is an isomorphism it therefore suffices to prove it is surjective,
	which may be checked after extension of scalars along the augmentation map
	$\Lambda_{A_r,r}\twoheadrightarrow A_r$ by Nakayama's Lemma.
	Consider the diagram
	\begin{equation}
	\begin{gathered}
		\xymatrix@C=38pt{
		{M_r \tens_{\Lambda_{A_r,r}} A_r} \ar[r]^-{(\ref{GroupRingDuality})\otimes 1}
		\ar[d]_-{\rho_{r,1}\otimes 1}^-{\simeq} & 
		{\Hom_{\Lambda_{A_r,r}}(M_r',\Lambda_{A_r,r})\tens_{\Lambda_{A_r,r}} A_r} \ar[r]^-{\xi}_-{\simeq} &
		{\Hom_{A_r}(M_r'\tens_{\Lambda_{A_r,r}} A_r, A_r)}  \\
		{M_1\tens_{A_1} A_r} \ar[rr]^-{\simeq} & &{\Hom_{A_r}(M_1'\tens_{A_1} A_r, A_r)}
		\ar[u]_-{(\rho_{r,1}'\otimes 1)^{\vee}}^-{\simeq}
		}
	\end{gathered}
	\label{XiDiagram}
	\end{equation}
	where $\xi$ is the canonical map sending $f\otimes \alpha$ to $\alpha(f\otimes 1)$,
	and the bottom horizontal arrow is obtained by $A_r$-linearly extending the canonical
	duality map $m\mapsto \langle m,\cdot\rangle_1$.  On the one hand, the vertical
	maps in (\ref{XiDiagram}) are isomorphisms thanks to Lemma \ref{Technical} (\ref{red2psurj}),
	while the map $\xi$ and the bottom horizontal arrow are isomorphisms
	because arbitrary extension of scalars commutes with linear duality 
	of {\em free} modules.\footnote{Quite generally, for any ring $R$, any $R$-modules 
	$M$, $N$, and any $R$-algebra $S$, the canonical map
	\begin{equation*}
	\xymatrix{
		{\xi_M:\Hom_R(M,N)\otimes_R S} \ar[r] & {\Hom_S(M\otimes_R S, N\otimes_R S)}
		}
	\end{equation*}
	sending $f\otimes s$ to $s(f\otimes \id_S)$ is an isomorphism if $M$ is finite and free over $R$.
	Indeed, the map $\xi_R$ is visibly an isomorphism, and one checks that $\xi_{M_1\oplus M_2}$
	is naturally identified with $\xi_{M_1}\oplus \xi_{M_2}$.
	}
	On the other hand, this diagram commutes because (\ref{pairingchangeinr}) guarantees the relation
	\begin{equation*}
	       \xymatrix@C=15pt{
	       {\langle \rho_{r,1}m , \rho_{r,1}'m' \rangle_1} \ar@^{=}[r]^-{(\ref{pairingchangeinr})} & 	    	       
	       {\displaystyle\sum_{\delta\in \Delta/\Delta_r} \langle  m,\delta^{-1}m'\rangle_r
	       \equiv (m,m')_r \bmod{I_{\Delta}}}
	       }  
	\end{equation*}
	where $I_{\Delta}=\ker(\Lambda_{A_r,r}\twoheadrightarrow A_r)$ is the augmentation ideal
	We conclude that (\ref{GroupRingDuality}) is an isomorphism, as desired.
	The argument that the corresponding map with the roles of $M_r$ and $M_r'$ interchanged
	is an isomorphism proceeds {\em mutatis mutandis}.
	
	Using the definition
	of $(\cdot,\cdot)_r$ and (\ref{pairingchangeinr}), one has more generally that
	\begin{equation*}
		(\rho_{r,s}m,\rho_{r,s}'m')_{s} \equiv (m,m')_r \bmod 
		\ker(\Lambda_{A_r,r}\twoheadrightarrow \Lambda_{A_r,s}) 
	\end{equation*}
	for all $r\ge s$.  In particular, the pairings $(\cdot,\cdot)_r$ induce, by extension 
	of scalars, a $\Lambda_B$-bilinear pairing
	\begin{equation*}
		\xymatrix{
			{(\cdot,\cdot)_{\Lambda_B}: M_B\times M_{B}'} \ar[r] & {\Lambda_B}
			}
	\end{equation*}
	which satisfies the specialization property
	\begin{equation}
		(\cdot,\cdot)_{\Lambda_B} \equiv (\cdot, \cdot)_{r} \bmod \ker(\Lambda_B \twoheadrightarrow \Lambda_{B,r}).
		\label{pairingspecialize}
	\end{equation}
	From $(\cdot,\cdot)_{\Lambda_B}$ we obtain in the usual way duality morphisms
	\begin{equation}
		\xymatrix@C=45pt{
			{M_B} \ar[r]^-{m\mapsto (m,\cdot)_{\Lambda_B}} & \Hom_{\Lambda_B}(M_{B}',\Lambda_B)
		}\quad\text{and}\quad
		\xymatrix@C=45pt{
			{M_B'} \ar[r]^-{m'\mapsto (\cdot,m')_{\Lambda_B}} & \Hom_{\Lambda_B}(M_{B},\Lambda_B)
		}\label{LambdaDualityMaps}
	\end{equation}
	which we wish to show are isomorphisms.  Due to Lemma \ref{Technical} (\ref{MBfree}),
	each of (\ref{LambdaDualityMaps}) is a map of finite free $\Lambda_B$-modules of the same rank,
	so we need only show that these mappings are surjective.  As the kernel of 
	$\Lambda_B\twoheadrightarrow \Lambda_{B,r}$ is contained in the radical of $\Lambda_B$,
	we may by Nakayama's Lemma check 
	such surjectivity after extension of scalars along $\Lambda_B\twoheadrightarrow \Lambda_{B,r}$
	for any $r$, where it follows from (\ref{pairingspecialize}) and the fact that $M_r$ and $M_{s}$
	are free $\Lambda_{A_r,r}$-modules, so that the extension of scalars of the perfect
	duality pairing $(\cdot,\cdot)_r$ along the canonical map 
	$\Lambda_{A_r,r}\rightarrow \Lambda_{B,r}$ is again perfect.	
\end{proof}

\subsection{Ordinary families of de Rham cohomology}\label{ordfamdR}

Let $\{\X_r/T_r\}_{r\ge 0}$ be the tower of modular curves 
introduced in \S\ref{tower}.  
As $\X_r$ is regular and proper flat over $T_r=\Spec(R_r)$ with geometrically reduced fibers, it is a curve
in the sense of Definition \ref{curvedef} (thanks to Corollary \ref{curvecorollary})
which moreover satisfies the hypotheses of Proposition \ref{HodgeIntEx}.
Abbreviating
\begin{align}
	& H^0(\omega_{r}):=H^0(\X_{r},\omega_{\X_{r}/S_{r}}),  &  & H^1_{\dR,r}:= H^1(\X_{r}/R_{r}), &  & H^1(\O_{r}):=H^1(\X_{r},\O_{\X_{r}}),\label{shfcoh}
\end{align}  
Proposition \ref{HodgeIntEx} (\ref{CohomologyIntegral}) provides a canonical short exact sequence
$H(\X_r/R_r)$ of finite 
free $R_r$-modules
\begin{equation}
	\xymatrix{
		0\ar[r] & {H^0(\omega_r)} \ar[r] & {H^1_{\dR,r}} \ar[r] & {H^1(\O_r)} \ar[r] & 0
	}\label{HodgeFilIntAbbrev}
\end{equation}
which recovers the Hodge filtration of $H^1_{\dR}(X_r/K_r)$ after inverting $p$.

The Hecke correspondences on $\X_r$ induce, via Proposition \ref{HodgeIntEx} (\ref{CohomologyFunctoriality})
(or by Proposition \ref{intcompare} and Remark \ref{canonicalproperty}),
canonical actions of $\H_r$ and $\H_r^*$ on $H(\X_r/R_r)$ via $R_r$-linear endomorphisms.
In particular, $H(\X_r/R_r)$ is canonically a short exact sequence of $\Z_p[(\Z/Np^r\Z)^{\times}]$-modules
via the diamond operators.
Similarly, pullback along (\ref{gammamaps}) yields $R_r$-linear morphisms
$H((\X_r)_{\gamma}/R_r)\rightarrow H(\X_r/R_r)$ for each $\gamma\in \Gamma$;
using the fact that hypercohomology commutes with flat base change 
(by \v{C}ech theory), we obtain an action of $\Gamma$ on $H(\X_r/R_r)$
which is $R_r$-semilinear over the canonical action of $\Gamma$ on $R_r$
and which commutes with the actions of $\H_r$ and $\H_r^*$ as the Hecke operators are defined over 
$K_0=\Q_p$.

For $r\ge s$, we will need to work with the base change $\X_s\times_{T_s} T_r$,
which is a curve over $T_r$ thanks to Proposition \ref{curveproperties}.
Although $\X_s\times_{T_s} T_r$ need no longer be regular as  $T_{r}\rightarrow T_{s}$ is not smooth when $r> s$,
we claim that it is necessarily {\em normal}.  
Indeed, this follows from the more general assertion:

\begin{lemma}\label{normalcrit}
		Let $V$ be a discrete valuation ring and $A$ a finite type Cohen-Macaulay $V$-algebra with smooth generic fiber and geometrically reduced
		special fiber.  Then $A$ is normal.
\end{lemma} 

\begin{proof}
	We claim that $A$ satisfies Serre's ``$R_1+S_2$"-criterion for normality \cite[Theorem 23.8]{matsumura}.  As $A$ is assumed
	to be CM, by definition of Cohen-Macaulay $A$ verifies $S_i$ for all $i\ge 0$, so we need only show that each localization of $A$ at
	a prime ideals of codimension 1 is regular.  Since $A$ has geometrically reduced special fiber, this special fiber is in particular
	smooth at its generic points.  As $A$ is flat over $V$ (again by definition of CM), we deduce that the (open) $V$-smooth locus   
	in $\Spec A$ contains the generic points of the special fiber and hence contains all codimension-1 points (as the generic
	fiber of $\Spec A$ is assumed to be smooth).  Thus $A$ is $R_1$, as desired.
\end{proof}

We conclude that $\X_{s}\times_{T_s} T_r$ is a normal curve, and 
we obtain from Proposition \ref{HodgeIntEx}
a canonical short exact sequence of finite free $R_r$-modules $H(\X_{s}\times_{T_s}T_r/R_r)$
which recovers the Hodge filtration of $H^1_{\dR}(X_{s}/K_r)$ after inverting $p$.
As hypercohomology commutes with flat base change and the formation of the 
relative dualizing sheaf and the structure sheaf are 
compatible with arbitrary base change, we have a natural isomorphism of short exact sequences of free $R_r$-modules
\begin{equation}
	H(\X_{s}\times_{T_s} T_r/R_r)	\simeq  H(\X_{s}/R_{s})\otimes_{R_{s}} R_r. \label{bccompat}
\end{equation}
In particular, we have $R_r$-linear actions of $\H_{s}^*$, $\H_r$ and an $R_r$-semilinear action of
$\Gamma$ on $H(\X_s\times_{T_s} T_r/R_r)$.  These actions moreover commute with one another.

Consider now the canonical degeneracy map $\pr: \X_r\rightarrow \X_{s}\times_{T_s} T_r$ of curves over $T_r$
induced by (\ref{rdegen}).  
As $\X_r$ and $\X_{s}\times_{T_s} T_r$ are normal and proper curves over $T_r$,
we obtain from Proposition \ref{HodgeIntEx} (\ref{CohomologyFunctoriality}) 
canonical trace mappings of short exact sequences
\begin{equation}
		\xymatrix{
		{\pr_* : H(\X_{r}/R_r)} \ar[r] & {H(\X_{s}\times_{T_s}T_r/R_r)
		\simeq H(\X_{s}/R_{s})\otimes_{R_{s}}R_r}
	}\label{trmap}
\end{equation}
which recover the usual trace mappings on de Rham cohomology after inverting $p$;
as such, these mappings are Hecke and $\Gamma$-equivariant, and
compatible with change in $r,s$ in the obvious way.
Tensoring these maps (\ref{trmap}) over $R_r$ with $R_{\infty}$, 
we obtain projective systems of free $R_{\infty}$ 
with semilinear $\Gamma$-action and commuting, linear 
$\H^*:=\varprojlim_r \H_r^*$ action:

\begin{definition}\label{limitmods}
	We write 
	\begin{align*}
			&H^0(\omega):=\varprojlim_r \left(H^0(\omega_r) \tens_{R_r} {R}_{\infty}\right),  
			&& H^1_{\dR}:=\varprojlim_r \left(H^1_{\dR,r}\tens_{R_r} {R}_{\infty} \right),
			&&  H^1(\O) := \varprojlim_r \left(H^1(\O_r)\tens_{R_r} {R}_{\infty}\right)
	\end{align*}
	for the projective limit with respect to the maps induced by $\pr_*$,
	each of which is naturally a module for 
	$\Lambda_{R_{\infty}}={R}_{\infty}[\![\Delta]\!]$, 
	and is equipped with a semilinear $\Gamma$-action 
	and a linear $\H^*$-action.  
\end{definition}

Although we have a left exact sequence of $\Lambda_{R_\infty}$-modules
with semilinear $\Gamma$-action and $\H^*$-action
\begin{equation*}
	\xymatrix{
		0\ar[r] & {H^0(\omega)} \ar[r] & {H^1_{\dR}} \ar[r] & {H^1(\O)}
	},
\end{equation*}
this sequence is almost certainly not right exact.  It is moreover unlikely that
any of the $\Lambda_{R_{\infty}}$-modules in 
Definition \ref{limitmods} are finitely generated.  The situation
is much better if we pass to {\em ordinary parts}:

\begin{theorem}\label{main}
	Let $e^*$ be the idempotent of $\H^*$ associated to $U_p^*$ and let $d$ be
	the positive integer defined as in Proposition $\ref{IgusaStructure}$ $(\ref{IgusaFreeness})$.
	Then $e^*H^0(\omega)$, $e^*H^1_{\dR}$ and $e^*H^1(\O)$ are free $\Lambda_{R_{\infty}}$-modules
	of ranks $d$, $2d$, and $d$ respectively, and there is a canonical short exact sequence
	of free $\Lambda_{R_{\infty}}$-modules with linear $\H^*$-action and $R_{\infty}$-semilinear
	$\Gamma$-action
	\begin{equation}
		\xymatrix{
			0\ar[r] & {e^*H^0(\omega)} \ar[r] & {e^*H^1_{\dR}} \ar[r] & {e^*H^1(\O)} \ar[r] & 0
			}.\label{mainthmexact}
	\end{equation}
	For each positive integer $r$, applying $\otimes_{\Lambda_{R_{\infty}}} R_{\infty}[\Delta/\Delta_r]$
	to $(\ref{mainthmexact})$ yields the short exact sequence
	\begin{equation}
		\xymatrix{
			0\ar[r] & {e^*H^0(\omega_r)\tens_{R_r} R_{\infty}} \ar[r] & 
			{e^*H^1_{\dR}\tens_{R_r} R_{\infty}} \ar[r] & 
			{e^*H^1(\O)\tens_{R_r} R_{\infty}} \ar[r] & 0
			},\label{mainthmexact2}	
	\end{equation}
	compatibly with the actions of $\H^*$ and $\Gamma$.	 
\end{theorem}

\begin{proof}
	
Applying $e^*$ to the short exact sequence $H(\X_r/R_r)$ yields a short exact sequence
\begin{equation}
	\xymatrix{
		0\ar[r] & {e^*H^0(\omega_r)}\ar[r] & {e^*H^1_{\dR,r}} \ar[r] & {e^*H^1(\O_r)} \ar[r] & 0
		}\label{hitwithidem}
\end{equation}
of $R_r[\Delta/\Delta_r]$-modules with linear $\H_r^*$-action and $R_r$-semilinear $\Gamma$-action in which each term is
free as an $R_r$-module.\footnote{Indeed, $e^*M$ is a direct summand of $M$ for any $\H_r^*$-module $M$,
and hence $R_r$-projective ($=R_r$-freee) if $M$ is.}
Similarly, for each pair of nonnegative integers $r\ge s$,
the trace mappings (\ref{trmap}) induce a commutative diagram with exact rows
\begin{equation}
\begin{gathered}
	\xymatrix{
		0\ar[r] & {e^*H^0(\omega_r)} \ar[r]\ar[d]_-{\pr_*} & 
		{e^*H^1_{\dR,r}} \ar[r]\ar[d]_-{\pr_*} & {e^*H^1(\O_r)} 
		\ar[r]\ar[d]^-{\pr_*} & 0\\
		0\ar[r] & {e^*H^0(\omega_{s})\otimes_{R_{s}}R_r} \ar[r] & 
		{e^*H^1_{\dR,{s}}\otimes_{R_{s}}R_r} \ar[r] & {e^*H^1(\O_{s})\otimes_{R_{s}}R_r} 
		\ar[r] & 0
	}
	\end{gathered}
	\label{piecetogether}
\end{equation}
We will apply Lemma \ref{Technical} with $A_r=R_r$, $I_r=(\pi_r)$, $B=R_{\infty}$ 
nd with $M_r$ each one of the terms in (\ref{hitwithidem}).
In order to do this, we must check that the hypotheses (\ref{freehyp}) and 
(\ref{surjhyp}) are satisfied.

Applying $\otimes_{R_r} \F_p$ to the short exact sequence (\ref{hitwithidem})
and using the fact that the idempotent $e^*$ commutes with tensor products,
we obtain, thanks to Lemma \ref{ReductionCompatibilities} (\ref{BaseChngDiagram}),
the short exact sequence of $\F_p$-vector spaces (\ref{sesincharp1}).
By Corollary \ref{FreenessInCharp}, the three terms of (\ref{sesincharp1}) are free 
$\F_p[\Delta/\Delta_r]$-modules
of ranks $d$, $2d$, and $d$ respecvitely, so (\ref{freehyp})
is satisfied for each of these terms.  Similarly, applying $\otimes_{R_r} \F_p$
to the diagram (\ref{piecetogether}) 
yields a diagram which by Corollary \ref{SplitIgusa} is naturally isomorphic to
the diagram of $\F_p[\Delta/\Delta_r]$-modules with split-exact rows
\begin{equation*}
	\xymatrix@C=15pt{
		0 \ar[r] & {H^0(I_r^{\infty},\Omega^1(\SS))^{V_{\ord}}} \ar[r]\ar[d]_-{\pr_*} &
		{H^0(I_r^{\infty},\Omega^1(\SS))^{V_{\ord}}\oplus
		H^1(I_r^0,\O(-\SS))^{F_{\ord}}} \ar[r]\ar[d]|-{\pr_*\oplus \pr_*} & 
		{H^1(I_r^0,\O(-\SS))^{F_{\ord}}} \ar[r]\ar[d]^-{\pr_*} & 0 \\
		0 \ar[r] & {H^0(I_{s}^{\infty},\Omega^1(\SS))^{V_{\ord}}} \ar[r] &
		{H^0(I_{s}^{\infty},\Omega^1(\SS))^{V_{\ord}}\oplus
		H^1(I_{s}^0,\O(-\SS))^{F_{\ord}}} \ar[r] & 
		{H^1(I_{s}^0,\O(-\SS))^{F_{\ord}}} \ar[r] & 0
	}
\end{equation*}
Each of the vertical maps in this diagram is surjective due to Proposition \ref{IgusaStructure} 
(\ref{IgusaControl}), and we conclude that the hypothesis (\ref{surjhyp}) is satisfied
as well.  Furthermore, the vertical maps in (\ref{piecetogether}) are then surjective by Nakayama's
Lemma, so applying $\otimes_{R_r} R_{\infty}$ 
yields an inverse system of short exact sequences in which the first term satisfies the Mittag-Leffler
condition.  Passing to inverse limits is therefore (right) exact, and we obtain the short 
exact sequence (\ref{mainthmexact}).
\end{proof}

Due to Proposition \ref{HodgeIntEx} (\ref{CohomologyDuality}), the short exact sequence
(\ref{HodgeFilIntAbbrev}) is auto-dual with respect to the canonical cup-product pairing $(\cdot,\cdot)_r$
on $H^1_{\dR,r}$.  We extend scalars along $R_r\rightarrow R_r':=R_r[\mu_N]$, so that the Atkin-Lehner
``invoultion" $w_r$ is defined, and consider the ``twisted" pairing on ordinary parts
\begin{equation}
\xymatrix{
	{\langle \cdot,\cdot\rangle _r : ({e^*}H^1_{\dR,r})_{R_r'} \times ({e^*}H^1_{\dR,r})_{R_r'}} \ar[r] & {R_r'}
	}\qquad\text{given by}\qquad \langle x, y\rangle_r := (x, w_r {U_p^*}^r y).
	\label{TwistdRpairing}
\end{equation}
It is again perfect and satisfies $\langle T^* x,y\rangle =\langle x, T^* y \rangle$
for all $x,y\in (e^*H^1_{\dR,r})_{R_r'}$ and $T^*\in \H_r^*$.
\begin{proposition}\label{dRDuality}
	The pairings $(\ref{TwistdRpairing})$ compile to give a perfect $\Lambda_{R_{\infty}'}$-linear
	duality pairing 
	\begin{equation*}
		\xymatrix{
		{\langle\cdot,\cdot\rangle_{\Lambda_{R_{\infty}'}}:   
		({e^*}H^1_{\dR})_{\Lambda_{R_{\infty}'}} \hspace{-1ex}\times ({e^*}H^1_{\dR})_{\Lambda_{R_{\infty}'}}}
		\ar[r] & {\Lambda_{R_{\infty}'}}
		}\ \text{given by}\ 
		\langle x , y \rangle_{\Lambda_{R_{\infty}'}} :=
		\varprojlim_r\sum_{\delta\in \Delta/\Delta_r} \langle x_r, \langle \delta^{-1}\rangle^* y_r\rangle_r\cdot\delta
	\end{equation*}
	for $x=\{x_r\}_r$ and $y=\{y_r\}_r$ in $(e^*H^1_{\dR})_{\Lambda_{R_{\infty}'}}$.
	The pairing $\langle \cdot,\cdot \rangle_{\Lambda_{R_{\infty}}'}$  induces
	a canonical isomorphism 
	\begin{equation*}
		\xymatrix{
				0\ar[r] & {e^*H^0(\omega)(\langle\chi\rangle\langle a\rangle_N)_{\Lambda_{R_{\infty}'}}}
				\ar[r]\ar[d]^-{\simeq} & 
				{e^*H^1_{\dR}(\langle\chi\rangle\langle a\rangle_N)_{\Lambda_{R_{\infty}'}}} 
				\ar[r]\ar[d]^-{\simeq} & 
				{e^*H^1(\O)(\langle\chi\rangle\langle a\rangle_N)_{\Lambda_{R_{\infty}'}}} 
				\ar[r]\ar[d]^-{\simeq} & 0\\
	0\ar[r] & {(e^*H^1(\O))^{\vee}_{\Lambda_{R_{\infty}'}}} \ar[r] & 
	{({e^*H^1_{\dR}})^{\vee}_{\Lambda_{R_{\infty}'}}} \ar[r] & 
	{(e^*H^0(\omega))^{\vee}_{\Lambda_{R_{\infty}'}}} \ar[r] & 0
		}
	\label{mainthmdualityisom}
	\end{equation*}
	that is $\H^*$-equivariant and compatible with the natural action 
	of $\Gamma \times \Gal(K_0'/K_0)\simeq \Gal(K_{\infty}'/K_0)$ 
	on the bottom row and the twist
	$\gamma\cdot m := \langle \chi(\gamma)\rangle\langle a(\gamma)\rangle_N \gamma m$
	of the natural action on the top, 	
	where 
	$a(\gamma) \in (\Z/N\Z)^{\times}$ is
	determined by
	$\zeta^{a(\gamma)}=\gamma\zeta$ for every $\zeta\in \mu_N(\Qbar_p)$.
\end{proposition}

\begin{proof}[Proof of Proposition $\ref{dRDuality}$]
	That $\langle\cdot,\cdot \rangle_{\Lambda_{R_{\infty}'}}$ is a perfect duality pairing
	follows easily from Lemma \ref{LambdaDuality}, using Theorem \ref{main} and the formalism of 
	\S\ref{TowerFormalism}, once we check that the twisted pairings (\ref{TwistdRpairing})
	satisfy the hypothesis (\ref{pairinghyp}).  By the definition (\ref{TwistdRpairing})
	of $\langle\cdot, \cdot\rangle_r$, this amounts to the computation
	\begin{align*}
		({\rho_1}_* x, w_r {U_p^*}^r{\rho_1}_* y)_r = (x, \rho_1^* w_r {U_p^*}^r{\rho_1}_*y)_{r+1}
		&=(x, w_{r+1} {U_p^*}^r \rho_2^*{\rho_1}_*y)_{r+1} \\				
		&= \sum_{\delta\in \Delta_r/\Delta_{r+1}}(x, w_{r+1} {U_p^*}^{r+1} \langle \delta^{-1}\rangle^* y)_{r+1}
	\end{align*}
	where we have used Proposition \ref{ALinv} and the identity 
	$\rho_2^*{\rho_1}_* = U_p^*\sum_{\delta\in \Delta_r/\Delta_{r+1}} \langle\delta^{-1}\rangle^*$
	on $H^1_{\dR,r+1}$, which follows from\footnote{The reader will check that our forward reference
	to \S\ref{BTfamily} does not involve any circular reasoning.} 
	Lemma \ref{MFtraceLem} by using Lemma \ref{LieFactorization} and 
	Proposition \ref{intcompare}.  We obtain an isomorphism of short exact sequences of
	$\Lambda_{R_{\infty}'}$-modules as in (\ref{mainthmdualityisom}), which it remains to show is 
	$\Gamma\times\Gal(K_0'/K_0)$-equivariant
	for the specified actions.  For this, we compute that for $\gamma\in\Gal(K_{\infty}'/K_0)$,
	\begin{equation*}
		\langle \gamma x,\gamma y\rangle_{r} =
		(\gamma x, w_r {U_p^*}^r \gamma y)_r = 
		 (\gamma x,  \gamma w_r {U_p^*}^r \langle \chi(\gamma)^{-1} \rangle\langle a(\gamma)^{-1}\rangle_N  y)_r
		 =  \gamma \langle x,\langle \chi(\gamma)^{-1} \rangle\langle a(\gamma)^{-1}\rangle_N y\rangle_r,
	\end{equation*}
	where we have used Proposition \ref{ALinv} and the fact that the cup product is Galois-equivarant.
	It now follows easily from definitions that 
	\begin{equation*}
	\langle \gamma x,\gamma y\rangle_{\Lambda_{R_{\infty}'}} = 
	\langle \chi(\gamma)^{-1} \rangle \gamma \langle x, \langle a(\gamma)^{-1}\rangle_N y	
	\rangle_{\Lambda_{R_{\infty}'}},
	\end{equation*}
	and the claimed $\Gamma\times\Gal(K_0'/K_0)$-equivariance of  (\ref{mainthmdualityisom}) is equivalent to this.
\end{proof}

\begin{remark}
	For an open subgroup $H$ of $\scrG_K$ and any $H$-stable subfield 
	$F$ of $\c_{K}$, denote by $\Rep_F(H)$ the category of finite-dimensional
	$F$-vector spaces that are equipped with a continuous semilinear 
	action of $H$.  Recall \cite{DSen} that classical Sen theory provides a functor
	$\D_{\Sen}:\Rep_{\c_K}(\scrG_K)\rightarrow \Rep_{K_{\infty}}(\Gamma)$
	which is quasi-inverse to $(\cdot)\otimes_{K_{\infty}} \c_K$.  Furthermore,
	for any $W\in \Rep_{\c_K}(\scrG_K)$, there is a unique $K_{\infty}$-linear
	operator $\Theta_{D}$ on $D:=\D_{\Sen}(W)$ with the property that 
	$\gamma x = \exp(\log \chi(\gamma)\cdot \Theta_D)(x)$ for all $x\in D$
	and all $\gamma$ in a small enough open neighborhood of $1\in \Gamma$.
	
	We expect that for $W$ any specialization of $e^*H^1_{\et}$ along a continuous homomorphism 
	$\Lambda\rightarrow K_{\infty}$, there is a canonical isomorphism between $D:=\D_{\Sen}(W\otimes {\c_K})$
	and the corresponding specialization of $e^*H^1_{\dR}$,
	with the Sen operator $\Theta_D$
	induced by the Gauss-Manin connections on $H^1_{\dR,r}$.  In this way, we might think
	of $e^*H^1_{\dR}$ 
	as a $\Lambda$-adic avatar of ``$\D_{\Sen}(e^*H^1_{\et}\otimes_{\Lambda} \Lambda_{\O_{\c_K}})$."
	We hope to pursue these connections in future work.
\end{remark}

\subsection{Ordinary \texorpdfstring{$\Lambda$}{Lambda}-adic modular forms}\label{ordforms}

In this section, we discuss the relation between $e^*H^0(\omega)$ 
and ordinary $\Lambda_{R_{\infty}}$-adic cuspforms as defined by Ohta \cite[Definition 2.1.1]{OhtaEichler}.

We begin with some preliminaries on modular forms.
For a ring $A$, a congruence subgroup $\Gamma$, and a nonnegative integer $k$, we will write $S_k(\Gamma;A)$ for the space of weight $k$ cuspforms for $\Gamma$ over $A$; we put $S_k(\Gamma):=S_k(\Gamma;\Qbar)$.
If $\Gamma',$ $\Gamma$ are congruence subgroups and $\gamma\in \GL_2(\Q)$ satisfies
$\gamma^{-1}\Gamma'\gamma\subseteq \Gamma$, then there is a 
canonical injective ``pullback" map 
on modular forms
$\xymatrix@1{{\iota_{\gamma}:S_k(\Gamma)} \ar@{^{(}->}[r] & {S_k(\Gamma')}}$ 
given by $\iota_{\gamma}(f):=f\big|_{\gamma^{-1}}$. 
When $\Gamma'\subseteq \Gamma$, {\em unless specified to the contrary}, we will
always view $S_k(\Gamma)$ as a subspace of $S_k(\Gamma')$ via $\iota_{\id}$.
As $\gamma\Gamma'\gamma^{-1}$ is necessarily of finite index in $\Gamma$,
one also has a canonical ``trace" mapping
\begin{equation}
	\xymatrix{
		{\tr_{\gamma}:S_k(\Gamma')} \ar[r] & {S_k(\Gamma)}
		}
		\qquad\text{given by}\qquad
		\tr_{\gamma}(f):=\sum_{\delta\in \gamma^{-1}\Gamma'\gamma\backslash\Gamma} (f\big|_{\gamma})\big|_{\delta}
		\label{MFtrace}
\end{equation}
with the property that $\tr_{\gamma}\circ\iota_{\gamma}$ is multiplication by $[\Gamma: \gamma^{-1}\Gamma'\gamma]$
on $S_k(\Gamma)$.

We define
	\begin{equation*}
		S_2^{\infty}(\Gamma_r;R_r):=S_2(\Gamma_r;R_r)\qquad\text{and}\qquad
		S_2^{0}(\Gamma_r;R_r):=
		\{f\in S_2(\Gamma_r; \Qbar_p)\ :\ f\big|_{w_r} \in S_2^{\infty}(\Gamma_r;R_r) \},
	\end{equation*}
By definition, $S_2^{\star}(\Gamma_r;R_r)$ for $\star=0,\infty$ are $R_r$-submodules
of $S_2(\Gamma_r;K_r')$ that are carried isomorphically onto eachother by the automorphism
$w_r$ of $S_2(\Gamma_r;K_r')$.	Note that $S_2^{\star}(\Gamma_r;R_r)$ is precisely the
$R_r$-submodule consisting of cuspforms whose formal expansion at the cusp $\star$
has coefficients in $R_r$.	
As the Hecke algebra $\H_r$ stabilizes 
$S_2^{\infty}(\Gamma_r; R_r)$, it follows immediately from Proposition \ref{AtkinInterchange} that
$S_2^0(\Gamma_R;R_r)$ is stable under the action of $\H_r^*$ on $S_2(\Gamma_r; K_r)$.
Furthermore, $\Gal(K_r'/K_0)$ acts on $S_2(\Gamma_r;K_r')\simeq S_2(\Gamma_r;\Q_p)\otimes_{\Q_p} K_r'$
through the second tensor factor, and this action leaves stable the $R_r$-submodule $S_2^{\infty}(\Gamma_r; R_r)$.
The second equality of Proposition \ref{ALinv} then implies that $S_2^0(\Gamma_r;R_r)$
is also a $\Gal(K_r'/K_0)$-stable $R_r$-submodule of $S_2(\Gamma_r;K_r')$.  
A straightforward computation shows that 
the direct factor $\Gal(K_0'/K_0)$ of $\Gal(K_r'/K_0)$ acts trivially on $S_2^{\infty}(\Gamma_r;R_r)$
and through $\langle a\rangle_N^{-1}$ on $S_2^0(\Gamma_r;R_r)$.

We can interpret $S_2^{\star}(\Gamma_r;R_r)$ geometrically as follows.
As in Remark \ref{MWGood}, for $\star= \infty, 0$
let $I_r^{\star}$ be the irreducible component 
of $\o{\X}_r$ passing through the cusp $\star$,
and denote by $\X_{r}^{\star}$ the complement
in $\X_r$ of all irreducible components of $\o{\X}_r$ distinct from $I_{r}^{\star}$.
By construction, $\X_r$ and $\X_r^{\star}$ have the same generic fiber $X_r\times_{\Q_p} K_r$.
Using Proposition \ref{redXr}, it is not hard to show that the diamond operators
induce automorphisms of $\X_r^{\star}$, and one checks via Proposition \ref{AtkinInertiaCharp}
that the ``semilinear" action (\ref{gammamaps}) of $\gamma\in \Gamma$ on $\X_r$
carries $\X_r^{\star}$ to $(\X_r^{\star})_{\gamma}$ for all $\gamma$.

\begin{lemma}\label{Edixhoven}
	Formal expansion at the $R_r$-point $\infty$ $($respectively $R_r'$-point $0$$)$ of $\X_r^{\star}$ 
	induces an isomorphism of $R_r$-modules
		\begin{equation}
			H^0(\X_r^{\infty},\Omega^1_{\X_r^{\infty}/R_r}) \simeq S_2^{\infty}(\Gamma_r;R_r)
			\quad\text{respectively}\quad
			H^0(\X_r^{0},\Omega^1_{\X_r^{0}/R_r})(\langle a\rangle_N^{-1}) \simeq S_2^{0}(\Gamma_r;R_r)
		\end{equation}
	which is equivariant for the natural action of $\Gamma$ and $\H_r$ $($respectively $\H_r^*$$)$ on source
	and target and, in the case of the second isomorphism, intertwines the action of $\Gal(K_0'/K_0)$
	via $\langle a\rangle_N^{-1}$ on source with the natural action on the target.
\end{lemma}

\begin{proof}
	The proof is a straightforward
	adaptation of the proof of  \cite[Proposition 2.5]{EdixhovenComparison}.
\end{proof}	
	
Now $\X_r\rightarrow S_r$ is smooth outside the supersingular points, so there is a canonical
closed immersion $\iota_r^{\star}:\X_r^{\star}\hookrightarrow \X_r^{\sm}$.
Using Lemmas \ref{ConcreteDualizingDescription} and \ref{Edixhoven}, 
pullback of differentials along $\iota_r^{\star}$ gives a natural map 
\begin{equation}
	\xymatrix{
		{H^0(\X_r,\omega_{\X_r/T_r})\simeq H^0(\X_r^{\sm},\Omega^1_{\X_r^{\sm}/T_r})} \ar[r]^-{(\iota_r^{\star})^*} & 
		{H^0(\X_r^{\star},\Omega^1_{\X_r^{\star}/T_r}) \simeq S_2^{\star}(\Gamma_r;R_r)}
		}\label{OmegasComparison}
\end{equation}
which is an isomorphism after inverting $p$.
In particular,
the map (\ref{OmegasComparison}) is injective, $\Gamma$-equivariant, and compatible with 
the natural action of $\H_r$ (respectively $\H_r^*$) on source and target for $\star=\infty$ (respectively
$\star=0$), and in the case of $\star=0$ intertwines the action of $\Gal(K_0'/K_0)$ via the character
$\langle a\rangle_N^{-1}$ on source with the natural action on the target.

\begin{remark}
	The image of $(\ref{OmegasComparison})$ for $\star=\infty$ 
	is naturally identified 
	with the space of weight $2$ cuspforms for $\Gamma_r$ whose formal expansion
	at {\em every} cusp has $R_r$-coefficients. 
\end{remark}

Applying the idempotent $e$ (respectively $e^*$) to (\ref{OmegasComparison}) with $\star=\infty$
(respectively $\star=0$) gives an injective homomorphism
\begin{subequations}
\begin{equation}
	\xymatrix{
		{eH^0(\X_r,\omega_{\X_r/T_r})} \ar@{^{(}->}[r] & {eS_2^{\infty}(Np^r;R_r)}
		}\label{OmegasComparisonOrd}
\end{equation}
respectively
\begin{equation}
	\xymatrix{
		{e^*H^0(\X_r,\omega_{\X_r/T_r})(\langle a\rangle_N^{-1})} \ar@{^{(}->}[r] & {e^*S_2^0(Np^r;R_r)}
		}\label{OmegasComparisonOrd0}
\end{equation}
\end{subequations}
which is compatible with the canonical actions of $\Gamma$ and of $\H_r$ (respectively $\H_r^*$) on
source and target and in the case of (\ref{OmegasComparisonOrd}) is $\Gal(K_0'/K_0)$-equivariant.

\begin{proposition}\label{MFGeometryIsom}
	The mappings $(\ref{OmegasComparisonOrd})$ and $(\ref{OmegasComparisonOrd0})$ are isomorphisms.
\end{proposition}

\begin{proof}
	We treat the case of $(\ref{OmegasComparisonOrd})$; the proof that (\ref{OmegasComparisonOrd})
	is an isomorphism goes through {\em mutatis mutandis}.
	We must show that (\ref{OmegasComparisonOrd}) is surjective.
	To do this, let $\nu\in e_rS_2^{\infty}(Np^r;R_r)$ be arbitrary.  Since (\ref{OmegasComparisonOrd})
	is an isomorphism after inverting $\pi_r$, there exists a least nonnegative integer $d$
	such that $\pi_r^d\nu$ is in the image of (\ref{OmegasComparisonOrd}).  
	Assume that $d\ge 1$, and let $\eta\in eH^0(\X_r,\omega_{\X_r/R_r})$
	be any element mapping to $\pi_r^d \nu$.  For an irreducible component $I$ of $\o{\X}_r$,
	write $I^h$ for the complement of the super-singular points in $I$, and denote by
	$i_r^{\infty}:I_r^{\infty,h}\hookrightarrow \X_r^{\infty}$ the canonical immersion.
	We then have a commutative diagram
	\begin{equation}
	\begin{gathered}
		\xymatrix@C=40pt{
			{H^0(\o{\X}_r,\omega_{\o{\X}_r/R_r})} \ar[r]^-{(\ref{OmegasComparison})\bmod \pi_r}\ar@{^{(}->}[d] & 
			{H^0(\X_r^{\infty},\Omega^1_{\X_r^{\infty}/R_r})\tens_{R_r} \F_p}\ar[d]^-{(i_r^{\infty})^*} \\
			{\displaystyle\prod\limits_{I\in \Irr(\o{\X}_r)} H^0(I^h,\Omega^1_{I^h/\F_p})}\ar[r]_-{\proj_{\infty}} & 
			{H^0(I_r^{\infty,h},\Omega^1_{I_r^{\infty,h}/\F_p})}
		}\label{etamaps20}
	\end{gathered}	
	\end{equation}
	where the left vertical mapping follows from Definition \ref{OmegaReg} and Remark \ref{OmegaRegMero}
	({\em cf}. the proof of Proposition \ref{charpord}), while the bottom map is simply projection.
	Our assumption that $d\ge 1$ implies that the image of $\o{\eta}:=\eta\bmod \pi_r$ 
	under the composite of the right vertical and top horizontal maps in (\ref{etamaps20})
	is zero and hence, viewing
	$\o{\eta}=(\eta_{(a,b,u)})$ as a meromorphic differential on the normalization of $\o{\X}_r$,
	we have $\eta_{(r,0,1)}=\proj_{\infty}(\o{\eta})=0$.  
	Using the formula (\ref{Upn1}), we deduce that $U_p^n\o{\eta}=0$ for $n$ sufficiently large.
	But $U_p$ acts invertibly on $\eta$ (and hence on $\o{\eta}$) so we necessarily have that
	$\o{\eta}=0$ or what is the same thing that $\eta\bmod \pi_r=0$.  We conclude that 
	$\pi^{d-1}\nu$ is in the image of (\ref{OmegasComparisonOrd}), contradicting the minimality
	of $d$.  Thus $d=0$ and (\ref{OmegasComparisonOrd}) is surjective.
\end{proof}

For $s \le r$, Ohta shows \cite[2.3.4]{OhtaEichler} that the trace mapping 
$\tr_{\id}:S_k(\Gamma_r; K_r)\rightarrow S_k(\Gamma_s; K_s)\otimes_{K_s} K_r$
attached to the inclusion $\Gamma_r\subseteq \Gamma_s$
carries $S_k^0(\Gamma_r;R_r)$ into $S_k^0(\Gamma_s;R_s)\otimes_{R_s} R_r$, so that the projective limit
\begin{equation*} 
	\s_k^*(N,R_{\infty}) : = \varprojlim_{\tr_{\id}} S_k^0(\Gamma_r; R_r)\otimes_{R_r} R_{\infty}
\end{equation*}
makes sense.  It is canonically a $\Lambda_{R_{\infty}}$-module, equipped with an action of $\H^*$,
a semilinear action of $\Gamma$, and a natural action of $\Gal(K_0'/K_0)$.
On the other hand,
let $eS(N;\Lambda_{R_{\infty}})\subseteq \Lambda_{R_{\infty}}[\![q]\!]$ 
be the space of ordinary $\Lambda_{R_{\infty}}$-adic
cuspforms of level $N$, as defined in \cite[2.5.5]{OhtaEichler}.
This space is equipped with an action of $\H$ via the usual formulae on formal
$q$-expansions (see, for example \cite[\S1.2]{WilesLambda}), as well as an action
of $\Gamma$ via its $q$-coefficient-wise action on $\Lambda_{R_{\infty}}[\![q]\!]$.

\begin{theorem}[Ohta]\label{OhtaThm}	
	Then there is a canonical isomorphism of $\Lambda_{R_{\infty}}$-modules 
	\begin{equation}
		\xymatrix{
			{eS(N;\Lambda_{R_{\infty}})} \ar[r]^-{\simeq} & {e^*\s_2^*(N,R_{\infty})}
			}\label{LambdaForms}
	\end{equation}
	that intertwines the action of $T\in \H$ on the source with that of $T^*\in \H^*$
	on the target, for all $T\in \H$.  This isomorphism is $\Gal(K_{\infty}'/K_0)$-equivariant 
	for the natural action of $\Gal(K_{\infty}'/K_0)$ on ${e^*\s_2^*(N,R_{\infty})}$
	and the twisted action $\gamma\cdot \scrF := \langle \chi(\gamma)\rangle^{-1}\langle a(\gamma)\rangle_N^{-1} 
	\gamma\scrF$ on $eS(N;\Lambda_{R_{\infty}})$.
\end{theorem}

\begin{proof}
	For the definition of the canonical map (\ref{LambdaForms}), as well as the proof that it is an isomorphism,
	see Theorem 2.3.6 and its proof in \cite{OhtaEichler}.  With the conventions of \cite{OhtaEichler},
	the claimed compatibility of (\ref{LambdaForms}) with Hecke operators is a consequence of \cite[2.5.1]{OhtaEichler},
	while the $\Gal(K_{\infty}'/K_0)$-equivariance of (\ref{LambdaForms}) follows from \cite[Proposition 3.5.6]{OhtaEichler}.
\end{proof}

\begin{corollary}\label{LambdaFormsRelation}
	There is a canonical isomorphism of $\Lambda_{R_{\infty}}$-modules 
	\begin{equation}
		 eS(N;\Lambda_{R_{\infty}})(\langle \chi\rangle^{-1})\simeq		e^* H^0(\omega) 
	\end{equation}
	that intertwines the action of $T\in \H$ on the source with $T^*\in \H^*$
	on the target and is $\Gamma$-equivariant for the canonical action of $\Gamma$
	on $e^*H^0(\omega)$ and the twisted action $\gamma\cdot \scrF:=\langle \chi(\gamma)\rangle^{-1} \gamma\scrF$
	on $eS(N;\Lambda_{R_{\infty}})$.
\end{corollary}

\begin{proof}
	This follows immediately from Proposition \ref{MFGeometryIsom} and Theorem \ref{OhtaThm}.
\end{proof}

\subsection{\texorpdfstring{$\Lambda$}{Lambda}-adic Barsotti-Tate groups}\label{BTfamily}

In order to define a crystalline analogue of Hida's ordinary $\Lambda$-adic \'etale cohomology,
we will apply the theory of \S\ref{PhiGammaCrystals} to a certain ``tower" 
$\{\G_r/R_r\}_{r\ge 1}$ of $p$-divisible groups (a $\Lambda$-adic Barsotti Tate group
in the sense of Hida \cite{HidaNotes}, \cite{HidaNotes2}) whose construction involves 
artfully cutting out certain $p$-divisible 
subgroups of $J_r[p^{\infty}]$ over $\Q$ and the ``good reduction'' theorems of Langlands-Carayol-Saito. 
The construction of $\{\G_r/R_r\}_{r\ge 1}$ is certainly well-known (e.g. \cite[\S1]{MW-Hida}, 
\cite[Chapter 3, \S1]{MW-Iwasawa}, \cite[Definition 1.2]{Tilouine}  
and \cite[\S 3.2]{OhtaEichler}), but as we shall need substantially finer information about the $\G_r$
than is available in the literature, we devote this section to recalling their construction and properties.

For nonnegative integers $i\le r$, write
$\Gamma_r^i:=\Gamma_1(Np^i)\cap \Gamma_0(p^r)$ for the intersection $($taken inside $\SL_2(\Z)$$)$,
so $\Gamma_r=\Gamma_r^r$.
We will need the following fact 
({\em cf.} \cite[pg. 339]{Tilouine}, \cite[2.3.3]{OhtaEichler})
concerning the trace mapping $(\ref{MFtrace})$ attached to the canonical inclusion 
$\Gamma_{r}\subseteq \Gamma_i$ for $r\ge i$; for notational clarity, 
we will write $\tr_{r,i}:S_k(\Gamma_r)\rightarrow S_k(\Gamma_i)$ for this map.

\begin{lemma}\label{MFtraceLem}
	Fix integers $i\le r$ and let $\tr_{r,i}:S_k(\Gamma_r)\rightarrow S_k(\Gamma_i)$ 
	be the trace mapping $(\ref{MFtrace})$ attached to the inclusion 
	$\Gamma_r\subseteq \Gamma_i$.  For $\alpha:=\left(\begin{smallmatrix} 1 & 0 \\ 0 & p\end{smallmatrix}\right)$, 
	we 	have an equality
	of $\o{\Q}$-endomorphisms of $S_k(\Gamma_{r})$
	\begin{equation}
		\iota_{\alpha^{r-i}}\circ \tr_{r,i} = (U_p^*)^{r-i} 
		\sum_{\delta\in \Delta_i/\Delta_{r}} \langle \delta \rangle.
		\label{DualityIdentity}
	\end{equation}
\end{lemma}

\begin{proof}
	We have
	index $p^{r-i}$ inclusions of groups $\Gamma_{r} \subseteq \Gamma_{r}^i \subseteq \Gamma_i$
	with $\Gamma_{r}$ normal in $\Gamma_{r}^i$, as it is the kernel of the canonical
	surjection $\Gamma_{r}^i\twoheadrightarrow \Delta_i/\Delta_{r}$.    
	For each $\delta\in \Delta_i/\Delta_{r}$, we
	fix a choice of $\sigma_{\delta}\in \Gamma_{r}^i$ mapping to $\delta$
	and calculate that	
	\begin{equation}
		\Gamma_i = \coprod_{\delta\in \Delta_i/\Delta_{r}} \coprod_{j=0}^{p^{r-i}-1}   		
		\Gamma_{r}\sigma_{\delta} \varrho_j
		\qquad\text{where}\qquad
		\varrho_j:=\begin{pmatrix} 1 & 0 \\ jNp^i & 1\end{pmatrix}.\label{CosetDecomp}
	\end{equation}  
	On the other hand, for each $0\le j < p^{r-i}$ one has the equality of matrices in $\GL_2(\Q)$
	\begin{equation}
		p^{r-i}\varrho_j \alpha^{-(r-i)} = \tau_{r} \begin{pmatrix} 1 & -j \\ 0 & p^{r-i} \end{pmatrix} \tau_{r}^{-1}
		\qquad\text{for}\qquad \tau_{r} := \begin{pmatrix} 0 & -1 \\ Np^{r} & 0 \end{pmatrix}. 
		\label{EasyMatCalc}
	\end{equation}
	The claimed equality (\ref{DualityIdentity}) follows easily from (\ref{CosetDecomp}) and (\ref{EasyMatCalc}), using
	the equalities of operators $(\cdot)\big|_{\sigma_{\delta}}=\langle \delta\rangle $ 
	and $U_p^* = w_{r} U_p w_{r}^{-1}$ on $S_k(\Gamma_{r})$ (see Proposition \ref{AtkinInterchange}).
\end{proof}

Perhaps the most essential ``classical" fact for our purposes is that 
the Hecke operator $U_p$ acting on spaces of  modular forms ``contracts" the $p$-level, as is made precise by the following:  

\begin{lemma}\label{UpContract}
	If $f\in S_k(\Gamma_r^i)$ then $U_p^{d}f$ is in the image of the canonical map
	$\iota_{\id}:S_k(\Gamma_{r-d}^i)\hookrightarrow S_k(\Gamma_r^i)$ for each integer $d\le r-i$.  In particular,
	$U_p^{r-i}f$ is in the image of $S_k(\Gamma_i)\hookrightarrow S_k(\Gamma_r^i)$.	
\end{lemma}

Certainly Lemma \ref{UpContract} is well-known 
(e.g. \cite{Tilouine}, \cite{HidaNotes}, \cite{Ohta1});
because of its importance in our subsequent applications, we sketch a proof (following
the proof of \cite[Lemma 1.2.10]{Ohta1}; see also \cite[\S 2]{HidaNotes}).
We note that $\Gamma_r\subseteq \Gamma_r^i$ for all $i\le r$, and the resulting inclusion
$S_k(\Gamma_r^i)\hookrightarrow S_k(\Gamma_r)$ has image 
consisting of forms on $\Gamma_r$ which are eigenvectors for the diamond operators
and whose associated character has conductor with $p$-part dividing $p^{i}$.

\begin{proof}[Proof of Lemma $\ref{UpContract}$]
	Fix $d$ with $0\le d\le r-i$ and let $\alpha:=\left(\begin{smallmatrix} 1 & 0 \\ 0 & p\end{smallmatrix}\right)$
	be as in Lemma \ref{MFtraceLem};
	then $\alpha^d$ is an element of the commeasurator of $\Gamma_{r}^i$ in $\SL_2(\Q)$.  Consider the following
	subgroups of $\Gamma_{r-d}^i$:
	\begin{align*}
		H&:= \Gamma_{r-d}^i \cap \alpha^{-d}\Gamma_{r}^i\alpha^d\\
		H'&:= \Gamma_{r-d}^i \cap \alpha^{-d}\Gamma_{r-d}^i\alpha^d,
	\end{align*}
	with each intersection taken inside of $\SL_2(\Q)$.  We claim that $H=H'$ inside $\Gamma_{r-d}^i$.
	Indeed, as $\Gamma_{r}^i\subseteq \Gamma_{r-d}^i$, the inclusion $H\subseteq H'$ is clear.
	For the reverse inclusion, if $\gamma:=\left(\begin{smallmatrix} * & * \\ x & *\end{smallmatrix}\right)\in \Gamma_{r-d}^i$,
	then we have $\alpha^{-d}\gamma\alpha^d = \left(\begin{smallmatrix} * & * \\ p^{-d}x & *\end{smallmatrix}\right)$,
	so if this lies in $\Gamma_{r-d}^i$ we must have $x\equiv 0\bmod p^r$ and hence $\gamma\in \Gamma_r^i$.
	We conclude that the coset spaces $H\backslash\Gamma_{r-d}^i$ and $H'\backslash\Gamma_{r-d}^i$
	are equal.  On the other hand, for {\em any} commeasurable subgroups $\Gamma,\Gamma'$ of a group $G$
	and any $g$ in the commeasurator of $\Gamma$ in $G$,
	an elementary computation shows that we have a bijection of coset spaces
	\begin{align*}
		(\Gamma'\cap g^{-1}\Gamma g )\backslash \Gamma' \simeq \Gamma\backslash\Gamma g\Gamma'
	\end{align*}
	via $(\Gamma'\cap g^{-1}\Gamma g)\gamma\mapsto \Gamma g\gamma$.
	Applying this with $g=\alpha^d$ in our situation and using the 
	decomposition
	\begin{equation*}
				\Gamma_{r-d}^i \alpha^d \Gamma_{r-d}^i = \coprod_{j=0}^{p^{d}-1} \Gamma_{r-d}^i\begin{pmatrix} 1 & j \\ 0 & p^{d}\end{pmatrix}
	\end{equation*}
	(see, e.g. \cite[proposition 3.36]{Shimura}), we deduce that we also have
	\begin{equation}\label{disjointHecke} 
		\Gamma_r^i \alpha^d \Gamma_{r-d}^i = \coprod_{j=0}^{p^{d}-1} \Gamma_r^i\begin{pmatrix} 1 & j \\ 0 & p^{d}\end{pmatrix}.
	\end{equation}
	Writing $U:S_k(\Gamma_r^i)\rightarrow S_k(\Gamma_{r-d}^i)$ for the ``Hecke operator" given by
	(e.g. \cite[\S3.4]{Ohta1})
	$\Gamma_r^i \alpha^d \Gamma_{r-d}^i$, an easy computation using \ref{disjointHecke} shows that the composite
	\begin{equation*}
		\xymatrix{
			S_k(\Gamma_r^i) \ar[r]^-{U} & S_{k}(\Gamma_{r-d}^i) \ar@{^{(}->}[r] & S_k(\Gamma_r^i) 
		}
	\end{equation*}
	coincides with $U_p^d$ on $q$-expansions.  By the $q$-expansion principle, we deduce that $U_p^d$ on $S_k(\Gamma_r^i)$
	indeed factors through the subspace $S_k(\Gamma_{r-d}^i)$, as desired.	
\end{proof}

For each integer $i$ and any character $\varepsilon:(\Z/Np^i\Z)^{\times}\rightarrow \Qbar^{\times}$, we denote by
$S_2(\Gamma_i,\varepsilon)$
the $\H_i$-stable subspace of weight 2 cusp forms for $\Gamma_i$ over $\Qbar$ 
on which the diamond operators act through $\varepsilon(\cdot)$. 
Define
\begin{equation}
	\o{V}_r := \bigoplus_{i=1}^r\bigoplus_{\varepsilon } S_2(\Gamma_i,\varepsilon)
	\label{VrDef}
\end{equation}
where the inner sum is over all Dirichlet characters defined modulo $Np^i$ whose $p$-parts are {\em primitive}
({\em i.e.} whose conductor has $p$-part exactly $p^i$).
We view $\o{V}_r$ as a $\Qbar$-subspace of $S_2(\Gamma_r)$ in the usual way
({\em i.e.} via the embeddings $\iota_{\id}$).  
We define $\o{V}_r^*$ as the direct sum (\ref{VrDef}), but viewed
as a subspace of $S_2(\Gamma_r)$ via the ``nonstandard" embeddings
$\iota_{\alpha^{r-i}}:S_2(\Gamma_i)\rightarrow S_2(\Gamma_r)$.

As in (\ref{TeichmullerIdempotent}), we write $f'$ for the idempotent of $\Z_p[\mu_{p-1}]$
corresponding to ``projection away from the trivial $\mu_{p-1}$-eigenspace."
From the formulae (\ref{GpRngIdem})
we see that $h':=(p-1)f'$ lies in the subring $\Z[\mu_{p-1}]$ of $\Z_p[\mu_{p-1}]$
and satisfies $h'^2 = (p-1)h'$.
We define endomorphisms of $S_2(\Gamma_r)$:
\begin{equation}
	U_r^*:=h'\circ (U_p^*)^{r+1} = (U_p^*)^{r+1}\circ h'\quad\text{and}\quad
	U_r:=h'\circ (U_p)^{r+1} = (U_p)^{r+1}\circ h'.
	\label{UrDefinition}
\end{equation}

\begin{corollary}\label{UpProjection}
	As subspaces of $S_2(\Gamma_r)$ we have $w_r(\o{V}_r^*)=\o{V}_r$.
	The space $\o{V}_r$ $($respectively $\o{V}_r^*$$)$ is naturally an $\H_r$ $($resp. $\H_r^*$$)$-stable 
	subspace of $S_2(\Gamma_r)$, and admits a canonical descent to $\Q$.
	Furthermore, the endomorphisms 
	$U_r$ and $U_r^*$ of $S_2(\Gamma_r)$ 
	factor through $\o{V}_r$ and $\o{V}_r^*$, respectively.
\end{corollary}

\begin{proof}
	The first assertion follows from the relation $w_r\circ \iota_{\alpha^{r-i}}=\iota_{\id}\circ w_i$
	as maps $S_2(\Gamma_i)\rightarrow S_2(\Gamma_r)$, together with the fact that $w_i$ on $S_2(\Gamma_i)$
	carries $S_2(\Gamma_i,\varepsilon)$ isomorphically onto $S_2(\Gamma_i,\varepsilon^{-1})$.
	The $\H_r$-stability of $\o{V}_r$ is clear as each of $S_2(\Gamma_i,\varepsilon)$ is an $\H_r$-stable subspace
	of $S_2(\Gamma_r)$; that $\o{V}_r^*$ is $\H_r^*$-stable follows from this and the comutation
	relation $T^* w_r = w_r T$ of Proposition \ref{AtkinInterchange}.  That $\o{V}_r$  and $\o{V}_r^*$
	admit canonical descents to $\Q$ is clear, as $\scrG_{\Q}$-conjugate Dirichlet characters have equal
	conductors.  The final assertion concerning the endomorphisms $U_r$ and $U_r^*$
	follows easily from Lemma \ref{UpContract}, 
	using the fact that $h':S_2(\Gamma_r)\rightarrow S_2(\Gamma_r)$
	has image contained in $\bigoplus_{i=1}^r S_k(\Gamma_r^i)$.
\end{proof}

\begin{definition}
	We denote by $V_r$ and $V_r^*$ the canonical descents to $\Q$ of $\o{V}_{r}$
	and $\o{V}_r^*$, respectively.  
\end{definition}

Following \cite[Chapter \Rmnum{3}, \S1]{MW-Iwasawa} and \cite[\S2]{Tilouine},
we recall the construction of certain ``good" quotient abelian varieties of $J_r$ whose
cotangent spaces are naturally identified with $V_r$ and $V_r^*$.  In what follows,
we will make frequent use of the following elementary result:  

\begin{lemma}\label{LieFactorization}
	Let $f:A\rightarrow B$ be a homomorphism of commutative group varieties 
	over a field $K$ of characteristic $0$.  Then:
\begin{enumerate}
	\item The formation of $\Lie$ and $\Cot$ commutes with the formation of kernels and images: the kernel $($respectively image$)$ of $\Lie(f)$ is canonically
	isomorphic to the Lie algebra of the kernel $($respectively image$)$ of $f$, and 
	similarly for cotangent spaces at the identity.	
	In particular, if $A$ is connected and $\Lie(f)=0$ $($respectively $\Cot(f)=0$$)$ 
	then $f=0$.\label{ExactnessOfLie}
	
	\item Let $i:B'\hookrightarrow B $ be a closed immersion of commutative
	group varieties over $K$ with $B'$ connected.  If $\Lie(f)$ factors through $\Lie(i)$
	then $f$ factors $($necessarily uniquely$)$ through $i$. 
	\label{InclOnLie}
	
	\item Let $j:A\twoheadrightarrow A''$ be a surjection of commutative 
	group varieties over $K$ with connected kernel.  If $\Cot(f)$
	factors through $\Cot(j)$ then $f$ factors $($necessarily uniquely$)$ through $j$.
	\label{LieFactorizationSurj} 
\end{enumerate}
\end{lemma}

\begin{proof}
	The key point is that because $K$ has characteristic zero, the functors $\Lie(\cdot)$
	and $\Cot(\cdot)$ on the category of commutative group schemes are {\em exact}.
	Indeed, since $\Lie(\cdot)$ is always left 
	exact, the exactness of $\Lie(\cdot)$ follows easily from the fact that any 
	quotient mapping 
	$G\twoheadrightarrow H$
	of group varieties in characteristic zero is smooth (as the kernel is
	a group variety over a field of characteristic zero and hence automatically smooth),
	so the induced map on Lie algebras is a surjection.
	By similar reasoning 
	one shows that the right exact $\Cot(\cdot)$ is likewise exact,
	and the first part of (\ref{ExactnessOfLie}) follows easily.  In particular,
	if $\Lie(f)$ is the zero map then $\Lie(\im(f))=0$, so $\im(f)$ is zero-dimensional.
	Since it is  also smooth, it must be \'etale.  Thus, if $A$ is connected,
	then $\im(f)$ is both connected and \'etale, whence it is a single point; by 
	evaluation of $f$ at the identity of $A$ we conclude that $f=0$.
	The assertions (\ref{InclOnLie}) and (\ref{LieFactorizationSurj})
	now follow immediately by using universal mapping properties.
\end{proof}

To proceed with the construction of good quotients of $J_r$, we now consider the diagrams of ``degeneracy mappings" of curves over $\Q$ for $i=1,2$
\begin{equation}
\addtocounter{equation}{1}
	\xymatrix{
		{X_r} \ar[r]^-{\pi} & {Y_r} \ar[r]^-{\pi_i} & {X_{r-1}}
	}
	\tag{$\arabic{section}.\arabic{subsection}.\arabic{equation}_i$}
	\label{DegeneracyDiag}
\end{equation}
where $\pi$ and $\pi_i$ are the maps induced by (\ref{XtoY}) and (\ref{Upcorr}),
respectively. These mappings covariantly (respectively contravariantly) induce
mappings on the associated Jacobians via Albanese (respectively Picard) functoriality.
Writing $JY_r:=\Pic^0_{Y_r/\Q}$ and setting $K_1^i:=JY_1$ for $i=1,2$
we inductively define abelian subvarieties $\iota_r^i:K_r^i\hookrightarrow JY_r$ and abelian variety quotients
$\alpha_r^i:J_r\twoheadrightarrow B_r^i$ as follows:
\begin{equation}
	\addtocounter{equation}{1}
	B_{r-1}^i:= J_{r-1}/\Pic^0(\pi)(K_{r-1}^i) 
	\qquad\text{and}\qquad
	K_{r}^i:=\ker(JY_r \xrightarrow{\alpha_{r-1}^i\circ\Alb(\pi_i)}  B_{r-1}^i)^0
	\tag{$\arabic{section}.\arabic{subsection}.\arabic{equation}_i$}
	\label{BrDef}
\end{equation}
for $r\ge 2$, $i=1,2$, with $\alpha_{r-1}^i$ and $\iota_r^i$ the obvious mappings;
here, $(\cdot)^0$ denotes the connected component of the identity of $(\cdot)$.
As in \cite[\S 3.2]{OhtaEichler}, we have modified Tilouine's construction
\cite[\S2]{Tilouine} so that kernel of $\alpha_r$ is connected; {\em i.e.} is an abelian
subvariety of $J_r$ ({\em cf.} Remark \ref{TilouineReln}).
Note that we have a commutative diagram of abelian varieties over $\Q$
for $i=1,2$
\begin{equation}
\addtocounter{equation}{1}
\begin{gathered}
	\xymatrix@C=50pt@R=26pt{
		& {J_{r-1}}\ar@{->>}[r]^-{\alpha_{r-1}^i}  & 
		{B_{r-1}^i} \ar@{=}[d] \\
		{K_r^i} \ar@{^{(}->}[r]^-{\iota_r^i}\ar@{=}[d] & 
		{JY_r} \ar[r]^-{\alpha_{r-1}^i\circ \Alb(\pi_i)} \ar[d]|-{\Pic^0(\pi)}\ar[u]|-{\Alb(\pi_i)} & 
		B_{r-1}^i \\
		K_r^i \ar[r]_-{\Pic^0(\pi)\circ \iota_r} & {J_r} \ar@{->>}[r]_-{\alpha_r^i} & {B_r^i}
	}
\end{gathered}
\tag{$\arabic{section}.\arabic{subsection}.\arabic{equation}_i$}
\label{BrDefiningDiag}
\end{equation}
with bottom two horizontal rows that are complexes.  

\begin{warning}\label{GoodQuoWarning}
While the bottom row of (\ref{BrDefiningDiag}) is exact in the middle by definition of $\alpha_r^i$,
the central row is {\em not} exact in the middle: it follows from
the fact that $\Alb(\pi_i)\circ\Pic^0(\pi_i)$ is multiplication by $p$
on $J_{r-1}$ that the component group of the kernel of
$\alpha_{r-1}^i\circ\Alb(\pi_i):JY_r\rightarrow  B_{r-1}^i$ is nontrivial
with order divisible by $p$.  
Moreover, there is no mapping $B_{r-1}^i\rightarrow B_r^i$
which makes the diagram (\ref{BrDefiningDiag}) commute.
\end{warning}

In order to be consistent with the literature, we adopt the following convention:
\begin{definition}\label{BalphDef}
	We set $B_r:=B_r^2$ and $B_r^*:=B_r^1$, with $B_r^i$ defined inductively by (\ref{BrDef}).
 	We likewise set $\alpha_r:=\alpha_r^2$ and $\alpha_r^*:=\alpha_r^1$.
\end{definition} 

\begin{remark}\label{TilouineReln}
	We briefly comment on the relation between our quotient $B_r$
	and the ``good" quotients of $J_r$ considered by Ohta \cite{OhtaEichler},
	by Mazur-Wiles \cite{MW-Iwasawa}, and by Tilouine \cite{Tilouine}.
	Recall \cite[\S2]{Tilouine} that Tilouine constructs\footnote{The notation Tilouine 
	uses for his quotient is the same as the notation we have used for our (slightly modified)
	quotient.  To avoid conflict, we have therefore chosen to 
	alter his notation.} an abelian variety quotient
	$\alpha_r':J_r\twoheadrightarrow B_r'$ via an inductive
	procedure nearly identical to the one used to define $B_r=B_r^1$:
	one sets $K_1':=JY_1$, and for $r\ge 2$ defines
\begin{equation*}
		B_{r-1}':= J_{r-1}/\Pic^0(\pi)(K_{r-1}') 
		\qquad\text{and}\qquad
		K_{r}':=\ker(JY_r \xrightarrow{\alpha_{r-1}'\circ\Alb(\pi_2)}  B_{r-1}').
\end{equation*}
	Using the fact that the 
	formation of images and identity components commutes, 
	one shows via a straightforward induction argument that 
	$\alpha_r:J_r\twoheadrightarrow B_r$
	identifies $B_r$ with $J_r/(\ker\alpha_r')^0$; in particular,
	our $B_r$ is the same as Ohta's \cite[\S3.2]{OhtaEichler}
	and Tilouine's quotient $\alpha_r':J_r\rightarrow B_r'$ uniquely
	factors through $\alpha_r$ via an isogeny $B_r\twoheadrightarrow B_r'$
	which has degree divisible by $p$ by Warning \ref{GoodQuoWarning}.
	Due to this fact, it is {\em essential} for our purposes to work with $B_r$ rather than $B_r'$.
	Of course, following \cite[3.2.1]{OhtaEichler}, we could have simply 
	{\em defined} $B_r$ as $J_r/(\ker\alpha_r')^0$,
	but we feel that the construction we have given is more natural.
	
	On the other hand, we remark that $B_r$
	is naturally a quotient of the ``good" quotient $J_r\twoheadrightarrow A_r$ constructed
	by Mazur-Wiles in \cite[Chapter \Rmnum{3}, \S1]{MW-Iwasawa},
	and the kernel of the corresponding surjective homomorphism 
	$A_r\twoheadrightarrow B_r$
	is isogenous to $J_0\times J_0$.
\end{remark}

\begin{proposition}\label{BrCotIden}
	Over $F:=\Q(\mu_{Np^r})$, the automorphism $w_r$ of ${J_r}_F$ induces
	an isomorphism of quotients ${B_r}_{F}\simeq {B_r^*}_F$.	
	The abelian variety $B_r$ $($respectively $B_r^*$$)$ is the unique quotient of $J_r$ by a $\Q$-rational 
	abelian subvariety with the property that
	the induced map on cotangent spaces 
	\begin{equation*}
		\xymatrix{
			{\Cot(B_r)} \ar@{^{(}->}[r]_-{\Cot(\alpha_r)} &  
			{\Cot(J_r)\simeq S_2(\Gamma_r;\Q)}
			}
			\quad\text{respectively}\quad
					\xymatrix{
			{\Cot(B_r^*)} \ar@{^{(}->}[r]_-{\Cot(\alpha_r^*)} &  
			{\Cot(J_r)\simeq S_2(\Gamma_r;\Q)}
			}
	\end{equation*}
	has image precisely $V_r$ $($respectively $V_r^*$$)$. 
	In particular, there are canonical actions
	of the Hecke algebras\footnote{We must warn the reader that
	Tilouine \cite{Tilouine} writes $\H_r(\Z)$ for the
	$\Z$-subalgebra of $\End(J_r)$ generated by the Hecke operators
	acting via the $(\cdot)^*$-action ({\em i.e.} by ``Picard" functoriality)
	whereas our $\H_r(\Z)$ is defined using the $(\cdot)_*$-action.
	This discrepancy is due primarily to the fact that Tilouine 
	identifies {\em tangent} spaces of modular abelian varieties
	with spaces of modular forms, rather than cotangent spaces as is our convention.  
	Our notation for regarding Hecke algebras
	as sub-algebras of $\End(J_r)$ agrees with that of Mazur-Wiles
	\cite[Chapter \Rmnum{2}, \S5]{MW-Iwasawa}, \cite[\S7]{MW-Hida} and 
	Ohta \cite[3.1.5]{OhtaEichler}.
	}
	$\H_r(\Z)$ on $B_r$ and $\H_r^*(\Z)$ on $B_r^*$ for which $\alpha_r$ and $\alpha_r^*$
	are equivariant.  
\end{proposition}

\begin{proof}

	By the construction of $B_r^i$ and Proposition \ref{ALinv},
	the automorphism $w_{r}$ of $J_{r,F}$ carries $\ker(\alpha_r)$
	to $\ker(\alpha_r^*)$ and induces an isomorphsm $B_{r,F} \simeq B_{r,F}^*$ over $F$
	that intertwines the action of $\H_r$ on $B_r$ with $\H_r^*$ on $B_r^*$.
	The isogeny $B_r\twoheadrightarrow B_r'$ of Remark \ref{TilouineReln} induces
	an isomorphism on cotangent spaces, compatibly with the inclusions into
	$\Cot(J_r)$.  Thus, the claimed identification of the image of
	$\Cot(B_r)$ with $V_r$ follows from \cite[Proposition 2.1]{Tilouine}
	(using \cite[Definition 2.1]{Tilouine}).  The claimed uniqueness
	of $J_r\twoheadrightarrow B_r$ follows easily from Lemma \ref{LieFactorization} 
	(\ref{LieFactorizationSurj}). Similarly, since the subspace $V_r$ 
	of $S_2(\Gamma_r)$ is stable under $\H_r$, we conclude from
	Lemma \ref{LieFactorization} (\ref{LieFactorizationSurj}) that for any 
	$T\in \H_r(\Z)$, the induced morphism $J_r\xrightarrow{T} J_r\twoheadrightarrow B_r$
	factors through $\alpha_r$, and hence that $\H_r(\Z)$ acts on $B_r$
	compatibly (via $\alpha_r$) with its action on $J_r$.
\end{proof}

\begin{lemma}\label{Btower}
	There exist unique morphisms $B_r^*\leftrightarrows B_{r-1}^*$
	of abelian varieties over $\Q$ making
		\begin{equation*}
			\xymatrix{
				{J_{r}} \ar[r]^-{\alpha_r^*} \ar[d]_-{\Alb(\ps)} &{B_r^*} \ar[d] \\
				{J_{r-1}} \ar[r]_-{\alpha_{r-1}^*} & {B_{r-1}^*}
			}\qquad\raisebox{-18pt}{and}\qquad
			\xymatrix{
				{J_{r}} \ar[r]^-{\alpha_r^*} &{B_r^*}  \\
				{J_{r-1}} \ar[u]^-{\Pic^0(\pr)} \ar[r]_-{\alpha_{r-1}^*} & {B_{r-1}^*}\ar[u]
			}
		\end{equation*}
	commute; these maps are moreover $\H_r^*(\Z)$-equivariant.
	By a slight abuse of notation, we will simply write $\Alb(\ps)$ and $\Pic^0(\pr)$
	for the induced maps on $B_r^*$ and $B_{r-1}^*$, respectively.
\end{lemma}

\begin{proof}
	Under the canonical identification of $\Cot(J_r)\otimes_{\Q}\o{\Q}$ with $S_2(\Gamma_r)$,
	the mapping on cotangent spaces induced by $\Alb(\ps)$ (respectively
	$\Pic^0(\pr)$) coincides with $\iota_{\alpha}:S_2(\Gamma_{r-1})\rightarrow S_2(\Gamma_r)$
	(respectively $\tr_{r,r-1}:S_2(\Gamma_r)\rightarrow S_2(\Gamma_{r-1})$).
	As the kernel of $\alpha_r^*:J_r\twoheadrightarrow B_r^*$ is connected by definition,
	thanks to Lemma \ref{LieFactorization} (\ref{LieFactorizationSurj}) it suffices
	to prove that $\iota_{\alpha}$ (respectively $\tr_{r,r-1}$)
	carries $V_{r-1}^*$ to $V_{r}^*$ (respectively $V_{r}^*$ to $V_{r-1}^*$).
	On one hand, the composite 
	$\iota_{\alpha}\circ \iota_{\alpha^{r-1-i}}:S_2(\Gamma_i,\varepsilon)\rightarrow S_2(\Gamma_r)$
	coincides with the embedding $\iota_{\alpha^{r-i}}$, and it follows immediately from the
	definition of $V_r^*$ that $\iota_{\alpha}$ carries $V_{r-1}^*$ into $V_r^*$.	
	On the other hand, an easy calculation using (\ref{DualityIdentity}) shows that 
	one has equalities of maps $S_2(\Gamma_i,\varepsilon)\rightarrow S_2(\Gamma_r)$
	\begin{equation*}
		\iota_{\alpha}\circ \tr_{r,r-1}\circ \iota_{\alpha^{(r-i)}} = \begin{cases}
			\iota_{\alpha^{(r-i)}}pU_p^* & \text{if}\ i< r \\
			0 & \text{if}\ i=r
		\end{cases}.
	\end{equation*}
	Thus,  
	the image of $\iota_{\alpha}\circ\tr_{r,r-1}:V_r^*\rightarrow S_2(\Gamma_r)$ 
	is contained in the image of $\iota_{\alpha}:V_{r-1}^*\rightarrow S_2(\Gamma_r)$;
	since $\iota_{\alpha}$ is injective, we conclude that 
	the image of $\tr_{r,r-1}:V_r^*\rightarrow S_2(\Gamma_{r-1})$ is contained in $V_{r-1}^*$
	as desired.
\end{proof}

\begin{proposition}\label{GoodRedn}
	The abelian varieties $B_r$ and $B_r^*$ acquire good reduction over $\Q_p(\mu_{p^r})$.
\end{proposition}

\begin{proof}
	See \cite[Chap \Rmnum{3}, \S2, Proposition 2]{MW-Iwasawa} 
	and {\em cf.} \cite[\S9, Lemma 9]{HidaGalois}.	
\end{proof}

As in \S\ref{OrdStruct}, we denote by ${e^*}':=f'e^*\in \H^*$ and $e':=f'e\in \H$ 
the sub-idempotents of $e^*$ and $e$, respectively, 
corresponding to projection away from the trivial eigenspace of $\mu_{p-1}$.

\begin{proposition}\label{GoodRednProp}	
	The maps $\alpha_r$ and $\alpha_r^*$
	induce isomorphisms of $p$-divisible groups over $\Q$
	\begin{equation}
			{e^*}'J_r[p^{\infty}] \simeq {e^*}'B_r^*[p^{\infty}]\quad\text{and}\quad
			{e}'J_r[p^{\infty}] \simeq {e}'B_r[p^{\infty}],
	\label{OrdBTisoms}
	\end{equation}
	respectively, that are $\H^*$ $($respectively $\H$$)$ equivariant 
	and compatible with change in $r$ via $\Alb(\ps)$ and $\Pic^0(\pr)$
	$($respectively $\Alb(\pr)$ and $\Pic^0(\ps)$$)$.
\end{proposition}

We view the maps (\ref{UrDefinition})
as endomorphisms of $J_r$ in the obvious way, and again write
$U_r^*$ and $U_r$ for the induced endomorphism of
$B_r^*$ and $B_r$, respectively.  To prove Proposition \ref{GoodRednProp},
we need the following geometric incarnation of
Corollary \ref{UpProjection}:

\begin{lemma}\label{UFactorDiagLem}
		There exists a unique $\H_r^*(\Z)$ $($respectively $\H_r(\Z)$$)$-equivariant 
		map $W_r^*:B_r^*\rightarrow J_r$ $($respectively $W_r:B_r\rightarrow J_r$$)$
		of abelian varieties over $\Q$ such that the diagram
	\begin{equation}
	\begin{gathered}
		\xymatrix@C=30pt@R=35pt{
			 {J_r}\ar[d]_-{U_r^*} 
			\ar@{->>}[r]^-{\alpha_r^*} & {B_r^*} \ar[dl]|-{W_r^*} \ar[d]^-{U_r^*} \\
		 {J_r}\ar@{->>}[r]_-{\alpha_r^*} & {B_r^*}
		}
		\quad\raisebox{-24pt}{respectively}\quad
		\xymatrix@C=30pt@R=35pt{
			 {J_r}\ar[d]_-{U_r} 
			\ar@{->>}[r]^-{\alpha_r} & {B_r} \ar[dl]|-{W_r} \ar[d]^-{U_r} \\
		 {J_r}\ar@{->>}[r]_-{\alpha_r} & {B_r}
		}\label{UFactorDiag}
	\end{gathered}
	\end{equation}
	commutes.  
\end{lemma}

\begin{proof}
	Consider the endomorphism of $J_r$ given by $U_r$.
	Due to Corollary \ref{UpProjection},
	the induced mapping on cotangent spaces factors through the inclusion
	$\Cot(B_r)\hookrightarrow \Cot(J_r)$. Since the kernel of the quotient
	mapping $\alpha_r:J_r\twoheadrightarrow B_r$ giving rise to this inclusion is connected,
	we conclude from Lemma \ref{LieFactorization} (\ref{LieFactorizationSurj})
	that $U_r$ factors uniquely through $\alpha_r$
	via an $\H_r$-equivariant morphism $W_r:B_r\rightarrow J_r$.  
	The corresponding statements for $B_r^*$ are proved similarly.
\end{proof}

\begin{proof}[Proof of Proposition $\ref{GoodRednProp}$]
	From (\ref{UFactorDiag}) we get commutative diagrams of $p$-divisible groups 
	over $\Q$
	\begin{equation}
	\begin{gathered}
		\xymatrix{
			{e^*}'{J_r}[p^{\infty}]\ar[d]_-{U_r^*}^-{\simeq} \ar[r]^-{\alpha_r^*} & 
			{e^*}'{B_r^*}[p^{\infty}] \ar[dl]|-{W_r^*} \ar[d]^-{U_r^*}_-{\simeq} \\
			{e^*}'{J_r}[p^{\infty}]\ar[r]_-{\alpha_r^*} & 
			{e^*}'{B_r^*}[p^{\infty}]
		}
		\quad\raisebox{-24pt}{and}\quad
		\xymatrix{
			e'{J_r}[p^{\infty}]\ar[d]_-{U_r}^-{\simeq} \ar[r]^-{\alpha_r} & 
			e'{B_r}[p^{\infty}] \ar[dl]|-{W_r} \ar[d]^-{U_r}_-{\simeq} \\
			e'{J_r}[p^{\infty}]\ar[r]_-{\alpha_r} & 
			e'{B_r}[p^{\infty}]
		}
		\label{UFactorDiagpDiv}
	\end{gathered}
	\end{equation}
	in which all vertical arrows are isomorphisms due to the very definition of the
	idempotents ${e^*}'$ and $e'$.  An easy diagram chase then shows that {\em all}
	arrows must be isomorphisms.
\end{proof}

We will write $\B_r$, $\B^*_r$, and $\J_r$, respectively, for the N\'eron models of
the base changes $(B_r)_{K_r}$, $(B_r^*)_{K_r}$ and $(J_r)_{K_r}$
over $T_r:=\Spec(R_r)$;  due to Proposition \ref{GoodRednProp}, both $\B_r$ and $\B_r^*$ are abelian
schemes over $T_r$.  
By the N\'eron mapping property, there are canonical actions of $\H_r(\Z)$ on $\B_r$, $\J_r$
and of $\H_r^*(\Z)$ on $\B_r^*$, $\J_r$ over $R_r$ extending the actions on generic fibers
as well as ``semilinear" actions of $\Gamma$
over the $\Gamma$-action on $R_r$ ({\em cf.} (\ref{GammaAction})).
For each $r$, the N\'eron mapping property further provides diagrams 
\begin{equation}
	\begin{gathered}
	\xymatrix{
		 {\J_r \times_{T_r} T_{r+1}}\ar@<-1ex>[d]_{\Pic^0(\pr)} \ar[r]^-{\alpha_r^*} 
		& {\B_r^* \times_{T_r} T_{r+1}}\ar@<1ex>[d]^{\Pic^0(\pr)} \\
		{\J_{r+1}} \ar[r]_-{\alpha_{r+1}^*} \ar@<-1ex>[u]_-{\Alb(\ps)} & \ar@<1ex>[u]^-{\Alb(\ps)} {\B_{r+1}^*} 
	}
	\quad\raisebox{-24pt}{respectively}\quad
	\xymatrix{
		 {\J_r \times_{T_r} T_{r+1}}\ar@<-1ex>[d]_{\Pic^0(\ps)} \ar[r]^-{\alpha_r} 
		& {\B_r \times_{T_r} T_{r+1}}\ar@<1ex>[d]^{\Pic^0(\ps)} \\
		{\J_{r+1}} \ar[r]_-{\alpha_{r+1}}\ar@<-1ex>[u]_-{\Alb(\pr)} & \ar@<1ex>[u]^-{\Alb(\pr)} {\B_{r+1}} 
	}
	\label{Nermaps}
	\end{gathered}
\end{equation}
of smooth commutative group schemes over $T_{r+1}$ in which the inner and outer rectangles commute, and
all maps are $\H_{r+1}^*(\Z)$ (respectively $\H_{r+1}(\Z)$) and $\Gamma$ equivariant.

\begin{definition}\label{ordpdivdefn}
	We define $\G_r:={e^*}'\left(\B_r^*[p^{\infty}]\right)$ and 
	we write $\G_r':=\G_r^{\vee}$ for its Cartier dual,
	each of which is canonically an object of $\pdiv_{R_r}^{\Gamma}$.
	For each $r\ge s$, noting that $U_p^*$ is an automorphism of $\G_r$,
	we obtain from (\ref{Nermaps}) canonical morphisms
	\begin{equation}
		\xymatrix@C=45pt{
		{\rho_{r,s}:\G_{s}\times_{T_{s}} T_{r}} \ar[r]^-{\Pic^0(\pr)^{r-s}} &  {\G_{r}}
		}
		\qquad\text{and}\qquad
		\xymatrix@C=70pt{
		{\rho_{r,s}' : \G_{s}'\times_{T_{s}} T_{r}} \ar[r]^-{{({U_p^*}^{-1}\Alb(\ps))^{\vee}}^{r-s}} & {\G_{r}'}
		}\label{pdivTowers}
	\end{equation}
	in $\pdiv_{R_r}^{\Gamma}$, where $(\cdot)^{i}$ denotes the $i$-fold composition, formed in the obvious manner.
	In this way, we get towers of $p$-divisible groups 
	$\{\G_r,\rho_{r,s}\}$ and $\{\G_r',\rho_{r,s}'\}$;
	we will write $G_r$ and $G_r'$ for the unique descents of the generic fibers of $\G_r$
	and $\G_r'$ to $\Q_p$, respectively.\footnote{Of course, $G_r'=G_r^{\vee}$.  
	Our non-standard notation $\G_r'$ for the Cartier  
	dual of $\G_r$ is preferrable, due to the fact that $\rho_{r,s}'$ is 
	{\em not} simply the dual of $\rho_{r,s}$; indeed, these two mappings go in opposite
	directions!}
	We let $T^*\in \H_r^*$ act on $\G_r$ 
	through the action of $\H_r^*(\Z)$ on $\B_r^*$,
	and on $\G_r'=\G_r^{\vee}$ by duality ({\em i.e.} as $(T^*)^{\vee}$).  
	The maps (\ref{pdivTowers}) are then $\H_r^*$-equivariant.
\end{definition}

By Proposition \ref{GoodRednProp}, $G_r$ 
is canonically isomorphic to ${e^*}'J_r[p^{\infty}]$, compatibly with the action of $\H_r^*$.
Since $J_r$ is a Jacobian---hence principally polarized---one might expect that $\G_r$
is isomorphic to its dual in $\pdiv_{R_r}^{\Gamma}$.  However, this is {\em not quite} the case
as the canonical isomorphism $J_r\simeq J_r^{\vee}$ intertwines the actions of $\H_r$
and $\H_r^*$, thus interchanging the idempotents ${e^*}'$ and $e'$.  To describe
the precise relationship between $\G_r^{\vee}$ and $\G_r$, we proceed as follows.
For each $\gamma\in \Gal(K_r'/K_0)\simeq \Gamma\times \Gal(K_0'/K_0)$, 
let us write $\phi_{\gamma}: {G_r}_{K_r'}\xrightarrow{\simeq} \gamma^*({G_r}_{K_r'})$
for the descent data isomorphisms encoding the unique $\Q_p=K_0$-descent of ${G_r}_{K_r'}$ furnished by $G_r$.
We ``twist" this descent data by the $\Aut_{\Q_p}(G_r)$-valued character $\langle \chi\rangle\langle a\rangle_N$
of $\Gal(K_{\infty}'/K_0)$:
explicitly, for $\gamma\in \Gal(K_{r}'/K_0)$ we set 
$\psi_{\gamma}:= \phi_{\gamma}\circ \langle \chi(\gamma)\rangle\langle a(\gamma)\rangle_N$
and note that since $\langle \chi(\gamma)\rangle\langle a(\gamma)\rangle_N$ 
is defined over $\Q_p$, the map $\gamma\rightsquigarrow \psi_{\gamma}$
really does satisfy the cocycle condition.  
We denote by $G_r(\langle \chi\rangle\langle a\rangle_N)$ the unique $p$-divisible group over $\Q_p$
corresponding to this twisted descent datum.
Since the diamond operators commute with the Hecke operators, there is a canonical
induced action of $\H_r^*$ on $G_r(\langle \chi\rangle\langle a\rangle_N)$.  
By construction, there is a canonical 
$K_r'$-isomorphism 
$G_r(\langle \chi\rangle\langle a\rangle_N)_{K_r'}\simeq {G_r}_{K_r'}$. Since
$G_r$ acquires good reduction over $K_r$ and the $\scrG_{K_r}$-representation
afforded by the Tate module of $G_r(\langle \chi\rangle\langle a\rangle_N)$
is the twist of $T_pG_r$ by the {\em unramified} character $\langle a\rangle_N$,
we conclude that $G_r(\langle \chi\rangle\langle a\rangle_N)$ also acquires 
good reduction over $K_r$, and we denote the resulting object of $\pdiv_{R_r}^{\Gamma}$
by $\G_r(\langle \chi\rangle\langle a\rangle_N)$.

\begin{proposition}\label{GdualTwist}
	There is a natural $\H_r^*$-equivariant isomorphism of $p$-divisible groups over $\Q_p$
	\begin{equation}
		G_r' \simeq G_r(\langle \chi\rangle \langle a\rangle_N)
		\label{GrprimeGr}
	\end{equation}
	which uniquely extends to an isomorphism of the corresponding objects in $\pdiv_{R_r}^{\Gamma}$
	and is compatible with change in $r$ using $\rho_{r,s}'$ on $G_r'$ and $\rho_{r,s}$ on $G_r$.
\end{proposition}

\begin{proof}
	Let $\varphi_r: J_r\rightarrow J_r^{\vee}$ be the canonical principal polarization over $\Q_p$;
	one then has the relation $\varphi_r\circ T = (T^*)^{\vee}\circ \varphi_r$
	for each $T\in \H_r(\Z)$.  On the other hand, the $K_r'$-automorphism 
	$w_r: {J_r}_{K_r'}\rightarrow {J_r}_{K_r'}$ intertwines $T\in \H_r(\Z)$ with $T^*\in \H_r^*(\Z)$.
	Thus, the $K_r'$-morphism 
	\begin{equation*}
		\xymatrix{
			{\psi_r:{J_r}_{K_r'}^{\vee}} \ar[r]^-{({U_p^*}^{r})^{\vee}} & 
			{{J_r}_{K_r'}^{\vee}} \ar[r]^-{\varphi_r^{-1}}_{\simeq} & {J_r}_{K_r'} 
			\ar[r]^-{w_r}_{\simeq} & {{J_r}_{K_r'}} 
		}
	\end{equation*}
	is $\H_r^*(\Z)$-equivariant. 
	Passing to the induced map on $p$-divisible groups and applying ${e^*}'$, we
	obtain from Proposition \ref{GoodRednProp} an $\H_r^*$-equivariant isomorphism 
	of $p$-divisible groups $\psi_r: {G_r'}_{K_r'} \simeq {G_r}_{K_r'}$. As
	\begin{equation*}
		\xymatrix@C=35pt{
			{{J_r}_{K_r'}} \ar[r]^-{\langle \chi(\gamma)\rangle \langle a\rangle_N w_r}\ar[d]_-{1\times \gamma} & 
			{{J_r}_{K_r'}}\ar[d]^-{1\times \gamma} \\
			{({J_r}_{K_r'})_{\gamma}} \ar[r]_-{\gamma^*(w_r)} & {({J_r}_{K_r'})_{\gamma}}	
		}
	\end{equation*}
	commutes for all $\gamma\in \Gal(K_r'/K_0)$ by Proposition \ref{ALinv},
	the $K_r'$-isomorphism $\psi_r$ uniquely descends to an
	$\H_r^*$-equivariant isomorphism (\ref{GrprimeGr})
	of $p$-divisible groups over $\Q_p$.
	By Tate's Theorem, this identification 
	uniquely extends to an isomorphism of the corresponding objects in $\pdiv_{R_r}^{\Gamma}$.
	The asserted compatibility with change in $r$ boils down to the commutativity of the diagrams
	\begin{equation*}
	\begin{gathered}
		\xymatrix{
 {{e^*}'J_s[p^{\infty}]^{\vee}} \ar[r]^-{({U_p^*}^{s})^{\vee}} \ar[d]_-{{({U_p^*}^{-1}\Alb(\ps))^{\vee}}^{r-s}}&
 {{e^*}'J_s[p^{\infty}]^{\vee}} \ar[d]^-{{\Alb(\ps)^{\vee}}^{r-s}} \\
{{e^*}'J_r[p^{\infty}]^{\vee}} \ar[r]_-{({U_p^*}^{r})^{\vee}} & {{e^*}'J_r[p^{\infty}]^{\vee}} \\
		}
		\quad\raisebox{-22pt}{and}\quad
		\xymatrix{
			 {{J_s}_{K_r'}^{\vee}} \ar[r]^-{\varphi_s^{-1}} \ar[d]_-{{\Alb(\ps)^{\vee}}^{r-s}}
			  & {J_s}_{K_r'} \ar[r]^-{w_s} \ar[d]|-{\Pic^0(\ps)^{r-s}} & {{J_s}_{K_r'}}\ar[d]^-{\Pic^0(\pr)^{r-s}}\\
			{{J_r}_{K_r'}^{\vee}} \ar[r]_-{\varphi_r^{-1}}
			 & {J_r}_{K_r'} \ar[r]_-{w_r} & {{J_r}_{K_r'}} 
		}
	\end{gathered}
	\end{equation*}
	for all $s\le r$.  The commutativity of the first diagram is clear, while that of the second follows
	from Proposition \ref{ALinv} and the fact that 
	for {\em any} finite morphism $f:Y\rightarrow X$ of smooth curves over a field $K$, 
	one has $\varphi_Y\circ \Pic^0(f)=\Alb(f)^{\vee}\circ \varphi_X$, where 
	$\varphi_{\star}:J_{\star}\rightarrow J_{\star}^{\vee}$ is the canonical
	principal polarization on Jacobians for $\star=X,Y$ (see, for example, the proof of Lemma 5.5 in \cite{CaisNeron}).
\end{proof}

We now wish to relate the special fiber of $\G_r$ to the $p$-divisible
group $\Sigma_r:={e^*}'\Pic^0_{\nor{\o{\X}}_r/\F_p}[p^{\infty}]$
of Definition \ref{pDivGpSpecial}.  In order to do this, we proceed as follows.
Since $\X_r$ is regular, and proper flat over $R_r$ with (geometrically) reduced special fiber,
$\Pic^0_{\X_r/R_r}$ is a smooth $R_r$-scheme by \S8.4 Proposition 2 and \S9.4 Theorem 2 of \cite{BLR}.
By the N\'eron mapping property, we thus have a natural mapping $\Pic^0_{\X_r/R_r}\rightarrow \J_r^0$
that recovers the canonical identification on generic fibers, and is in fact an isomorphism
by \cite[\S9.7, Theorem 1]{BLR}.  Composing with the map $\alpha_r^*:\J_r\rightarrow \B_r^*$ 
and passing to special fibers
yields a homomorphism of smooth commutative algebraic groups over $\F_p$
\begin{equation}
	\xymatrix{
 		{\Pic^0_{\o{\X}_r/\F_p}} \ar[r]^-{\simeq} & {\o{\J}_r^0} \ar[r] & {\o{\B}^*_r}
	}\label{PicToB}
\end{equation}
Due to \cite[\S9.3, Corollary 11]{BLR},
the normalization map $\nor{\o{\X}}_r\rightarrow \o{\X}$ induces a surjective homomorphism 
$\Pic^0_{\o{\X}_r/\F_p}\rightarrow {\Pic^0_{\nor{\o{\X}}_r/\F_p}}$
with kernel that is a smooth, connected {\em linear} algebraic group over $\F_p$.
As any homomorphism from an affine group variety to an abelian variety is zero, 
we conclude that (\ref{PicToB}) uniquely factors through this quotient, and we obtain
a natural map of abelian varieties:
\begin{equation}
	\xymatrix{
		{\Pic^0_{\nor{\o{\X}}_r/\F_p}} \ar[r] & {\o{\B}_r^*}
		}\label{AbVarMaps}
\end{equation}
that is necessarily equivariant for the actions of $\H_r^*(\Z)$ and $\Gamma$.
As in \ref{AlbPicIncl},
we write
$j_r^{\star}:=\Pic^0_{I_r^{\star}/\F_p}$ the Jacobian of $I_r^{\star}$ for $\star=0,\infty$.
The following Proposition relates the special fiber of $\G_r$ to the 
$p$-divisible group $\Sigma_r$ of Definition \ref{pDivGpSpecial}, and thus
enables an explicit description of the special fiber
of $\G_r$ in terms of the $p$-divisible groups of $j_r^{\star}$
({\em cf}. \S3 and \S4, Proposition 1 of \cite{MW-Hida} and pgs. 267--274 of \cite{MW-Iwasawa}).

\begin{proposition}\label{SpecialFiberOrdinary}
	The mapping $(\ref{AbVarMaps})$ induces an isomorphism
	of $p$-divisible groups over $\F_p$
	\begin{equation}
		\o{\G}_r := {{e^*}'\o{\B}_r^*[p^{\infty}]} \simeq 
		{{e^*}'\Pic^0_{\nor{\o{\X}}_r/\F_p}[p^{\infty}]}=:\Sigma_r
		\label{OnTheNose}
	\end{equation}
	that is $\H_r^*$ and $\Gamma$-equivariant and compatible with change in $r$ 
	via the maps $\pr_{r,s}$ on $\o{\G}_r$ and the maps $\Pic^0(\pr)^{r-s}$ on $\Sigma_r$.
	In particular, 
	$\G_r/R_r$ is an ordinary $p$-divisible group, and for each $r$ there is a canonical exact sequence,
	compatible with change in $r$ via $\pr_{r,s}$ on $\o{\G}_r$ and $\Pic^0(\pr)^{r-s}$ on $j_r^{\star}[p^{\infty}]$
	\begin{equation}
 		\xymatrix@C=45pt{
			0 \ar[r] & {f'j_r^{0}[p^{\infty}]^{\mult}} \ar[r]^-{\Alb(i_r^{0})\circ V^r} &
			{\o{\G}_r} \ar[r]^-{\Pic^0(i_r^{\infty})} & {f'j_r^{\infty}[p^{\infty}]^{\et}} \ar[r] & 0  
		}\label{GrSpecialExact}
	\end{equation}
	where $i_r^{\star}:I_r^{\star}\hookrightarrow \nor{\o{\X}}_r$ are the canonical closed immersions
	for $\star=0,\infty$.  Moreover, $(\ref{GrSpecialExact})$ is compatible with the actions of
	$\H^*$ and $\Gamma$, with $U_p^*$ $($respectively $\gamma\in \Gamma$$)$
	acting on $f'j_r^{0}[p^{\infty}]^{\mult}$ as $\langle p\rangle_N V$ 
	$($respectively $\langle \chi(\gamma)\rangle^{-1}$$)$ 
	and on $f'j_r^{\infty}[p^{\infty}]^{\et}$ as $F$ $($respectively $\id$$)$.
\end{proposition}

\begin{proof}
	The diagram (\ref{UFactorDiag}) induces a corresponding diagram of
	N\'eron models over $R_r$ and hence of special fibers over $\F_p$.
	Arguing as above, we obtain a commutative diagram of abelian
	varieties
	\begin{equation}
	\begin{gathered}
		\xymatrix@C=30pt@R=35pt{
			{\Pic^0_{\nor{\o{\X}}_r/\F_p}}\ar[d]_-{U_r^*} 
			\ar[r]^-{\o{\alpha}^*_r} & {\o{\B}_r^*} \ar[dl]^-{W_r^*} \ar[d]^-{U_r^*} \\
		 	{\Pic^0_{\nor{\o{\X}}_r/\F_p}}\ar[r]_-{\o{\alpha}^*_r} & {\o{\B}_r^*}
		}\label{UFactorDiagmodp}
	\end{gathered}
	\end{equation}
	 over $\F_p$.  The proof of \ref{GoodRednProp} now goes through {\em mutatis mutandis} to give the claimed
	 isomorphism (\ref{OnTheNose}).  The rest follows immediately from Proposition \ref{MWSharpening}.
\end{proof}

\subsection{Ordinary families of Dieudonn\'e modules}\label{OrdDieuSection}

Let $\{\G_r/R_r\}_{r\ge 1}$ be the tower of $p$-divisible groups given by Definition \ref{ordpdivdefn}.  
From the canonical morphisms $\rho_{r,s}: \G_{s}\times_{T_{s}} T_r\rightarrow \G_{r}$ we obtain
a map on special fibers $\o{\G}_{s}\rightarrow \o{\G}_r$ over $\F_p$
for each $r\ge s$; applying the contravariant Dieudonn\'e module functor
$\D(\cdot):=\D(\cdot)_{\Z_p}$ yields a projective
system of finite free $\Z_p$-modules $\{\D(\o{\G}_r)\}_r$ with compatible linear endomorphisms $F,V$
satisfying $FV=VF=p$.
 
\begin{definition}\label{DinftyDef}
	We write $\D_{\infty}:=\varprojlim_r \D(\o{\G}_r)$ for the projective limit
	of the system $\{\D(\o{\G}_r)\}_r$.  For $\star\in \{\et,\mult\}$
	we write $\D_{\infty}^{\star}:=\varprojlim_r \D(\o{\G}_r^{\star})$
	for the corresponding projective limit.
\end{definition}

Since $\H_r^*$ acts by endomorphisms on $\o{\G}_r$, compatibly with change in $r$,
we obtain an action of $\H^*$ on $\D_{\infty}$ and on $\D_{\infty}^{\star}$.
Likewise, the ``geometric inertia action" of $\Gamma$ on $\o{\G}_r$ by automorphisms
of $p$-divisible groups over $\F_p$ gives an 
action of $\Gamma$ on $\D_{\infty}$ and $\D_{\infty}^{\star}$.
As $\o{\G}_r$ is ordinary by Proposition \ref{SpecialFiberOrdinary}, applying $\D(\cdot)$
to the (split) connected-\'etale squence of $\o{\G}_r$ gives, for each $r$, 
a functorially split exact sequence 
\begin{equation}
	\xymatrix{
		0 \ar[r] & {\D(\o{\G}_r^{\et})} \ar[r] & {\D(\o{\G}_r)} \ar[r] & 
		{\D(\o{\G}_r^{\mult})} \ar[r] & 0
	}\label{DieudonneFiniteExact}
\end{equation}
with $\Z_p$-linear actions of $\Gamma$, $F$, $V$, and $\H_r^*$.
Since projective limits commute with finite direct sums, we obtain
a split short {\em exact} sequence of $\Lambda$-modules with linear $\H^*$ and $\Gamma$-actions
and commuting linear endomorphisms $F,V$ satisfying $FV=VF=p$:
\begin{equation}
	\xymatrix{
		0 \ar[r] & {\D_{\infty}^{\et}} \ar[r] & {\D_{\infty}} \ar[r] & {\D_{\infty}^{\mult}} \ar[r] & 0
	}.\label{DieudonneInfiniteExact}
\end{equation}

\begin{theorem}\label{MainDieudonne}
	As in Proposition $\ref{NormalizationCoh}$, set $d':=\sum_{k=3}^p \dim_{\F_p} S_k(N;\F_p)^{\ord}$. 
	Then:
	\begin{enumerate} 
		\item $\D_{\infty}$ is a free $\Lambda$-module of rank $2d'$, and $\D_{\infty}^{\star}$
		is free of rank $d'$ over $\Lambda$ for $\star\in \{\et,\mult\}$.  
		\label{MainDieudonne1}
		
		\item For each $r\ge 1$, applying $\otimes_{\Lambda} \Z_p[\Delta/\Delta_r]$ to 
		$(\ref{DieudonneInfiniteExact})$  yields the short exact sequence $(\ref{DieudonneFiniteExact})$, 
		compatibly with $\H^*$, $\Gamma$, $F$ and $V$.
		\label{MainDieudonne2}
		
		\item Under the canonical splitting of $(\ref{DieudonneInfiniteExact})$, $\D_{\infty}^{\et}$ 
		is the maximal subspace of $\D_{\infty}$ on which $F$ acts invertibly, while 
		$\D_{\infty}^{\mult}$ corresponds to the maximal subspace of $\D_{\infty}$ on which $V$ acts 
		invertibly.  
		\label{MainDieudonne3}
		
		\item The Hecke operator $U_p^*$ acts as $F$ on $\D_{\infty}^{\et}$ and as $\langle p\rangle_NV$ on $\D_{\infty}^{\mult}$.
		\label{MainDieudonne4}
	
		\item $\Gamma$ acts trivially on $\D_{\infty}^{\et}$ and via $\langle \chi\rangle^{-1}$
		on $\D_{\infty}^{\mult}$.
		\label{MainDieudonne5}
\end{enumerate}
\end{theorem}

\begin{proof}
	We apply Lemma \ref{Technical} with $A_r=\Z_p$, $I_r=(p)$, 
	and with $M_r$ each one of the terms in (\ref{DieudonneFiniteExact}).
	Due to Proposition \ref{GisOrdinary}, there is a natural isomorphism of split short exact sequences
	\begin{equation*}
		\xymatrix{
			0 \ar[r] & {\D(\o{\G}_r^{\et})_{\F_p}} \ar[r]\ar[d]^-{\simeq} & 
			{\D(\o{\G}_r)_{\F_p}} \ar[r] \ar[d]^-{\simeq}& 
			{\D(\o{\G}_r^{\mult})_{\F_p}} \ar[r] \ar[d]^-{\simeq}& 0 \\
			0 \ar[r] & {f'H^1(I_r^0,\O)^{F_{\ord}}}\ar[r] &
			{f'H^0(I_r^{\infty},\Omega^1)^{V_{\ord}}\oplus f'H^1(I_r^0,\O)^{F_{\ord}}} \ar[r] &
			{f'H^0(I_r^{\infty},\Omega^1)^{V_{\ord}}} \ar[r] & 0
		}
	\end{equation*}
	that is compatible with change in $r$ using the trace mappings attached to 
	$\rho:I_r^{\star}\rightarrow I_{s}$ and the maps on Dieudonn\'e modules 
	induced by $\o{\rho}_{r,s}:\o{\G}_{s} \rightarrow \o{\G}_r$.
	The hypotheses (\ref{freehyp}) and (\ref{surjhyp})
	of Lemma \ref{Technical} are thus satisfied with $d'$ as in the statement of the theorem,
	thanks to Proposition \ref{IgusaStructure} (\ref{IgusaFreeness})--(\ref{IgusaControl})
	and Lemma \ref{CharacterSpaces}.  We conclude from Lemma \ref{Technical}
	that (\ref{MainDieudonne1}) and (\ref{MainDieudonne2}) hold. 
	As $F$ (respectively $V$) acts invertibly on $\D(\o{\G}_r^{\et})$ (respectively 
	$\D(\o{\G}_r^{\mult})$) for all $r$, assertion (\ref{MainDieudonne3}) is clear, while
	(\ref{MainDieudonne4}) and (\ref{MainDieudonne5}) follow immediately from 
	Proposition \ref{SpecialFiberOrdinary}.  
\end{proof}

As in Proposition \ref{dRDuality}, the short exact sequence (\ref{DieudonneInfiniteExact})
is very nearly ``auto dual":

\begin{proposition}\label{DieudonneDuality}		
		There is a canonical isomorphism of short exact sequences of $\Lambda_{R_0'}$-modules
		\begin{equation}
		\begin{gathered}
			\xymatrix{
	0 \ar[r] & {\D_{\infty}^{\et}(\langle \chi \rangle\langle a\rangle_N)_{\Lambda_{R_0'}}} \ar[r]\ar[d]^-{\simeq} & 
	{\D_{\infty}(\langle \chi \rangle\langle a\rangle_N)_{\Lambda_{R_0'}}}\ar[r]\ar[d]^-{\simeq} & 
	{\D_{\infty}^{\mult}(\langle \chi \rangle\langle a\rangle_N)_{\Lambda_{R_0'}}}\ar[r]\ar[d]^-{\simeq} & 0 \\		
		0\ar[r] & {(\D_{\infty}^{\mult})^{\vee}_{\Lambda_{R_0'}}} \ar[r] & 
				{(\D_{\infty})^{\vee}_{\Lambda_{R_0'}}} \ar[r] & 
				{(\D_{\infty}^{\et})^{\vee}_{\Lambda_{R_0'}}}\ar[r] & 0
			}
		\end{gathered}
		\label{DmoduleDuality}	
		\end{equation}
		that is $\H^*$ and $\Gamma\times \Gal(K_0'/K_0)$-equivariant, 
		and intertwines $F$
		$($respectively $V$$)$ on the top row with $V^{\vee}$
		$($respectively $F^{\vee}$$)$ on the bottom.
\end{proposition}

\begin{proof}	
	We apply the duality formalism 
	of Lemma \ref{LambdaDuality}.  
	Let us write
	$\pr_{r,s}':\o{\G}_r'\rightarrow \o{\G}_s'$ for the maps on special fibers
	induced by (\ref{pdivTowers}). 
	Thanks to Proposition \ref{GdualTwist}, the definition \ref{ordpdivdefn}  of $\o{\G}_r':=\o{\G}_r^{\vee}$,
	the natural isomorphism $\G_r\times_{R_r} R_r' \simeq \G_r(\langle \chi\rangle\langle a\rangle_N)\times_{R_r} R_r'$,
	and the compatibility of the Dieudonn\'e module functor with duality, there are natural
	isomorphisms of $R_0'$-modules
	\begin{equation}
		\D(\o{\G}_r)(\langle \chi\rangle\langle a\rangle_N) \tens_{\Z_p} R_0' \simeq
		\D(\o{\G_r(\langle \chi\rangle\langle a\rangle_N)})\tens_{\Z_p} R_0' 
		\simeq \D(\o{\G}_r')\tens_{\Z_p} R_0' = \D(\o{\G}_r^{\vee})\tens_{\Z_p} R_0'\simeq 
		(\D(\o{\G}_r))_{R_0'}^{\vee}
		\label{evpairingDieudonne}
	\end{equation}
	that are $\H^*_r$-equivariant, $\Gal(K_r'/K_0)$-compatible 
	for the standard action $\sigma\cdot f (m):=\sigma f(\sigma^{-1}m)$
	on the $R_0'$-linear dual of $\D(\o{\G}_r)\otimes_{\Z_p} R_0'$, 
	and compatible with change in $r$ using $\pr_{r,s}$
	on $\D(\o{\G}_r)$ and $\pr_{r,s}'$ on $\D(\o{\G}_r')$.  We claim that the resulting 
	perfect ``evaluation" pairings 
	\begin{equation}
		\xymatrix{
			{\langle\cdot,\cdot\rangle_r : \D(\o{\G}_r)(\langle \chi\rangle\langle a\rangle_N)\tens_{\Z_p}{R_0'} 
			\times  \D(\o{\G}_r)\tens_{\Z_p}{R_0'}} \ar[r] & {R_0'}
			}\label{pdivSpecialTwistPai}
	\end{equation}
	satisfy the compatibility hypothesis (\ref{pairingchangeinr}) of Lemma \ref{LambdaDuality}.
	Indeed, the stated compatibility of (\ref{evpairingDieudonne}) with change in $r$ 
	and the very definition (\ref{pdivTowers}) of the transition maps $\pr_{r,s}'$
	implies that for $r\ge s$
	\begin{equation*}
		\langle \D(\Pic^0(\pr)^{r-s} x), y\rangle_s = \langle x , \D({U_p^*}^{s-r}\Alb(\ps)^{r-s}) y\rangle_r,
	\end{equation*}
	so our claim follows from the equality in $\End_{\Q_p}(J_{r+1})$
	\begin{equation}
		\Pic(\pr)\circ \Alb(\ps) = U_p^* \sum_{\delta\in \Delta_r/\Delta_{r+1}} \langle \delta^{-1}\rangle,
		\label{PicAlbRelation}
	\end{equation}
	which, as in the proof of Proposition \ref{dRDuality}, follows from Lemma \ref{MFtraceLem} via Lemma 		
	\ref{LieFactorization}.
	Again, by the $\H_r^*$-compatibility of (\ref{evpairingDieudonne}), the action of $\H_r^*$
	is self-adjoint with resect to (\ref{pdivSpecialTwistPai}), so 
	Lemma \ref{LambdaDuality} gives a perfect $\Gal(K_{\infty}'/K_0)$-compatible duality pairing
	$\langle\cdot,\cdot \rangle: \D_{\infty}(\langle \chi\rangle\langle a\rangle_N)
	\otimes_{\Lambda} \Lambda_{R_0'} \times \D_{\infty}\otimes_{\Lambda} \Lambda_{R_0'} \rightarrow \Lambda_{R_0'}$
	with respect to which $T^*$ is self-adjoint for all $T^*\in \H^*$.
	That the resulting isomorphism (\ref{DmoduleDuality})
	intertwines 
	$F$ with $V^{\vee}$
	and $ V $ with $F^{\vee}$ is an immediate consequence of the compatibility of the Dieudonn\'e module
	functor with duality. 
\end{proof}

We can interpret $\D_{\infty}^{\star}$ in terms of the crystalline cohomology of the
Igusa tower as follows.  Let $I_r^0$ and $I_r^{\infty}$ be the two ``good" 
components of $\o{\X}_r$ as in Remark \ref{MWGood}, and form the projective limits
\begin{equation*}
	H^1_{\cris}(I^{\star}) := \varprojlim_{r} H^1_{\cris}(I_r^{\star})
\end{equation*}
for $\star\in \{\infty,0\}$, taken with respect to the trace maps on crystalline
cohomology (see \cite[\Rmnum{7}, \S2.2]{crystal2}) induced by the canonical degeneracy mappings 
$\rho:I_{r}^{\star}\rightarrow I_{s}^{\star}$.  
Then $H^1_{\cris}(I^{\star})$ 
is naturally a $\Lambda$-module (via the diamond operators), equipped with 
a commuting action of $F$ (Frobenius) and $V$ (Verscheibung) satisfying $FV=VF=p$.  
Letting $U_p^*$ act as $F$ (respectively $\langle p\rangle_N V$) on $H^1_{\cris}(I^{\star})$ for $\star=\infty$
(respectively $\star=0$) and the Hecke operators outside $p$ (viewed as correspondences
on the Igusa curves) act via pullback and trace at each level $r$, we obtain 
an action of $\H^*$ on $H^1_{\cris}(I^{\star})$.  Finally, we 
let $\Gamma$ act trivially on $H^1_{\cris}(I^{\star})$ for $\star=\infty$
and via $\langle\chi^{-1}\rangle$ for $\star=0$.

\begin{theorem}\label{DieudonneCrystalIgusa}
	There is a canonical $\H^*$ and $\Gamma$-equivariant
	isomorphism of $\Lambda$-modules
	\begin{equation*}
		\D_{\infty} = \D_{\infty}^{\mult}\oplus \D_{\infty}^{\et} \simeq
		f'H^1_{\cris}(I^{0})^{V_{\ord}} \oplus 
		f'H^1_{\cris}(I^{\infty})^{F_{\ord}}
	\end{equation*}
	which respects the given direct sum decompositions and is compatible with $F$ and $V$.
\end{theorem}

\begin{proof}
	From the exact sequence (\ref{GrSpecialExact}), we obtain for each $r$ isomorphisms
	\begin{equation}
		\xymatrix@C=55pt{
			{\D(\o{\G}_r^{\mult})} \ar[r]^-{\simeq}_-{V^r \circ \D(\Alb(i_r^{0}))} &
			 {f'\D(j_r^{0}[p^{\infty}])^{V_{\ord}}}
		}\qquad\text{and}\qquad
		\xymatrix@C=55pt{
			{f'\D(j_r^{\infty}[p^{\infty}])^{F_{\ord}}} \ar[r]^-{\simeq}_-{\D(\Pic^0(i_r^{\infty}))} &
			{\D(\o{\G}_r^{\et})}
		}\label{IgusaInterpretation}
	\end{equation}
	that are $\H^*$ and $\Gamma$-equivariant (with respect to the actions
	specified in Proposition \ref{SpecialFiberOrdinary}), and compatible with change in $r$
	via the mappings $\D(\pr_{r,s})$ on $\D(\o{\G}_r^{\star})$ and $\D(\pr)$
	on $\D(j_r^{\star}[p^{\infty}])$.  On the other hand, for {\em any} smooth and proper curve
	$X$ over a perfect field $k$ of characteristic $p$, 
	thanks to \cite{MM} and \cite[\Rmnum{2}, \S3 C Remarque 3.11.2]{IllusiedR}
	there are natural isomorphisms
	of $W(k)[F,V]$-modules  
	\begin{equation}
		\D(J_X[p^{\infty}]) \simeq H^1_{\cris}(J_X/W(k)) \simeq H^1_{\cris}(X/W(k))\label{MMIllusie}
	\end{equation}
	that for any finite map of smooth proper curves $f:Y\rightarrow X$ over $k$
	intertwine $\D(\Pic(f))$ and $\D(\Alb(f))$ with trace and pullback by $f$ on crystalline cohomology,
	respectively.  Applying this to $X=I_r^{\star}$ for $\star=0,\infty$, appealing to 
	the identifications (\ref{IgusaInterpretation}), and passing to inverse limits completes the proof.
\end{proof}

Applying the idempotent $f'$ of (\ref{TeichmullerIdempotent}) to the Hodge filtration (\ref{mainthmexact})
yields a short exact sequence of free $\Lambda_{R_{\infty}}$-modules with
semilinear $\Gamma$-action and linear commuting action of $\H^*$:
\begin{equation}
	\xymatrix{
		0 \ar[r] & {{e^*}'H^0(\omega)} \ar[r] & {{e^*}'H^1_{\dR}} \ar[r] & {{e^*}'H^1(\O)} \ar[r] & 0
	}.\label{LambdaHodgeFilnomup}
\end{equation}
The key to relating (\ref{LambdaHodgeFilnomup}) to the slope filtration (\ref{DieudonneInfiniteExact}) 
is the following comparison isomorphism:

\begin{proposition}\label{KeyComparison}
	For each positive integer $r$, there is a natural isomorphism of short exact sequences
	\begin{equation}
	\begin{gathered}
		\xymatrix{
			0\ar[r] & {\omega_{\G_r}} \ar[r]\ar[d]_-{\simeq} & {\D(\G_{r,0})_{R_r}} \ar[r]\ar[d]^-{\simeq} & 
			{\Lie(\Dual{\G}_r)} \ar[r]\ar[d]^-{\simeq} & 0 \\
			0\ar[r] & {{e^*}'H^0(\omega_r)} \ar[r] & {{e^*}'H^1_{\dR,r}} \ar[r] & {{e^*}'H^1(\O_r)} \ar[r] & 0
		}
	\end{gathered}\label{CollectedComparisonIsom}
	\end{equation}
	that is compatible with $\H_r^*$, $\Gamma$, and change in $r$ using
	the mappings $(\ref{pdivTowers})$ on the top row and the maps $\pr_*$ on the bottom.
	Here, the bottom row is obtained from $(\ref{HodgeFilIntAbbrev})$ by applying ${e^*}'$
	and the top row is the Hodge filtration of $\D(\G_{r,0})_{R_r}$ given by
	Proposition $\ref{BTgroupUnivExt}$.  
\end{proposition}

\begin{proof}
	Let $\alpha_r^*: J_r\twoheadrightarrow B_r^*$ be the map of Definition \ref{BalphDef}.
	We claim that $\alpha_r^*$ induces a canonical isomorphism of short exact sequences of free 
	$R_r$-modules
	\begin{equation}
	\begin{gathered}
		\xymatrix{
			0 \ar[r] & {\omega_{\G_r}}\ar[d]_-{\simeq} \ar[r] & {\D(\G_{r,0})_{R_r}} \ar[d]_-{\simeq}\ar[r] & 
			{\Lie(\G_r^t)} \ar[d]^-{\simeq}\ar[r] & 0 \\
			0 \ar[r] & {{e^*}'\omega_{\J_r}} \ar[r] & {{e^*}'\Lie\scrExtrig(\J_r,\Gm)} \ar[r] & 
			{{e^*}'\Lie({\J_r^t}^0)} \ar[r] & 0
		}
	\end{gathered}\label{HodgeToExtrigMap}
	\end{equation}
	that is $\H_r^*$ and $\Gamma$-equivariant and compatible with change in $r$ using
	the map on N\'eron models induced by $\Pic^0(\pr)$ and the maps (\ref{pdivTowers})
	on $\G_r$.
	Granting this claim, the proposition then follows immediately from Proposition \ref{intcompare}.

	To prove our claim, we introduce the following notation:
	set $V:=\Spec(R_r)$, and for $n\ge 1$ put $V_n:=\Spec(R_r/p^nR_r)$. For any scheme 
	(or $p$-divisible group) $X$ over $V$, we put $X_n:=X\times_V V_n$.
	If $\A$ is a N\'eron model over $V$,
	we will write $H(\A)$ for the short exact sequence of free $R_r$-modules obtained by
	applying $\Lie$ to the canonical extension (\ref{NeronCanExt}) of $\Dual{\A}^0$.
	If $G$ is a $p$-divisible group over $V$, we similalry
	write $H(G_n)$ for the short exact sequence of Lie algebras associated to the universal extension
	of $G_n^t$ by a vector group over $V_n$ (see Theorem \ref{UniExtCompat}, (\ref{UniExtCompat2})).  
	If $\A$ is an abelian scheme over $V$ 
	then we have natural and compatible (with change in $n$) isomorphisms
	\begin{equation}
		H(\A_n[p^{\infty}])\simeq H(\A_n)\simeq H(\A)/p^n,\label{AbSchpDiv}
	\end{equation}
	thanks to Theorem \ref{UniExtCompat}, (\ref{UniExtCompat3}) and (\ref{UniExtCompat1}); in particular, this
	justifies our slight abuse of notation.  
	
	Applying the contravariant functor ${e^*}'H(\cdot)$ to the diagram of N\'eron models over $V$ 
	induced by (\ref{UFactorDiag}) yields a commutative diagram of short exact sequences of 
	free $R_r$-modules
	\begin{equation}
	\begin{gathered}
		\xymatrix{
			{{e^*}'H(\J_r)} & {{e^*}'H(\B_r^*)}\ar[l] \\
			{{e^*}'H(\J_r)} \ar[u]^-{U_r^*}\ar[ur] & {{e^*}'H(\B_r)}\ar[u]_-{U_r^*}\ar[l] 
			}
	\end{gathered}		
	\end{equation}
	in which both vertical arrows are isomorphisms by definition of ${e^*}'$.  As in the proofs of Propositions
	\ref{GoodRednProp} and \ref{SpecialFiberOrdinary}, it follows that
	the horizontal maps must be isomorphisms as well:
	\begin{equation}
		{e^*}'H(\J_r)\simeq {e^*}'H(\B_r^*)
		\label{alphaIdenOrd}
	\end{equation}
	Since these isomorphisms are induced via the N\'eron mapping property and the functoriality
	of $H(\cdot)$ by the $\H_r^*(\Z)$-equivariant map $\alpha_r^*:J_r\twoheadrightarrow B_r^*$,
	they are themselves $\H_r^*$-equivariant.  Similarly, since $\alpha_r^*$ is defined 
	over $\Q$ and compatible with change in $r$ as in Lemma \ref{Btower}, the isomorphism
	(\ref{alphaIdenOrd}) is compatible with the given actions of $\Gamma$ (arising via the N\'eron
	mapping property from the semilinear action of $\Gamma$ over $K_r$ giving the descent
	data of ${J_r}_{K_r}$ and ${B_r}_{K_r}$ to $\Q_p$) and change in $r$.
	Reducing (\ref{alphaIdenOrd}) modulo $p^n$ and using the canonical isomorphism (\ref{AbSchpDiv}) yields
	the identifications
	\begin{equation}
		{e^*}'H(\J_r)/p^n\simeq {e^*}'H(\B_r^*)/p^n \simeq {e^*}'H(\B_{r,n}^*[p^{\infty}]) 
		\simeq H({e^*}'\B_{r,n}^*[p^{\infty}]) =: H(\G_{r,n})\label{ModPowersIsom}
	\end{equation}
	which are clearly compatible with change in $n$, and which are easily checked
	(using the naturality of (\ref{AbSchpDiv}) and our remarks above) to be
	$\H_r^*$ and $\Gamma$-equivariant, and compatible with change in $r$.
	Since the surjection $R_r\twoheadrightarrow R_r/pR_r$ is a PD-thickening,
	passing to inverse limits (with respect to $n$) on (\ref{ModPowersIsom}) and using 
	Proposition \ref{BTgroupUnivExt} now completes the proof.	
\end{proof}

\begin{corollary}\label{RelationToHodgeCor}
	Let $r$ be a positive integer. Then the short exact sequence of free $R_r$-modules
	\begin{equation}
		\xymatrix{
			0\ar[r] & {{e^*}'H^0(\omega_r)} \ar[r] & {{e^*}'H^1_{\dR,r}} \ar[r] & {{e^*}'H^1(\O_r)} \ar[r] & 0
		}\label{TrivialEigenHodge}
	\end{equation}
	is functorially split; in particular, 
	it is split compatibly with the actions of $\Gamma$ and $\H_r^*$.
	Moreover, $(\ref{TrivialEigenHodge})$ admits a functorial descent to $\Z_p$:
	there is a natural isomorphism of split short exact sequences 
	\begin{equation}
	\begin{gathered}
			\xymatrix{
			0\ar[r] & {{e^*}'H^0(\omega_r)} \ar[r]\ar[d]_-{\simeq} & 
			{{e^*}'H^1_{\dR,r}} \ar[r]\ar[d]^-{\simeq} & {{e^*}'H^1(\O_r)} \ar[r]\ar[d]^-{\simeq} & 0\\
			0 \ar[r] & {\D(\o{\G}_r^{\mult})\tens_{\Z_p} R_r} \ar[r] &
			{\D(\o{\G}_r)\tens_{\Z_p} R_r}\ar[r] &
			{\D(\o{\G}_r^{et})\tens_{\Z_p} R_r} \ar[r] & 0
			}
	\end{gathered}\label{DescentZp}		
	\end{equation}
	that is $\H^*$ and $\Gamma$ equivariant, with
	$\Gamma$ acting trivially on $\o{\G}_r^{\et}$ and through $\langle \chi\rangle^{-1}$ on $\o{\G}_r^{\mult}$.
	The identification $\ref{DescentZp}$ is compatible with change in $r$ using the maps $\pr_*$ on the top
	row and the maps induced by 
	\begin{equation*}
			\xymatrix@C=35pt{
				{\o{\G}_r=\o{\G}_r^{\mult} \times \o{\G}_r^{\et}} \ar[r]^{V^{-1}\times F} &
				 {\o{\G}_r^{\mult} \times \o{\G}_r^{\et}=\o{\G}_r} \ar[r]^-{\o{\rho}} & 
				 {\o{\G}_{r+1}} 
			}
	\end{equation*}	
	on the bottom row.
\end{corollary}

\begin{proof}
	Consider the isomorphism (\ref{CollectedComparisonIsom}) of Proposition \ref{KeyComparison}.
	As  $\G_r$ is an ordinary $p$-divisible group by Proposition \ref{SpecialFiberOrdinary},
	the top row of (\ref{CollectedComparisonIsom}) is functorially split
	by Lemma \ref{HodgeFilOrdProps},  and this gives our first assertion.
	Composing the inverse of (\ref{CollectedComparisonIsom})
	with the isomorphism (\ref{DescentToWIsom}) of	Lemma \ref{HodgeFilOrdProps} gives
	the claimed identification (\ref{DescentZp}).
	That this isomorphism is compatible with change in $r$ via the specified maps
	follows easily from the construction of (\ref{DescentToWIsom}) via
	(\ref{TwistyDieuIsoms}).
\end{proof}

We can now prove Theorem \ref{dRtoDieudonne}.  Let us recall the statement:

\begin{theorem}\label{dRtoDieudonneInfty}
	There is a canonical isomorphism of 
	short exact sequences of finite free $\Lambda_{R_{\infty}}$-modules 
	\begin{equation}
	\begin{gathered}
		\xymatrix{
		0 \ar[r] & {{e^*}'H^0(\omega)} \ar[r]\ar[d]^-{\simeq} & 
		{{e^*}'H^1_{\dR}} \ar[r]\ar[d]^-{\simeq} & {{e^*}'H^1(\O)} \ar[r]\ar[d]^-{\simeq} & 0 \\
		0 \ar[r] & {\D_{\infty}^{\mult}\tens_{\Lambda} \Lambda_{R_{\infty}}} \ar[r] &
		{\D_{\infty}\tens_{\Lambda} \Lambda_{R_{\infty}}} \ar[r] &
		{\D_{\infty}^{\et}\tens_{\Lambda} \Lambda_{R_{\infty}}} \ar[r] & 0
		}
	\end{gathered}
	\end{equation}
	that is $\Gamma$ and $\H^*$-equivariant.  
	Here, the mappings on bottom row are the canonical inclusion and projection morphisms
	corresponding to the direct sum decomposition $\D_{\infty}=\D_{\infty}^{\mult}\oplus \D_{\infty}^{\et}$.
	In particular, the Hodge filtration exact sequence $(\ref{LambdaHodgeFilnomup})$ is canonically 
	split, and admits a canonical descent to $\Lambda$.
\end{theorem}

\begin{proof}
	Applying $\otimes_{R_r} R_{\infty}$ to $(\ref{DescentZp})$ and passing to projective limits 
	yields an isomorphism of split exact sequences
	\begin{equation*}
		\xymatrix{
			0\ar[r] & {{e^*}'H^0(\omega)} \ar[r]\ar[d]_-{\simeq} & 
			{{e^*}'H^1_{\dR}} \ar[r]\ar[d]^-{\simeq} & {{e^*}'H^1(\O)} \ar[r]\ar[d]^-{\simeq} & 0\\
0 \ar[r] & {\varprojlim\limits_{\o{\rho}\circ V^{-1}} \left(\D(\o{\G}_r^{\mult})\tens_{\Z_p} R_{\infty}\right)} \ar[r] &
{\varprojlim\limits_{\o{\rho}\circ (V^{-1}\times F)}\left(\D(\o{\G}_r)\tens_{\Z_p} R_{\infty}\right)}\ar[r] &
			{\varprojlim\limits_{\o{\rho}\circ F}\left(\D(\o{\G}_r^{et})\tens_{\Z_p} R_{\infty}\right)} \ar[r] & 0
			}
		\end{equation*}
	On the other hand, the isomorphisms
$		\xymatrix@1@C=37pt{
			{\o{\G}_r = \o{\G}_r^{\mult}\times \o{\G}_r^{\et} } \ar[r]^-{V^{-r}\times F^{r}} &
			{\o{\G}_r^{\mult}\times \o{\G}_r^{\et} =\o{\G}_r}
		}
$
	induce an isomorphism of projective limits
	\begin{equation*}
		\xymatrix{
			{\varprojlim\limits_{\o{\rho}}\left(\D(\o{\G}_r)\tens_{\Z_p} R_{\infty}\right)}	\ar[r]^-{\simeq} &
			{\varprojlim\limits_{\o{\rho}\circ (V^{-1}\times F)}\left(\D(\o{\G}_r)\tens_{\Z_p} R_{\infty}\right)}
		}
	\end{equation*}
	which is visibly compatible with the the canonical splittings of source and target.  
	The result now follows from	
	Lemma \ref{Technical} (\ref{CompletedBaseChange}) and the proof of Theorem \ref{MainDieudonne},
	which guarantee that the canonical mapping 
	$\D_{\infty}\otimes_{\Lambda}\Lambda_{R_{\infty}}\rightarrow 
	\varprojlim_{\o{\rho}} (\D(\o{\G}_r)\otimes_{\Z_p}R_{\infty})$
	is an isomorphism respecting the natural splittings.
\end{proof}

As in \S\ref{ordforms}, for any subfield $K$ of $\c_p$ with ring of integers $R$, we denote by
$eS(N;\Lambda_R)$ the module of ordinary $\Lambda_R$-adic cuspforms of level $N$
in the sense of \cite[2.5.5]{OhtaEichler}.  Following our convention of \S\ref{OrdStruct},
we write $e'S(N;\Lambda_R)$ for the direct summand of $eS(N;\Lambda_R)$
on which $\mu_{p-1}\hookrightarrow \Z_p^{\times}\subseteq \H$ acts 
nontrivially.

\begin{corollary}\label{MFIgusaDieudonne}
	There is a canonical isomorphism of finite free $\Lambda$-modules 
	\begin{equation}
		{e}'S(N;\Lambda) \simeq \D_{\infty}^{\mult}   
		\label{LambdaFormsCrystalline}
	\end{equation}
	that intertwines $T\in \H$ on $e'S(N;\Lambda)$ with $T^*\in \H^*$
	on $\D_{\infty}^{\mult}$, 
	where $U_p^*$ acts on $\D_{\infty}^{\mult}$as $\langle p\rangle_N V$.
\end{corollary}

\begin{proof}
	We claim that there are natural isomorphisms of  finite free $\Lambda_{R_{\infty}}$-modules
	\begin{equation}
		\D_{\infty}^{\mult} \otimes_{\Lambda} \Lambda_{R_{\infty}} \simeq 
		{e^*}'H^0(\omega) \simeq {e}'S(N,\Lambda_{R_{\infty}}) \simeq 
		e'S(N,\Lambda)\otimes_{\Lambda} \Lambda_{R_{\infty}}\label{TakeGammaInvariants}
	\end{equation}		
	and that the resulting composite isomorphism 
	intertwines $T^*\in \H^*$ on $\D_{\infty}^{\mult}$ with $T\in \H$ on 
	$e'S(N,\Lambda)$ and is $\Gamma$-equivariant, with $\gamma\in\Gamma$
	acting as $\langle \chi(\gamma)\rangle^{-1}\otimes \gamma$ on each tensor product.  
	Indeed, the first and second isomorphisms are
	due to Theorem \ref{dRtoDieudonneInfty} and Corollary \ref{LambdaFormsRelation}, respectively,
	while the final isomorphism is a consequence of the definition 
	of $e'S(N;\Lambda_R)$ and the facts that this $\Lambda_R$-module is free of finite rank
	\cite[Corollary 2.5.4]{OhtaEichler} and specializes as in \cite[2.6.1]{OhtaEichler}.
	Twisting the
	$\Gamma$-action on the source and target of the composite (\ref{TakeGammaInvariants}) by 
	$\langle \chi \rangle$ therefore gives a $\Gamma$-equivariant isomorphism
	\begin{equation}
		\D_{\infty}^{\mult} \otimes_{\Lambda} \Lambda_{R_{\infty}} \simeq 
		S(N,\Lambda)\otimes_{\Lambda} \Lambda_{R_{\infty}}\label{TwistedGammaIsom}
	\end{equation}
	with $\gamma\in \Gamma$ acting as $1\otimes \gamma$ on source and target.  Passing to $\Gamma$-invariants on
	(\ref{TwistedGammaIsom}) yields (\ref{LambdaFormsCrystalline}).	
\end{proof}

\begin{remark}\label{MFIgusaCrystal}
	Via Proposition \ref{DieudonneDuality} and the natural $\Lambda$-adic
	duality between $e\H$ and $eS(N;\Lambda)$ \cite[Theorem 2.5.3]{OhtaEichler}, we obtain
	a canonical $\Gal(K_0'/K_0)$-equivariant isomorphism of $\Lambda_{R_0'}$-modules
	\begin{equation*}
		e'\H\tens_{\Lambda} \Lambda_{R_0'} \simeq \D_{\infty}^{\et}(\langle a\rangle_N)\tens_{\Lambda}{\Lambda_{R_0'}}
	\end{equation*}
	that intertwines $T\otimes 1$ for $T\in \H$ acting on $e'\H$ by multiplication
	with $T^*\otimes 1$, with $U_p^*$ acting on $\D_{\infty}^{\et}(\langle a\rangle_N)$ as $F$.
	From Theorem \ref{DieudonneCrystalIgusa} and Corollary \ref{MFIgusaDieudonne}  
	we then obtain canonical
	isomorphisms 
	\begin{equation*}
		e'S(N;\Lambda)\simeq f'H^1_{\cris}(I^0)^{V_{\ord}}\qquad\text{respectively}\qquad
		e'\H\tens_{\Lambda}\Lambda_{R_0'} \simeq 
		f'H^1_{\cris}(I^{\infty})^{F_{\ord}}(\langle a\rangle_{N})\tens_{\Lambda}\Lambda_{R_0'}
	\end{equation*}
	intertwing $T$ (respectively $T\otimes 1$) with $T^*$ (respectively $T^*\otimes 1$)
	where $U_p^*$ acts on crystalline cohomology as $\langle p\rangle_N V$
	(respectively $F\otimes 1$).  The second of these isomorphisms is moreover $\Gal(K_0'/K_0)$-equivariant.
\end{remark}

In order to relate the slope filtration (\ref{DieudonneInfiniteExact}) of $\D_{\infty}$
to the ordinary filtration of ${e^*}'H^1_{\et}$, we 
require:
\begin{lemma}
	Let $r$ be a positive integer let $G_r={e^*}'J_r[p^{\infty}]$ 
	be the unique $\Q_p$-descent of the generic fiber of $\G_r$, as in Definition $\ref{ordpdivdefn}$.
	There are canonical isomophisms of free $W(\o{\F}_p)$-modules
	\begin{subequations}	
		\begin{equation}
			\D(\o{\G}_r^{\et})\tens_{\Z_p} W(\o{\F}_p) \simeq \Hom_{\Z_p}(T_pG_r^{\et},\Z_p)\tens_{\Z_p} W(\o{\F}_p)
		\label{etalecase}
		\end{equation}
		\begin{equation}
			\D(\o{\G}_r^{\mult})(-1)\tens_{\Z_p} W(\o{\F}_p) \simeq \Hom_{\Z_p}(T_pG_r^{\mult},\Z_p)
			\tens_{\Z_p} W(\o{\F}_p).	
		\label{multcase}	
		\end{equation}
		that are $\H_r^*$-equivariant and $\scrG_{\Q_p}$-compatible for the diagonal action on source and target,
		with $\scrG_{\Q_p}$ acting trivially on $\D(\o{\G}^{\et}_r)$ and via $\chi^{-1}\cdot \langle \chi^{-1}\rangle$
		on $\D(\o{\G}_r^{\mult})(-1):=\D(\o{\G}_r^{\mult})\otimes_{\Z_p} \Z_p(-1)$.  The isomorphism
		$(\ref{etalecase})$ intertwines $F\otimes\sigma$ with $1\otimes \sigma$
		while $(\ref{multcase})$ intertwines $V\otimes\sigma^{-1}$ with $1\otimes\sigma^{-1}$.
	\end{subequations}
\end{lemma}

\begin{proof}
	Let $\G$ be any object of $\pdiv_{R_r}^{\Gamma}$ and write $G$ for the unique
	descent of the generic fiber $\G_{K_r}$ to $\Q_p$.  We recall that the semilinear $\Gamma$-action
	on $\G$ gives the $\Z_p[\scrG_{K_r}]$-module
	$T_p\G:=\Hom_{\O_{\C_p}}(\Q_p/\Z_p,\G_{\O_{\C_p}})$ 
	the natural structure of $\Z_p[\scrG_{\Q_p}]$-module via $g\cdot f:= g^{-1}\circ g^*f\circ g$.
	It is straightforward to check that the natural map $T_p\G\rightarrow T_pG$, which is an isomorphism
	of $\Z_p[\scrG_{K_r}]$-modules by Tate's theorem, is an isomorphism of $\Z_p[\scrG_{\Q_p}]$-modules
	as well.

	For {\em any} \'etale $p$-divisible group $H$ over a perfect field $k$, one has a canonical
	isomorphism of $W(\o{k})$-modules with semilinear $\scrG_k$-action
	\begin{equation*}
		\D(H)\tens_{W(k)} W(\o{k}) \simeq \Hom_{\Z_p}(T_pH,\Z_p)\tens_{\Z_p} W(\o{k})
	\end{equation*}
	that intertwines $F\otimes\sigma$ with $1\otimes\sigma$ and $1\otimes g$
	with $g\otimes g$ for $g\in \scrG_k$; for example, this can be deduced
	by applying \cite[\S4.1 a)]{BBM1} to $H_{\o{k}}$ 
	and using the fact that
	the Dieudonn\'e crystal is compatible with base change.
	In our case, the \'etale $p$-divisible group $\G_r^{\et}$ lifts $\o{\G}_r^{\et}$
	over $R_r$, and we obtain a natural isomorphism of $W(\o{\F}_p)$-modules
	with semilinear $\scrG_{K_r}$-action
	\begin{equation*}
		\D(\o{\G}_r^{\et})\tens_{\Z_p} W(\o{\F}_p) \simeq \Hom_{\Z_p}(T_p\G_r^{\et},\Z_p)\tens_{\Z_p} W(\o{\F}_p).
	\end{equation*}
	By naturality in $\G_r$, this identification respects the semilinear $\Gamma$-actions
	on both sides (which are trivial, as $\Gamma$ acts trivially on $\G_r^{\et}$); as explained
	in our initial remarks, it is precisely this action which allows us to view $T_p\G_r^{\et}$ as a 
	$\Z_p[\scrG_{\Q_p}]$-module, and we deduce (\ref{etalecase}).  The proof of (\ref{multcase})
	is similar, using the natural isomorphism (proved as above) for any multiplicative $p$-divisible group $H/k$
	\begin{equation*}
		\D(H)\tens_{W(k)} W(\o{k}) \simeq T_p\Dual{H}\tens_{\Z_p} W(\o{k}),
	\end{equation*}
	which intertwines $V\otimes\sigma^{-1}$ with $1\otimes\sigma^{-1}$
	and $1\otimes g$ with $g\otimes g$, for $g\in \scrG_k$.
\end{proof}

\begin{proof}[Proof of Theorem $\ref{FiltrationRecover}$ and Corollary $\ref{MWmainThmCor}$]
	For a $p$-divisible group $H$ over a field $K$, we will write $H^1_{\et}(H):=\Hom_{\Z_p}(T_pH,\Z_p)$;
	our notation is justified by the standard fact that, for $J_X$ the Jacobian
	of a curve $X$ over $K$, there is a natural isomorphisms of $\Z_p[\scrG_K]$-modules
	\begin{equation}
		H^1_{\et}(J_X[p^{\infty}]) \simeq H^1_{\et}(X_{\Kbar},\Z_p).\label{etalecohcrvjac}
	\end{equation}

	It follows from (\ref{etalecase})--(\ref{multcase}) and 
	Theorem \ref{MainDieudonne} (\ref{MainDieudonne1})--(\ref{MainDieudonne2})
	that $H^1_{\et}(G_r^{\star})\otimes_{\Z_p} W(\o{\F}_p)$ is a free
	$W(\o{\F}_p)[\Delta/\Delta_r]$-module of rank $d'$ for $\star\in \{\et,\mult\}$,
	and hence that $H^1_{\et}(G_r^{\star})$ is a free $\Z_p[\Delta/\Delta_r]$-module
	of rank $d'$ by Lemma \ref{fflatfreedescent}.  In a similar manner, using
	the faithful flatness of $W(\o{\F}_p)[\Delta/\Delta_r]$ over $\Z_p[\Delta/\Delta_r]$,
	we deduce that the canonical trace mappings
	\begin{equation}  
		\xymatrix{
			{H^1_{\et}(G_r^{\star})} \ar[r] & {H^1_{\et}(G_{r'}^{\star})}
			}\label{cantraceetale}
	\end{equation}
	are surjective for all $r\ge r'$.
	By Lemma \ref{Technical}, we conclude that
		$H^1_{\et}(G_{\infty}^{\star}):=\varprojlim_r H^1_{\et}(G_r^{\star})$
	is a free $\Lambda$-module of rank $d'$ and that there are canonical
	isomorphisms of $\Lambda_{W(\o{\F}_p)}$-modules
	\begin{equation*}
		H^1_{\et}(G_{\infty}^{\star})\tens_{\Lambda} \Lambda_{W(\o{\F}_p)} \simeq \varprojlim_r \left(H^1_{\et}(G_r^{\star})\tens_{\Z_p} W(\o{\F}_p)\right)
	\end{equation*}
	for $\star\in \{\et,\mult\}$.  Since we likewise have canonical identifications
	\begin{equation*}
		\D_{\infty}^{\star}\tens_{\Lambda} \Lambda_{W(\o{\F}_p)} \simeq \varprojlim_r \left(\D(G_r^{\star})\tens_{\Z_p} W(\o{\F}_p)\right)
	\end{equation*}
	thanks to Lemma \ref{Technical} and (the proof of) Theorem \ref{MainDieudonne}, 
	passing to inverse limits on (\ref{etalecase})--(\ref{multcase}) gives a canonical
	isomorphism of $\Lambda_{W(\o{\F}_p)}$-modules
	\begin{equation}
		\D_{\infty}^{\star}\tens_{\Lambda} \Lambda_{W(\o{\F}_p)} \simeq
		H^1_{\et}(G_{\infty}^{\star})\tens_{\Lambda} \Lambda_{W(\o{\F}_p)}\label{dieudonneordfillimit}
	\end{equation}
	for $\star\in \{\et,\mult\}$.
	
		Applying the functor $H^1_{\et}(\cdot)$ to the connected-\'etale sequence of $G_r$
	yields a short exact sequence of $\Z_p[\scrG_{\Q_p}]$-modules
	\begin{equation*}
		\xymatrix{
			0\ar[r] & {H^1_{\et}(G_r^{\et})} \ar[r] & {H^1_{\et}(G_r)} \ar[r] & {H^1_{\et}(G_r^{\mult})}\ar[r] & 0
		}
	\end{equation*}
	which naturally identifies ${H^1_{\et}(G_r^{\star})}$ with the invariants (respectively covariants)
	of $H^1_{\et}(G_r)$ under the inertia subgroup $\I\subseteq \scrG_{\Q_p}$  for $\star=\et$ (respectively 
	$\star=\mult$).
	As $G_r={e^*}'J_r[p^{\infty}]$ by definition, we deduce from this and (\ref{etalecohcrvjac})
	a natural isomorphism of short exact sequences of $\Z_p[\scrG_{\Q_p}]$-modules
	\begin{equation}
	\begin{gathered}
		\xymatrix{
			 0 \ar[r] & {H^1_{\et}(G_r^{\et})} \ar[r]\ar[d]^-{\simeq} & 
			 {H^1_{\et}(G_r)} \ar[r]\ar[d]^-{\simeq} & 
			 {H^1_{\et}(G_r^{\mult})} \ar[r]\ar[d]^-{\simeq} & 0 \\
			 0 \ar[r] & {({e^*}'H^1_{\et,r})^{\I}} \ar[r] &
			 {{e^*}'H^1_{\et,r}} \ar[r] &
			{({e^*}'H^1_{\et,r})_{\I}} \ar[r] & 0 
			}
	\end{gathered}\label{inertialinvariantsseq}
	\end{equation}
	where for notational ease abbreviate $H^1_{\et,r}:=H^1_{\et}({X_r}_{\Qbar_p},\Z_p)$.
	As the trace maps (\ref{cantraceetale}) are surjective, passing to inverse limits
	on (\ref{inertialinvariantsseq}) yields an isomorphism of short exact sequences
	\begin{equation}
	\begin{gathered}
		\xymatrix{
			0 \ar[r] & {H^1_{\et}(G_{\infty}^{\et})} \ar[r]\ar[d]^-{\simeq} & 
			{H^1_{\et}(G_{\infty})} \ar[r]\ar[d]^-{\simeq} & 
			{H^1_{\et}(G_{\infty}^{\mult})} \ar[r]\ar[d]^-{\simeq} & 0 \\
			0 \ar[r] & {\varprojlim_r ({e^*}'H^1_{\et,r})^{\I}} \ar[r] &
			{\varprojlim_r {e^*}'H^1_{\et,r}} \ar[r] &
			{\varprojlim_r ({e^*}'H^1_{\et,r})_{\I}} \ar[r] & 0
		}
	\end{gathered}\label{limitetaleseq}
	\end{equation}
	Since inverse limits commute with group invariants, the bottom row of (\ref{limitetaleseq})
	is canonically isomorphic to the ordinary filtration of Hida's ${e^*}'H^1_{\et}$,
	and Theorem \ref{FiltrationRecover} follows immediately from (\ref{dieudonneordfillimit}).
	Corollary \ref{MWmainThmCor} is then an easy consequence of Theorem \ref{FiltrationRecover} and 
	Lemma \ref{fflatfreedescent};
	alternately one can prove Corollary \ref{MWmainThmCor} directly from Lemma \ref{Technical},
	using what we have seen above.
\end{proof}

\subsection{Ordinary families of \texorpdfstring{$\s$}{S}-modules}\label{OrdSigmaSection}

We now study the family of Dieudonn\'e crystals attached to 
the tower of $p$-divisible groups $\{\G_{r}/R_r\}_{r\ge 1}$.
For each pair of positive integers $r\ge s$,
we have a morphism $\rho_{r,s}: \G_{s}\times_{T_{s}} T_r\rightarrow \G_{r}$
in $\pdiv_{R_{r}}^{\Gamma}$;  
applying the contravariant functor $\m_r:\pdiv_{R_r}^{\Gamma}\rightarrow \BT_{\s_r}^{\Gamma}$
studied in \S\ref{pDivPhiGamma} to the map on connected-\'etale sequences induced by $\rho_{r,s}$ 
and using the exactness of $\m_r$ and its compatibility with base change
(Theorem \ref{CaisLauMain}), we obtain
maps of exact sequences in  
$\BT_{\s_{r}}^{\Gamma}$
\begin{equation}
\begin{gathered}
	\xymatrix{
		0\ar[r] & {\m_r(\G_r^{\et})} \ar[r]\ar[d]_-{\m_r(\rho_{r,s})} & 
		{\m_r(\G_r)} \ar[r]\ar[d]^-{\m_r(\rho_{r,s})} & {\m_r(\G_r^{\mult})} 
		\ar[r]\ar[d]^-{\m_r(\rho_{r,s})} & 0\\
		0\ar[r ] & {\m_{s}(\G_{s}^{\et})\tens_{\s_{s}} \s_r} \ar[r] & 
		{\m_r(\G_{s})\tens_{\s_{s}} \s_r} \ar[r] & {\m_r(\G_{s}^{\mult})\tens_{\s_{s}} \s_r} \ar[r] & 0
	}\label{BTindLim}
\end{gathered}	
\end{equation}

\begin{definition}\label{DieudonneLimitDef}
	Let $\star=\et$ or $\star=\mult$ and define
	\begin{align}
			\m_{\infty}&:=\varprojlim_r \left(\m_r(\G_r) \tens_{\s_r} \s_{\infty}\right)
			& 
			\m_{\infty}^{\star}&:=\varprojlim_r \left(\m_r(\G_r^{\star}) \tens_{\s_r} \s_{\infty}\right),
			\label{UncompletedDieudonneLimit}
	\end{align}
	with the projective limits taken with respect to the mappings induced by (\ref{BTindLim}).
\end{definition}	

Each of (\ref{UncompletedDieudonneLimit}) is naturally a module over 
the completed group ring $\Lambda_{\s_{\infty}}$
and is equipped with a semilinear action of $\Gamma$ and a $\varphi$-semilinear
Frobenius morphism defined by $F:=\varprojlim (\varphi_{\m_r}\otimes \varphi)$.
Since $\varphi$ is bijective on $\s_{\infty}$, we also have a $\varphi^{-1}$-semilinear
Verscheibung morphism defined as follows.  For notational ease, we provisionally 
set $M_r:=\m_r(\G_r)\otimes_{\s_r} \s_{\infty}$ and we define 
\begin{equation}
	\xymatrix@C=50pt{
		{V_r: M_r} \ar[r]^{m\mapsto 1\otimes m} & 
		{{\varphi^{-1}}^*M_r }
		\ar[r]^-{{\varphi^{-1}}^*(\psi_{\m_r}\otimes 1)} &
		{{\varphi^{-1}}^*\varphi^*M_r\simeq M_r}
		}
\end{equation}
with $\psi_{\m_r}$ as above Definition \ref{DualBTDef}.  It is easy to see
that the $V_r$ are compatible with $r$, and we
put $V:=\varprojlim V_r$ on $\m_{\infty}$.  We define Verscheibung
morphisms on $\m_{\infty}^{\star}$ for $\star=\et,\mult$
similarly.  As the composite of $\psi_{\m_r}$ and $1\otimes\varphi_{\m_r}$
in either order is multiplication by $E_r(u_r) = u_0/u_1=:\omega$, we have 
\begin{equation*}
	FV = VF = \omega.
\end{equation*}
Due to the functoriality of $\m_r$, we moreover have a
$\Lambda_{\s_{\infty}}$-linear action of 
$\H^*$
on each of (\ref{UncompletedDieudonneLimit})
which commutes with $F$, $V$, and $\Gamma$.

\begin{theorem}\label{MainThmCrystal}
	As in Proposition $\ref{NormalizationCoh}$, set $d':=\sum_{k=3}^p \dim_{\F_p} S_k(N;\F_p)^{\ord}$. 
	Then $\m_{\infty}$ $($respectively $\m_{\infty}^{\star}$ for $\star=\et,\mult$$)$
	is a free $\Lambda_{\s_{\infty}}$-module of rank $2d'$ $($respectively $d'$$)$
	and there is a canonical short exact sequence of $\Lambda_{\s_{\infty}}$-modules
	with linear $\H^*$-action and semi linear actions of $\Gamma$, $F$ and $V$
		\begin{equation}
		\xymatrix{
		0\ar[r] & {\m_{\infty}^{\et}} \ar[r] & {\m_{\infty}} \ar[r] & {\m_{\infty}^{\mult}} \ar[r] & 0
		}.\label{DieudonneLimitFil}
	\end{equation}
	Extension of scalars of $(\ref{DieudonneLimitFil})$
	along the quotient 
	$\Lambda_{\s_{\infty}}\twoheadrightarrow  \s_{\infty}[\Delta/\Delta_r]$
	recovers the exact sequence
	\begin{equation}
		\xymatrix{
			0\ar[r] & {\m_r(\G_r^{\et})\tens_{\s_r} \s_{\infty}} \ar[r] &
			{\m_r(\G_r)\tens_{\s_r} \s_{\infty}} \ar[r] &
			{\m_r(\G_r^{\mult})\tens_{\s_r} \s_{\infty}} \ar[r] & 0
		}.
	\end{equation}
	for each integer $r>0$, compatibly with $\H^*$, $\Gamma$, $F$, and $V$.	
\end{theorem}

\begin{proof}

	Since $\varphi$ is an automorphism
	of $\s_{\infty}$, pullback by $\varphi$ commutes with 
	projective limits of $\s_{\infty}$-modules.
	As the canonical $\s_{\infty}$-linear map $\varphi^*\Lambda_{\s_{\infty}}\rightarrow \Lambda_{\s_{\infty}}$  
	is an isomorphism of rings (even of $\s_{\infty}$-algebras), it therefore suffices
	to prove the assertions of Theorem \ref{MainThmCrystal} after pullback by $\varphi$,
	which will be more convenient due to the relation between
	$\varphi^*\m_r(\G_r)$ and the Dieudonn\'e crystal of $\G_r$.
	
	Pulling back (\ref{BTindLim}) by $\varphi$ gives a commutative diagram
	with exact rows
	\begin{equation}
	\begin{gathered}
		\xymatrix{
			0\ar[r] & {\varphi^*\m_r(\G_r^{\et})} \ar[r]\ar[d] & 
			{\varphi^*\m_r(\G_r)} \ar[r]\ar[d] & {\varphi^*\m_r(\G_r^{\mult})} 
			\ar[r]\ar[d] & 0\\
			0\ar[r] & {\varphi^*\m_{s}(\G_{s}^{\et})\tens_{\s_{s}} \s_r} \ar[r] & 
			{\varphi^*\m_r(\G_{s})\tens_{\s_{s}} \s_r} \ar[r] &
			 {\varphi^*\m_r(\G_{s}^{\mult})\tens_{\s_{s}} \s_r} \ar[r] & 0
		}
	\end{gathered}\label{BTindLimPB}	
	\end{equation}	
	and we apply Lemma \ref{Technical} with $A_r:=\s_r$, $I_r:=(u_r)$, 
	$B=\s_{\infty}$, and with $M_r$ each one of the
	terms in the top row of (\ref{BTindLimPB}).  
	The isomorphism (\ref{MrToDieudonneMap})
	of Proposition \ref{MrToHodge} 
	ensures, via Theorem \ref{MainDieudonne} (\ref{MainDieudonne1}), that the hypothesis (\ref{freehyp})
	is satisfied.
	
	Due to the functoriality of (\ref{MrToDieudonneMap}), 
	for any $r\ge s$, the mapping obtained from (\ref{BTindLimPB}) by reducing
	modulo $I_r$ is identified with the mapping on (\ref{DieudonneFiniteExact}) induced (via functoriality
	of $\D(\cdot)$) by $\o{\pr}_{r,s}$.
	As was shown in the proof of Theorem (\ref{MainDieudonne}), these mappings are surjective
	for all $r\ge s$, and we conclude that hypothesis (\ref{surjhyp}) holds as well. 
	Moreover,
	the vertical mappings of (\ref{BTindLimPB}) are then surjective by Nakayama's Lemma,
	so as in the proof of Theorems \ref{main} and \ref{MainDieudonne} (and keeping in mind
	that pullback by $\varphi$ commutes with projective limits of $\s_{\infty}$-modules), 
	we obtain, by applying $\otimes_{\s_r} \s_{\infty}$ to (\ref{BTindLimPB}), passing
	to projective limits, and pulling back by $(\varphi^{-1})^*$, 
	the short exact sequence (\ref{DieudonneLimitFil}).	
\end{proof}

\begin{remark}
	In the proof of Theorem \ref{MainThmCrystal}, we could have alternately applied Lemma \ref{Technical}
	with $A_r=\s_r$ and $I_r:=(E_r)$, appealing to the identifications (\ref{MrToHodgeMap})
	of Proposition \ref{MrToHodge} and (\ref{CollectedComparisonIsom}) of Proposition \ref{KeyComparison},
	and to Theorem \ref{main}.  
\end{remark}

The short exact sequence (\ref{DieudonneLimitFil}) is closely related to its $\Lambda_{\s_{\infty}}$-linear
dual.  In what follows, we write $\s_{\infty}':=\varinjlim_r \Z_p[\mu_N][\![ u_r]\!]$,
taken along the mappings $u_r\mapsto \varphi(u_{r+1})$; it is naturally a $\s_{\infty}$-algebra.

\begin{theorem}\label{CrystalDuality}
		Let $\mu:\Gamma\rightarrow \Lambda_{\s_{\infty}}^{\times}$ be the crossed homomorphism 
		given by $\mu(\gamma):=\frac{u_1}{\gamma u_1}\chi(\gamma) \langle \chi(\gamma)\rangle$.
		There is a canonical $\H^*$ and $\Gal(K_{\infty}'/K_0)$-equivariant isomorphism of 
		exact sequences of $\Lambda_{\s_{\infty}'}$-modules
		\begin{equation}
		\begin{gathered}
			\xymatrix{
			0\ar[r] & {\m_{\infty}^{\et}(\mu \langle a\rangle_N)_{\Lambda_{\s_{\infty}'}}} \ar[r]\ar[d]_-{\simeq} & 
			{\m_{\infty}(\mu \langle a\rangle_N)_{\Lambda_{\s_{\infty}'}}} \ar[r]\ar[d]_-{\simeq} & 
				{\m_{\infty}^{\mult}(\mu \langle a\rangle_N)_{\Lambda_{\s_{\infty}'}}} \ar[r]\ar[d]_-{\simeq} & 0\\
	0\ar[r] & {(\m_{\infty}^{\mult})_{\Lambda_{\s_{\infty}'}}^{\vee}} \ar[r] & 
				{(\m_{\infty})_{\Lambda_{\s_{\infty}'}}^{\vee}} \ar[r] & 
				{(\m_{\infty}^{\et})_{\Lambda_{\s_{\infty}'}}^{\vee}} \ar[r] & 0 
		}
		\end{gathered}
		\label{MinftyDuality}	
		\end{equation}
		that intertwines $F$ 
		$($respectively $V$$)$ on the top row with  
		$V^{\vee}$ $($respectively $F^{\vee}$$)$  on the bottom.
\end{theorem}

\begin{proof}
	We first claim that there is a natural isomorphism of $\s_{\infty}'[\Delta/\Delta_r]$-modules
	\begin{equation}
		\m_r(\G_r)(\mu\langle a\rangle_N)\otimes_{\s_r} \s_{\infty}' \simeq 
		\Hom_{\s_{\infty}'}(\m_r(\G_r)\otimes_{\s_r}\s_{\infty}', \s_{\infty}')
		\label{twistyisom}
	\end{equation}
	that is $\H^*$-equivariant and $\Gal(K_{\infty}'/K_0)$-compatible for the standard action 
	$\gamma\cdot f(m):=\gamma f(\gamma^{-1}m)$ on the right side, and that intertwines
	$F$ and $V$ with $V^{\vee}$ and $F^{\vee}$, respectively.
	Indeed, this follows immediately from the identifications 
	\begin{equation}
			{\m_r(\G_r)(\langle \chi \rangle\langle a\rangle_N)\tens_{\s_r} \s_{\infty}'} \simeq 
			{\m_r(\G_r')\tens_{\s_r} \s_{\infty}'=:\m_r(\G_r^{\vee})\tens_{\s_r}\s_{\infty}'} 
			\simeq {\m_r(\G_r)^{\vee}_{\s_{\infty}'}}
		\label{GrTwist}
	\end{equation}
	and the definition (Definition \ref{DualBTDef}) of duality in $\BT_{\s_r}^{\varphi,\Gamma}$; here, the
	first isomorphism in (\ref{GrTwist}) results from Proposition \ref{GdualTwist}
	and Theorem \ref{CaisLauMain} (\ref{BaseChangeIsom}), while the final 
	identification is due to Theorem \ref{CaisLauMain} (\ref{exequiv}). 
	The identification (\ref{twistyisom}) carries $F$ (respectively $V)$
	on its source to $V^{\vee}$ (respectively $F^{\vee}$) on its target due to the compatibility 
	of the functor $\m_r(\cdot)$ with duality (Theorem \ref{CaisLauMain} (\ref{exequiv})).
		
	From (\ref{twistyisom})
	we obtain a natural $\Gal(K_r'/K_0)$-compatible evaluation pairing of $\s_{\infty}'$-modules
	\begin{equation}
		\xymatrix{
			{\langle\cdot,\cdot\rangle_r: \m_r(\G_r)(\mu\langle a\rangle_N) \tens_{\s_r} \s_{\infty}' 
			\times \m_r(\G_r)\tens_{\s_r} \s_{\infty}'} \ar[r] & {\s_{\infty}'}
			}\label{crystalpairingdefs}
	\end{equation}
	with respect to which the natural action of $\H^*$ is self-adjoint, due to the
	fact that (\ref{GrTwist}) is $\H^*$-equivariant by Proposition \ref{GdualTwist}.
	Due to the compatibility with change in $r$ of the identification (\ref{GrprimeGr}) of Proposition \ref{GdualTwist}
	together with the definitions (\ref{pdivTowers}) of $\pr_{r,s}$ and 
	$\pr_{r,s}'$,
	the identification (\ref{GrTwist}) intertwines the map induced by $\Pic^0(\pr)$ on its source
	with the map induced by ${U_p^*}^{-1}\Alb(\ps)$ on its target. For $r\ge s$, we therefore have
	\begin{equation*}
		\langle \m_r(\rho_{r,s})x , \m_r(\rho_{r,s})y  \rangle_s = 
		\langle x, \m_r({U_p^*}^{s-r}\Pic^0(\pr)^{r-s}\Alb(\ps)^{r-s})y\rangle_r =
		\sum_{\delta\in \Delta_s/\Delta_r} \langle x, \delta^{-1} y \rangle_r, 
	\end{equation*}
	where the final equality follows from (\ref{PicAlbRelation}).
	Thus, the perfect pairings (\ref{crystalpairingdefs}) satisfy the compatibility condition 
	(\ref{pairingchangeinr}) of Lemma \ref{LambdaDuality} which, together with
	Theorem \ref{MainThmCrystal}, completes the proof.
\end{proof}

The $\Lambda_{\s_{\infty}}$-modules $\m_{\infty}^{\et}$ and $\m_{\infty}^{\mult}$ admit
canonical descents to $\Lambda$: 

\begin{theorem}\label{etmultdescent}
	There are canonical $\H^*$, $\Gamma$, $F$ and $V$-equivariant isomorphisms
	of $\Lambda_{\s_{\infty}}$-modules
	\begin{subequations}
	\begin{equation}
		\m_{\infty}^{\et} \simeq \D_{\infty}^{\et}\tens_{\Lambda} \Lambda_{\s_{\infty}},
	\end{equation}	
	intertwining $F$ $($respetcively $V$$)$ with 
	$F\otimes \varphi$ $($respectively $F^{-1}\otimes \omega\cdot \varphi^{-1}$$)$ and $\gamma\in \Gamma$
	with $\gamma\otimes\gamma$, and
	\begin{equation}
		\m_{\infty}^{\mult}\simeq \D_{\infty}^{\mult}\tens_{\Lambda} \Lambda_{\s_{\infty}},	
	\end{equation}
	\end{subequations}
	intertwing $F$ $($respectively $V$$)$ with $V^{-1} \otimes \omega \cdot\varphi$
	$($respectively $V\otimes\varphi^{-1}$$)$
	and $\gamma$ with $\gamma\otimes \chi(\gamma)^{-1} \gamma u_1/u_1$.
	In particular, $F$ $($respectively $V$) 
	acts invertibly on $\m_{\infty}^{\et}$ $($respectively $\m_{\infty}^{\mult}$$)$.
\end{theorem}

\begin{proof}
	We twist the identifications (\ref{EtMultSpecialIsoms}) of Proposition 
	\ref{EtaleMultDescription} to obtain natural isomorphisms 
	\begin{equation*}
		\xymatrix@C=40pt{
			{\m_r(\G_r^{\et})} \ar[r]^-{\simeq}_-{F^r \circ (\ref{EtMultSpecialIsoms})} & 
			{\D(\o{\G}_r^{\et})_{\Z_p}\otimes_{\Z_p} \s_r}
		}\qquad\text{and}\qquad
		\xymatrix@C=40pt{
			{\m_r(\G_r^{\mult})} \ar[r]^-{\simeq}_-{V^{-r} \circ (\ref{EtMultSpecialIsoms})} & 
			{\D(\o{\G}_r^{\mult})_{\Z_p}\otimes_{\Z_p} \s_r}
		}
	\end{equation*}
	that are $\H_r^*$-equivariant and, Thanks to \ref{EtMultSpecialIsomsBC}, 
	compatible with change in $r$ using the maps on source and target
	induced by $\pr_{r,s}$.  Passing to inverse limits and appealing to Lemma \ref{Technical}
	and (the proof of) Theorem \ref{MainDieudonne}, we deduce for $\star=\et,\mult$
	natural isomorphisms of $\Lambda_{\s_{\infty}}$-modules
	\begin{equation*}
		\m_{\infty}^{\star} \simeq \varprojlim_r \left( \D(\o{\G}_r^{\star})_{\Z_p}\otimes_{\Z_p} \s_{\infty}\right)
		\simeq \D_{\infty}^{\star}\otimes_{\Lambda} \Lambda_{\s_{\infty}}
	\end{equation*}
	that are $\H^*$-equivariant and satisfy the asserted compatibility with respect to
	Frobenius, Verscheibung, and the action of $\Gamma$ due to 
	Proposition \ref{EtaleMultDescription} and the definitions (\ref{MrEtDef})--(\ref{MrMultDef}).
\end{proof}

We can now prove Theorem \ref{MinftySpecialize}, which asserts that
the slope filtration (\ref{MinftySpecialize}) of $\m_{\infty}$
specializes, on the one hand, to the slope filtration (\ref{DieudonneInfiniteExact})
of $\D_{\infty}$, and on the other hand to the Hodge filtration (\ref{LambdaHodgeFilnomup})
(in the opposite direction!) of ${e^*}'H^1_{\dR}$.  We recall the precise statement:

\begin{theorem}\label{SRecovery}
	Let $\tau:\Lambda_{\s_{\infty}}\twoheadrightarrow \Lambda$ be the $\Lambda$-algebra
	surjection induced by $u_r\mapsto 0$.  
	There is a canonical $\Gamma$ and $\H^*$-equivariant 
	isomorphism of split exact sequences of finite free $\Lambda$-modules
	\begin{equation}
	\begin{gathered}
		\xymatrix{
			0 \ar[r] & {\m_{\infty}^{\et}\tens_{\Lambda_{\s_{\infty}},\tau} \Lambda}\ar[d]_-{\simeq} \ar[r] &
			{\m_{\infty}\tens_{\Lambda_{\s_{\infty}},\tau} \Lambda}\ar[r] \ar[d]_-{\simeq}&
			{\m_{\infty}^{\mult}\tens_{\Lambda_{\s_{\infty}},\tau} \Lambda} \ar[r]\ar[d]_-{\simeq} & 0\\
			0 \ar[r] & {\D_{\infty}^{\et}} \ar[r] & {\D_{\infty}} \ar[r] &
			{\D_{\infty}^{\mult}} \ar[r] & 0
		}
	\end{gathered}\label{OrdFilSpecialize}
	\end{equation}
	which carries $F\otimes 1$ to $F$ and $V\otimes 1$ to $V$.
	
	Let $\theta\circ\varphi:\Lambda_{\s_{\infty}}\rightarrow \Lambda_{R_{\infty}}$
	be the $\Lambda$-algebra surjection induced by $u_r\mapsto (\varepsilon^{(r)})^p-1$.
	There is a canonical $\Gamma$ and $\H^*$-equivariant
	isomorphism of split exact sequences of finite free $\Lambda_{R_{\infty}}$-modules
	\begin{equation}
	\begin{gathered}
		\xymatrix{
				0 \ar[r] & {\m_{\infty}^{\et}\tens_{\Lambda_{\s_{\infty}},\theta\varphi} \Lambda_{R_{\infty}}}
				\ar[d]_-{\simeq} \ar[r] &
			{\m_{\infty}\tens_{\Lambda_{\s_{\infty}},\theta\varphi} \Lambda_{R_{\infty}}}\ar[r] \ar[d]_-{\simeq}&
			{\m_{\infty}^{\mult}\tens_{\Lambda_{\s_{\infty}},\theta\varphi} \Lambda_{R_{\infty}}} \ar[r]\ar[d]_-{\simeq} & 0\\
		0 \ar[r] & {{e^*}'H^1(\O)} \ar[r]_{i} & 
		{{e^*}'H^1_{\dR}} \ar[r]_-{j} & {{e^*}'H^0(\omega)} \ar[r] & 0 
		}
	\end{gathered}
	\end{equation}
	where $i$ and $j$ are the canonical sections 
	given by the splitting in Theorem $\ref{dRtoDieudonne}$.
\end{theorem}

\begin{proof}
	To prove the first assertion, we apply Lemma \ref{Technical} 
	with $A_r=\s_r,$ $I_r=(u_r)$, $B=\s_{\infty}$, $B'=\Z_p$ (viewed as 
	a $B$-algebra via $\tau$) and $M_r=\m_r^{\star}$ for $\star\in \{\et,\mult,\Null\}$.
	Thanks to (\ref{MrToDieudonneMap}) in the case $G=\G_r$,
	we have a canonical identification $\o{M}_r:=M_r/I_rM_r  \simeq \D(\o{\G}_r^{\star})_{\Z_p}$
	that is compatible with change in $r$ in the sense that the induced projective
	system $\{\o{M}_r\}_{r}$ is identified with that of Definition \ref{DinftyDef}.
	It follows from this and 
	Theorem \ref{MainDieudonne} (\ref{MainDieudonne1})--(\ref{MainDieudonne2}) that
	the hypotheses (\ref{freehyp})--(\ref{surjhyp}) are satisfied,
	and (\ref{OrdFilSpecialize}) is an isomorphism by Lemma \ref{Technical} (\ref{CompletedBaseChange}).

	In exactly the same manner,
	the second assertion follows by appealing to 
	Lemma \ref{Technical} with $A_r=\s_r$, $I_r=(E_r)$, $B=\s_{\infty}$, $B'=R_{\infty}$
	(viewed as a $B$-algebra via $\theta\circ\varphi$)
	and $M_r=\m_r^{\star}$, using (\ref{MrToHodgeMap}) and
	Theorem \ref{main} to verify the hypotheses (\ref{freehyp})--(\ref{surjhyp}).
\end{proof}

\begin{proof}[Proof of Theorem $\ref{RecoverEtale}$ and Corollary $\ref{HidasThm}$]
	Applying Theorem \ref{comparison} to (the connected-\'etale sequence of) $\G_r$
	gives a natural isomorphism of short exact sequences 
	\begin{equation}
	\begin{gathered}
		\xymatrix{
			0 \ar[r] &{\m_r(\G_r^{\et})\tens_{\s_r,\varphi} \a_r } \ar[r]\ar[d]^-{\simeq} & 
			{\m_r(\G_r)\tens_{\s_r,\varphi} \a_r} \ar[r]\ar[d]^-{\simeq} & 
			{\m_r(\G_r^{\mult})\tens_{\s_r,\varphi} \a_r} \ar[r]\ar[d]^-{\simeq} & 0 \\
			0 \ar[r] & {H^1_{\et}(\G_r^{\et})\tens_{\Z_p} \a_r} \ar[r] &
			{H^1_{\et}(\G_r)\tens_{\Z_p} \a_r} \ar[r] & 
			{H^1_{\et}(\G_r^{\mult})\tens_{\Z_p}\a_r}\ar[r] & 0
		}
	\end{gathered}	
		\label{etalecompdiag}
	\end{equation}
	Due to Theorem \ref{MainThmCrystal}, the terms in the top row of \ref{etalecompdiag}
	are free of ranks $d'$, $2d'$, and $d'$ over $\wt{\a}_r[\Delta/\Delta_r]$, respectively,
	so we conclude from Lemma \ref{fflatfreedescent} (with $A=\Z_p[\Delta/\Delta_r]$
	and $B=\a_r[\Delta/\Delta_r]$) that $H^1_{\et}(\G_r^{\star})$ is a free $\Z_p[\Delta/\Delta_r]$-module
	of rank $d'$ for $\star=\{\et,\mult\}$ and that $H^1_{\et}(\G_r)$ is free of rank $2d'$
	over $\Z_p[\Delta/\Delta_r]$.  Using the fact that 
	$\Z_p\rightarrow \a_r$ is faithfully flat, it then follows 
	from the surjectivity of the vertical maps in (\ref{BTindLimPB}) 
	(which was noted in the proof of Theorem \ref{MainThmCrystal})
	that the canonical trace mappings $H^1_{\et}(\G_r^{\star})\rightarrow H^1_{\et}(\G_{r'}^{\star})$
	for $\star\in \{\et,\mult,\Null\}$ are surjective for all $r\ge r'$.
	Applying Lemma \ref{Technical} with $A_r=\Z_p$, $M_r:=H^1_{\et}(\G_r^{\star})$,
	$I_r=(0)$, $B=\Z_p$ and $B'=\wt{\a}$, we conclude that $H^1_{\et}(\G_{\infty}^{\star})$
	is free of rank $d'$ (respectively $2d'$) over $\Lambda$ for $\star=\et,$ $\mult$ (respectively $\star=\Null$),
	that the specialization mappings 
	\begin{equation*}
	\xymatrix{
		{H^1_{\et}(\G_{\infty}^{\star})\tens_{\Lambda} \Z_p[\Delta/\Delta_r]} \ar[r] & 
		{H^1_{\et}(\G_r^{\star})}
		}
	\end{equation*}
	are isomorphisms, and that the canonical mappings for $\star\in \{\et,\mult,\Null\}$
	\begin{equation}
		\xymatrix{
			{H^1_{\et}(\G_{\infty}^{\star})\tens_{\Lambda} \Lambda_{\wt{\a}}} \ar[r] & 
			{\varprojlim_r \left(H^1_{\et}(\G_r^{\star})\tens_{\Z_p} \wt{\a}\right)}
		}\label{etaleswitcheroo}
	\end{equation}
	are isomorphisms.  Invoking the isomorphism (\ref{limitetaleseq})
	gives Corollary \ref{HidasThm}. By Lemma \ref{Technical} with $A_r=\s_r$, $M_r=\m_r(\G_r^{\star})$,
	$I_r=(0)$, $B=\s_{\infty}$ and $B'=\wt{\a}$, we similarly conclude from (the proof of) Theorem 
	\ref{MainThmCrystal} that the canonical mappings for $\star\in \{\et,\mult,\Null\}$
	\begin{equation}
		\xymatrix{
			{\m_{\infty}^{\star}\tens_{\s_{\infty},\varphi} \Lambda_{\wt{\a}}} \ar[r] & 
			{\varprojlim_r \left(\m_r(\G_r^{\star})\tens_{\s_r} \wt{\a}\right)}
		}\label{crystalswitcheroo}
	\end{equation}
	are isomorphisms.
	Applying $\otimes_{\a_r} \wt{\a}$ to the diagram (\ref{etalecompdiag}),
	passing to inverse limits, and using the isomorphisms
	(\ref{etaleswitcheroo}) and (\ref{crystalswitcheroo}) gives (again invoking (\ref{limitetaleseq}))
	the isomorphism (\ref{FinalComparisonIsom}).  
	Using the fact that the inclusion $\Z_p\hookrightarrow \wt{\a}^{\varphi=1}$
	is an equality, the isomorphism (\ref{RecoverEtaleIsom}) follows immediately from 
	(\ref{FinalComparisonIsom}) by taking $F\otimes\varphi$-invariants.	
\end{proof}

Using Theorems \ref{RecoverEtale} and \ref{CrystalDuality} we can give a new proof of
Ohta's duality theorem \cite[Theorem 4.3.1]{OhtaEichler} for the $\Lambda$-adic ordinary 
filtration of ${e^*}'H^1_{\et}$ (see Corollary \ref{OhtaDuality}):

\begin{theorem}\label{OhtaDualityText}
	There is a canonical $\Lambda$-bilinear and perfect duality pairing	
	\begin{equation}
		\langle \cdot,\cdot\rangle_{\Lambda}: {e^*}'H^1_{\et}\times {e^*}'H^1_{\et}\rightarrow \Lambda
		\quad\text{determined by}\quad
		\langle x,y\rangle_{\Lambda} \equiv \sum_{\delta\in \Delta/\Delta_r} 
		(x , w_r {U_p^*}^r\langle\delta^{-1}\rangle^*y)_r \delta \bmod I_r
		\label{EtaleDualityPairing}
	\end{equation}
	with respect to which the action of $\H^*$ is self-adjoint; here,
	$(\cdot,\cdot)_r$ is the usual cup-product pairing on $H^1_{\et,r}$ and 
	$I_r:=\ker(\Lambda\twoheadrightarrow \Z_p[\Delta/\Delta_r])$.  
	Writing $\nu:\scrG_{\Q_p}\rightarrow \H^*$ for the character 
	$\nu:=\chi\langle\chi\rangle \lambda(\langle p\rangle_N)$, the 
	pairing $(\ref{EtaleDualityPairing})$ induces a canonical
	$\scrG_{\Q_p}$ and $\H^*$-equivariant isomorphism of exact sequences 
	\begin{equation*}
		\xymatrix{
			0 \ar[r] & {({e^*}'H^1_{\et})^{\I}(\nu)}
			\ar[d]^-{\simeq} \ar[r] & 
			{{e^*}'H^1_{\et}(\nu)}\ar[d]^-{\simeq} \ar[r] &
			{({e^*}'H^1_{\et})_{\I}(\nu)}
			\ar[d]^-{\simeq}\ar[r] & 0 \\
			0 \ar[r] & {\Hom_{\Lambda}(({e^*}'H^1_{\et})_{\I},\Lambda)} \ar[r] & 
			{\Hom_{\Lambda}({e^*}'H^1_{\et},\Lambda)} \ar[r] &
			{\Hom_{\Lambda}(({e^*}'H^1_{\et})^{\I},\Lambda)}\ar[r] & 0
		}
	\end{equation*}
\end{theorem}

\begin{proof}
	The proof is similar to that of Proposition \ref{dRDuality}, using 
	Corollary \ref{HidasThm} and applying Lemma \ref{LambdaDuality} ({\em cf.} 
	the proof of \cite[Theorem 4.3.1]{OhtaEichler} and of \cite[Proposition 4.4]{SharifiConj}).    
	Alternatively, one can prove Theorem \ref{OhtaDualityText} by appealing to Theorem \ref{CrystalDuality}
	and isomorphism (\ref{RecoverEtaleIsom}) of Theorem \ref{RecoverEtale}.
\end{proof}

\begin{proof}[Proof of Theorem $\ref{SplittingCriterion}$]
	Suppose first that (\ref{DieudonneLimitFil}) admits a $\Lambda_{\s_{\infty}}$-linear
	splitting $\m_{\infty}^{\mult}\rightarrow \m_{\infty}$
	which is compatible with $F$, $V$, and $\Gamma$.
	Extending scalars along $\Lambda \rightarrow \Lambda_{\wt{\a}}\xrightarrow{\varphi}\Lambda_{\wt{\a}}$ 
	and taking $F\otimes\varphi$-invariants
	yields, by Theorem \ref{RecoverEtale}, a $\Lambda$-linear and $\scrG_{\Q_p}$-equivariant map
	$({e^*}'H^1_{\et})_{\I}\rightarrow {e^*}'H^1_{\et}$
	whose composition with the canonical projection 
	${e^*}'H^1_{\et}\twoheadrightarrow ({e^*}'H^1_{\et})_{\I}$
	is necessarily the identity.
	
	Conversely, suppose that the ordinary filtration of ${e^*}'H^1_{\et}$ is $\Lambda$-linearly
	and $\scrG_{\Q_p}$-equivariantly split.  Applying $\otimes_{\Lambda} \Z_p[\Delta/\Delta_r]$
	to this splitting gives, thanks to Corollary \ref{HidasThm} and the isomorphism
	(\ref{inertialinvariantsseq}), a $\Z_p[\scrG_{\Q_p}]$-linear splitting of
	\begin{equation*}
		\xymatrix{
			0 \ar[r] & {T_pG_r^{\mult}} \ar[r] & {T_pG_r} \ar[r] & {T_pG_r^{\et}}\ar[r] & 0
		}
	\end{equation*}
	which is compatible with change in $r$ by construction.
	By $\Gamma$-descent and Tate's theorem, there is a natural isomorphism
	\begin{equation*}
			{\Hom_{\pdiv_{R_r}^{\Gamma}}(\G_r^{\et},\G_r)}\simeq {\Hom_{\Z_p[\scrG_{\Q_p}]}(T_pG_r^{\et},T_pG_r)}
	\end{equation*}
	and we conclude that the connected-\'etale sequence of $\G_r$ is split (in the category 
	$\pdiv_{R_r}^{\Gamma}$), compatibly with change in $r$.  Due to the functoriality
	of $\m_r(\cdot)$, this in turn implies that
	the top row of (\ref{BTindLim}) is split in $\BT_{\s_r}^{\Gamma}$,
	compatibly with change in $r$, which is easily seen to imply the splitting of (\ref{DieudonneLimitFil}).
\end{proof}

\bibliographystyle{amsalpha_noMR}
\bibliography{mybib}
\end{document}